\documentclass[10pt]{article}
\usepackage{amsmath,amsthm}
\usepackage{amssymb}
\usepackage{amsfonts}
\usepackage{amscd}
 at 9truept
\font \smallrm=cmr10 at 9truept
\font \smallbf=cmbx10 at 9truept
\font \smallsl=cmsl10 at 9truept

\newcounter{amoi}
\newtheorem{thm}{Theorem}[subsection]
\newtheorem{theorem}[thm]{Theorem}

\newtheorem{corollary}[thm]{Corollary}
\newtheorem{lemma}[thm]{Lemma}
\newtheorem{proposition}[thm]{Proposition}
\newtheorem{definition}[thm]{Definition}

\theoremstyle{definition}
\newtheorem{example}[thm]{Example}
\newtheorem{examples}[thm]{Examples}
\newtheorem{remark}[thm]{Remark}
\newtheorem{remarks}[thm]{Remarks}

\newtheorem{notation}[thm]{Notation}
\newtheorem{free text}[thm]{}

\newcommand{\N} {\mathbb{N}}

\newcommand{\J} {{\mathfrak{J}}}

\begin{document}

{\ }

\vskip-57pt

   \centerline{\smallrm {\smallsl Journal of Noncommutative Geometry}  {\smallbf 9}  (2015), no.\  2, 287--358.   \ \ \ --- \ \ \  {\smallbf DOI:} 10.4171/JNCG/194}
 \vskip1pt

   \centerline{\smallrm {\smallsl The original publication is available at\/}
\  http://www.ems-ph.org/journals/journal.php?jrn=jncg}

\vskip27pt   {\ }

\centerline{\Large \bf DUALITY FUNCTORS FOR}
 \vskip9pt
\centerline{\Large \bf QUANTUM GROUPOIDS}

\vskip27pt

\centerline{ Sophie CHEMLA }
 \vskip3pt
   \centerline{\it UPMC Universit\'e Paris 06  ---  UMR 7586, Institut de math\'ematiques de Jussieu}
   \centerline{F-75005- Paris - FRANCE  \ / \
 e-mail:  {\tt schemla@math.jussieu.fr}}

\vskip11pt

   \centerline{ Fabio GAVARINI }
 \vskip3pt
   \centerline{\it Dipartimento di Matematica, Universit\`a di Roma ``Tor Vergata'' }
   \centerline{via della ricerca scientifica 1  ---  I-00133 Roma, ITALY  \ / \
 e-mail:  {\tt gavarini@mat.uniroma2.it}}

\vskip21pt

\begin{abstract}
 \vskip2pt
 We present a formal algebraic language to deal with quantum deformations of Lie-Rine\-hart algebras   --- or Lie algebroids, in a geometrical setting.  In particular, extending the ice-breaking ideas introduced by Xu in  \cite{Xu2},  we provide suitable notions of ``quantum groupoids''.  For these objects, we detail somewhat in depth the formalism of linear duality; this yields several fundamental antiequivalences among (the categories of) the two basic kinds of ``quantum groupoids''.  On the other hand, we develop a suitable version of a ``quantum duality principle'' for quantum groupoids, which extends the one for quantum groups --- dealing with Hopf algebras ---   originally introduced by Drinfeld  (cf.~\cite{Drinfeld}, \S 7)  and later detailed in  \cite{Gavarini}.
 \footnote{\ 2010 {\it MSC}\;: \, Primary 17B37, 16T05; Secondary 16T15, 17B62, 53D55.}
\end{abstract}

\bigskip

 \tableofcontents

\bigskip

\section{Introduction}

\smallskip

   {\ } \quad   The classical theory of Lie groups, or of algebraic groups, has a quantum counterpart in the theory of ``quantum groups''.  In Drinfeld's language, quantum groups are suitable topological Hopf algebras which are formal deformations either of the algebra of functions on a formal group, or of the universal enveloping algebras of a Lie algebra.  These deformations add further structure on the classical object: the formal group inherits a structure of  {\sl Poisson\/}  formal group, and the Lie algebra a structure of  {\sl Lie bialgebra}.  Linear duality for topological Hopf algebras reasonably adapts to quantum groups, lifting the analogous duality for their semiclassical limits.  On the other hand, Drinfeld revealed a more surprising feature of quantum groups, later named ``quantum duality'', which somehow lifts the Poisson duality among Poisson (formal) groups.  Namely, there exists an equivalence of categories between quantized enveloping algebras and quantized formal groups, which shifts from a quantization of a given Lie bialgebra  $ L $  to one of the  {\sl dual\/}  Lie bialgebra  $ L^* $.
                                                                                  \par
   Another extension of Lie group theory is that of Lie-Rinehart algebras (sometimes loosely called ``Lie algebroids''), developed by Rinehart, Huebschmann and others.  The notion of Lie-Rinehart algebra  $ \big( L \, , [\,\ , \ ] \, , \omega \big) $  over a commutative ring  $ A $  lies inbetween  $ A $--Lie  algebras and  $ k $--Lie  algebras of derivations of the form  $ \text{\sl Der}\,(A) \, $.  Well-known examples come from geometry, such as the global sections of a Lie algebroid, for example the  $ 1 $--forms  over a Poisson manifold  (cf.~\cite{Dazord and Sondaz}, \cite{Huebschmann1},  \cite{Evens and Lu and Weinstein}).
                                                                                  \par
   The natural algebraic gadgets attached with a Lie-Rinehart algebra are its universal enveloping algebra  $ V^\ell(L) $  and its algebra of jets  $ J^r(L) \, $,  which are in linear duality with each other.  Any Lie-Rinehart algebra  $ L $  can also be seen as a  {\sl right\/}  Lie-Rinehart algebra: thus one can also consider its right enveloping algebra, call it  $ V^r(L) \, $,  anti-isomorphic to  $ V^\ell(L) \, $, and its dual  $ J^\ell(L) \, $.
                                                                                  \par
   All these algebraic objects   ---  $ V^\ell(L) \, $,  $ V^r(L) \, $,  $ J^r(L) $  and  $ J^\ell(L) $  ---  are  {\sl (topological) bialgebroid}   ---  {\sl left\/}  ones when a superscript ``$ \, \ell \, $''  occurs, and  {\sl right\/}  when  ``$ \, r \, $''  does.  Indeed, they also have an additional property, about their Hopf-Galois map, such that these left/right bialgebroids are actually  {\sl left\/  {\rm or}  right Hopf left/right bialgebroids}   --- an important generalization of Hopf algebras.
                                                                                           \par

\smallskip

   Linear duality for (left/right) bialgebroids is twofold: any (left/right) bialgebroid  $ U $  is naturally a left  $ A $--module  and a right  $ A $--module,  thus one may consider its  {\sl left dual}  $ U_* $  as well as its  {\sl right dual}  $ U^* \, $.  Under mild conditions,  $ U^* $  and  $ U_* $  are naturally (right/left) bialgebroids (see  \cite{Kadison and Szlachanyi}).  The  $ \big( V^\ell(L) \, , J^r(L) \big) $  is tied together by such a linear duality, and similarly for  $ \big( V^r(L) \, , J^\ell(L) \big) \, $.

\smallskip

   When looking for quantizations of Lie-Rinehart algebras, one should consider formal deformations of either  $ V^{\ell/r}(L) $  or  $ J^{r/\ell}(L) \, $,  among left/right (topological) bialgebroids: these deformations automatically inherit from their semiclassical limits the additional property of being  {\sl left/right Hopf left/right bialgebroids}.  We shall loosely call such deformations ``quantum groupoids''.
                                                                                           \par
   The first step in this direction was made by Ping Xu  (cf.~\cite{Xu2}):  he introduced a notion of quantization of  $ V^\ell(L) \, $,  called quantum universal enveloping algebroid (LQUEAd in short).  Then he noticed that any such quantization endows  $ L $  with a richer structure of  {\sl Lie-Rinehart bialgebra}.  This is a direct extension of the notion of Lie bialgebra, in particular, it is a self-dual notion, so if  $ L $  is a Lie-Rinehart bialgebra then its dual  $ L^* $  is a Lie-Rinehart bialgebra as well (see  \cite{Kosmann}).  Finally, Xu also provided an example of construction of a non-trivial LQUEAd  $ {\mathcal D}_X[[h]]^{\mathcal F} \, $,  by ``twisting'' the trivial deformation  $ \, {\mathcal D}_X[[h]] \, $  of  $ \, {\mathcal D}_X := V^\ell\big(\Gamma(TX)\big) \, $,  where  $ X $  is a Poisson manifold.

\medskip

   The purpose of this paper is to move some further steps in the theory of ``quantum groupoids''.

\smallskip

   After recalling some basics of the theory of Lie-Rinehart algebras and bialgebras  (Sec.~\ref{L-R_algs-bialgs}),  we introduce also some basics of the theory of bialgebroids  (Sec.~\ref{l/r-bialgbds}):  in particular, we dwell on the relevant examples, i.e.~universal enveloping algebras and jet spaces for Lie-Rinehart algebras.

\smallskip

   Then we introduce ``quantum groupoids''  (Sec.~\ref{q-groupds}).  Besides Xu's original notion of LQUEAd, we introduce its right counterpart (in short RQUEAd): a topological right bialgebroid which is a formal deformation of some  $ V^r(L) \, $.  Similarly, we introduce quantizations of jet spaces;
a topological right bialgebroid which is a formal deformation of some  $ J^r(L) $  will be called a  {\sl right quantum formal series algebroid\/} (RQFSAd in short); similarly, the left-handed version of this notion gives rise to the definition of  {\sl left quantum formal series algebroid\/}  (LQFSAd in short).  Altogether, this gives us four kinds of quantum groupoids; each one of these induces a Lie-Rinehart {\sl bialgebra\/}  structures on the original Lie-Rinehart algebra one deals with, extending what happens with LQUEAd's.

\smallskip

   As a next step, we discuss linear duality for quantum groupoids  (Sec.~\ref{lin-dual_q-grpds}).  The natural language is that of linear duality for bialgebroids, with some precisions.  First, by infinite rank reasons we are lead to consider  {\sl topological\/}  duals.  Second, both left and right duals are available, thus taking duals might cause a proliferation of objects.  Nevertheless, we can keep this phenomenon under control, so eventually we can bound ourselves to deal with only a handful of duality functors.
                                                                           \par
  In the end, our main result on the subject claims the following: our duality functors provide (well-defined) anti-equivalences between the category of all LQUEAd's and the category of all RQFSAd's (on a same, fixed ring  $ A_h $),  and similarly also anti-equivalences between the category of all RQUEAd's and the category of all LQFSAd's (on  $ A_h $  again).  In addition, if one starts with a given quantum groupoid, which induces a specific (Lie-Rinehart) bialgebra structure on the underlying Lie-Rinehart algebra, then the dual quantization yields the same or the coopposite Lie-Rinehart bialgebra structure   --- see  Theorems \ref{dual_QUEAd's=QFSAd's}  and  \ref{dual_QFSAd's=QUEAd's}  for further precisions.

\smallskip

   Finally  (Sec.~\ref{Drinf-functs_q-duality}),  we develop a suitable ``Quantum Duality Principle'' for quantum groupoids.  Indeed, we introduce functors ``\`a la Drinfeld'', denoted by  $ {(\ )}^\vee $  and  $ {(\ )}' \, $,  which turns (L/R)QFSAd's into (L/R)QUEAd's and viceversa, so to provide an equivalence between the category of LQFSAd's and that of LQUEAd's, and a similar equivalence between RQFSAd's and RQUEAd's.  In addition, if one starts with a quantization of some Lie-Rinehart bialgebra  $ L \, $,  then the (appropriate) Drinfeld's functor gets out of it a quantization of the  {\sl dual\/}  Lie-Rinehart bialgebra  $ L^* $.
                                                                \par
  For the functor  $ {(\ )}^\vee $,  Drinfeld's original definition for quantum groups can be easily extended to quantum groupoids.  Instead, this is not the case for the functor  $ {(\ )}' $:  therefore we have resort to a different characterization (for quantum groups) of it, and adopt that as a definition (for quantum groupoids): this requires linear duality, which sets a strong link with the first part of the paper.

\smallskip

   It is worth remarking that linear duality for quantum groupoids interchanges ``left'' and ``right''; instead, quantum duality takes either one to itself: at the end of the day, this means that if one aims to have both linear duality and quantum duality then he/she is forced to deal with  {\sl all four types of quantum groupoids\/}  that we introduced   --- none of them can be left apart.

\medskip

   At the end  (Sec.~\ref{example_twistor})  we present an example, just to illustrate some of our main results on a single   --- and simple, yet significant enough ---   toy model.

\vskip15pt

   \centerline{ACKNOWLEDGEMENTS}
 \vskip5pt
   \centerline{The authors thank Niels Kowalzig for his valuable comments and hints.}

\bigskip

\section{Lie-Rinehart algebras and bialgebras}  \label{L-R_algs-bialgs}

\smallskip

   {\ } \quad   Throughout this paper,  $ k $  will be a field and  $ A $  will be a unital, associative  $ k $--algebra;  we assume  $ k $  to have characteristic zero (though for most definitions and constructions this is not necessary).  Moreover, for all objects defined in  {\sl this\/}  section we assume in addition that  $ A $  is also commutative.

\smallskip

 \subsection{Lie-Rinehart algebras}  \label{L-R_algebras}

\smallskip

   {\ } \quad   To begin with, we introduce the notion of  {\sl (left) Lie-Rinehart algebra\/}  ( or ``Lie algebroid'').

\smallskip

\begin{definition}
A  {\sl (left) Lie-Rinehart algebra}  (see  \cite{Rinehart})  is a triple  $ (A,L,\omega) $  where:
  $ L $  is a  $ k $--Lie  algebra,
  $ L $  is an  $ A $--module,
  and  $ \omega $  is an  $ A $--linear  morphism of Lie  $ k $--algebras from  $ L $  to  $ \text{\sl Der}(A) $,  called  {\sl anchor (map)},  such that the following compatibility relation holds:
  $$  \forall \;\, D, D' \in L \, , \;\; \forall \; f \in A \, ,  \qquad
\big[ D, f \, D' \big] \; = \; \omega(D)(f) \, D' + f \, \big[ D, D' \big]  $$
In particular, if  $ L $  is finitely generated projective as an  $ A $--module,  then  $ (A,L,\omega) $  will be called a finite projective Lie-Rinehart algebra.
 \vskip4pt
   {\sl  $ \underline{\text{Notation}} $:}  \, when there is no ambiguity, the Lie-Rinehart algebra  $ (A,L,\omega) $  will be written  $ L \, $.
\end{definition}

\smallskip

\begin{examples}  \label{examples L-R_alg}   Any Lie algebra over  $ A $  is a Lie-Rinehart algebra whose anchor map  $ \omega $  is  $ 0 $  (and conversely).  On the other hand,  $ \big( A, Der(A), \text{\sl id} \big) $  is a Lie-Rinehart algebra too.  Another example is the  $ A $--module  $ D_A $  of K{\"a}hler differentials on any Poisson algebra  $ A $  (see \cite{Huebschmann1}).
                                                                        \par
   In the setup of differential geometry, natural examples of Lie-Rinehart algebras arise as spaces of global sections of Lie algebroids; for instance, such an example (of Lie algebroid) is given by the vector bundle  $ \Omega_P^1 $  of differential forms of degree 1 on a Poisson manifold  $ P \, $  (see, e.g., \cite{Dazord and Sondaz}).
\end{examples}

\smallskip

\begin{free text}
 {\bf Differentials for Lie-Rinehart algebras.}     Given a finite projective Lie-Rinehart algebra  $ (A,L,\omega) \, $,  it is known that  $ \, \bigwedge_A L^* \, = \, \oplus_n \bigwedge_A^n L^* \, $  admits a differential  $ d_L $  that makes it into a differential algebra. Here  $ \; d_L : \bigwedge_A^n L^* \longrightarrow \bigwedge_A^{n+1} L^* \; $  is defined as follows: for all  $ \, \lambda \in \bigwedge_A^n L^* \, $ and for   all $ \, (X_1,X_2,\dots,X_{n+1}) \in L^{n+1} \, $,  one has
  $$  \begin{array}{rcl}
   (d_L \lambda) (X_1, \dots, X_{n+1})  &  =  &  \sum_{i=1}^{n+1} {(-1)}^{i+1}
\omega(X_{i}) \big( \lambda (X_1, \dots, \widehat{X_{i}}, \dots, X_{n+1} \big) \;\; +  \\
                                        &  +  &  \sum_{i<j} {(-1)}^{i+j} \lambda \big( [X_i, X_j], X_1, \dots, \widehat{X_i}, \dots, \widehat{X_j}, \dots, X_{n+1} \big)
      \end{array}  $$
In the case where  $ \, L = TX \, $,  the differential  $ d_L $  coincides with the de Rham differential.
\end{free text}

\smallskip

\begin{definition}  \label{def_V^ell(L)}
 Let  $ (A,L,\omega) $  be a (left) Lie-Rinehart algebra.  The  {\sl (left) universal enveloping algebra}  of  $ L $  is the  $ k $--algebra
 $ \; V^\ell(L) := T_k^+(A \oplus L) \Big/ I \; $
 where  $ \, T_k^+(A \oplus L) $  is the positive part of the tensor  $ k $--algebra  over  $ \, A \oplus L \, $  and  $ \, I $  is the two-sided ideal in  $ T_k^+(A \oplus L) $  generated by the elements
  $$  a \otimes b - a \, b \; ,  \;\;  a \otimes \xi - a \, \xi \; ,
 \;\;  \xi \otimes \eta -\eta \otimes \xi - [\xi,\eta] \; ,
 \;\;  \xi \otimes a -a \otimes \xi -\omega (\xi)(a)   \eqno \quad \forall \; a, b \in A \, , \, \xi , \eta \in L  $$
\end{definition}

\smallskip

\begin{remarks}  \label{canon-filtr_V^ell(L)}
   {\it (a)}\,  Note that  $ V^\ell(L) $  is a filtered ring, its (increasing) filtration  $ {\big\{ V^\ell_n(L) \big\}}_{n \in \N} $  being defined by  $ \; V^\ell_0(L) := A \, $,  $ \; V^\ell_{n+1}(L) := V^\ell_n(L) + V^\ell_n(L) \cdot L \; $  ($ \, n \in \N \, $).  We denote by  $ \text{\sl Gr}\big(V^\ell(L)\big) $  the associated graded algebra.  It is known  (cf.~\cite{Rinehart})  that if  $ L $  is  projective as an  $ A $--module,  then  $ \; \text{\sl Gr}\big(V^\ell(L)\big) \cong S_A(L) \; $. Moreover,  $ \; \iota_A : A \longrightarrow V^\ell(L) \; $  and  $ \; \iota_L : L \longrightarrow V^\ell(L) \; $  are monomorphisms.
 \vskip3pt
   {\it (b)}\,  The Lie-Rinehart algebras  $ \, L = \big( L , A , [\ ,\ ] , \omega \big) \, $  and  $ \, L^{\text{\it op}} := \big( L , A , -[\ ,\ ] , -\omega \big) \, $  are isomorphic via the isomorphims  $ F $  defined by  $ \; F(D) := -D \; $  for all  $ \, D \in L \, $  and  $ \; F(a) := a \; $  for all  $ \, a \in A \, $  .
 \vskip3pt
   {\it (c)}\,  If  $ X $  is a (smooth) manifold and  $ \, A = \mathcal{C}^\infty(X) \, $,  then  $ V^\ell\big(\text{\sl Der}(A)\big) $  for the Lie-Rinehart algebra  $ \, \big( A, \text{\sl Der}(A), \text{\sl id} \big) \, $  is the  $k$--algebra  of global differential operators on  $ X \, $.
\end{remarks}

\medskip

\begin{free text}  \label{L-R_proj->free}
 {\bf From a finite projective Lie-Rinehart algebra to a free Lie-Rinehart algebra.}  Most of the time, we will work with finite projective Lie-Rinehart algebras.  This is a reasonable hypothesis as Lie-Rinehart algebras coming from the geometry are finite projective.  Several times in this article, we will prove results for (finite) free Lie-Rinehart algebra and then extend them to finite projective Lie-Rinehart algebras.  We now explain the key step for this.
                                                              \par
   Let  $ L $  be a finite projective Lie-Rinehart algebra. There exist a finite projective  $ A $--module  $ Q $  such that  $ \, F = L \oplus Q \, $  is a finite rank free  $ A $--module.  We can endow  $ F $  with the following Lie-Rinehart algebra structure:
 for all  $ \, D \, , D_1 \, , D_2 \in L \, $,  $ \, E \, , E_1 \, , E_2 \in Q \, $,  we set
  $$  \omega_F(D+E) \; := \; \omega_L (D) \;\; ,  \qquad  [D_1 + E_1 \, , D_2 + E_2] \; := \; [D_1 \, , D_2] $$
that is, the structure of  $ L $  is extended trivially to  $ \, F = L \oplus Q \, $.  Then  $ \, V^\ell(F) = V^\ell(L) \otimes_A S(Q) \, $.
 \vskip5pt
   The  $ A $--module  $ \; L_{Q} := L \oplus Q \oplus L \oplus Q \oplus \cdots = F \oplus F \oplus F \oplus \cdots \; $  is a free  $ A $--module.  Define  $ \, R = Q \oplus L \oplus Q \oplus L \oplus \cdots \, $;  then  $ R $  is a free  $ A $--module  such that  $ \, L_Q = L \oplus R \, $  is a free  $ A $--module  (cf.~\cite{Hilton and Stammbach}).  We set on  $ L_Q $  the Lie-Rinehart algebra structure (for all  $ \, D , D_1 , D_2 \in L \, , \, B , B_1 , B_2 \in R \, $)
  $$  \omega_{L_Q}(D+B) \; := \; \omega_L (D) \;\; ,  \qquad  [D_1 + B_1 \, , D_2 + B_2] \; := \; [D_1 \, , D_2]  $$
in other words, the Lie-Rinehart algebra structure of  $ L $  is extended trivially to  $ \, L_Q = L \oplus R \, $.

\smallskip

\begin{definition}  \label{def_good-basis}
  Let an  $ A $--basis  $ \, \{b_1,\dots,b_n\} $  of  $ F $  be given.  Then one can construct a basis  $ \, {\{v_t\}}_{t \in T} \, $  of  $ R $  and an  $ A $--basis  $ \, {\{e_t\}}_{t \in T} \, $  of  $ \, L_Q $  both indexed by  $ \, T := \N \times \{1,\dots,n\} \, $.  Any such basis for  $ L_Q $  will be called a  {\sl good basis}.  For later use, if  $ \, i = (i_1,i_2) \in T \, $  we set  $ \, \varpi(e_i) := i_1 \, $.
\end{definition}

\smallskip

   We set  $ \; Y \! = \oplus_{i=1}^n k b_i \; $  and  $ \; Z = {\displaystyle \oplus_{t \in T} k \, v_t \simeq Y \oplus Y\oplus Y \cdots} \;\, $  so that  $ \, F = A \otimes_k Y \, $  and  $ \; R = A \otimes_k Z \; $.  We have then  $ \; V^\ell(L_Q) = V^\ell(L) \otimes_k S(Z) \; $.
\end{free text}

\smallskip

\begin{free text}

 {\bf Right Lie-Rinehart algebras.}  For the sake of completeness, we have to mention that one can also, in a symmetric way, consider the notion of  {\sl right\/}  Lie-Rinehart algebra, as follows:

\smallskip

\begin{definition}
 A  {\sl right Lie-Rinehart algebra}   is a triple  $ (A,L,\omega ) $  where
%  £
  $ L $  is a  $ k $--Lie  algebra,
  $ L $  is a right  $ A $--module,
  and  $ \omega $  is an  $ A $--linear  morphism of Lie  $ k $--algebras from  $ L $  to  $ \text{\sl Der}(A) $,  called  {\sl anchor (map)},  such that the following compatibility relation holds:
  $$  \forall \;\, D, D' \in L \, , \;\; \forall \; f \in A \, ,  \qquad
\big[ D,  D'\cdot f \big] \; = \; D' \cdot \omega(D)(f)  +  \big[ D, D' \big] \cdot f $$
\end{definition}

\smallskip

\begin{remark}
   As  $ A $  is commutative, a Lie-Rinehart algebra can be considered as a right Lie-Rinehart algebra and viceversa.  However, the enveloping algebra defined by the notion of right Lie-Rinehart algebra is different from that defined by a (left) Lie-Rinehart algebra.
\end{remark}

\smallskip

\begin{definition}  \label{def_V^r(L)}
 Let  $ (A,L,\omega) $  be a right Lie-Rinehart algebra.  The  {\sl (right) universal enveloping algebra}  of  $ L $  is the  $ k $--algebra
 $ \; V^r(L) := T_k^+(A \oplus L) \Big/ I \; $
 where  $ \, T_k^+(A \oplus L) $  is the positive part of  the tensor  $ k $--algebra  over  $ \, A \oplus L \, $  and  $ \, I $  is the two-sided ideal in  $ T_k^+(A \oplus L) $  generated by the elements
  $$  a \otimes b - a \, b \; ,  \;\;  \xi \otimes a  -  \xi \cdot a \; ,
 \;\;  \xi \otimes \eta -\eta \otimes \xi - [\xi,\eta] \; ,
 \;\;  \xi \otimes a -a \otimes \xi -\omega (\xi)(a)   \eqno \quad \forall \; a, b \in A \, , \, \xi , \eta \in L  $$
\end{definition}

\smallskip

   Next result clarifies the link between left and right enveloping algebras of a single Lie-Rinehart algebra  $ L \, $.  Hereafter,  $ L^{\text{\it op}} $  denotes the ``opposite'' Lie-Rinehart algebra to  $ L $   --- cf.~Remarks \ref{canon-filtr_V^ell(L)}  ---   while  $ \mathfrak{A}^{\text{\it op}} $  denotes the opposite of any (associative) algebra  $ \mathfrak{A} \, $.
\end{free text}

\bigskip

\begin{proposition}  \label{alg_anti-isom_Xi}
 For any Lie-Rinehart algebra  $ L \, $,  the algebras  $ {V^r(L)}^{op} $  and  $ V^\ell(L^{op}) $  are equal, and there is an algebra isomorphism  $ \; \Xi : V^\ell(L) \relbar\joinrel\longrightarrow \! {V^r(L)}^{\text{\it op}} \, $,  $ \, a \mapsto a \, $,  $ \, D \mapsto -D \; $  (for all  $ \, a \in \! A \, $,  $ D \in \! L \, $).
\end{proposition}

\smallskip

 \subsection{Lie-Rinehart bialgebras}  \label{L-R_bialgebras}

\smallskip

   {\ } \quad   We are now ready to introduce the notion of  {\sl Lie-Rinehart bialgebra\/} (cf.~\cite{Mackenzie and Xu}, \cite{Kosmann}, \cite{Huebschmann2}).

\smallskip

\begin{definition}  \label{def_L-R_bialgebras}
 A  {\sl Lie-Rinehart bialgebra}  is a pair  $ \big( L_1 , L_2 \big) $  of finitely generated projective  $ A $--modules  in duality   --- that is,  $ \, L_1 \cong L_2^{\,*} \, $  and  $ \, L_2 \cong L_1^{\,*} \, $ ---   each of them being endowed with Lie-Rinehart algebra structures such that the differential  $ d_1 $  on  $ \bigwedge_A \! L_1 $  arising from the Lie-Rinehart structure on  $ \, L_2 \cong L_1^* \, $    is a derivation of the Lie bracket of   $  L_1 \, $,  that is
  $$  d_1\!\big([X,Y]\big)  \; = \;  \big[ d_1(X) \, , Y \big] + \big[ X , d_1(Y) \big]   \eqno \text{for all}  \quad  X, Y \in L_1 \; .   \qquad  $$
   \indent   In general, if  $ L $  is a finitely generated projective  $ A $--module,  then its linear dual $ L^* $  (as an  $ A $--module)  is finitely generated projective as well: in this case, in the following we shall say that  ``$ L $  is a Lie-Rinehart bialgebra''  to mean that   $ \big( L , L^* \big) $  has a structure of Lie-Rinehart bialgebra, and we shall denote the differential of  $ \, \bigwedge_A \! L $  mentioned above by  $ \, d_{L^*} $  or  $ \delta_L \, $.
\end{definition}

\smallskip

\begin{remarks}  \label{props_Lie-bialg}  {\ }
 \vskip3pt
   {\it (a)}\,  The conditions of  Definition \ref{def_L-R_bialgebras}  do not change if one switches  $ L_1 $  and  $ L_2 $  (cf.~\cite{Kosmann}).
 \vskip3pt
   {\it (b)}\,  It follows from the definition that the differential  $ \delta_L $  of  $ L $  in a Lie-Rinehart bialgebra  $ \big( L, L^* \big) $  is uniquely determined by its restriction to  $ A $  and  $ L $   --- the degree 0 and degree 1 pieces of  $ \, \bigwedge_A \! L \; $.
 \vskip3pt
    {\it (c)}\,  If  $ \big( L, L^* \big) $  is a Lie-Rinehart bialgebra, we can read off the explicit relation between the Lie-Rinehart structure of  $ L^* $   --- its anchor map  $ \omega_{L^*} $  and its Lie bracket  $ {[\,\ ,\ ]}_{L^*} $  ---   and the differential  $ \delta_L $  of  $ L $  as follows: if  $ \, D^*, E^* \in L^* \, $,  $ \, X \in L \, $,  $ \, a \in A \, $,  then
 $ \; \omega_{L^*}(D^*)(a) = \big\langle\, \delta_L(a) \, , D^* \big\rangle \; $,
 $ \; \big\langle\, X \, , {[D^*, E^*]}_{L^*} \big\rangle  \, + \,  \big\langle\, \delta_L(X) \, , \, D^* \!\wedge E^* \big\rangle  \; = \;  \omega_{L^*}(D^*) \big( \langle\, X \, , E^* \rangle \big)  \, - \,  \omega_{L^*}(E^*)\big( \langle\, X \, , D^* \rangle \big) \; $,  \,
 where  $ \, \langle\,\ ,\ \rangle \, $  denotes the natural pairing between  $ L $  and  $ L^* \, $.  Indeed, one can use these formulas either to deduce  $ \omega_{L^*} $  and  $ {[\,\ ,\ ]}_{L^*} $  from  $ \delta_L \, $,  or to deduce the latter from  $ \omega_{L^*} $  and  $ {[\,\ ,\ ]}_{L^*} $
 \vskip3pt

    {\it (d)}\,  Let  $ \big( L \, , L^* \big) $  be a Lie-Rinehart bialgebra.  Denote by  $ d $  the differential on  $ \bigwedge_A \! L^* $  arising from the Lie-Rinehart structure on  $ L $  and  $ \, d_* \, (= \delta_L) \, $  the differential on  $ \bigwedge_A \! L $  coming from the Lie-Rinehart structure  on  $ L^* \, $.  Then  $ A $  inherits a Poisson algebra structure by
 $ \; \{f,g\} \, := \, \big\langle\, df , d_*g \big\rangle \; $  for all  $ \, f \, , g \in A \; $.
 (see   \cite{Kosmann}, \cite{Xu1});  moreover, one has  $ \; [\, df , dg \,] \, = \, d\{f,g\} \; $  and  $ \; d_*\{f,g\} = \, - [\, d_*f , d_*g \,] \; $.
 \vskip3pt
    {\it (e)}\,  If  $ \big( L \, , L^* \big) $  is a Lie-Rinehart bialgebra, then  $ \big( L^{\text{\it op}} , L^* \big) \, $,  $ \big( L \, , {(L^*)}^{\text{\it op}} \big) $  and  $ \big( L^{\text{\it op}} , {(L^*)}^{\text{\it op}} \big) $  are Lie-Rinehart bialgebras too.  Identifying any Lie-Rinehart bialgebra, as a pair, with the left-hand of the pair, say  $ \, L \equiv \big( L \, , L^* \big) \, $,  we write  $ \, L^{\text{\it op}} \equiv \big( L^{\text{\it op}} , L^* \big) \, $   --- the ``opposite'' to  $ L $  ---   $ \, L_{\text{\it coop}} \equiv \big( L \, , {(L^*)}^{\text{\it op}} \big) \, $   --- the ``coopposite'' ---   and  $ \, L^{\text{\it op}}_{\text{\it coop}} \equiv \big( L^{\text{\it op}} , {(L^*)}^{\text{\it op}} \big) \, $   --- the ``opposite-coopposite''.
 \vskip3pt
    {\it (f)}\,  By means of the so called  {\sl  $ r $--matrices\/}  one can introduce the class of  {\sl coboundary\/}  Lie-Rinehart bialgebras, and among them the triangular ones: we refer the interested readers to  \cite{Xu2}.
\end{remarks}

\bigskip

\section{Left and right bialgebroids}  \label{l/r-bialgbds}

\smallskip

   {\ } \quad   Let again  $ k $  be a field, and  $ A $  a unital, associative  $ k $--algebra.  We define  $ \, A^e := A \otimes_k A^{\text{\it op}} \; $.

\smallskip

 \subsection{$ A $--rings,  $ A $--corings}  \label{A-rings A-corings}

\smallskip

   {\ } \quad   We begin this section introducing the notions of  {\sl  $ A $--ring\/}  and  {\sl  $ A $--coring},  which are direct generalizations of the notions of algebra and coalgebra over a commutative ring.

\smallskip

\begin{definition}  \label{def_A-ring}
 Let  $ A $  be a  $ k $--algebra as above.  An  $ A $--ring  is a triple  $ (H , m_H , \iota) $  where $ H $  is an  $ A^e $--module,  $ \; m_H : H \otimes_A H \longrightarrow H \; $  and  $ \; \iota : A \longrightarrow H \; $  are  $ A^e $--module  morphisms such that
  $$  m_H \circ (m_H \otimes \text{\sl id}_H) \; = \; m_H \circ (\text{\sl id}_H \otimes m_H)  \quad , \qquad  m_H \circ (\iota \otimes \text{\sl id}_H) \; = \; id_H \, = \; m_H \circ (\text{\sl id}_H \otimes \iota)  $$
where in the second set of identities we make the identifications  $ \, H\otimes_A A \cong H \, $  and  $ \, A \otimes_A H \cong H \; $.
\end{definition}

\smallskip

   It is well known (see  \cite{Bohm2})  that  $ A $--rings  $ H $  correspond bijectively to  $ k $--algebra  homomorphisms  $ \, \iota : A \longrightarrow H \; $.  With this characterization, the  $ A^e $--module  structure on  $ H $  can be expressed as follows:
%
%   $$  a.h.b  \; = \; \iota(a) \, h \, \iota(b)   \eqno  \forall \;\; a, b \in A \; , \,\; h \in H
% \qquad $$
%
 $ \; a\cdot h\cdot b \, = \iota(a) \, h \, \iota(b) \; $  for  $ \, a, b \in A \, $,  $ \, h \in H \, $.  \,The  {\sl dual\/}  notion (``$ A $--coring'')  is the following:

\smallskip

\begin{definition}  \label{def_A-coring}
 An  $ A $--coring  is a triple  $ (C , \Delta , \epsilon) $  where  $ C $  is an  $ A^e $--module  (with left action  $ L_A $  and right action  $ R_A $),  $ \; \Delta : C \longrightarrow C \otimes_A C \; $  and  $ \; \epsilon : C \longrightarrow A \; $  are  $ A^e $--module  morphisms such that
  $$  (\Delta \otimes \text{\sl id}_C) \circ \Delta \; = \; (\text{\sl id}_C \otimes \Delta ) \circ \Delta  \quad  ,  \qquad  L_A \circ (\epsilon \otimes \text{\sl id}_C) \circ \Delta \; = \; \text{\sl id}_C \; = \; R_A \circ (\text{\sl id}_C \otimes \epsilon) \circ \Delta  $$
   \indent   As usual, we adopt Sweedler's  $ \Sigma $--notation  $ \, \Delta(c) = c_{(1)} \otimes c_{(2)} \, $  or  $ \, \Delta(c) = c^{(1)} \otimes c^{(2)} \, $  for  $ \, c \in C \, $.
\end{definition}

\smallskip

%%%
  \label{A-mods_spec-prods}
 Let  $ A $  be as above, and consider now  $ A^e $  as base  $ k $--algebra.  An  $ A^e $--ring  $ H $  can be described by a  $ k $--algebra  morphism  $ \; \iota : A^e \longrightarrow H \; $.  Let us consider its restrictions
 $ \; s^\ell \, := \, \iota(\text{--} \otimes_k 1_{{}_H}) : A \longrightarrow H \; $,
 $ \; t^\ell \, := \, \iota (1_{{}_H} \otimes_k \text{--}) : A \longrightarrow H \; $,
 \, which are called respectively  {\it source\/}  and  {\it target\/}  maps.  Thus an  $ A^e $--ring  $ H $  carries two (left)  $ A $--module  structures and two (right)  $ A^{\text{\it op}} $--module  structures: for all  $ \; a, a' \in A \, $,  $ \; h \in H \, $,  we write
  $$  a \triangleright h \triangleleft \tilde{a} \, := \, s^\ell(a) \, t^\ell(\tilde{a}) \, h  \quad ,  \qquad  a \blacktriangleright h \blacktriangleleft \tilde{a} \, := \, h \, t^\ell(a) \, s^\ell(\tilde{a})  $$
As usual, the tensor product of  $ H $  with itself (as an  $ A $--bimodule,  i.e.~an  $ A^e $--module)  is defined as
  $$  H_\triangleleft \,{\mathop \otimes_A}\, {}_\triangleright H  \; := \;
   \;  H \otimes_k H \Big/ {\big\{ \big(  u \; \triangleleft a \big) \otimes u' -
   u \otimes \big( a \triangleright  \, u' \big) \big\}}_{a \in A, \, u, u' \in H}    $$
   \indent   Now we  {\it define\/}  the  {\sl left Takeuchi product\/}  (of the  $ A^e $--ring  $ H $  with itself)
 $ \; H_\triangleleft \,{\mathop \times\limits_A}\, {}_\triangleright H \, \subseteq \, H_\triangleleft \,{\mathop \otimes\limits_A}\, {}_\triangleright H \; $
 $$
   H_\triangleleft \,{\mathop \times_A}\, {}_\triangleright H  \; := \;
    \Big\{\, {\textstyle \sum_i} u_i \otimes u'_i \in H_\triangleleft \,{\mathop \otimes_A}\, {}_\triangleright H \,\Big|\; {\textstyle \sum_i}
     \big(a \blacktriangleright  u_i \,  \big) \otimes_A u'_i =
     {\textstyle \sum_i} u_i \otimes \big( u'_i \, \blacktriangleleft a \big) \; , \,\; \forall \; a \in A \,\Big\}   $$
 \vskip4pt
   By construction,  $ \, H_\triangleleft \displaystyle{\,{\mathop \times_A}\,} {}_\triangleright H \, $  has a natural structure of  $ A^e $--module,  induced by that of  $ \, H_\triangleleft \displaystyle{\,{\mathop \otimes_A}\,} {}_\triangleright H \, $.  Even more,  $ \, H_\triangleleft \displaystyle{\,{\mathop \times_A}\,} {}_\triangleright H \, $  is an  $ A^e $--ring,  via factorwise multiplication, with unit element  $ \, 1_{{}_H} \otimes 1_{{}_H} \, $  and with  $ \; \iota_{H_\triangleleft \,{\mathop \times\limits_A}\, {}_\triangleright H}(a \otimes \tilde{a}) := s^\ell(a) \otimes t^\ell(\tilde{a}) \; $.  Note that this instead is  {\sl not\/}  the case for  $ \, H_\triangleleft \,{\mathop \otimes\limits_A}\, {}_\triangleright H \, $.

\smallskip

 \subsection{Left bialgebroids}  \label{l-bialgds}

\smallskip

   {\ } \quad   We introduce now the notion of  {\sl left bialgebroid},  as well as some related items (see  \cite{Takeuchi},  \cite{Lu},  \cite{Xu2},  \cite{Bohm2}  and  \cite{Kowalzig}, Chapter 2, for a detailed history of this notion).  We begin with the very definition:

\begin{definition}
 A left  $ A $--bialgebroid  is a  $ k $--module  $ H $  that carries simultaneously a structure of an
$ A^e $--ring  $ (H, s^\ell, t^\ell) $  and of an  $ A $--coring  $ (H, \Delta_\ell , \epsilon) $  subject to the following compatibility relations:
 \vskip4pt
   (i) \,  The $ A^e $--module structure on the  $ A $-coring  $ (H, \Delta_\ell, \epsilon) $  is that of  $ \, {}_{\triangleright}H_{\triangleleft} \, $,  namely (for all  $ \, a, \tilde{a} \in A \, $,  $ \, h \in H \, $)  $ \; a \triangleright h \triangleleft \tilde{a} \, := \, s^\ell(a) \, t^\ell(\tilde{a}) \, h \; $.
 \vskip4pt
   (ii) \,  The coproduct map  $ \Delta^\ell $  is a unital  $ k $--algebra  morphism taking values in  $ \, \displaystyle{ H {{}_\triangleleft}{\mathop \times_A}{{}_\triangleright} H } \, $.
 \vskip2pt
   (iii) \,  The (left) counit map  $ \epsilon $  has the following property: for  $ \, a, \tilde{a} \in A \, $,  $ \, u, u' \in H \, $,  \, one has
  $$  \epsilon \big( s^\ell(a) \, t^\ell(\tilde{a}) \, u \big)  \; = \;  a \, \epsilon(u) \, \tilde{a}  \quad ,  \qquad \epsilon(u\,u')  \; = \;  \epsilon \big( u \, s^\ell\big(\epsilon(u')\big)\big)  \; = \;  \epsilon \big( u \, t^\ell\big(\epsilon(u') \big) \big)\quad , \qquad \epsilon (1)=1  $$
\end{definition}

\smallskip

\begin{remarks}  \label{left-bialgd-propts}
 A left bialgebroid  $ H $  over  $ A $  has the following properties (for  $ \, a \in A \, $,  $ \, u \in H \, $):
  $$  \displaylines{
   \quad \text{\it (a)}   \phantom{|} \hfill   \Delta_\ell(a \triangleright u \triangleleft \tilde{a})  \; = \;  (a \triangleright u_{(1)}) \otimes (u_{(2)} \triangleleft \tilde{a})  \quad ,
\qquad  \Delta_\ell\big(a \blacktriangleright u \blacktriangleleft \tilde{a}\big)  \; = \;  (u_{(1)} \blacktriangleleft \tilde{a}) \otimes (a \blacktriangleright u_{(2)})   \hfill  \cr
   \text{(using  $ \Sigma $--notation  $ \; \Delta(u) = u_{(1)} \otimes u_{(2)} \; $  as usual,  cf.~Definition \ref{def_A-coring})}   \hfill  }  $$
 \vskip-7pt
\noindent
   \quad \text{\it (b)}   \phantom{$ \Big| $}  $ H $  acts on  $ A $  on the left  (cf.~\cite{Kowalzig})  by
 $ \; u.a \, := \, \epsilon\big( u \, s^\ell(a) \big) \, = \, \epsilon\big( u \, t^\ell(a) \big) \; $;
 we shall also use the notation  $ \; u(a) := u.a \; $,  and call this  {\sl the left anchor of the left bialgebroid  $ H $}  (cf.~\cite{Xu2}).
 \vskip5pt
\noindent
   \quad \text{\it (c)}  \qquad  $ \displaystyle{ t^\ell\big(\epsilon(x)\big) \otimes 1  \, = \, t^\ell\big(\epsilon\big(x_{(1)}\big)\big) \otimes s^\ell\big(\epsilon\big(x_{(2)}\big)\big) \, = \,  1 \otimes s^\ell\big(\epsilon(x)\big) } $  \qquad  for all  $ \; x \in H \; $.
 \vskip6pt
\noindent
   \quad \text{\it (d)}  \,  as a matter of notation, if  $ \, \big( H, A, s^\ell, t^\ell, \Delta, \epsilon \big) \, $  is a left bialgebroid, we set  $ \, H^+ := \text{\sl Ker}(\epsilon) \, $.
\end{remarks}

\vskip11pt

\begin{definition}  \label{def-categ_left-algd}
 Let  $ \, {\mathcal H} = \big( H, A, s^\ell, t^\ell, \Delta, \epsilon \big) \, $  and  $ \, \hat{\mathcal H} = \big( \hat{H}, \hat{A}, \hat{s}^\ell, \hat{t}^\ell, \hat{\Delta}, \hat{\epsilon} \big) \, $  be two left bialgebroids.  A morphism of left bialgebroids from  $ {\mathcal H} $  to  $ \hat{\mathcal H} $  is a pair  $ (f,F) $  where
 $ \; f : A \longrightarrow \hat{A} \; $  is a morphism of algebras,
 $ F: H \longrightarrow \hat{H} \; $  is a morphism of rings and of corings,
 and  $ \,\; F \circ s^\ell \, = \, \hat{s}^\ell \circ f \; $,  $ \; F \circ t^\ell \, = \, \hat{t}^\ell \circ f \; $.
 \vskip3pt
   We denote by  {\rm (LBialg)}  the category of left bialgebroids, whose objects are left bialgebroids and morphisms are defined as above.  Inside it,  {\rm ($ \text{LBialg}_A $)}  is the subcategory whose objects are all the left bialgebroids over  $ A \, $,  and whose morphisms are all the morphisms in  {\rm (LBialg)}  of the form  $ (\text{\sl id}, F) \, $.
\end{definition}

\smallskip

\begin{free text}
 {\bf Twistors of left bialgebroids.}  Let  $ H $  be a left bialgebroid.  Given  $ \, \mathcal{F} = \sum_i x_i \otimes y_i \in H_\triangleleft{\mathop \otimes\limits_A}{}_{\,\triangleright} H \, $  (with  $ \, x_i , \, y_i \in H \, $),  define  $ \; s^\ell_{\mathcal{F}} : A \longrightarrow H \; $  by  $ \; s^\ell_{\mathcal{F}}(a) := {\textstyle \sum_i} \, s^\ell\big(x_i (a) \big) \, y_{i} \; $  and  $ \; t^\ell_{\mathcal{F}} : A \longrightarrow H \; $  by  $ \; t^\ell_{\mathcal{F}}(a) = {\textstyle \sum_i} \, t^\ell\big(y_i (a)\big) \, x_i \; $.  Moreover, for any  $ \, a , b \in A \, $  set  $ \; a *_{\mathcal{F}} b := s^\ell_{\mathcal{F}}(a)(b) = {\textstyle \sum_i} x_i(a) \, y_i(b) \; $.
\end{free text}

\smallskip

\begin{proposition}  \label{twistor_properties}
    (cf.~\cite{Xu2})
 Assume that  $ \, \mathcal{F} \in H \,{\mathop \otimes\limits_A}\, H \, $  satisfies the following conditions:
 \vskip-11pt
  $$  \displaylines{
   \qquad (i)  \qquad   (\Delta \otimes \text{\sl id})(\mathcal{F}) \cdot \mathcal{F}_{1,2}  \; = \; (\text{\sl id} \otimes \Delta)(\mathcal{F}) \cdot \mathcal{F}_{2,3}  \qquad  \text{inside\ }\;\,  H \otimes_{\!{}_{\scriptstyle A}} \! H \otimes_{\!{}_{\scriptstyle A}} \! H   \hfill  \cr
   \qquad (ii)  \qquad \qquad   m\big((\epsilon \otimes \text{\sl id})(\mathcal{F})\big)  \; = \; 1_{\!{}_H} \quad ,  \qquad  m\big((\text{\sl id} \otimes \epsilon)\mathcal{F}\big)  \; = \;  1_{\!{}_H}   \hfill  }  $$
 \vskip-3pt
\noindent
 where  $ \; \mathcal{F}_{1,2} \, = \, \mathcal{F} \otimes 1_{\!{}_H} \, \in \, H \otimes_{\!{}_{\scriptstyle A}} \! H \otimes_{\!{}_{\scriptstyle A}} \! H \; $  and  $ \; \mathcal{F}_{2,3} \, = \, 1_{\!{}_H} \otimes \mathcal{F} \, \in \, H \otimes_{\!{}_{\scriptstyle A}} \! H \otimes_{\!{}_{\scriptstyle A}} \! H \; $.
                                                                             \par
   Then one has
 $ \,\; \mathcal{F} \cdot \big( t^\ell_{\mathcal{F}}(a) \otimes 1_{\!{}_H} - 1_{\!{}_H} \otimes s^\ell_{\mathcal{F}}(a) \big) \, = \, 0 \;\, $  inside  $ \, H \otimes_{\!{}_{\scriptstyle A}} \! H \, $  (for all  $ \, a \in A \, $).
 \vskip4pt
   Moreover, if  $ \mathcal{F} $  satisfies (i) and (ii) above, then
 \vskip4pt
   \hskip-11pt   (a)  $ \;\; (A, *_{\mathcal{F}}) \, $  is an associative algebra, denoted  $ A_{\mathcal{F}} \, $,  and  $ \; a *_{\mathcal{F}} 1 \, = \, a \, = \, 1 *_{\mathcal{F}} a \; $  for all  $ \, a \in A \, $;
 \vskip3pt
   \hskip-11pt   (b)  $ \;\; s^\ell_{\mathcal{F}} : A_{\mathcal{F}} \longrightarrow H \; $  is an algebra morphism and  $ \; t^\ell_{\mathcal{F}} : A_{\mathcal{F}} \longrightarrow H \; $  is an algebra antimorphism.
\end{proposition}

\medskip

  Now let  $ M $  be a module over  $ H $  (as an algebra): then  $ M $  has also a natural $ A^e $--module  structure.  If  $ \mathcal{F} $  is a twistor of  $ H $,  then  $ M $  has also a natural  $ A_{\mathcal{F}}^e $--module  structure.  Consequently, if  $ M_1 $  and  $ M_2 $  are two  $ H $--modules,  the tensor products
 $ \; M_1 {{}_{\triangleleft\,}}{\mathop\otimes\limits_A}{{}_{\,\triangleright}} M_2 \; $
and
 $ \; M_1 {{}_{\triangleleft}}{\mathop\otimes\limits_{\,A_{\mathcal{F}}}}{{}_\triangleright} M_2 \; $
are well defined.

\smallskip

\begin{corollary}
 (cf.~\cite{Xu2}) Let  $ M_1 $  and  $ M_2 $  be two left  $ H $--modules.  Then
there exists a well defined  $ k $--linear  map
 $ \; \mathcal{F}^\# : M_1 {{}_{\triangleleft}}{\mathop\otimes\limits_{\,A_{\mathfrak F}}}{{}_\triangleright} M_2 \relbar\joinrel\longrightarrow
  M_1 {{}_{\,\triangleleft}}{\mathop\otimes\limits_A}{{}_{\,\triangleright}} M_2 \; $  given by  $ \; m_1 \otimes m_2 \mapsto \mathcal{F} \cdot (m_1 \otimes m_2) \; $.
\end{corollary}

\smallskip

   We say that  $ \mathcal{F} $  {\it is invertible\/}  if  $ \mathcal{F}^{\#} $  is a  $ k $--vector  space isomorphism for any choice of  $ M_1 $  and  $ M_2 \, $.  In this case, taking  $ \, M_1 = M_2 = H \, $  we get a  $ k $--linear  isomorphism  $ \,\; \mathcal{F}^\# : H \otimes_{\!{}_{\scriptstyle A_{\mathcal{F}}}} \!\! H \longrightarrow H \otimes_{\!{}_{\scriptstyle A}} \! H \;\, $.

\smallskip

\begin{definition}  \label{def-twistor}
 An element  $ \, \mathcal{F} \in H \,{\mathop \otimes\limits_A}\, H \, $  is called a  {\sl twistor}  (of  $ H $)  if it satisfies equations (i) and (ii) in  Proposition \ref{twistor_properties}  and it is invertible (in the just explained sense).
\end{definition}

\smallskip

   Let now  $ \mathcal{F} $  be a twistor of  $ H \, $.  Then we may define a new coproduct
$ \; \Delta_{\mathcal{F}} : H \longrightarrow H \otimes_{\!{}_{\scriptstyle A_{\mathcal{F}}}} \!\!\! H \; $  of  $ H $  by the formula  $ \,\; \Delta_{\mathcal{F}}(x) \, := \, {\big( \mathcal{F}^\# \big)}^{-1} \big( \Delta (x) \cdot \mathcal{F} \big) \;\, $.  The key result is then the following (see  \cite{Xu2}):

\smallskip

\begin{theorem}  \label{twisted_left-bialgd}
 Let  $ \, \big( H , A \, , s^\ell \, , t^\ell \, , m \, , \Delta \, , \epsilon \big) $  be a left bialgebroid.  Then  $ \, \big( H , A_{\mathcal{F}} , s^\ell_{\mathcal{F}} , t^\ell_{\mathcal{F}} , m \, , \Delta_{\mathcal{F}} , \epsilon \big) \, $  is a left bialgebroid as well.
\end{theorem}

\smallskip

\begin{free text}  \label{V^ell(L)_left-bialgd}
 {\bf Left bialgebroid structures on universal enveloping algebras  $ V^\ell(L) \, $.}  Given a Lie-Rinehart algebra  $ L \, $,  there is a standard left bialgebroid structure on  $ V^\ell(L) \, $.
                                                                \par
   Source and target maps are equal and given by  $ \, \iota_A : A \longrightarrow V^\ell(L) \, $.   Then the  $ A^e $--module structure  $ \, {}_\triangleright V^\ell(L)_\triangleleft \, $  is given by  $ \; a \triangleright u \triangleleft \tilde{a} \, := \, a \, \tilde{a} \, u \; $.  The coproduct  $ \; \Delta_\ell : V^\ell(L) \longrightarrow V^\ell(L)_{\triangleleft } \,{\mathop \otimes\limits_A}\, {}_{\triangleright}V^\ell(L) \; $  and the counit map  $ \; \epsilon : V^\ell(L) \longrightarrow A \; $  are determined by
  $$  \Delta_\ell(a) \, = \, a \otimes 1 \; ,  \quad  \Delta_\ell(X) \, = \, X \otimes 1 \, + \, 1 \otimes X \; ,  \qquad  \epsilon(a) \, = \, a \; ,  \quad  \epsilon(X) \, = \, 0  \qquad   \eqno  \forall \; a \in A \; ,  \; X \in L  \quad  $$
Note that the anchor map  $ \omega $  endows  $ A $  with an obvious left  $ V^\ell(L) $--module  structure, given by  $ \, u.a := \omega(u)(a) \, $  for  $ \, u \! \in \! V^\ell(L) \, $,  $ \, a \! \in \! A \, $,  that coincides with the anchor of the left bialgebroid  $ V^\ell(L) \, $;  cf.\  Remarks \ref{left-bialgd-propts}{\it (b)}.
                                                                    \par
 More in general, left  $ V^\ell(L) $--module  structures on  $ A $  correspond to left bialgebroid structures on  $ V^\ell(L) $  over  $ A \, $ (see  \cite{Kowalzig}).  Finally, one can recover the anchor of  $ L $  from the left bialgebroid structure of  $ V^\ell(L) $  as follows:
 $ \; \omega_L(X)(a) = \epsilon_{\scriptscriptstyle V^\ell(L)}(X \, a) \; $  for all  $ \, X \in L \, $,  $ \, a \in A \, $.
\end{free text}

\smallskip

\begin{remark}
 Let  $ (A,L) $  and  $ (A',L') $  be two Lie-Rinehart algebras.  Endow  $ V^\ell(L) $  and  $ V^\ell(L') $  with their standard left bialgebroid structure.  A Lie-Rinehart algebra morphism from  $ (A,L) $  to  $ (A',L') $  as it was defined in  \cite{Huebschmann1}   --- see also ``morphisms of Lie pseudo-algebras'' as they are defined in  \cite{Higgins and Mackenzie}  ---   gives rise to a left bialgebroid morphism from  $ V^\ell(L) $  to  $ V^\ell(L') \, $.
\end{remark}

\smallskip

   Our next theorem is a suitable version for left bialgebroids of the well-known Cartier-Milnor-Moore theorem (for Hopf algebras).  A similar result is given in  \cite{Moerdijk and Mrcun},  yet in this paper we do need (later on) that kind of result exactly as stated here below.

\smallskip

\begin{theorem}  \label{left-bialgd_V^ell(L)}
 Assume that  $ A $  is a unital commutative algebra over the field  $ k \, $.
 \vskip5pt
   (a) \, Let  $ (U, A,s^\ell,t^\ell, \Delta_\ell, \epsilon) $  be a left bialgebroid such that  $ \, s^\ell = t^\ell \, $.  Set
  $$  P^\ell(U)  \; := \;  \big\{\, u \in U \,\big|\, \Delta_\ell(u) = u \otimes 1 + 1 \otimes u \,\big\}  $$
(the set of ``left primitive elements'' of  $ U \, $).  Then the pair  $ \big( A, P^\ell(U) \big) $  is a Lie-Rinehart algebra.
 \vskip4pt
   (b) \, Assume in addition that  $ \, P^\ell(U) $  is projective as an  $ A $--module,  and that  $ P^\ell(U) $  and  $ s^\ell(A) $  generate  $ U $  as an algebra.  Then  $ U $  is isomorphic to  $ V^\ell\big(P^\ell(U)\big) $  as a left bialgebroid.
\end{theorem}

\begin{proof}
   {\it (a)} \,  On  $ P^\ell(U) $  we set the following  $ A $--module structure:
 $ \; a \cdot D := s^\ell(a)D \; $  for all  $ \, a \in A \, $,  $ \, D \in P^\ell(U) \, $.
Moreover, if  $ \, D , D' \in P^\ell(U) \, $  then  $ \, \big[D,D'\big] := D \cdot D' - D' \cdot D \in P^\ell(U) \, $,  by direct check: this defines a Lie bracket on  $ P^\ell(U) \, $.

 Finally, we define  $ \; \omega : P^\ell(U) \relbar\joinrel\longrightarrow \text{\sl Der}(A) \; $  by  $ \; D \mapsto \big( b \;{\buildrel {\omega(D)} \over {\relbar\joinrel\relbar\joinrel\relbar\joinrel\rightarrow}}\; \epsilon(D\,s^\ell(b)) \big) \; $.
                                                              \par
   It is proved in  \cite{Kowalzig}  (Proposition 4.2.1) that  $ \big( A, P^\ell(U), \omega \big) $  is a Lie-Rinehart algebra.
 \vskip7pt
   {\it (b)} \,  By assumption, the natural algebra morphism  from  $ \, T_k\big( A \oplus P^\ell(U) \big) \, $  to  $ U $  is surjective and it induces a surjective algebra morphism  $ \, f : V^\ell\big( P^\ell(U) \big) \rightarrow U \, $.  As  $ P^\ell(U) $  and  $ s^\ell(A) $  generates  $ V^\ell\big( P^\ell(U) \big) $  as an algebra, this map is also a morphism of corings.  By the same argument as in  \cite{Montgomery},  Lemma 5.3.3, one shows that  $ f $  is also injective because  $ f{\big|}_{P^\ell(U)} $  is injective (which is obvious).
\end{proof}

\smallskip

 \subsection{Right bialgebroids}  \label{r-bialgds}

\smallskip

   {\ } \quad   Just like for left bialgebroids, one can consider the notion of  {\sl right bialgebroids\/}  (cf.~\cite{Kadison and Szlachanyi}  and  \cite{Bohm2}).  We will need a second type of  ``Takeuchi product''.  In order to distinguish it from the previous one, we shall now denote the base  $ k $--algebra  by  $ B $  instead of  $ A \, $.  Hereafter,  $ B $  is a (unital, associative)  $ k $--algebra,  and we use notations as in  \S~\ref{A-rings A-corings}.  Let  $ H $  be a  $ B^e $--ring  given by a  $ k $--algebra  morphism  $ \; \eta^r : B^e \longrightarrow H \; $,  a source map  $ \; s^r := \eta^r(\,\text{--} \otimes 1) \; $  and a target map  $ \; t^r := \eta^r(1 \otimes \text{--}\,) \; $.  We consider now the right  $ B^e $--module  structure on  $ H $  given by  $ \; h\cdot(b \otimes \tilde{b}) := h \cdot \eta^r(b \otimes \tilde{b}) \; $,  for  $ \; b, \tilde{b} \in B \, $,  $ \, h \in H \, $.  Then the tensor product of  $ H $  with itself (as a  $ B $--bimodule,  i.e.~a  $ B^e $--module)  is defined as
  $$  \displaylines{
   H_\blacktriangleleft \,{\mathop \otimes_B}\, {}_\blacktriangleright H  \; := \;
H \otimes_k H \Big/ {\big\{ (u \blacktriangleleft b) \otimes u' - u \otimes (b \blacktriangleright u') \big\}}_{b \in B, \, u, u' \in H} }$$
   \indent   Now we  {\it define\/}  the  {\sl right Takeuchi product\/}  (of the  $ B^e $--ring  $ H $  with itself)
 $ \; H_\blacktriangleleft {\mathop \times\limits_B}\, {}_\blacktriangleright H \subseteq \, H_\blacktriangleleft {\mathop \otimes\limits_B} {}_{\,\blacktriangleright} H \; $
  $$  \displaylines{
   H_\blacktriangleleft \,{\mathop \times\limits_B}\, {}_\blacktriangleright H  \; := \;
     \Big\{\, {\textstyle \sum_i} u_i \otimes u'_i \in H_\blacktriangleleft \,{\mathop \otimes\limits^B}\, {}_\blacktriangleright H \,\Big|\; {\textstyle \sum_i} (a \triangleright u_i) \otimes u'_i = {\textstyle \sum_i} u_i \otimes (u'_i \triangleleft a) \,\Big\}  }  $$

\begin{definition}
 A right  $ B $--bialgebroid  is a  $ k $--module  $ H $  that carries simultaneously a structure of a  $ B^e$--ring  $ (H, s^r, t^r) $  and of a  $ B $--coring  $ (H, \Delta_r, \partial) $  subject to the following compatibility relations:
 \vskip4pt
   (i) \,  The  $ B^e $--module  structure on the  $ B $--coring  $ (H, s^r, t^r) $  is that of  $ \, {}_\blacktriangleright H_\blacktriangleleft \, $,  namely (for all  $ \, b, \tilde{b} \in B \, $,  $ \, h \in H ,\, $)  $ \; b \blacktriangleright h \blacktriangleleft \tilde{b} \, := \, h \, s^r(\tilde{b}) \, t^r(b) \, = \, h \, \eta(\tilde{b} \otimes b) \; $.
 \vskip1pt
   (ii) \,  The coproduct map  $ \Delta_r $  is a unital  $ k $--algebra  morphism taking values in  $ \, \displaystyle{ H {{}_\blacktriangleleft}{\mathop \times^B}{{}_\blacktriangleright} H } \, $.
 \vskip4pt
   (iii) \,  The (right) counit map  $ \partial $  has the following property: for all  $ \, b, \tilde{b} \in B $,  $ \, u, u' \in H \, $,  one has
  $$  \partial \big( u \, s^r\big(\tilde{b}\big) \, t^r(b) \big)  \; = \;  b \, \partial(u) \, \tilde{b}  \!\quad ,  \!\!\qquad  \partial(u\,u')  \; = \;  \partial \big( s^r\big(\partial(u) \big) \, u' \big)  \; = \;  \partial \big( t^r\big(\partial(u) \big) \, u' \big)  \!\quad ,  \!\!\qquad  \partial(1)  \; = \;  1  $$
\end{definition}

\smallskip

\begin{remarks}  \label{right-bialgd-propts} {\ }
                                               \par
   {\it (a)}\,  The definition of a right bialgebroid is obtained from the definition of a left bialgebroid by exchanging the role of black triangles ($ \, \blacktriangleright \, $,  $ \! \blacktriangleleft \, $)  and white triangles  ($ \, \triangleright \, $,  $ \! \triangleleft \, $).  Consequently, the properties of a right bialgebroid are obtained from those of a left bialgebroid  --- see  Remarks  \ref{left-bialgd-propts})  ---   by exchanging the roles of black triangles and white triangles.
                                               \par
   {\it (b)}\,  (cf.~\cite{Kowalzig})  The ``opposite'' of a left bialgebroid  $ \, U = \big( U, A, s^\ell, t^\ell, \Delta_\ell , \epsilon \big) \, $  is defined as  $ \, U^{\text{\it op}} := \big( U^{\text{\it op}}, A, t^\ell, s^\ell, \Delta_\ell, \epsilon \big) \, $:  this can be shown to be a right bialgebroid.  The ``coopposite'' is given by  $ \, U_{\text{\it coop}} := \big( U, A^{\text{\it op}}, t^\ell, s^\ell, \Delta_\ell^{\text{\it coop}} , \epsilon \big) \, $  with  $ \, \Delta_\ell^{\text{\it coop}} : U \longrightarrow {}_{\triangleright}U\otimes_{A^{\text{\it op}}}U_{\triangleleft} \, $,  $ \, u \mapsto u_{(2)} \otimes u_{(1)} \, $;  this is still a left bialgebroid.  As a consequence, $ \, U^{\text{\it op}}_{\text{\it coop}} \, $  is a right bialgebroid.

 \end{remarks}

The definition of a right bialgebroid morphism is analogous to that of a left bialgebroid morphism.    We denote by  {\rm (RBialg)}  the category of right bialgebroids, whose objects are right bialgebroids and morphisms are defined mimicking  Definition  \ref{def-categ_left-algd}.  Inside it,  $ \text{(RBialg)}_A $  is the subcategory whose objects are all the right bialgebroids over  $ A $  and whose morphisms are all those in  {\rm (RBialg)}  of the form  $ (\text{\sl id} \, , F) \, $.

\medskip

\begin{free text}  \label{V^r(L)_right-bialgd}
 {\bf Right bialgebroid structures on universal enveloping algebras  $ V^r(L) \, $.}  Given a Lie-Rinehart algebra  $ L \, $,  now considered as a  {\sl right\/}  one, its right universal enveloping algebra  $ V^r(L) $  bears a natural structure of right bialgebroid over  $ A \, $ obtained by  pulling-back the left bialgebroid structure of  $ V^\ell(L) $  via the isomorphism  $ \, V^r(L) \cong {V^\ell(L)}^{\text{\it op}} \, $.  More explicitly, the  $ A^e $--module  structure  $ \, {}_\blacktriangleright V^r(L)_\blacktriangleleft \, $  is given by  $ \; a \blacktriangleright u \blacktriangleleft \tilde{a} \, := \,  u \, a \, \tilde{a} \; $;  the coproduct  $ \, \Delta_r : V^r(L) \longrightarrow V^r(L)_{\blacktriangleleft } \,{\mathop \otimes\limits_A}\, {}_{\blacktriangleright}V^r(L) \, $  and the counit  $ \; \partial : V^r(L) \longrightarrow A \; $  are determined by
 \vskip-13pt
  $$  \Delta_r(a) \, = \, a \otimes 1 \; ,  \quad  \Delta_r(X) \, = \, X \otimes 1 \, + \, 1 \otimes X \; ,  \qquad  \partial(a) \, = \, a \; ,  \quad  \partial(X) \, = \, 0  \qquad   \eqno  \forall \; a \in A \; ,  \; X \in L  \;\;  $$
   \indent    Finally, one can recover the anchor map of  $ L $  from the right bialgebroid structure of  $ V^r(L) $  as  $ \,\; \omega_L(X)(a) \, = \, - \partial_{\scriptscriptstyle V^r(L)}(a\; X) \;\, $  for all  $ \, X \in L \, $,  $ \, a \in A \, $.
\end{free text}

\smallskip

   We also have an analogue for right bialgebroids of  Theorem \ref{left-bialgd_V^ell(L)}  (with similar proof):

\medskip

\begin{theorem}  \label{right-bialgd_V^r(L)}
 Assume that  $ A $  is a unital commutative algebra over the field  $ k \, $.
 \vskip5pt
   (a) \, Let  $ (W, A,s^r,t^r, \Delta_r, \partial) $  be a right bialgebroid such that  $ \, s^r = t^r \, $.  Set
 \vskip-5pt
  $$  P^r(W)  \; := \;  \big\{\, w \in W \,\big|\, \Delta_r(w) = w \otimes 1 + 1 \otimes w \,\big\}  $$
(the set of ``right primitive elements'' of  $ W \, $).  Then the pair  $ \big( A, P^r(W) \big) $  is a right Lie-Rinehart algebra for the following right action and anchor map
 \vskip-17pt
  $$  w.a \, := \, w \, s^r(a) \;\; ,  \quad  \omega(D)(a) \, := \, -\partial\big( s^r(a)\,D \big) \;\; , \qquad  \forall \;\, w \in W \, ,  \;\; \forall \;\, D \in P^r(W) \, ,  \;\; \forall \;\, a \in A  $$
 \vskip-3pt
   (b) \, Assume in addition that  $ \, P^r(W) $  is projective as an  $ A $--module,  and that  $ P^r(W) $  and  $ s^r(A) $  generate  $ W $  as an algebra.  Then  $ W $  is isomorphic to  $ V^r\big(P^r(W)\big) $  as a right bialgebroid.
\end{theorem}

\smallskip

\begin{remark}  \label{rPrim-Vr(L)}
  The previous result improves a bit as follows.  Let  $ \big( W, A, s^r, t^r, \Delta_r, \partial \big) $  be a right bialgebroid for which  $ A $  is commutative and  $ \, s^r = t^r \, $.  Let  $ \, Q \subseteq P^r(U) \, $  be a right Lie-Rinehart subalgebra of  $ P^r(W) $  such that \,  {\it (i)}\,  $ Q $  is a projective  $ A $--module  and \,  {\it (ii)}\,  $ Q $  and  $ A $  generate  $ W $  as an algebra.  Then  $ W $  is isomorphic to  $ V^r(Q) $  as a right bialgebroid, and  $ \, Q = P^r(W) \; $.
                                                              \par
   An entirely similar remark also applies to  Theorem \ref{left-bialgd_V^ell(L)}.
\end{remark}

\smallskip

 \subsection{Duals of bialgebroids}  \label{dual_bialgds}

\smallskip

   {\ } \quad   We shall now consider left and right dual of (left and right) bialgebroids, and investigate their main properties.  We begin with dual of left bialgebroids, then we pass on to dual of right ones.

\smallskip

\begin{definition}  \label{dual-bialgd's}  {\ }
 \vskip3pt
   (a) \,  Let  $ U $  be a left  $ A $--bialgebroid,  with structure maps as before.  The  {\sl left dual}  and the  {\sl right dual}  of  $ \, U $  respectively are the spaces
 \vskip-17pt
  $$  \displaylines{
   U_*  \; := \;  \big\{\, \phi : U \longrightarrow A \,\big|\, \phi(u'+u'') = \phi(u') + \phi(u'') \, , \; \phi\big( s^\ell(a) \, u \big) = a \, \phi(u) \,\big\}  \; = \;
\text{\sl Hom}_A\big(\, {}_\triangleright U , {}_A A \,\big)  \cr
   U^*  \; := \;  \big\{\, \phi : U \longrightarrow A \,\big|\, \phi(u'+u'') = \phi(u') + \phi(u'') \, , \; \phi\big( t^\ell(a) \, u \big) = \phi(u) \, a \,\big\}  \; = \;
\text{\sl Hom}_A\big( U_\triangleleft \, , A_A \big)  }  $$
   \indent   (b) \,  Let  $ W $  be a right  $ A $--bialgebroid,  with structure maps as before.  The  {\sl left dual}  and the  {\sl right dual}  of  $ \, W $  respectively are the spaces
 \vskip-17pt
  $$  \displaylines{
   {}_*W  \; := \;  \big\{\, \psi : W \!\longrightarrow A \,\big|\, \psi(w'+w'') = \psi(w') + \psi(w'') \, , \; \psi\big( w \, t^r(a) \big) = a \, \psi(w) \,\big\}  \; = \;
\text{\sl Hom}_A\big(\, {}_\blacktriangleright{}W_{\,} , {}_A{}A \,\big)  \cr
   {}^*W  \; := \;  \big\{\, \psi : W \!\longrightarrow A \,\big|\, \psi(w'+w'') = \psi(w') + \psi(w'') \, , \; \psi\big( w \, s^r(a) \big) = \psi(w) \, a \,\big\}  \; = \;
\text{\sl Hom}_A\big(\, W_\blacktriangleleft \, , A_A \big)  }  $$
\end{definition}

\smallskip

\begin{free text}  \label{bialgd-struct-dual}
 {\bf Bialgebroid structures on dual spaces.}  Let  $ U $  be a  {\sl left}  $ A $--bialgebroid  as above.  We shall now introduce on its dual spaces  $ U_* $  and  $ U^* $  a structure of right  $ A $--bialgebroid; most of the structure is well-defined in general, but for the coproduct we need an additional assumption, namely  $ U $  as an  $ A $--module  (on the left, or the right, see below) has to be  {\sl projective}.

\vskip7pt

   {\sl  $ \underline{\text{Product structure}} $:}  First we recall (see  \cite{Kadison and Szlachanyi},  and also  \cite{Kowalzig})  that  $ U_* $  and  $ U^* $  can be equipped with a product, for which the counit map  $ \epsilon $  is a two-sided unit.  For  $ \, \phi, \phi' \in U_* \, $,  $ \, \psi, \psi' \in U^* \, $,  $ \, u \in U \, $,  set
  $$  \displaylines{
   \big(\phi \, \phi'\big)(u)  \; \equiv \;  m_{U_*}\big(\phi \otimes \phi'\big)(u)  \; := \;  \phi' \Big( m_{U^{\text{\it op}}}\big( \text{\sl id} \otimes (t^\ell \circ \phi) \big) \big(\Delta_\ell(u)\big) \Big)  \; = \;
\phi' \big( t^\ell (\phi(u_{(2)})) \cdot u_{(1)} \big)  \cr
   \big(\psi \, \psi'\big)(u)  \; \equiv \;  m_{U^*}\big(\psi \otimes \psi'\big)(u)  \; := \;  \psi' \Big( m_U\big((s^\ell \circ \psi) \otimes \text{\sl id} \big) \big(\Delta_\ell(u)\big) \Big)  \; = \;
\psi' \big( s^\ell (\psi(u_{(1)})) \cdot u_{(2)} \big)  }  $$
 \vskip7pt
   {\sl  $ \underline{\text{$ A $--module  structures}} $:} \,  For the left dual space  $ U_* \, $,  the left dual source map  $ \, s^r_* : A \longrightarrow U_* \, $  and the right dual target map  $ \, t^r_* : A \longrightarrow U_* \, $  are defined as follows:
  $$  \big(s^r_*(a)\big)(u)  \, := \,  \epsilon\big( t^\ell(a) \, u \big)  \, = \; \epsilon(u) \, a  \;\; ,  \qquad  \big(t^r_*(a)\big)(u)  \, := \,  \epsilon\big( u \, t^\ell(a) \big)   \eqno \forall \; a \in A \, ,  \; u \in U \; .  \quad  $$
   \indent   Then one has, in the usual way, two left and two right actions of  $ A $  on  $ U_* \, $,  given by
  $$  \displaylines{
   (a \triangleright \phi)(u)  \, := \,  \big( s^r_*(a) \, \phi \big)(u)  \, = \,  \phi\big( t^\ell(a) \, u \big)  \;\; ,  \qquad
   (\phi \triangleleft a)(u)  \, := \,  \big( t^r_*(a) \, \phi \big)(u)  \, = \,  \phi\big( u \, t^\ell(a) \big)  \cr
   (a \blacktriangleright \phi)(u)  \, := \,  \big( \phi \, t^r_*(a) \big)(u)  \, = \,  \phi\big( u \, s^\ell(a) \big)  \;\; ,  \qquad
   (\phi \blacktriangleleft a)(u)  \, := \,  \big( \phi \, s^r_*(a) \big)(u)  \, = \,  \phi(u) \, a  }  $$
 \vskip3pt
   Similarly, for the right dual  $ U^* $  the source  $ \, s_r^* : A \longrightarrow U^* \, $  and the target  $ \, t_r^* : A \longrightarrow U^* \, $  are
  $$  \big(s_r^*(a)\big)(u)  \, := \,  \epsilon\big( u \, s^\ell(a) \big)  \;\; ,  \qquad  \big(t_r^*(a)\big)(u)  \, := \,  \epsilon\big( s^\ell(a) \, u \big)  \, = \;  a \, \epsilon(u)   \eqno \forall \; a \in A \, ,  \; u \in U \; .  \quad  $$
Then one has, like before, two left and two right  $ A $--actions  on  $ U^* \, $,  given by
  $$  \displaylines{
   (a \triangleright \psi)(u)  \, := \,  \big( s_r^*(a) \, \psi \big)(u)  \, = \,  \psi\big( u \, s^\ell(a) \big)  \;\; ,  \qquad
   (\psi \triangleleft a)(u)  \, := \,  \big( t_r^*(a) \, \psi \big)(u)  \, = \,  \psi\big( s^\ell(a) \, u \big)  \cr
   (a \blacktriangleright \psi)(u)  \, := \,  \big( \psi \, t_r^*(a) \big)(u)  \, = \,  a \, \psi(u)  \;\; ,  \qquad
   (\psi \blacktriangleleft a)(u)  \, := \,  \big( \psi \, s_r^*(a) \big)(u)  \, = \,  \psi\big( u \, t^\ell(a) \big)  }  $$

\vskip5pt

   {\sl  $ \underline{\text{Coproduct  structure}} $:}  Now  {\sl assume that  $ {}_{\triangleright}U $  as an  $ A $--module  be projective}.  Then we now endow the left dual  $ U_*^{\phantom{|}} $  with a coproduct  $ \Delta_r $  which, eventually, makes it into a right bialgebroid.
                                                                \par
   Consider the injective map  $ \; \chi : U_{*\,\blacktriangleleft } \!\otimes {}_{\blacktriangleright}U_* \longrightarrow \text{\sl Hom}_{(A,-)} \big(\, {}_\triangleright{( U_{\scriptscriptstyle \blacktriangleleft} \! \otimes {}_{\triangleright} U \,)} \, , \, {}_A{}A \big) \; $  given by
 \vskip-4pt
  $$  \phi \otimes \phi'  \; \mapsto \;  \chi\big( \phi \otimes \phi' \big) \Big(\, u \otimes u' \, \mapsto \, \chi\big( \phi \otimes \phi' \big)\big(u \otimes u'\big) \, := \, \phi'\big( u \, s^\ell(\phi(u')) \big) \Big)  $$
 \vskip-1pt
\noindent
 Now, if  $ U $  is finite projective (as an  $ A $--module)  then the previous map is even an isomorphism.  If instead  $ U $  is projective but not finite, one can endow  $ \, U_{*\,\blacktriangleleft } \!\otimes {}_{\blacktriangleright}U_* \, $  with a suitable topology (typically, the ``weak'' one), and denote by  $ \, U_{*\,\blacktriangleleft } \,\widetilde{\otimes}\, {}_{\blacktriangleright}U_* \, $  the corresponding completion: then the above map extends   --- by continuity, using the notion of basis for a projective module  (cf.~\cite{Anderson and Fuller})  --- to an isomorphism from  $ \, U_{*\,\blacktriangleleft} \widetilde{\otimes}\, {}_{\blacktriangleright}U_* \, $  to  $ \, \text{\sl Hom}_{(A,-)} \big(\, {}_\triangleright{( U_{\scriptscriptstyle \blacktriangleleft} \! \otimes {}_{\triangleright} U \,)} \, , \, {}_A{}A \,\big) \, $.  This allows us to define a coproduct  $ \Delta_r $  on  $ U_* $  as the transpose of the multiplication on  $ U \, $,  namely
 \vskip-12pt
  $$  \Delta_r : U_* \relbar\joinrel\longrightarrow \text{\sl Hom}_{(A,-)} \big(\, {}_\triangleright{( U_{\scriptscriptstyle \blacktriangleleft} \! \otimes {}_{\triangleright} U \,)} \, , \, {}_A{}A \,\big) \mathop{\cong}\limits_{\,\;\chi^{-1}} \,  U_{*\,\blacktriangleleft } \widetilde{\otimes}\, {}_{\blacktriangleright}U_* \,\; ,  \;\;\;
  \phi \, \mapsto \Delta_r(\phi) \, \big(\, u \otimes u' \, \mapsto \, \phi\big(u\,u'\big) \big)  $$
 \vskip-4pt
\noindent
 This coproduct makes  $ U_* $  into a (topological)  $ A $--coring,  with counit  $ \; \eta_* : U_* \longrightarrow A \; $,  $ \; \eta_*(\phi ) := \phi(1) \; $.
 \vskip3pt
   Similarly,  {\sl if  $ U_\triangleleft $  as an  $ A $--module  is projective},  then for its right dual  $ U^* $  a coproduct is defined as follows.  Consider the injective map
 $ \; \vartheta : U^*_\blacktriangleleft \!\otimes {}_{\blacktriangleright}U^* \longrightarrow \text{\sl Hom}_{(-,A)} \big(\, {( U_\triangleleft \! \otimes {}_{\scriptscriptstyle \blacktriangleright}U \,)}_\triangleleft \, , \, A_A \,\big) \; $  given by
 \vskip-4pt
  $$  \psi \otimes \psi'  \; \mapsto \;  \vartheta\big( \psi \otimes \psi' \big) \Big(\, u \otimes u' \, \mapsto \, \vartheta\big( \psi \otimes \psi' \big)\big(u \otimes u'\big) \, := \, \psi\big( u' \, t^\ell(\psi'(u)) \big) \Big)  $$
 \vskip-1pt
\noindent
 Again, if  $ U $  is finite projective (as an  $ A $--module)  then this map is an isomorphism.  If instead  $ U $  is projective but not finite, one can endow  $ \, {U^*}_{\!\!\blacktriangleleft} \! \otimes {}_{\blacktriangleright}U^* \, $  with a suitable topology (like the weak one), and denote by  $ \, {U^*}_{\!\!\blacktriangleleft} \widetilde{\otimes}\, {}_{\blacktriangleright}U^* \, $  the corresponding completion: then the above map extends   --- by continuity ---   to an isomorphism  $ \widetilde\vartheta $  from  $ \, {U^*}_{\!\!\blacktriangleleft} \widetilde{\otimes}\, {}_{\blacktriangleright}U^* \, $  to  $ \, \text{\sl Hom}_{(-,A)} \big(\, {( U_\triangleleft \! \otimes {}_{\scriptscriptstyle \blacktriangleright}U \,)}_\triangleleft \, , \, A_A \,\big) \, $.  Thus we can define a coproduct  $ \Delta_r $  on  $ U^* $  as the transpose of the  {\sl opposite\/}  multiplication on  $ U \, $,  namely
 \vskip-12pt
  $$  \Delta_r : U^* \relbar\joinrel\longrightarrow \text{\sl Hom}_{(-,A)} \big(\, {( U_\triangleleft \! \otimes {}_{\scriptscriptstyle \blacktriangleright}U \,)}_\triangleleft \, , \, A_A \,\big)  \mathop{\cong}\limits_{\,\;\widetilde{\vartheta}^{-1}} \,  {U^*}_{\!\!\blacktriangleleft} \widetilde{\otimes}\, {}_{\blacktriangleright}U^* \,\; , \;\;\;
  \psi \, \mapsto \Delta_r(\psi) \, \big(\, u \otimes u' \, \mapsto \, \psi\big(u'\,u\big) \big) $$
 \vskip-4pt
\noindent
 This makes  $ U^* $  into a (topological)  $ A $--coring,  with counit  $ \; \partial_* : U^* \longrightarrow A \; $  given by  $ \; \partial_*(\psi ) := \psi(1) \; $.

\vskip7pt

   {\sl  $ \underline{\text{Conclusion}} $:}  If  $ U $  is any  {\sl left\/}  bialgebroid over  $ A \, $, projective as an  $ A $--module,  then  $ U_* $  and  $ U^* $  with the structures introduced above are both (topological)  {\sl right\/}  bialgebroids over  $ A \, $.

\vskip11pt

   Similarly, we consider the case of a  {\sl right\/}  $ A $--bialgebroid  $ W \, $,  and we introduce canonical structures of (topological)  {\sl left\/}  $ A $--bialgebroids  on its left and right dual spaces  $ {}_*W $  and  $ {}^*W \, $:  indeed, everything is strictly similar to what occurs in the previous case for  $ U $,  so we skip details.
\vskip13pt

   {\bf Notation:}  in the following, if  $ v $  is an element of some (left or right)  $ A $--module,  and  $ \phi $  is an element of the (left or right) dual of that module, we shall write  $ \; \big\langle \phi \, , \, v \big\rangle := \phi(v) \; $  or  $ \; \big\langle v \, , \, \phi \big\rangle := \phi(v) \; $.
\end{free text}

\vskip9pt

\begin{remark}
 \,  If  $ U $  is a left bialgebroid which is projective of finite type as an  $ A $--module,  then it is isomorphic to  $ {}^*(U_*) $  and to  $ {}_*(U^*) $   --- as a left bialgebroid.  This follows from the equalities
 \vskip-13pt
  $$  \big\langle\, u \, , \, \phi \, s_*^r(a) \,\big\rangle \, = \, \big\langle\, u \, , \, \phi \,\big\rangle \, a  \;\; ,  \;\quad
     \big\langle\, u \, , \, \psi \, t^*_r(a) \,\big\rangle \, = \, a \, \big\langle\, u \, , \, \psi \,\big\rangle   \eqno \quad \forall \;\; a \in A \, , \, u \in U \, , \, \phi \in U_* \, , \, \psi \in U^*  $$
\end{remark}

\medskip

   We introduce now the natural vocabulary of ``pairings'', which we shall use in computations.

\medskip

\begin{definition}  \label{left and right 1}
 \, (a) Let  $ \big( U, s_\ell, t_\ell \big) $  and  $ \big( W, s^r, t^r \big) $  be two  $ A^e $--modules. An  $ A^e $--left pairing is a  $ k $--bilinear  map  $ \; \big\langle\ ,\ \big\rangle : U \times W \longrightarrow A \; $  such that, for any  $ \, u \in U $,  $ \, w \in W $  and  $ \, a \in A \, $,  one has
  $$  \begin{array}l
   \hskip4pt   \big\langle\, u \, , a \triangleright w \,\big\rangle  \; = \;  \big\langle\, u, s^r(a) \, w \,\big\rangle  \; = \;  \big\langle\, t_\ell(a) \, u \, , w \,\big\rangle  \; = \;  \big\langle\, u \triangleleft a \, , w \,\big\rangle  \\
   \hskip2pt   \big\langle\, u \, , w \triangleleft a \,\big\rangle  \; = \;  \big\langle\, u \, , t^r(a) \, w \,\big\rangle  \; = \;  \big\langle\, u \, t_\ell(a) \, , w \,\big\rangle  \; = \;  \big\langle\, a \blacktriangleright u \, , w \,\big\rangle^{\phantom{\big|}}  \\
   \big\langle\, u \, , a \blacktriangleright w \,\big\rangle  \; = \;  \big\langle\, u \, , w \, t^r(a) \,\big\rangle  \; = \;  \big\langle\, u \, s_\ell(a) \, , w \,\big\rangle  \; = \;  \big\langle\, u \blacktriangleleft a \, , w \,\big\rangle^{\phantom{\big|}}  \\
  \hskip43pt   \big\langle\, u \, , w \blacktriangleleft a \big\rangle  \; = \;  \big\langle u \, , w \, s^r(a) \,\big\rangle  \; = \;  \big\langle\, u \, , w \,\big\rangle^{\phantom{\big|}} \! a  \\
  \hskip45pt   \big\langle\, a \triangleright u \, , w \,\big\rangle  \; = \;  \big\langle\, s^\ell(a) \, u \, , w \,\big\rangle  \; = \;  a \, \big\langle\, u \, , w \,\big\rangle^{\phantom{\big|}}
      \end{array}  $$

Then there exist natural morphisms of  $ A^e $--modules  $ \, W \longrightarrow U_* \, $  and  $ \, U \longrightarrow {}^*W \, $.  The pairing is  {\sl non degenerate}  if the left and right kernels of this pairing are trivial, that is to say
  $$  \big\langle\, u \, , w \,\big\rangle = 0 \; ,  \;\; \forall \;\, w \in W  \quad  =\joinrel\Longrightarrow  \quad  u = 0 \;\; ,  \qquad \quad
      \big\langle\, u \, , w \,\big\rangle = 0 \; ,  \;\; \forall \;\, u \in U  \quad  =\joinrel\Longrightarrow  \quad w = 0  $$
 In other words, the pairing is non degenerate if and only if the above maps  $ \, W \longrightarrow U_* \, $  and  $ \, U \longrightarrow {}^*W \, $  (which are morphisms  of  $ A^e $--modules)  are  {\sl injective}.

\vskip5pt

   (b)  Let  $ \big( U, s^\ell, t^\ell \big) $  and  $ \big( W, s_r, t_r \big) $  be two  $ A^e $--modules. An  $ A^e $--right  pairing is a  $ k $--bilinear  map  $ \; \big\langle\ ,\ \big\rangle : U \times W \longrightarrow A \; $  such that, for any  $ \, u \in U $,  $ \, w \in W $  and  $ \, a \in A \, $,  one has
  $$  \begin{array}l
   \hskip1pt   \big\langle\, u \, , w \triangleleft a \,\big\rangle  \; = \;  \big\langle\, u \, , t_r(a) \, w \,\big\rangle  \; = \;  \big\langle\, s^\ell(a) \, u \, , w \,\big\rangle  \; = \;  \big\langle\, a \triangleright u \, , w \,\big\rangle  \\
   \big\langle\, u \, , a \triangleright w \,\big\rangle  \; = \;  \big\langle\, u, s_r(a) \, w \,\big\rangle  \; = \;  \big\langle\, u \, s^\ell(a) \, , w \,\big\rangle  \; = \;  \big\langle\, u \blacktriangleleft a \, , w \,\big\rangle^{\phantom{\big|}}  \\
   \hskip3pt   \big\langle\, u \, , w \blacktriangleleft a \big\rangle  \; = \;  \big\langle u \, , w \, s_r(a) \,\big\rangle  \; = \;  \big\langle\, u \, t^\ell(a) \, , w \,\big\rangle  \; = \;  \big\langle a \blacktriangleright u \, , w \,\big\rangle^{\phantom{\big|}}  \\
   \hskip43pt   \big\langle\, u \, , a \blacktriangleright w \,\big\rangle  \; = \;  \big\langle\, u \, , w \, t_r(a) \,\big\rangle  \; = \;  a \, \big\langle\, u \, , w \,\big\rangle^{\phantom{\big|}}  \\
   \hskip45pt   \big\langle\, u \triangleleft a \, , w \,\big\rangle  \; = \;  \big\langle\, t^\ell(a) \, u \, , w \,\big\rangle  \; = \;  \big\langle\, u \, , w \,\big\rangle \, a{\phantom{\big\rangle^{\big|}}}
      \end{array}  $$
Then there exist natural morphisms of  $ A^e $--modules  $ \, W \longrightarrow U^* \, $  and  $ \, U \longrightarrow {}_*W \, $.  The pairing is  {\sl non degenerate}  if the left and right kernels of this pairing are trivial, that is to say
  $$  \big\langle\, u \, , w \,\big\rangle = 0 \; ,  \;\; \forall \;\, w \in W  \quad  =\joinrel\Longrightarrow  \quad  u = 0 \;\; ,  \qquad \quad
      \big\langle\, u \, , w \,\big\rangle = 0 \; ,  \;\; \forall \;\, u \in U  \quad  =\joinrel\Longrightarrow  \quad w = 0  $$
 In other words, the pairing is non degenerate if and only if the above maps  $ \, W \longrightarrow U^* \, $  and  $ \, U \longrightarrow {}_*W \, $  (which are morphisms  of  $ A^e $--modules)  are  {\sl injective}.
\end{definition}

\smallskip

\begin{definition}\label{left and right 2} {\ }
 \vskip5pt
   (a)  Let  $ \big( U, s_\ell, t_\ell, \Delta, \epsilon \big) $  be a left  $ A $--bialgebroid  and  $ \big( W, s^r, t^r, \Delta, \eta \big) $  be a right  $ A $--bialgebroid.  A  {\sl bialgebroid left pairing}  is a non degenerate $ A^e $--left  pairing  $ \; \big\langle \ ,\ \big\rangle : U \times W \longrightarrow A \; $  such that
  $$  \displaylines{
   \big\langle\, u\, u'\, , \, w \,\big\rangle  \; = \;  \Big\langle\, u \, , \, w_{(2)} \, t^r\big( \big\langle\, u' \, , \, w_{(1)} \,\big\rangle \big) \Big\rangle  \; = \;
\Big\langle\, u \, s_\ell\big( \big\langle\, u'\, , \, w_{(1)} \,\big\rangle \big) \, , \, w_{(2)} \Big\rangle  \quad ,  \qquad  \big\langle\, 1 \, , \, w \,\big\rangle  \, = \,  \eta(w)  \cr
   \big\langle\, u \, , \, w \, w' \,\big\rangle  \; = \;  \Big\langle\, t^\ell\big( \big\langle\, u_{(2)} \, , \, w \,\big\rangle \big) \, u_{(1)} \, , \, w' \,\Big\rangle  \; = \;
\Big\langle\, u_{(1)} \, , \, s_r\big( \big\langle\, u_{(2)} \, , \, w \,\big\rangle \big) \, w' \,\Big\rangle  \quad ,  \qquad  \big\langle\, u \, , \, 1 \,\big\rangle  \, = \,  \epsilon(u)  }  $$
for any  $ \, u $,  $ u' \in U \, $  and any  $ \, w $,  $ w' \in W \, $.  In other words, the natural maps  $ \, W \relbar\joinrel\longrightarrow U_* \, $  and  $ \, U \relbar\joinrel\longrightarrow {}^*W \, $  are (injective) morphisms of right and left bialgebroids respectively.

\vskip5pt

   (b)  Let  $ \big( U, s^\ell, t^\ell, \Delta, \epsilon \big) $  be a left  $ A $--bialgebroid  and  $ \big( W, s_r, t_r, \Delta, \eta \big) $  be a right  $ A $--bialgebroid.  A  {\sl bialgebroid right pairing}  is a non degenerate $ A^e $--right  pairing  $ \; \big\langle \ ,\ \big\rangle : U \times W \longrightarrow A \; $  such that
  $$  \displaylines{
   \big\langle\, u\, u'\, , \, w \,\big\rangle  \; = \;  \Big\langle\, u \, t_\ell\big( \big\langle\, u' \, , \, w_{(2)} \,\big\rangle \big) \, , \, w_{(1)} \Big\rangle  \; = \;
\Big\langle\, u \, , \, w_{(1)} \, s^r\big( \big\langle\, u'\, , \, w_{(2)} \,\big\rangle \big) \Big\rangle  \quad ,  \qquad  \big\langle\, 1 \, , \, w \,\big\rangle  \, = \,  \eta(w)  \cr
   \big\langle\, u \, , \, w \, w' \,\big\rangle  \; = \;  \Big\langle\, s^\ell\big( \big\langle\, u_{(1)} \, , \, w \,\big\rangle \big) \, u_{(2)} \, , \, w' \,\Big\rangle  \; = \;
\Big\langle\, u_{(2)} \, , \, t_r\big( \big\langle\, u_{(1)} \, , \, w \,\big\rangle \big) \, w' \,\Big\rangle  \quad ,  \qquad  \big\langle\, u \, , \, 1 \,\big\rangle  \, = \,  \epsilon(u)  }  $$
for any  $ \, u $,  $ u' \in U \, $  and any  $ \, w $,  $ w' \in W \, $.  In other words, the natural maps  $ \, W \relbar\joinrel\longrightarrow U^* \, $  and  $ \, U \relbar\joinrel\longrightarrow {}_*W \, $  are (injective) morphisms of right and left bialgebroids respectively.
\end{definition}

\smallskip

\begin{remarks}  \label{dual_vs._op-coop}  {\ }
 \vskip3pt
   {\it (a)}\,  If  $ U $  is a left bialgebroid, then the couple  $ \big(U,U^*\big) $  bears a bialgebroid right pairing, whereas  $ \big(U,U_*\big) $  bears a bialgebroid left pairing.
 \vskip3pt
   {\it (b)}\,  Let  $ U $  be a left bialgebroid.  Then the left bialgebroids  $ (U^*)^{\text{\it op}}_{\text{\it coop}} $  and  $ {}_*(U^{\text{\it op}}_{\text{\it coop}}) $  are isomorphic:  indeed, the right  $ A^e $--pairings  between  $ U $  and  $ U^* $  and between  $ {}_*\big( U^{\text{\it op}}_{\text{\it coop}} \big) $  and  $ U^{\text{\it op}}_{\text{\it coop}} $  give rise to the same formulas.  Similarly, the left bialgebroids  $ {(U_*)}^{\text{\it op}}_{\text{\it coop}} $  and  $ {}^*\big( U^{\text{\it op}}_{\text{\it coop}} \big) $  are isomorphic.
\end{remarks}

\smallskip

 \subsection{The jet space(s) of a Lie-Rinehart algebra}  \label{def_J(L)}

\smallskip

\begin{free text}  \label{def_J^r(L)}
 {\bf Bialgebroids of jets: the right version.}  Let  $ (L,A) $  be a Lie-Rinehart algebra, projective as an  $ A $--module.  Consider its enveloping algebra  $ V^\ell(L) $  endowed with its standard left bialgebroid structure and define the  {\sl right jet space\/}  of the Lie-Rinehart algebra  $ L $  as
  $$  J^r(L)  \; := \;  {V^\ell(L)}^*  \; = \;  \text{\sl Hom}_{(-,A)} \big( {V^\ell(L)}_\triangleleft \, , A_A \big)  $$
   \indent   As in  \S \ref{bialgd-struct-dual},  a multiplication in  $ J^r(L) $  can be given by
 $ \; \big( \phi \, \phi' \big)(u) \, = \, \phi\big(u_{(1)}\big) \, \phi'\big(u_{(2)}\big) \; $  for  $ \, \phi $,  $ \phi' \in J^r(L) \, $,  $ \, u \in V^\ell(L) \, $.
In particular, this multiplication is commutative, and the counit map of  $ V^\ell(L) $  is the unit element of  $ J^r(L) \, $.  Also, the map  $ \; \partial = \partial_{\scriptscriptstyle J^r(L)} : J^r(L) \longrightarrow A \, $,  $ \, \phi \mapsto \partial(\phi) := \phi(1_{\scriptscriptstyle V^\ell(L)}) \, $,  will play the role of counit map of  $ J^r(L) \, $;  hereafter, we write  $ \, \J_{\scriptscriptstyle J^r(L)} := \text{\sl Ker}\,(\partial_{\scriptscriptstyle J^r(L)}) \, $.  Moreover, we have a structure of  $ A^e $--ring  on  $ J^r(L) \, $,  whose source and target maps are given --- for all  $ \, a \in A \, $,  $ \, u \in V^\ell(L) \, $  ---   by the formulas  $ \; \big(s^r(a)\,\big)(u) := \epsilon\big( u \, s^\ell(a) \big) \; $,  $ \; \big(t^r(a)\big)(u) := \epsilon\big(s^\ell(a)\,u\big) = a \, \epsilon(u) \; $.  Note that  $ J^r(L) $  is complete for the  $ \, \J_{\scriptscriptstyle J^r(L)}$-adic  topology.

\smallskip

   To define a coproduct on  $ \, J^r(L) := {V^\ell(L)}^* \, $  we adapt the construction in  \S \ref{bialgd-struct-dual}  (cf.\ also  \cite{Kowalzig},  \cite{Kowalzig and Posthuma},  \cite{Calaque and VandenBergh}).  Consider the injective map
 $ \; \vartheta : {J^r\!(L)}_\blacktriangleleft \otimes_{{}_{\scriptstyle \blacktriangleright}}\! J^r\!(L) = {V^\ell\!(L)}^*_\blacktriangleleft \otimes_{{}_{\scriptstyle \blacktriangleright}}\! {V^\ell\!(L)}^* \! \longrightarrow {\big( {V^\ell\!(L)}_\triangleleft \!\otimes_{{}_\blacktriangleright}\!\! V^\ell\!(L) \big)}^{\!*} \; $
given (as in  \S \ref{bialgd-struct-dual})  by
 $ \;\; \psi \otimes \psi'  \, \mapsto \,  \vartheta\big( \psi \otimes \psi' \big) \Big(\, u \otimes u' \, \mapsto \, \vartheta\big( \psi \otimes \psi' \big)\big(u \otimes u'\big) \, := \, \psi\big( u' \, t^\ell(\psi'(u)) \big) \Big) \; $.
                                                               \par
   Consider in  $ \, {J^r\!(L)}_\blacktriangleleft \!\otimes_{{}_{\scriptstyle \blacktriangleright}}\! J^r\!(L) \, $  the  $ \J_\otimes $--adic  filtration, with
 $ \, \J_\otimes := \J_{\scriptscriptstyle J^r\!(L)} \otimes J^r\!(L) + J^r\!(L) \otimes \J_{\scriptscriptstyle J^r\!(L)} =\allowbreak
 \text{\sl Ker}\,\big( \partial_{\scriptscriptstyle J^r(L)} \otimes \partial_{\scriptscriptstyle J^r(L)} \big) \, $,
 and the corresponding topology defined by it in  $ \, {J^r(L)}_\blacktriangleleft \otimes_{\,{}_{\scriptstyle \blacktriangleright}}\! J^r(L) \, $;  then denote by  $ \, {J^r(L)}_\blacktriangleleft \widetilde{\otimes}_{\;{}_{\scriptstyle \blacktriangleright}}\! J^r(L) \, $  the  $ \J_\otimes $--adic  completion of  $ \, {J^r(L)}_\blacktriangleleft \otimes_{\;{}_{\scriptstyle \blacktriangleright}}\! J^r(L) \, $.  The completion of $\theta$,
 $ \; \widetilde\vartheta : {J^r(L)}_\blacktriangleleft \widetilde{\otimes}_{\,{}_{\scriptstyle \blacktriangleright}}\! J^r(L) \longrightarrow {\big( {V^\ell(L)}_\triangleleft \! \otimes_{\,{}_\blacktriangleright}\!\! V^\ell(L) \big)}^{\!*} \; $,
 is an isomorphism.  Therefore, we can complete the procedure explained in  \S \ref{bialgd-struct-dual}  and define a coproduct  $ \; \Delta : J^r(L) \relbar\joinrel\longrightarrow {J^r(L)}_\blacktriangleleft \widetilde{\otimes}_{\,{}_{\scriptstyle \blacktriangleright}} J^r(L) \; $  on  $ \, J^r(L) := {V^\ell(L)}^* \, $  as  $ \, \Delta := \widetilde{\vartheta}^{-1} \circ \nabla \, $  where
 $ \; \nabla : J^r(L) := {V^\ell(L)}^* \longrightarrow {\big( {V^\ell(L)}_\triangleleft \! \otimes_{\,{}_\blacktriangleright}\!\! V^\ell(L) \big)}^{\!*} \; $  is given by  $ \; \nabla : \psi \mapsto \nabla(\psi) \, \big(\, u \otimes u' \, \mapsto \, \psi\big(u'\,u\big) \big) \; $  for all  $ \, \psi \in J^r(L) \, $,  $ \, u, u' \in V^\ell(L) \, $.  As an outcome,
 we have
  $$  \Delta(\psi) = \psi_{(1)} \otimes \psi_{(2)} \, \in \, {J^r(L)}_\blacktriangleleft \widetilde{\otimes}_{\,{}_{\scriptstyle \blacktriangleright}} J^r(L)  \qquad \text{with} \qquad  \psi_{(1)}\big( u \, t^\ell\big( \psi_{(2)}(u') \big) \big) = \, \psi\big(u\,u'\big)  $$
This  $ \Delta $  makes  $ J^r\!(L) $  into a (topological)  $ A $--coring,  with counit map  $ \; \partial = \partial_{\scriptscriptstyle J^r\!(L)} : J^r\!(L) \longrightarrow \! A \, $  given as above by  $ \, \phi \mapsto \partial(\phi) := \phi(1_{\scriptscriptstyle V^\ell\!(L)}) \, $.  All in all, this makes  $ J^r\!(L) $  into a right bialgebroid over  $ A \, $.
\end{free text}

\smallskip

\begin{remarks}  \label{remarks_J^r(L)}
 We have to mention some extra features of the right bialgebroid  $ \, J^r\!(L) := {V^\ell(L)}^* \, $,  namely the following:
 \vskip3pt
   {\it (a)}\,  as  $ J^r(L) $  is commutative, it is equal to  $ J^r(L)^{\text{\it op}} $  hence it is also a  {\sl left\/}  bialgebroid;
 \vskip3pt
   {\it (b)}\,  it is known that  $J^r(L)$ is a  {\sl Hopf algebroid\/}  (see \cite{Kowalzig}, \cite{Calaque and VandenBergh}, \cite{Nest and Tsygan}): in particular, there exists a standard right bialgebroid isomorphism   --- called the ``antipode'' ---   from  $ J^r(L) $  to  $ \, {J^r(L)}_{\text{\it coop}} \; $.
\end{remarks}

\medskip

\begin{free text}  \label{def_J^ell(L)}
 {\bf Bialgebroids of jets: the left version.}  Let again  $ L $  be a Lie-Rinehart algebra over  $ A \, $,  again projective as an  $ A $--module.  Considering now  $ L $  as a  {\sl right\/}  $ A $--module,  we look at its  {\sl right\/}  enveloping algebra  $ V^r(L) $,  endowed with its natural structure of right bialgebroid  (cf.~\S \ref{V^r(L)_right-bialgd}).
                                                                   \par
   We define the  {\sl left jet space\/}  of the Lie-Rinehart algebra  $ L $  as the left dual space
  $$  J^\ell(L)  \; := \;  {}_*{V^r(L)}  \; = \;  \text{\sl Hom}_{(-,A)} \big( {V^r(L)}_\blacktriangleleft \, , A_A \big)  $$
   \indent   Again from  \S \ref{bialgd-struct-dual}  we have a multiplication in  $ J^\ell(L) $  given (for  $ \, \psi $,  $ \psi' \in J^\ell(L) \, $,  $ \, u \in V^r(L) \, $)  by  $ \; \big( \psi \, \psi' \big)(u) \, = \, \psi\big(u_{(1)}\big) \, \psi'\big(u_{(2)}\big) \; $;  in particular this multiplication is commutative in  $ J^\ell(L) \, $,  and the unit element of  $ J^\ell(L) $  is the counit map of  $ V^r(L) \, $.  Moreover, the map  $ \; \epsilon = \epsilon_{\scriptscriptstyle J^\ell(L)} : J^\ell(L) \longrightarrow A \, $,  $ \, \psi \mapsto \epsilon(\psi) := \psi(1_{\scriptscriptstyle V^r(L)}) \, $,  works as counit map of  $ J^\ell(L) \, $;  in the sequel we write  $ \, \J_{\scriptscriptstyle J^\ell(L)} := \text{\sl Ker}\,(\epsilon_{\scriptscriptstyle J^\ell(L)}) \, $.
                                                                   \par
   Still from  \S \ref{bialgd-struct-dual}  we get a structure of  $ A^e $--ring  on  $ J^\ell(L) \, $,  with source and target maps given by  $ \; \big(s^\ell(a)\big)(u) := \partial\big(\,a \, u\big) \; , \,\; \big(t^\ell(a)\,\big)(u) : =  \, \partial(u) \, a \; $,   --- for all  $ \, a \in A \, $,  $ \, u \in V^r(L) \, $.
                                                                   \par
   Finally,    we can also endow  $ J^\ell(L) $  with a suitable (topological) coproduct, just adapting the recipe given in  \S \ref{bialgd-struct-dual}.  Eventually, all this makes  $ J^\ell(L) $  into a (topological) left bialgebroid.
\end{free text}

\smallskip

\begin{remark}  \label{remark_J^ell(L)}
 As  $ \, J^\ell(L) = {J^r(L^{\text{\it op}})}^{\text{\it op}}_{\text{\it coop}} \; $,  it follows from  Remarks \ref{remarks_J^r(L)}{\it (b)\/}  that our  $ \, J^\ell(L) := {}_*{V^r(L)} \, $  is also a  {\sl Hopf algebroid\/}:  in particular, there exists a standard right bialgebroid isomorphism   --- the ``antipode'' of  $ J^\ell(L) $  ---   from  $ J^\ell(L) $  to  $ \, {J^\ell(L)}_{\!\text{\it coop}} \; $.
\end{remark}

\medskip

\begin{free text} \label{further_J(L)_comp}
 {\bf Further jet spaces, and comparison.}  Besides the jet spaces  $ J^r(L) $  and  $ J^\ell(L) \, $,  further possibilities exist.  All in all we can consider  {\sl four\/}  different types of ``jet bialgebroids'', namely
 \vskip5pt
   \centerline{ $ {V^\ell(L)}^* \, =: \, J^r(L)  \quad ,  \qquad  {}_*{V^r(L)} \, =: \, J^\ell(L)  \quad ,  \qquad  {V^\ell(L)}_* \, =: \, {}^r{\!}J(L)  \quad ,  \qquad  {}^*V^r(L) \, =: \, {}^\ell{\!}J(L) $ }
 \vskip4pt
\noindent
 One can also establish some relevant links among all these bialgebroids of jets: for instance, we have already noticed that  that  $ \, J^\ell(L) \cong {J^r(L^{op})}^{\text{\it op}}_{\text{\it coop}} \cong {J^r(L^{op})}_{\text{\it coop}}\; $  (cf.~also  Remark \ref{dual_vs._op-coop}).
We also saw above that  $ \; {V^\ell(L)}_* = {\big( {V^\ell(L)}^* \big)}_{\text{\it coop}} = {J^r(L)}_{\text{\it coop}} \cong J^r\big(L\big) \; $  (cf.~Remarks \ref{remarks_J^r(L)})  and  $ \, {}^*V^r(L) = {\big( {}_*{}V^r(L) \big)}_{\text{\it coop}} ={J^\ell(L)}_{\text{\it coop}} \cong J^\ell\big(L\big) \; $  (cf.~Remark \ref{remark_J^ell(L)}).  Thus, in the end, jet bialgebroids of type  $ J^r(L) $  or  $ J^\ell(L) $  are enough to consider all possible situations, for every possible  $ L \, $.
\end{free text}

\medskip

   We introduce now suitable ``topological duals'' for jet bialgebroids  $ J^r(L) $  and  $ J^\ell(L) \, $:

\bigskip

\begin{definition}
 Let  $ \, K = J^r(L) \, $  be a right jet bialgebroid, for some Lie-Rinehart algebra  $ L \, $.  Set
 $ \; I \, := \,  \big\{ \lambda \in J^r(L) \,\big|\, \langle\, 1 \, , \lambda \,\rangle \, = \, 0 \big\} \, = \, \text{\sl Ker}\,\big(\partial_{\scriptscriptstyle J^r(L)}\big) \; $ --- which is a (two-sided) ideal in  $ J^r(L) \, $,  as one easily sees.  Then we introduce the following subsets of  $ \, {}^*K $  and  $ \, {}_*K \; $:
  $$  {}^\star{}K \, := \, \big\{\, u \in {}^*K \,\big|\; u\big(I^n\big) = 0 \;\; \forall \;\; n \gg 0 \,\big\}  \quad ,  \qquad
      {}_\star{}K \, := \, \big\{\, u \in {}_*K \,\big|\; u\big(I^n\big) = 0 \;\; \forall \;\; n \gg 0 \,\big\}  $$
   Similarly, if  $ \, K := J^\ell(L) \, $  is a left jet bialgebroid, and  $ \, I := \text{\sl Ker}\,\big(\partial_{\scriptscriptstyle J^\ell(L)}\big) \, $,  we define
  $$  K^\star \, := \, \big\{\, u \in K^* \,\big|\; u\big(I^n\big) = 0 \;\; \forall \;\; n \gg 0 \,\big\}  \quad ,  \qquad
   K_\star \, := \, \big\{\, u \in K_* \,\big|\; u\big(I^n\big) = 0 \;\; \forall \;\; n \gg 0 \,\big\}  $$
\end{definition}

\smallskip

   It should be clear by the very definition that, in the first case,  $ {}^\star{}K \, $,  resp.~$ {}_\star{}K \, $,  is nothing but the subset of those functions in  $ {}^*K \, $,  resp.~in  $ {}_*K \, $,  which are  {\sl continuous\/}  with respect to the  $ I $--adic  topology in  $ {}^*K \, $,  resp.~in  $ {}_*K \, $,  and the discrete topology in  $ A \, $.  Similarly for  $ K^\star $  and  $ K_\star $  in the second case.  The key reason of interest for these objects lies in the following, well-known result:

\smallskip

\begin{theorem}  \label{J(L)star=V(L)}
 Let  $ L $  be a Lie-Rinehart algebra which, as an  $ A $--module,  is finite projective.
 \vskip4pt
   {\it (a)}\,  Consider the right bialgebroid  $ \, J^r(L) := {V^\ell(L)}^* \, $.  Then  $ \! \phantom{\big|}_{\star\,}J^r(L) \, $,  as a left bialgebroid, is isomorphic to  $ V^\ell(L) \, $:  more precisely, the canonical map  $ \; V^\ell(L) \relbar\joinrel\longrightarrow \!\! \phantom{\big|}_\star{\big( {V^\ell(L)}^* \big)} \, = \! \phantom{\big|}_{\star\,}J^r(L) \; $  given by evaluation is an iso\-mor\-phism of left bialgebroids.
                                                               \par
   Similarly, replacing  $ \, J^r(L) := {V^\ell(L)}^* \, $  with the right bialgebroid  $ {V^\ell(L)}_* $  one has a corresponding isomorphism of left bialgebroids  $ \; V^\ell(L) \relbar\joinrel\longrightarrow \!\! {}^{{}^{\scriptstyle \star}}\!\big( {V^\ell(L)}_* \big) \; $  still given by evaluation.
 \vskip3pt
   {\it (b)}\,  Consider the left bialgebroid  $ \, J^\ell(L) := \!\!\phantom{\big|}_*{}V^r(L) \; $.  Then  $ \! {J^\ell(L)}^\star \, $,  as a right bialgebroid, is isomorphic to  $ V^r(L) \, $:  more precisely, the canonical map  $ \; V^r(L) \relbar\joinrel\longrightarrow {\big(\!\! \phantom{\big|}_*{}V^r(L) \big)}^\star = {J^\ell(L)}^\star \; $  given by evaluation is an iso\-mor\-phism of right bialgebroids.
                                                               \par
   Similarly, replacing  $ \, J^\ell(L) := \!\!\phantom{\big|}_*{}V^r(L) \, $  with the left bialgebroid  $ {\phantom{|}}^*{V^r(L)} $  one has a corresponding isomorphism of right bialgebroids  $ \; V^r(L) \relbar\joinrel\longrightarrow \!\! {\big(\!\! {\phantom{|}}^*{V^r(L)} \big)}_\star \; $  still given by evaluation.
\end{theorem}

\smallskip

\begin{remark}  \label{diff_double-dual}
 The standard isomorphism between  $ \, J^r\!(L) := {V^\ell(L)}^* \, $  and  $ \, {J^r(L)}_{\text{\it coop}} = {V^\ell(L)}_* \, $  (see  Remarks \ref{remarks_J^r(L)}{\it (b)\/})  induces an isomorphism
 $ \phantom{\big|}^\star{\big( \big( {V^\ell(L)}^*  \big)} \cong V^\ell(L) \; $.  Similarly,
we have also an analogous isomorphism
 $ \phantom{\big|}{\big( {}_*\!{V^r(L)}  \big)}_\star \cong V^r(L) \; $.
\end{remark}

\smallskip

\begin{remark}  \label{def-J^r_f(L_Q)}
 Let  $ L $  be a finite projective Lie-Rinehart algebra and  $ Q $  be a (finite projective)  $ A $--module  such that  $ \, L \oplus Q = F \, $  is a finite rank free  $ A $--module.  We resume notation of  \S \ref{L-R_proj->free}:  so we take an  $ A $--basis  $ \; \{ b_1 \, , \dots \, , b_n \} \; $  of  $ F \, $,  and we set  $ \, Y = k\,b_1 \oplus \cdots \oplus k\,b_n \, $,  \, so that  $ \, F = A \otimes_k Y \; $;  moreover,  $ \; L_Q = L \oplus (A \otimes_k Z) \; $  is a Lie-Rinehart algebra with  $ \; Z = Y \oplus Y \oplus Y \oplus \cdots \;\, $.  One has  $ \; {S(Y)}^{\otimes \infty} := \, S(Z) \, = \, S(Y) \otimes S(Y) \otimes \cdots \; $  (recall that elements of an infinite tensor product of algebras are sums of tensor products with only finitely many factors different from  $ 1 \, $).  For  $ \, T \in \{Y,Z\} \, $,  we let  $ \; \epsilon : S(T) \longrightarrow k \; $   --- the counit map of  $ S(T) $  ---   be the unique  $ k $--algebra  morphism given by  $ \, S(t) := 0 \, $  for  $ \, t \in T \, $,  and we set  $ \, {S(T)}^+ := \text{\sl Ker}(\epsilon) \, $.
                                                                    \par
   For any  $ n \, $,  denote by  $ \; J^r_{f,n}(L_Q) \equiv {V^\ell(L_Q)}^*_{f,n} \; $  the subset of  $ \, {V^\ell(L_Q)}^* $  whose elements are all the  $ \, \lambda \in {V^\ell(L_Q)}^* \, $  such that  $ \; \lambda{\big|}_{V^\ell(L) \otimes {S(Y)}^{\otimes n} \otimes {S(Z)}^+} = \, 0 \; $  and set  $ \; J^r_f(L_Q) \equiv {V^\ell(L_Q)}^*_f := \bigcup_{n \in \N} J^r_{f,n}(L_Q) \; $.  Then one can describe  $ J^r_{f,n}(L_Q) $  as  $ \; J^r_{f,n}(L_Q) \, = \, J^r(L) \,\widetilde{\otimes}\, \widetilde{S}(Y^*)^{\widetilde{\otimes}^n} \widetilde{\otimes}\, 1 \,\widetilde{\otimes}\, 1 \,\widetilde{\otimes}\, \cdots \; $,  where  $ \widetilde{S}\big(Y^*\big) $  denotes the completion of  $ S\big(Y^*\big) $  with respect to the weak topology; so we have also
 \vskip4pt
   \centerline{ $  J^r_f(L_Q)  \,\; \cong \;\,  {\textstyle \sum_{n \in \N}}\, J^r_{f,n}(L_Q)  \,\; = \;\,  {\textstyle \sum_{n \in \N}}\, J^r(L) \,\widetilde{\otimes}\, \widetilde{S}(Y^*)^{\widetilde{\otimes}^n} \widetilde{\otimes}\, 1 \,\widetilde{\otimes}\, 1 \,\widetilde{\otimes}\, \cdots $ }
 \vskip7pt
\noindent
 This  $ J^r_f(L_Q) $  is a sub-bialgebroid of  $ J^r(L_Q) \, $:  indeed, its right bialgebroid structure is described by
  $$  \displaylines{
   s_r : A \longrightarrow  \; J^r_f(L_Q) \;\, ,  \quad  a \mapsto s_r(a) \otimes 1 \;\, ,   \qquad \qquad
 t_r : A \longrightarrow J^r_f(L_Q) \;\, ,  \quad  a \mapsto t_r(a) \otimes 1  \cr
   (\phi\otimes s) \, \big(\phi' \otimes s'\big) \, := \, \phi \, \phi' \otimes s \, s' \;\, ,  \quad  \Delta(\phi \otimes s) \, := \, \big( \phi_{(1)} \otimes s_{(1)} \big) \otimes
\big( \phi_{(2)} \otimes s_{(2)} \big) \;\, ,  \quad  \partial(\phi \otimes s) \, := \, \partial(\phi) \epsilon(s)  }  $$
for all  $ \, a \in A \, $,  $ \, \phi, \phi' \in J^r(L) \, $,  $ \, s, s' \in \sum_{n \in \N} \widetilde{S}(Y^*)^{\widetilde{\otimes}^n} \,\widetilde{\otimes}\, 1 \,\widetilde{\otimes}\, 1 \,\widetilde{\otimes}\, \cdots \, (n \in \N) \, $.
 \vskip3pt
   Last, let  $ \, {}^{\star_f}\!{J^r_f(L_Q)}\, $  be the subset of all  $ \, \delta \in \! {}^\star\!{J^r_f(L_Q)} \, $
   \hbox{such that  $ \, \delta{\big|}_{J^r(L) \,\widetilde{\otimes}\, {S(Y^*)}^{\widetilde{\otimes} n} \,\widetilde{\otimes}\, {S(Z^*)}^+} = 0 \, $  for}
%%%
 $ n \! \gg \! 0 \, $.  It is easy to see that  $ {}^{\star_f}\!{J^r_f(L_Q)} $  is a left sub-bialgebroid of  $ {}^\star{J^r_f(L_Q)} \, $,  isomorphic to  $ V^\ell(L_Q) \, $.
\end{remark}

\bigskip

\section{Quantum groupoids}  \label{q-groupds}

\smallskip

  {\ } \quad   In this section we introduce quantum groupoids   --- i.e.~topological bialgebroids which are formal deformations of those attached to Lie-Rinehart algebras.  Then we show that taking suitable ``(linear) duals'' we get an antiequivalence among the categories of objects of these two types.

\smallskip

\begin{free text}  \label{h-adic_topology}
 {\bf  The  $ h $--adic  topology.}  If  $ V $  is any  $ k[[h]] $--module,  it is endowed with the following decreasing filtration:  $ \; V \supseteq h\,V \supseteq h^2 V \supseteq \cdots \supseteq h^n V \supseteq h^{n+1} V \supseteq \cdots \supseteq \; $.  Then  $ V $  is also endowed with the  $ h $--adic  topology, which is the unique one for which  $ V $  is a topological  $ k[[h]] $--module  in which  $ {\big\{\, h^m V \big\}}_{m \in \N} $  is a basis of neighborhoods of  $ \, 0 \, $.  Indeed,  $ V $  is then a pseudo-metric space, as the  $ h $--adic  topology is the one induced by the following pseudo-metric:
  $$  d(x,y)  \, : = \,  \|x-y\|  \, = \,  2^{-m}   \quad  \text{with}  \quad  m := \sup\big\{ s \in \N \;\big|\; (x-y) \in h^s \, V \,\big\}   \eqno  \forall \;\, x, y \in V  \quad  $$
The topological space  $ V $  is Hausdorff if and only if the pseudo-metric  $ d $  is a metric: in turn, this occurs if and only if  $ \; \bigcap_{m \in \N} h^m V \, = \, \{0\} \; $,  which means that each point in  $ V $  forms a closed subset.
\end{free text}

\smallskip

 \subsection{Quantum groupoids}  \label{def_q-groupds}

\smallskip

   {\ } \quad   In this subsection we introduce the notion of ``quantum groupoids'': these are special ``quantum bialgebroids'', namely (topological) bialgebroids which are formal deformations of those of type  $ \, V^\ell(L) \, $,  $ V^r(L) \, $,  $ J^r(L) \, $  or  $ J^\ell(L) \, $.  We begin with the ones associated with the first two cases:

\smallskip

\begin{definition}  \label{def-QUEAd}
 A  {\sl left quantum universal enveloping algebroid}  (=LQUEAd) is a  {\sl topological}  left bialgebroid  $ \, \big( H_h \, , A_h \, , s^\ell_h \, , t^\ell_h \, , m_h \, , \Delta_h \, , \epsilon_h \big) \, $  over a topological  $ k[[h]] $--algebra  $ A_h $  such that:
 \vskip2pt
   (i)  $ \, A_h $  is isomorphic to  $ A[[h]] $  as a topological  $ k[[h]] $--module,  for some  $ k $--algebra  $ A \, $,  and this isomorphism induces an algebra isomorphism  $ \; A_h \big/ \, h \, A_h \, \cong \, A[[h]] \big/ \, h\,A[[h]] \, \cong \, A \; $;
 \vskip2pt
   (ii)  $ \, H_h $  is isomorphic to  $ V^\ell(L)[[h]] $  as a topological  $ k[[h]] $--module   where  $ V^\ell(L) $  is the left bialgebroid associated with some Lie-Rinehart  $ A $--algebra  $ L \, $,  as in  \S \ref{V^ell(L)_left-bialgd};
 \vskip2pt
   (iii)  $ \, H_h \big/ \, h\,H_h \, \cong \, V^\ell(L)[[h]] \big/ \, h\,V^\ell(L)[[h]] \; $  is isomorphic to  $ V^\ell(L) $  as a left  $ A $--bialgebroid  via the isomorphism  $ \; A_h \big/ \, h\,A_h \cong A[[h]] \big/ \, h\,A[[h]] \cong A \; $    mentioned in (i);
 \vskip2pt
   (iv)  denote by  $ \, H_{h\,\triangleleft} \! \mathop{\widehat{\otimes}}\limits_{\;A_h} {}_{\triangleright\,}H_h \, $  the completion of  $ \, H_{h\,\triangleleft} \! \mathop{\otimes}\limits_{\;A_h} {}_{\triangleright\,} H_h \, $  with respect to the  $ h $--adic  topology, and define the  ($ h $--adically  completed)  {\sl Takeuchi product}  as
 \vskip-9pt
  $$  H_{h\,\triangleleft} \! \mathop{\widehat{\times}}\limits_{\;A_h} {}_{\triangleright\,}H_h  \;\; := \;\;  \Big\{\, {\textstyle \sum_i} \, u_ i \otimes u'_i \in H_{h\,\triangleleft} \! \mathop{\widehat{\otimes}}\limits_{\;A_h} {}_{\triangleright\,}H_h \;\Big|\; {\textstyle \sum_i} \, (a \blacktriangleright u_i) \otimes u'_i = {\textstyle \sum_i} \, u_i \otimes \big( u'_i \blacktriangleleft a) \,\Big\}  $$
 \vskip-4pt
\noindent
 then the coproduct  $ \Delta_h $  of  $ H_h $  takes values in  $ \; H_{h\,\triangleleft} \! \mathop{\widehat{\otimes}}\limits_{\;A_h} {}_{\triangleright\,}H_h \; $.
 \vskip2pt
   In this setting, we shall say that  $ H_h $  is a  {\sl quantization},  or a  {\sl quantum deformation},  of  $ V^\ell(L) \, $;  we shall resume it in short using notation  $ \, {V^\ell(L)}_h := H_h \, $.
 \vskip4pt
   In a similar, parallel way, we define the notion of {\sl right quantum universal enveloping algebroid}  (=RQUEAd) as well, just replacing ``left'' with ``right'' and  $ V^\ell(L) $  with  $ V^r(L) \, $,  cf.~\S \ref{V^r(L)_right-bialgd}.
 \vskip5pt
   We define morphisms among left, resp.~right, quantum universal enveloping algebroids like in  Definition \ref{def-categ_left-algd};  moreover, we use notation  {\rm (LQUEAd)},  resp.~{\rm RQUEAd},  to denote the category of all left, resp.~right, quantum universal enveloping algebroids.  If  $ A_h $  is a fixed ground  $ k[[h]] $--algebra,  then we write  {\rm  $ \text{(LQUEAd)}_{A_h} \, $},  resp.~{\rm  $ \text{(RQUEAd)}_{A_h} \, $},  to denote the subcategory   --- in  {\rm (LQUEAd)},  resp.~{\rm (RQUEAd)}  ---   whose objects are all the left, resp.~right, quantum universal enveloping algebroids over  $ A_h \, $,  and whose morphisms are selected as in  Definition \ref{def-categ_left-algd}.
\end{definition}

\smallskip

\begin{remarks}  \label{remarks-QUEAd's}   {\ }
 \vskip3pt
   {\it (a)}\,  $ \;\; U $  is a LQUEAd  $ \,\iff\, $  $ U^{\text{\it op}}_{\text{\it coop}} $  is a RQUEAd  $ \,\iff\, $  $ U^{\text{\it op}} $  is a RQUEAd\;\,.
 \vskip3pt
   {\it (b)}\,  If  $ \big(\, {V^\ell(L)}_h \, , A_h \, , s^\ell_h \, , t^\ell_h \, , m_h \, , \Delta_h \, , \epsilon_h \big) $  is any LQUEAd, then  $ A_h $  is a deformation of the algebra  $ A \, $:  then, as usual, one can define a Poisson structure on the base algebra  $ A $  as follows:
  $$  \{f,g\}  \; := \;  {{\; f' *_h g' - g' *_h f' \;} \over h}  \mod h\,A_h   \eqno  \forall \;\; f, g \in A   \qquad  $$
where  $ \, f' \in A_h \, $  and  $ \, g' \in A_h \, $  are such that  $ \, f' \mod h\,A_h = f \, $  and   $ \, g' \mod h\,A_h = g \, $.  The same observations makes sense if one has to do with a RQUEAd  $ \, {V^r(L)}_h \; $.
 \vskip3pt
   {\it (c)}\,  The definitions given so far make sense for any Lie-Rinehart algebra  $ L \, $.  However, {\it from now on we shall assume in addition that  $ L \, $,  as an  $ A $--module,  is  {\sl finitely generated projective}}.
\end{remarks}

\smallskip

   The following theorem is proved in  \cite{Xu2}  (Theorem 5.16) :

\smallskip

\begin{theorem}  \label{semiclassical_limit-V^ell(L)}
 Let  $ \big( {V^\ell(L)}_h \, , A_h \, , s^\ell_h \, , t^\ell_h \, , m_h \, , \Delta_h \, , \epsilon_h \big) $  be a LQUEAd.  Define
  $$  \displaylines{
   \hfill   \delta(a)  \; := \;  {{\; t^\ell_h\big(a'\big) - s^\ell_h\big(a'\big) \;} \over h}  \mod h\,{V^\ell(L)}_h
        \hfill   \forall \;\; a \in A  \qquad  \cr
  \hfill   \delta(X)  \; := \;  {\Delta^{[1]}(X)}_{2,1} - \Delta^{[1]}(X)  \; \in \; V^\ell(L) \, \mathop{\otimes}_A \, V^\ell(L)
        \hfill   \forall \;\; X \in L  \qquad  \cr
  \text{with}  \qquad  \Delta^{[1]}(X)  \; := \;  {{\; \Delta_h\big(X'\big) - X' \otimes 1 - 1 \otimes X' \;} \over h}  \mod h \, \Big( {V^\ell(L)}_h \!\mathop{\widehat{\otimes}}_{\;A_h} {V^\ell(L)}_h \Big)  \cr
  \text{and}  \qquad  {\Delta^{[1]}(X)}_{2,1} \, := \, {\textstyle \sum_{[X]}}\, X_{[2]} \otimes X_{[1]}  \qquad  \text{if}  \qquad  \Delta^{[1]}(X) \, = \, {\textstyle \sum_{[X]}}\, X_{[1]} \otimes X_{[2]}  }  $$
where  $ \, X' \in {V^\ell(L)}_h \, $  is any lift of  $ X $  (i.e.~$ \, X' \mod h \, {V^\ell(L)}_h = X \, $)  and  $ \, a' \in A_h \, $  is any lift of  $ a \, $.
 \vskip4pt
   Then  $ \, \delta(a) \in L \, $  and  $ \, \delta(X) \in \bigwedge_A^2 \! L \, $;  this gives to  $ L $  the structure of a Lie-Rinehart bialgebra.  Also, the Poisson structure on  $ A $  induced by this Lie-Rinehart bialgebra  (cf.~Remarks \ref{props_Lie-bialg}{\it (c)\/})  coincides with the one obtained as the classical limit of the base $ * $--algebra  $ A_h \, $  (cf.~Remarks \ref{remarks-QUEAd's}{\it (b)\/}).
\end{theorem}

\smallskip

\begin{remark}
 In the above statement, we took formulas opposite to those in  \cite{Xu2}:  indeed, this allows us to deduce the very last claim.
\end{remark}

\smallskip

\begin{example}
 (cf.~\cite{Xu2})  Let  $ P $  be a smooth manifold,  $ D $  the algebra of global differential operators on  $ P $  and  $ \, A := {\mathcal C}^\infty(P) \, $.  Let  $ D[[h]] $  be the trivial deformation of  $  D \, $.  Let
  $$  \mathcal{F}  \; = \;  1 \otimes 1 + h \, B_1 + \cdots  \, \in \,  \big( D \otimes_A D \big)[[h]]  \; \cong \;  D[[h]] \,\widehat{\otimes}_{A[[h]]} D[[h]]  $$
be a formal series of bidifferential operators.  It is easy to see that  $ \mathcal{F} $  is a twistor  (cf.~Definition \ref{def-twistor})  iff the multiplication on  $ A[[h]] $ defined by
$ \; f *_h g \, = \, \mathcal{F}(f,g) \; $  for all  $ \, f $,  $ g \in A[[h]] \, $,
is asso\-ciative, with identity being the constant function  $ 1 \, $,  i.~e.~iff  $ \, *_h \, $  is a  {\sl star product\/}  on  $ P \, $.  The twisted bialgebroid structure on  $ \, D_h := D[[h]] \, $  can be easily described:  $ \, A_h = A[[h]] \, $  has the star product defined above,  $ \, s^\ell_h : A_h \longrightarrow D_h \, $  and  $ \, t^\ell_h : A_h \longrightarrow D_h \, $  are given by
 $ \; s^\ell_h(f) \, g \, = \, f \ast_h g \, $,  $ \; t^\ell_h(f) \, g = \, g \ast_h f \, $,  for  $ \, f $,  $ g \in A \, $,
the coproduct  $ \; \Delta_h : D_h \longrightarrow D_h \widehat{\otimes}_{A_h} D_h \; $  is  $ \; \Delta_h(x) := {\mathcal F}^{\# -1}\big(\Delta(x) \cdot {\mathcal F} \big) \; $  for  $ \, x \in D_h \, $.
 \vskip7pt
   In  Section \ref{example_twistor}  later on we shall explicitly provide a specific example of this kind.
\end{example}

\medskip

   Theorem \ref{semiclassical_limit-V^ell(L)}  has a natural counterpart for RQUEAd's as follows:

\medskip

\begin{theorem}  \label{semiclassical_limit-V^r(L)}
 Let  $ \big( {V^r(L)}_h \, ,  A_h \, , s^r_h \, , t^r_h \, , m_h \, , \Delta_h \, , \epsilon_h \big) $  be a RQUEAd.  Define
  $$  \displaylines{
   \hfill   \delta(a)  \; := \;  {{\; s^r_h(a) - t^r_h(a) \;} \over h}  \mod h \, {V^r(L)}_h
        \hfill   \forall \;\; a \in A  \qquad  \cr
  \hfill   \delta(X)  \; := \;  {\Delta^{[1]}(X)}_{2,1} - \Delta^{[1]}(X)  \; \in \; V^r(L) \, \mathop{\otimes}_A \, V^r(L)
        \hfill   \forall \;\; X \in L  \qquad  \cr
  \text{with}  \qquad  \Delta^{[1]}(X)  \; := \;  {{\; \Delta_h\big(X'\big) - X' \otimes 1 - 1 \otimes X' \;} \over h}  \mod h \, \Big( {V^r(L)}_h \!\mathop{\widehat{\otimes}}_{\;A_h} {V^r(L)}_h \Big)  \cr
  \text{and}  \qquad  {\Delta^{[1]}(X)}_{2,1} \, := \, {\textstyle \sum_{[X]}}\, X_{[2]} \otimes X_{[1]}  \qquad  \text{if}  \qquad  \Delta^{[1]}(X) \, = \, {\textstyle \sum_{[X]}}\, X_{[1]} \otimes X_{[2]}  }  $$
where  $ \, X' \in {V^r(L)}_h \, $  is any lift of  $ X $  (i.e.~$ \, X' \mod h \, {V^r(L)}_h = X \, $)  and  $ \, a' \in A_h \, $  is any lift of  $ a \, $.
 \vskip4pt
   Then  $ \, \delta(a) \in L \, $  and  $ \, \delta(X) \in \bigwedge_A^2 \! L \, $;  this gives to  $ L $  the structure of a Lie-Rinehart bialgebra.  Moreover, the Poisson structure induced on  $ A $  by this Lie-Rinehart  bialgebra  structure  is opposite to the one obtained as the classical limit of the base  $ * $--algebra  $ A_h \, $  (cf.~Remarks \ref{remarks-QUEAd's}{\it (b)\/}).
\end{theorem}

\smallskip

\begin{remark}
 The previous result can be proved just like  Theorem \ref{semiclassical_limit-V^ell(L)}  in  \cite{Xu2}.  Otherwise, one can deduce  Theorem \ref{semiclassical_limit-V^r(L)}  from  Theorem \ref{semiclassical_limit-V^ell(L)}  applied to  $ \, {V^\ell(L)}_h := {V^r(L)}_h^{\,\text{\it op}} \, $,  which is a LQUEAd   --- cf.~Remarks \ref{remarks-QUEAd's}{\it (a)}.  In particular, the Lie-Rinehart bialgebra structure induced on  $ L $  by the RQUEAd  $ {V^r(L)}_h $  is opposite to that induced by the LQUEAd  $ \, {V^\ell(L)}_h := {V^r(L)}_h^{\,\text{\it op}} \, $.
\end{remark}

\smallskip

   We introduce now a second type of ``quantum bialgebroids'', namely quantizations of jets:

\medskip

\begin{definition}  \label{def-QFSAd}
 A  {\sl right quantum formal series algebroid}  (=RQFSAd) is a  {\sl topological}  right bialgebroid  $ \, \big( K_h \, , A_h \, , s^r_h \, , t^r_h \, , m_h \, , \Delta_h \, , \partial_h \big) \, $  over a topological  $ k[[h]] $--algebra  $ A_h $  such that:
 \vskip2pt
   (i)  $ \, A_h $  is isomorphic to  $ A[[h]] $  as a topological  $ k[[h]] $--module,  for some  $ k $--algebra  $ A \, $,  and this isomorphism induces an algebra isomorphism  $ \; A_h \big/ \, h \, A_h \, \cong \, A[[h]] \big/ \, h\,A[[h]] \, \cong \, A \; $;
 \vskip2pt
   (ii)  $\, K_h $  is isomorphic to  $ J^r(L)[[h]] $  as a topological  $k [[h]] $--module   where  $ J^r(L) $  is the right bialgebroid of jets associated with some finite projective Lie-Rinehart  $ A $--algebra  $ L $  as in  \S \ref{def_J^r(L)};
 \vskip2pt
   (iii)  $ \, K_h \big/ \, h\,K_h \, \cong \, J^r(L)[[h]] \big/ \, h\,J^r(L)[[h]] \; $  is isomorphic to  $ J^r(L) $  as a right  $ A $--bialgebroid via the isomorphism  $ \; A_h \big/ \, h\,A_h \, \cong \, A[[h]] \big/ \, h\,A[[h]] \, \cong \, A \; $    mentioned in (i);
 \vskip2pt
   (iv)  letting  $ \; I_h \, := \, \big\{\, \psi \in K_h \,\big|\, \partial(\psi) \in h\, A_h \,\big\} \; $   --- which is easily seen to be a two-sided ideal in  $ K_h $  ---   we have that  $ K_h $  is complete in the  $ I_h $--adic  topology;
 \vskip2pt
%
%%%
%    (v)  let  $ \, I^\otimes_h := \big( I_h \otimes_{A_h} K_h + K_h \otimes_{A_h} I_h \big) \, $,
% denote by  $ \, K_{h\,\blacktriangleleft} \! \mathop{\widetilde{\otimes}} \limits_{\;A_h}
% {}_{\blacktriangleright\,}K_h \, $  the  $ I^\otimes_h $--adic  completion of  $ \,
% K_{h\,\blacktriangleleft} \! \mathop{\otimes}\limits_{\;A_h} {}_{\blacktriangleright\,} K_h \, $;
% also, define the  ($ I^\otimes_h $--adically  completed)  {\sl Takeuchi product}  as
%%%
%
   (v)  denote by  $ \, K_{h\,\blacktriangleleft} \! \mathop{\widetilde{\otimes}} \limits_{\;A_h} {}_{\blacktriangleright\,}K_h \, $  the completion of  $ \, K_{h\,\blacktriangleleft} \! \mathop{\otimes}\limits_{\;A_h} {}_{\blacktriangleright\,} K_h \, $  with respect to the topology defined by the filtration  $ \; {\Big\{ \sum_{p+q=n} I_h^{\,p} \otimes I_h^{\,q} \Big\}}_{n \in \N} \; $;  also, define the  {\sl Takeuchi product}  as
 \vskip-9pt
  $$  K_{h\,\blacktriangleleft} \! \mathop{\widetilde{\times}}\limits_{\;A_h} {}_{\blacktriangleright\,}K_h  \;\; := \;\;  \Big\{\, {\textstyle \sum_i} \, u_ i \otimes u'_i \in K_{h\,\blacktriangleleft} \! \mathop{\widetilde{\otimes}}\limits_{\;A_h} {}_{\blacktriangleright\,}K_h \;\Big|\; {\textstyle \sum_i} \, (a \triangleright u_i) \otimes u'_i = {\textstyle \sum_i} \, u_i \otimes \big( u'_i \triangleleft a) \,\Big\}  $$
 \vskip-4pt
\noindent
 then the coproduct  $ \Delta_h $  of  $ \, K_h $  takes values in
$ \; K_{h\,\blacktriangleleft} \! \mathop{\widetilde{\times}}\limits_{\;A_h}
{}_{\blacktriangleright\,}K_h \; $.
 \vskip2pt
   In this setting, we shall say that  $ K_h $  is a  {\sl quantization},  or a  {\sl quantum deformation},  of  $ J^r(L) \, $;  we shall resume it in short using notation  $ \, {J^r(L)}_h := K_h \, $.
 \vskip4pt
   In a similar, parallel way, we define the notion of {\sl left formal series algebroid}  (=LQFSAd) as well, just replacing ``left'' with ``right'' and  $ J^r(L) $  with  $ J^\ell(L) \, $.
 \vskip5pt
   We define morphisms among right, resp.~left, quantum formal series algebroids like in  Definition   \ref{def-categ_left-algd};  moreover, we use notation  {\rm (RQFSAd)},  resp.~{\rm LQFSAd},  to denote the category of all right, resp.~left, quantum formal series algebroids.  If  $ A_h $  is a fixed ground (topological)  $ k[[h]] $--algebra,  then we write  {\rm  $ \text{(RQFSAd)}_{A_h} \, $},  resp.~{\rm  $ \text{(LQFSAd)}_{A_h} \, $},  to denote the subcategory   --- in  {\rm (RQFSAd)},  resp.~{\rm (LQFSAd)}  ---   whose objects are all the right, resp.~left, quantum formal series algebroids over  $ A_h \, $,  and whose morphisms are selected as in  Definition \ref{def-categ_left-algd}.
\end{definition}

\smallskip

\begin{remarks}  \label{remarks-QFSAd's}  {\ }
 \vskip3pt
   {\it (a)}\,  From the analysis in  \S \ref{further_J(L)_comp}  we can argue that one could define a RQFSAd also as a deformation of the right bialgebroid  $ \, {V^\ell(L)}_* \, $,  and a LQFSAd as a deformation of the left bialgebroid  $ {}^*{V^r(L)} \, $.  On the other hand, the very conclusion of  \S \ref{further_J(L)_comp}  itself also tells us that it is enough to consider the notions of RQFSAd and LQFSAd introduced in  Definition \ref{def-QFSAd}  above.
 \vskip3pt
   {\it (b)}\,  $ \;\; K_h $ is a LQFSAd  $ \,\iff\, $  $ {\big( K_h^{\,\text{\it op}} \big)}_{\!\text{\it coop}} \, $  is a RQFSAd  $ \,\iff\, $  $ K_h^{\,\text{\it op}} $  is a RQFSAd\;\,.
\end{remarks}

\smallskip

\begin{free text}
 {\bf Further ``half Hopf'' structures on quantum groupoids.}  By construction, our quantum groupoids are just (left or right) bialgebroids, namely deformations of such (left or right) bialgebroids as  $ V^{\ell/r}(L) $  and  $ J^{\ell/r}(L) \, $.  However,  {\it  $ V^{\ell/r}(L) $  and  $ J^{\ell/r}(L) $  actually bear further structure, which ``automatically inheriteld'' by their quantizations too}.  To explain it, we fix some terminology.
                                                              \par
   Let
 $ \,\; {}_\blacktriangleright U \!\!\mathop{\otimes}\limits_{\,A^{\text{\it op}}}\! U_\triangleleft
\; := \; U \mathop{\otimes}\limits_k U \Big/ \text{\sl Span}\Big( {\big\{\, (a \blacktriangleright u) \otimes v - u \otimes (v \triangleleft a) \,\big\}}_{u, v \in U}^{\, a \in A} \Big) \;\, $
for some left bialgebroid  $ U $ over  $ A \, $;  then define a ``Hopf-Galois'' map
 $ \; {}_\blacktriangleright U \!\!\mathop{\otimes}\limits_{\,A^{\text{\it op}}}\! U_\triangleleft \;{\buildrel {\alpha_\ell} \over {\relbar\joinrel\longrightarrow}}\; U_\triangleleft \mathop{\otimes}\limits_A {}_\triangleright U  \; , \; u \otimes v \mapsto u_{(1)} \otimes u_{(2)}\,v \; $.
 Similarly, one can consider an analogous tensor product
 $ \,\; {}_{\triangleright\,} U \!\!\mathop{\otimes}\limits_{\,A^{\text{\it op}}}\! U_\blacktriangleleft \;\, $
 and a corresponding ``Hopf-Galois'' map
 $ \; {}_{\triangleright\,} U \!\!\mathop{\otimes}\limits_{\,A^{\text{\it op}}}\! U_\blacktriangleleft \;{\buildrel {\alpha_r} \over {\relbar\joinrel\longrightarrow}}\; U_\triangleleft \mathop{\otimes}\limits_A {}_\triangleright U  \; , \; u \otimes v \mapsto v_{(1)} \, u \otimes v_{(2)} \; $.
 On the other hand, for a right bialgebroid  $ W $  over  $ B \, $  one consider suitable tensor products
 $ \; {}_{\blacktriangleright\,} W \!\!\mathop{\otimes}\limits_{\,B^{\text{\it op}}}\! W_\triangleleft
\; $  and
 $ \; {}_{\triangleright\,} W \!\!\mathop{\otimes}\limits_{\,B^{\text{\it op}}}\! W_\blacktriangleleft \; $
 and Hopf-Galois maps
 $ \; {}_{\blacktriangleright\,} W \!\!\mathop{\otimes}\limits_{\,B^{\text{\it op}}}\! W_\triangleleft \;{\buildrel {\beta_\ell} \over {\relbar\joinrel\longrightarrow}}\; W_\blacktriangleleft \mathop{\otimes}\limits_B {}_\blacktriangleright W \; $  and
 $ \; {}_{\triangleright\,} W \!\!\mathop{\otimes}\limits_{\,B^{\text{\it op}}}\! W_\blacktriangleleft \;{\buildrel {\beta_r} \over {\relbar\joinrel\longrightarrow}}\; W_\blacktriangleleft \mathop{\otimes}\limits_B {}_\blacktriangleright W \; $
involving them.
 Then  $ U $  is called a  {\sl left},  resp.\ a  {\sl right},  {\sl Hopf left bialgebroid\/}  iff the map  $ \alpha_\ell \, $,  resp.~$ \alpha_r \, $,  is a bijection; similarly,  $ W $  is called a  {\sl left Hopf right bialgebroid},  respectively a  {\sl left},  resp.\ a  {\sl right},  {\sl Hopf right bialgebroid},  iff the map  $ \beta_\ell \, $,  resp.~$ \beta_r \, $,  is a bijection  (cf.\ \cite{Bohm2}, \cite{Kowalzig}, \cite{Kowalzig and Posthuma}, \cite{Kowalzig and Krahmer}).
                                                              \par
   It is known that  $ V^\ell(L) \, $,  resp.\  $ V^r(L) \, $,   is both a  {\sl left\/}  and  {\sl right Hopf left},  resp.\  {\sl right},  {\sl bialgebroid}.  The same holds for  $ J^\ell(L) $  and  $ J^r(L) $  too   --- actually because they have even stronger properties, namely they are  {\sl Hopf algebroids}, in the sense of B\H{o}hm-Szlach\`anyi  (cf.\ \cite{Bohm-S}, \cite{Bohm1}, \cite{Lu}, \cite{Khalkhali and Rangipour}, \cite{Kowalzig}).
                                                              \par
   Any quantum groupoid has its own Hopf-Galois maps, whose semiclassical specialization are the analogous maps for its semiclassical limit: e.g., the Hopf-Galois map  $ \alpha_\ell $  for any  $ \, {V^\ell(L)}_h \, $  yields by specialization the same name Hopf-Galois map for  $ V^\ell(L) \, $.  The latter map is bijective (since  $ V^\ell(L) $  is a left Hopf left bialgebroid) hence, by a standard argument, its deformation   --- the map  $ \alpha_\ell $  for  $  {V^\ell(L)}_h $  ---   is bijective too: thus in turn  $ {V^\ell(L)}_h $  is a left Hopf left bialgebroid as well!  With similar reasonings, looking all types of quantum groupoids we find that any  $ {V^\ell(L)}_h $  and any  $ {J^\ell(L)}_h $  are both right and left Hopf left bialgebroids, while any  $ {V^r(L)}_h $  and any  $ {J^r(L)}_h $  are both left and right Hopf right bialgebroids.
\end{free text}

\smallskip

\begin{free text}
 {\bf Liftings in a (R/L)QFSAd.}  Let  $ L $  be a Lie-Rinehart algebra which is finite projective (as an  $ A $--module).  Set  $ \, \J_{\scriptscriptstyle J^r(L)} := \text{\sl Ker}\,\big( \partial_{\scriptscriptstyle J^r(L)} \big) \, $:  then  $ \, \J_{\scriptscriptstyle J^r(L)} \big/ \J_{\scriptscriptstyle J^r(L)}^{\,2} \cong L^* \, $ as  $ A $--modules,  by definitions.  Given  $ \, \Phi \in L^* \, $,  we shall call a  {\sl lift of  $ \, \Phi $  in  $ J^r(L) $}  any  $ \, \phi \in \J_{\scriptscriptstyle J^r(L)} \, $  such that through the above isomorphism one has  $ \; \phi \mod \J_{\scriptscriptstyle J^r(L)}^{\,2} = \, \Phi \; $.  Now let  $ \, K_h = {J^r(L)}_h \, $  be a RQFSAd, deformation of  $ J^r(L) \, $.  For any  $ \, \Phi \in L^* \, $,  we shall call a  {\sl lift of  $ \, \Phi $  in  $ {J^r(L)}_h $}  any element  $ \, \phi' \in {J^r(L)}_h \, $
%%%
 such that\break
%%%
 $ \, \phi' \! \mod h \, {J^r(L)}_h $  is a lift of  $ \, \Phi $  in  $ J^r(L) \, $.  In short, this means  $ \, \Phi \, \equiv \big( \phi' \!\! \mod h \, {J^r(L)}_h \big) \mod \J_{\scriptscriptstyle J^r(L)}^{\,2} \, $.
                                                                      \par
   Also, if  $ \, a \in A \, $  we call a  {\sl lift of  $ a $  in  $ A_h $}  any  $ \, a' \in A_h \, $  such that  $ \, a' \!\mod h\,A_h = \, a \; $.
 \vskip3pt
   Changing ``right'' into ``left'', similar remarks and terminology apply for defining ``lifts'' of elements of  $ J^\ell(L) $  in some associated LQFSAd, say  $ \, {J^\ell(L)}_h \; $.
\end{free text}

\smallskip

   Next result introduces semiclassical structures induced on a Lie-Rinehart algebra  $ L $  by quantizations of the form  $ J^r(L) $  or  $ J^\ell(L) \, $.  Indeed, this is the dual counterpart of  Theorem \ref{semiclassical_limit-V^ell(L)}.

\medskip

\begin{theorem}  \label{semiclassical_limit-J^r(L)}
 Let  $ {J^r(L)}_h $  be a RQFSAd, namely a deformation of  $ J^r(L) $  as above.  Then  $ L $  inherits from this quantization a structure of Lie-Rinehart bialgebra, namely the unique one such that the Lie bracket and the anchor map of  $ \, L^* $  are given (notation as above) by
  $$  \displaylines{
    [\Phi\,,\Psi\,]  \; := \;  \left( {{\;\phi' \, \psi' - \psi' \, \phi'\;} \over h}  \mod h \, {J^r(L)}_h \right)  \!\mod \J_{\scriptscriptstyle J^r(L)}^{\,2}  \cr
   \omega(\Phi)(a)  \, :=  \left(\! {{\;\phi' \, r\big(a'\big) - r\big(a'\big) \, \phi'\;} \over h}  \hskip-5pt  \mod h \, {J^r\!(L)}_h \!\right)  \hskip-9pt  \mod \J_{\scriptscriptstyle J^r\!(L)}  =
    \partial \! \left(\! {{\;\phi' \, r\big(a'\big) - r\big(a'\big) \, \phi'\;} \over h} \!\right)  \hskip-9pt  \mod h A_h  }  $$
for all  $ \, \Phi , \Psi  \in L^* \, $  and  $ \, a \in A \, $,  where  $ \phi' $  and  $ \psi' $  are liftings in  $ {J^r(L)}_h $  of  $ \, \Phi $  and  $ \, \Psi $  respectively,  $ a' $  is a lifting in  $ A_h $  of  $ \, a \in A \, $,  and finally  $ \, r\big(a'\big) \, $  stands for either  $ s^r_h\big(a'\big) $  or  $ \, t^r_h\big(a'\big) \, $.
\end{theorem}

\begin{proof}
 First, it is easy to see that the maps  $ \, [\,\ ,\ ] \, $  and  $ \omega $  as given in the statement are well-defined, i.e.~they do not depend on the choice of liftings, nor of the choice of either of  $ s^r_h\big(a'\big) $  or  $ \, t^r_h\big(a'\big) \, $  acting as  $ r\big(a'\big) \, $.  Moreover, by construction we have  $ \, [\Phi\,,\Psi\,] \in \J_{\scriptscriptstyle J^r(L)} \big/ \J_{\scriptscriptstyle J^r(L)}^{\,2} \cong L^* \, $.  Also, again by construction we have  $ \, \omega(\Phi)(a) \in J^r(L) \big/ \J_{\scriptscriptstyle J^r(L)} \, $;  now the latter space identifies with  $ \, \partial\big(J^r(L)\big) = A \, $,  thus  $ \, \omega(\Phi)(a) \in A \, $  via these identifications, so that  $ \, \omega(\Phi) \, $  is a  $ k $--linear endomorphism of  $ A \, $.
                                                                             \par
   Now, the definition of both  $ [\,\ ,\ ] $  and  $ \omega $  is made via a commutator in  $ {J^r(L)}_h \, $.  As the commutator   --- in any associative  $ k $--algebra ---   is a  $ k $--bilinear  Lie bracket and satisfies the Leibniz identity (involving the associative product), one can easily argue at once from definitions that  $ L^* $  with the given bracket and anchor map is indeed a Lie-Rinehart algebra (over  $ A $).
                                                                      \par
   What is more demanding is to prove that with this structure the pair  $ \big(L\,,L^*\big) $  of Lie-Rinehart  $ A $--algebras  fulfills all constraints to be a Lie-Rinehart bialgebra.
  Indeed, we shall  {\sl not\/}  provide a  {\sl direct\/}  proof for that:
 instead, we have recourse to a duality argument, using the notions and results of  Subsec.~\ref{lin-dual_QUEAd's}  later on.  Indeed, there we shall see that  $ {}_\star{}J^r(L)_h $  is a LQUEAd, hence by  Theorem \ref{semiclassical_limit-V^ell(L)}  we know that  $ \big(L\,,L^*\big) $  is a Lie-Rinehart bialgebra.
\end{proof}

\medskip

   The analogue of  Theorem \ref{semiclassical_limit-J^r(L)}  for LQFSAd's (with essentially the same proof) is the following:

\medskip

\begin{theorem}  \label{semiclassical_limit-J^ell(L)}
 Let  $ {J^\ell(L)}_h $  be a LQFSAd, namely a deformation of  $ J^\ell(L) \, $.  Then  $ L $  inherits from this quantization a structure of Lie-Rinehart bialgebra, namely the unique one for which the Lie bracket and the anchor map of  $ \, L^* $  are given (notation as above) by
  $$  \displaylines{
   [\Phi\,,\Psi\,]  \; := \;  \left( {{\;\phi' \, \psi' - \psi' \, \phi'\;} \over h}  \mod h \, {J^\ell(L)}_h \right)  \!\mod \J_{\scriptscriptstyle J^r(L)}^{\,2}  \cr
   \omega(\Phi)(a)  \; := \;  \left( {{\;\phi' \, r\big(a'\big) - r\big(a'\big) \, \phi'\;} \over h}  \mod h \, {J^\ell(L)}_h \right)  \!\mod \J_{\scriptscriptstyle J^r(L)}  }  $$
for all  $ \, \Phi \, , \Psi \in L^* \, $  and  $ \, a \in A \, $,  where  $ \phi' $  and  $ \psi' $  are liftings in  $ {J^\ell(L)}_h $  of  $ \, \Phi $  and  $ \, \Psi $  respectively,  $ a' $  is a lifting in  $ A_h $  of  $ \, a \in A \, $,  and finally  $ \, r\big(a'\big) \, $  stands for either  $ s^r_h\big(a'\big) $  or  $ \, t^r_h\big(a'\big) \, $.
\end{theorem}

\smallskip

\begin{remark}
 The result above can be proved like its analogue for RQFSAd's,  i.e.~Theorem \ref{semiclassical_limit-J^r(L)}.  Otherwise, one can get the former from the latter applied to  $ \, {J^r(L)}_h := {J^\ell(L)}_h^{\,\text{\it op}} \, $,  which is a RQFSAd   --- cf.~Remarks \ref{remarks-QFSAd's}{\it (c)}.  In particular, the Lie-Rinehart bialgebra structure induced on  $ L $  by the LQFSAd  $ {J^\ell(L)}_h $  is opposite-coopposite
 to that induced by the RQFSAd  $ \, {J^r(L)}_h := {J^\ell(L)}_h^{\,\text{\it op}} \, $.
\end{remark}

\smallskip

 \subsection{Extending quantizations: from the finite projective to the free case} \label{quant_proj-free}

\smallskip

   {\ } \quad   Let  $ L $  be a Lie-Rinehart algebra over  $ A $  which is finite projective as an  $ A $--module.  With the procedure presented in  \S \ref{L-R_proj->free},  we can find a projective  $ A $--module  $ Q $  (a complement of  $ L $  in a finite free  $ A $--module  $ F $)  and use it to build a new Lie-Rinehart algebra  $ \, L_Q := L \oplus \big( Q \oplus L \oplus Q \oplus L \oplus \cdots \big) \, = L\oplus R $,  which as an  $ A $--module  is free.  Then we fix an  $ A $--basis  $ \, \{b_1,\dots,b_n\} \, $  of  $ F $  from which we construct a good basis   $ \, {\{e_i\}}_{i \in T} \, $  of  $ L_Q \, $  and a good basis  $ \, {\{v_t\}}_{t \in T} \, $  of  $ R \, $.  Set  $ \, Y := \mathop{\oplus}\limits_{i=1}^n k\,b_i \, $  so that  $ \, F \! = \! A \! \otimes_k \! Y \, $,
 $ \, T \! := \N \times \{1,\dots,n\} \, $,  $ \, Z \! := \! \mathop{\oplus}\limits_{t \in T} \! k \, v_t \, $,  hence  $ \, R \! = \! A \! \otimes_k \! (Y \! \oplus Y \! \oplus \cdots) = A \! \otimes_k \! Z \; $.  Moreover, one has also  $ \, V^\ell(L_Q) = V^\ell(L) \otimes_k S(Z) \, $  with  $ \; S(Z) = S(Y) \otimes S(Y) \otimes \cdots \;\, $.

\medskip

\begin{free text}  \label{ext-QUEAd's}
 {\bf Extending QUEAd's.}  Let  $ L $  be a finite projective Lie-Rinehart algebra, for which we consider for it all the objects and constructions mentioned just above.
 \vskip4pt
   Let  $ \, {V^\ell(L)}_h \in {\text{(LQUEAd)}}_{A_h} \, $  be a (left) quantization of the left bialgebroid  $ V^\ell(L) \; $.  Consider
  $$  {V^\ell(L)}_{h,Y}  \,\; := \;\,  \text{$ h $--adic completion of}
\;\; {V^\ell(L)}_h \otimes_k S(Y) \otimes_k S(Y) \otimes \,\cdots  \,\; =: \;\, {V^\ell(L)}_h \widehat{\otimes}_k \, S(Z)  $$
In order to describe it, for  $ \, \underline{d} \in T^{(\N)} \, $  we set
 $ \; e^{\underline{d}} := \prod_{t\in T} e_t^{\underline{d}(t)} \, $  and  $ \; \varpi\big(\underline{d}\big) := \max \big\{ \varpi(e_t) \,\big|\, \underline{d}(t) \neq 0 \,\big\} \; $
(cf.~Definition \ref{def_good-basis});

\smallskip

\begin{proposition}
 Any element of  $ {V^\ell(L)}_{h,Y} $  can be written in a unique way as
 \vskip5pt
   \centerline{ $ \sum\limits_{\underline{d} \in T^{(\N)}} \hskip-3pt t^\ell\big( a_{\underline{d}} \big) \, e^{\underline{d}} \; = \lim\limits_{n \to +\infty} \hskip-11pt \sum\limits_{\stackrel{\underline{d} \in T^{(\N)}}
   {\mid \underline{d} \mid + \varpi(\underline{d}) \leq n}} \hskip-13pt
   t^\ell\big( a_{\underline{d}} \big) \, e^{\underline{d}} $  \qquad  with  \quad  $ \lim\limits_{\mid \underline{d} \mid +\varpi(\underline{d}) \to +\infty} \big\|a_{\underline{d}}\big\| = 0 $
   \quad  (notation of\/ \S \ref{h-adic_topology}) }
\end{proposition}

\begin{proof}
 It is obvious that any element of the given lies in  $ V(L)_{hY} \, $.  Conversely, let  $ \, u \in V(L)_{hY} \, $.  Write  $ \, u = u_0 + h \, u_1 + \cdots + h^n u_n + \cdots \, $  with  $ \, u_i \in V(L_Q) \, $  for all  $ i \, $.  Now, for all  $ \, i \in \N \, $,  each  $ u_i $  can be written as  $ \; u_i = \sum_{\underline{\alpha} \in \N^{(T)}} t^\ell\big(u_i^{\underline{\alpha}}\big) \, e^{\underline{\alpha}} \; $  where all but a finite number of the  $ u_i^{\underline{\alpha}} \, $'s  are zero.  Set  $ \, u_{\underline{\alpha}} := \sum_i h^i u_i^{\underline{\alpha}} \, $,  so  $ \, u = \sum_{\underline{\alpha}} u_{\underline{\alpha}} \, e^{\underline{\alpha}} \;\, $;  we show that
 $ \lim\limits_{\mid \alpha \mid +\varpi(\underline{\alpha}) \to +\infty}
 \big\|u_{\underline{\alpha}}\big\| = 0 \; $.  Pick  $ \, n_0 \in \N \, $;  choosing  $ \, A > \max \big\{\mid \underline{\alpha} \mid + \omega(\underline{\alpha}) \,\big|\, \exists \, i \leq n_0 : u_i^{\underline{\alpha}} \not= 0 \,\big\} \, $,  \, for any  $ \, |\underline{\alpha}| + \omega(\underline{\alpha}) > A \, $  we have  $ \, \big\| u_{\underline{\alpha}} \big\| < 2^{-n_0} \; $.
\end{proof}

\medskip

   Now, there exists a unique left bialgebroid structure on  $ {V^\ell(L)}_{h,Y} $  given as follows:
 $$  \displaylines{
   t^{\scriptscriptstyle Y}_\ell : A_h \longrightarrow V^\ell(L)_{h,Y} \; ,  \quad  a \mapsto t_\ell(a) \otimes 1 \; ,  \qquad \qquad
      s^{\scriptscriptstyle Y}_\ell : A_h \longrightarrow V^\ell(L)_{h,Y} \; ,  \quad  a \mapsto s_\ell(a) \otimes 1  \cr
   \Delta_{V^\ell(L)_{h,Y}}(a \otimes s) \, := \, \big( a_{(1)} \otimes s_{(1)} \big) \otimes
\big( a_{(2)} \otimes s_{(2)} \big) \;\;\;\; {\rm if}\;\;\;\;
\Delta_h(a) = a_{(1)}\otimes a_{(2)} \; , \;\;
\Delta_{S(Z)}(s) = s_{(1)} \otimes s_{(2)}  \cr
   \epsilon_h(a \otimes s) \, := \, \epsilon_h(a) \epsilon(s) \; ,  \qquad
\qquad  (a\otimes s)\big(a' \otimes s'\big) \, := \, a \, a' \otimes s \, s'  }  $$
where the right-hand side factor map  $ \epsilon $  above is just the standard counit map  $ \, \epsilon: S(Z) \longrightarrow k \, $  of the Hopf  $ k $--algebra  $ S(Z) \, $,  uniquely determined by  $ \, \epsilon(z) = 0 \, $  for every  $ \, z \in Z \, $.  It is easy to see that
 \vskip4pt
   \quad\hskip-15pt  {\it  (a)  $ \; {V^\ell(L)}_{h,Y} $  {\it is a quantization of the left bialgebroid}  $ \, V^\ell(L_Q) \, $};
 \vskip3pt
   \quad\hskip-15pt  {\it  (b)  $ \; \pi_{\scriptscriptstyle Y} \! := \text{\sl id}_{{V^\ell(L)}_{h,Y}} \! \widehat{\otimes}_k \, \epsilon : {V^\ell(L)}_{h,Y} \! := {V^\ell(L)}_h \widehat{\otimes}_k \, S(Z) \rightarrow {V^\ell(L)}_h \; $  is an epimorphism of left bialgebroids}.
 \vskip7pt
   {\sl A similar construction is possible if we take a RQUEAd  $ {V^r(L)}_h $  instead of the LQUEAd  $ {V^\ell(L)}_h \, $.}
\end{free text}

\smallskip

\begin{free text}  \label{ext-QFSAd's}
 {\bf Extending QFSAd's.}  Let  $ L $  be a finite projective Lie-Rinehart algebra, and adopt again notations as before.  Recall also that in  Remark \ref{def-J^r_f(L_Q)}  we have introduced  $ \, J^r_f(L_Q) := {V^\ell(L_Q)}^{*_f} \, $.
 \vskip4pt
   Let  $ \, {J^r(L)}_h \in \text{(RQFSAd)}_{A_h} \, $  be a quantization of  $ J^r(L) \, $.  Keeping notation as in  \S \ref{ext-QUEAd's},  consider
  $$  {J^r(L)}_{h,Y}  \; := \;  h\text{--adic completion of \ } {\textstyle \sum_{n \in \N}} \, {J^r(L)}_h \widetilde{\otimes}_k\, {S(Y^*)}^{\widetilde{\otimes}\, n} \otimes 1 \otimes 1 \otimes 1 \cdots  $$
where  $ \, J^r(L)_h \widetilde{\otimes} {S(Y^*)}^{\widetilde{\otimes}\, n} \otimes 1 \otimes 1 \otimes \cdots \, $  is the  $ \, \big( {\big( {S(Y^*)}^{\otimes n} \big)}^+ \! \otimes 1 \otimes 1 \otimes 1 \cdots \big) $--adic  completion of  $ \, J^r(L)_h \otimes  {S(Y^*)}^{\otimes n} \otimes 1 \otimes 1 \otimes 1 \cdots \; $.  There exists a unique right bialgebroid structure on  $ {J^r(L)}_{h,Y} $
  $$  \displaylines{
   t^{\scriptscriptstyle Y}_r : A \longrightarrow J^r(L)_{h,Y} \; ,  \quad  a \mapsto t_r(a) \otimes 1 \; ,   \qquad
      s^{\scriptscriptstyle Y}_r : A_h \longrightarrow  \; J^r(L)_{h,Y},  \quad  a \mapsto s_r(a) \otimes 1  \cr
   (a\otimes s)(a' \otimes s') \, := \, a \, a' \otimes s \, s' \;\, ,  \;\quad  \Delta(a \otimes s) \, := \big( a_{(1)} \otimes s_{(1)} \big) \otimes \big( a_{(2)} \otimes s_{(2)} \big) \;\, ,  \;\quad  \partial_h(a \otimes s) \, := \, \partial_h(a) \, \epsilon^*\!(s)  }  $$
Then one easily sees that
 \vskip4pt
   \quad\hskip-19pt  {\it  (a)  $ \; {J^r(L)}_{h,Y} $  {\it is a quantization of the right bialgebroid}  $ \, J^r_f(L_Q) \, $};
 \vskip3pt
   \quad\hskip-19pt  {\it  (b)  $ \; \pi^{\scriptscriptstyle Y} \! := \text{\sl id}_{{J^r(L)}_{h,Y}} \!\! \otimes_k \epsilon^* : {J^r(L)}_{h,Y}  \rightarrow {J^r(L)}_h \; $  is an epimorphism of right bialgebroids}.
 \vskip7pt
   {\sl An entirely similar construction is possible if  $ {J^r(L)}_h $  is replaced with a LQFSAd  $ {J^\ell(L)}_h \, $}.
 \vskip11pt
   {\sl $ \underline{\text{Remark}} $:}\,  Let  $ \, {V^\ell(L)}_h \in \text{(LQUEAd)}_{A_h} \, $  be a quantization of  $ V^\ell(L) \, $.  We have seen in  \S \ref{ext-QUEAd's}  that  $ \, {V^\ell(L)}_{h,Y} := {V^\ell(L)}_h \,\widehat{\otimes}_k \, S(Z) \, $  is a LQUEAd which quantizes  $ V^\ell(L_Q) \, $.  If  $ \, n \in \N \, $,  let  $ \, {S(Z)}^+ := \text{\sl Ker}\,\big( \epsilon : S(Z) \longrightarrow k \big) \, $  be the kernel of the counit of  $ S(Z) \, $,  and let  $ \, {V^\ell(L)}_{h,Y}^{\;*_{f,n}} \, $  be the subspace of  $ \, {V^\ell(L)}_{h,Y}^{\;*} \, $  given by  $ \; {V^\ell(L)}_{h,Y}^{\;*_{f,n}} := \big\{\, \lambda \in {V^\ell(L)}_{h,Y}^{\;*} \,\big|\; \lambda\big( {V^\ell(L)}_h \otimes_k {S(Y)}^{\otimes n} \otimes_k {S(Z)}^+ \big) = 0 \,\big\} \; $.  Then set
  $$  {V^\ell(L)}_{h,Y}^{\;*_f}  \; := \;  h\text{--adic completion of}\;\,
{\textstyle \sum}_{n \in \N} {V^\ell(L)}_{h,Y}^{\;*_{f,n}}  $$
Then  $ {V^\ell(L)}_{h,Y}^{\;*_f} $  is a right subbialgebroid of  $ {V^\ell(L)}_{h,Y}^{\;*} \, $,  which is isomorphic to the right bialgebroid  $ {\big( {V^\ell(L)}_h^* \big)}_Y \, $.  Note also that  $ {V^\ell(L)}_{h,Y}^{\;*_{f,n}} $  is isomorphic to
 $ \; {V^\ell(L)}_h^* \, \widetilde{\otimes}_k \, {S(Y^*)}^{\widetilde{\otimes} n}
\otimes_k 1 \otimes_k 1 \otimes_k 1 \cdots \; $.
 \vskip7pt
   In a similar way, one can define also the right bialgebroid  $ {\big( {V^\ell(L)}_{h,Y} \big)}_{*_f} \, $:  this is a right subbialgebroid of  $ {\big( {V^\ell(L)}_{h,Y} \big)}_* \, $  isomorphic to the right bialgebroid  $ {\big( {\big( {V^\ell(L)}_h \big)}_* \big)}_Y \, $.

 \vskip9pt

   Parallel ``right-handed versions'' of the previous constructions and results also make sense if one starts with some  $ \, {V^r(L)}_h \in \text{(RQUEAd)}_{A_h} \, $  instead of  $ \, {V^\ell(L)}_h \in \text{(LQUEAd)}_{A_h} \, $:  in a nutshell, one still finds that ``extension commutes with dualization''.  Details are left to the reader.
\end{free text}

\bigskip

\section{Linear duality for quantum groupoids}  \label{lin-dual_q-grpds}

\smallskip

   {\ } \quad   In this section we explore the relationship among quantum groupoids ruled by linear duality (i.e., by taking left or right duals).  We shall see that the ``(left/right) full dual'' and the ``(left/right) continuous dual'' altogether provide category antiequivalences between  $ \text{\rm (LQUEAd)}_{A_h} $  and  $ \text{\rm (RQFSAd)}_{A_h} $  and between  $ \text{\rm (RQUEAd)}_{A_h} $  and  $ \text{\rm (LQFSAd)}_{A_h} \, $.
                                                                                             \par
   Essentially, we implement the construction of ``dual bialgebroids'' presented in  Subsection \ref{dual_bialgds},  but still we need to make sure that several technical aspects do turn round.

\smallskip

 \subsection{Linear duality for QUEAd's}  \label{lin-dual_QUEAd's}

\smallskip

   {\ } \quad   We begin with the construction of duals for (L/R)QUEAD's.  In this case, we consider ``full duals'' (versus topological ones,  cf.~Subsection \ref{lin-dual_QFSAd's}  later on.  Before giving the main result, we need a couple of auxiliary, technical lemmas.

\smallskip

\begin{lemma}  \label{property of the coproduct}
 Let  $ \, {V^\ell(L)}_h \in \text{\rm (LQUEAd)}_{A_h} \, $  and  $ \, u \in {V^\ell(L)}_h \, $.  For any  $ \, r \in \N \, $,  there exists  $ \, t_r \in \N \, $  such that
 $ \; \Delta^{t_r}(u) \, = \, \delta_0 + h \, \delta_1 + h^2 \, \delta_2 + \dots + h^{r-1} \, \delta_{r-1} + h^r \, \delta_r \; \Big(\! \in {V^\ell(L)}_h^{\widehat{\otimes}\, t_r} \Big) \; $
and, for any  $ \, i = 0, \dots , r\!-\!1 \, $,  each homogeneous tensor in an expansion of  $ \, \delta_i $  has at least  $ r $  factors equal to  $ 1 \, $.
\end{lemma}

\begin{proof}
 For any  $ \, w \in {V^\ell(L)}_h^{\otimes s} \, $,  we denote by  $ \overline{w} $  the coset of  $ w $  modulo  $ \, h \, {V^\ell(L)}_h^{\otimes s} \, $.
                                                                  \par
   We expand the given  $ u $  as  $ \; u \, = \, u_0 + h \, u_1 + \cdots + h^r \, u_r + \cdots \; $. Then there exists  $ \, t'_0 \in \N \, $  such that (each homogeneous tensor in)  $ \, \overline{\Delta^{t'_0}(u_0)} \, $  contains at least  $ r $  terms equal to  $ 1 \, $.  We lift  $ \, \overline{\Delta^{t'_0}(u_0)} \, $  to some  $ \, \delta_0^0 \in {V^\ell(L)}_h^{\widehat{\otimes}\, t'_0} \, $  containing (i.e., its homogeneous tensors contain) at least  $ r $  terms equal to  $ 1 \, $.  Then  $ \; \Delta^{t'_0}(u) \, = \, \delta_0^0 + h \, \delta_1^0 + h^2 \, \delta_2^0 + \cdots + h^r \, \delta_r^0 + \cdots \; $  for suitable elements  $ \; \delta_1^0 \, , \, \dots \, , \, \delta_r^0 \in {V^\ell(L)}_h^{\widehat{\otimes}\, t'_0} \; $.
                                                                  \par
   Now we can find  $ \, t'_1 \in \N \, $  such that  $ \; \overline{\big( \text{\sl id}^{t'_0 - 1} \otimes \Delta^{t'_1}\big)(\delta _1^0)} \; $  contains at least  $ r $  terms equal to $ 1 \, $.  We lift  $ \; \overline{\big( \text{\sl id}^{t'_0 - 1} \otimes \Delta^{t'_1}\big)(\delta _0^0)} \; $  and  $ \; \overline{\big( \text{\sl id}^{t'_0 - 1} \otimes \Delta^{t'_1}\big)(\delta _1^0)} \; $  to elements  $ \; \delta_0^1 \, , \, \delta_1^1 \in {V^\ell(L)}_h^{\widehat{\otimes}\, t'_0 + t'_1} \; $  which both contain at least  $ r $  terms equal to  $ 1 \, $.  Thus we find
 $ \; \Delta^{t'_0 + t'_1}(u) \, = \, \delta_0^1 + h \, \delta_1^1 + h^2 \, \delta_2^1 + \cdots + h^r \, \delta_r^1 + \cdots \; $
 for suitable  $ \; \delta_2^1 \, , \, \dots \, , \, \delta_r^1 \in {V^\ell(L)}_h^{\widehat{\otimes}\, t'_0 + t'_1} \; $.
 Iterating finitely many times, we complete the proof.
\end{proof}

\smallskip

\begin{notation}  \label{notation_V^ell(L)_h^*}
   Before next lemma, we need some more notation: given  $ \, {V^\ell(L)}_h \in \text{\rm (LQUEAd)}_{A_h} $,  consider  $ \, K_h := {V^\ell(L)}_h^{\,*} \, $  and its subset  $ \; I_{K_h} \, := \, \big\{\, \chi \in K_h \,\big|\, \big\langle 1_{{V^\ell(L)}_h} , \chi \big\rangle \in h \, A_h \,\big\} \; $.
 \end{notation}

\vskip1pt

\begin{remark}
   As  $ {V^\ell(L)}_h $  is a left bialgebroid, by  \S \ref{bialgd-struct-dual}  we know that its right dual  $ \, K_h := {V^\ell(L)}_h^{\,*} \, $  has a canonical structure of  $ A^e $--ring;  then, with respect to this structure, one easily sees that  $  I_{K_h} $  is a two-sided ideal of  $ K_h \, $. Moreover
   $\Delta (I_h) \subset I_{K_h} \otimes K_h + K_h \otimes I_{K_h} . $ Indeed, given any  $ \, \phi \in I_{K_h} \, $,  we write
   $ \, \Delta(\phi ) = \phi_{(1)} \otimes \phi_{(2)} \, $   --- a formal series (in  $ \Sigma $--notation)  ---   convergent in the  $ I_{K_h \,\widetilde{\otimes}\, K_h} $--adic  topology of  $ \, K_h \,\widetilde{\otimes}\, K_h \, $.  Writing  $ \, \phi_{(1)} \, $  and  $ \, \phi_{(2)} \, $  as
 \vskip-13pt
  $$  \displaylines{
   \phi_{(1)} \, = \, \phi_{(1)}^{+_s} + s^r\big(\partial_h\big(\phi_{(1)}\big)\big) \; ,  \quad \text{with} \quad  \phi_{(1)}^{+_s} \, := \, \phi_{(1)} - s^r\big(\partial_h\big(\phi_{(1)}\big)\big) \, \in \, I_{K_h}
    \,  \cr
   \phi_{(2)} \, = \, \phi_{(2)}^{+_t} + t^r\big(\partial_h\big(\phi_{(2)}\big)\big) \; ,  \quad \text{with} \quad  \phi_{(2)}^{+_t} \, := \, \phi_{(2)} - t^r\big(\partial_h\big(\phi_{(2)}\big)\big) \, \in \,  \,  \, I_{K_h}  }  $$
 \vskip-3pt
\noindent
 we can expand  $ \, \Delta(\phi) = \phi_{(1)} \otimes \phi_{(2)} \, $  as
  $$  \Delta(\phi)  \! =  \phi_{(1)}^{+_s} \otimes \phi_{(2)} \! + s^r\big(\partial_h\big(\phi_{(1)}\big)\!\big) \! \otimes \phi_{(2)}^{+_t} \! + s^r\!\big(\partial_h(\phi)\big) \otimes 1  \, \in \,
   \big( I_{K_h} \widetilde{\otimes}_{A_h} K_h + K_h \,\widetilde{\otimes}_{A_h} I_{K_h}
   + \, h \, s^r\!(A_h) \,\widetilde{\otimes}_{A_h} 1 \,\big)  $$
where we took into account the identity  $ \; s^r\big(\partial_h\big(\phi_{(1)}\big)\big) \otimes t^r\big(\partial_h\big(\phi_{(2)}\big)\big) \, = \, s^r\big(\partial_h(\phi)\big) \otimes 1 \; $,  due to Remarks \ref{right-bialgd-propts},  and the fact that  $ \, \partial_h(\phi) \in h\,A_h \, $,  since
$ \, \phi \in I_{K_h} \, $  by assumption.
\end{remark}

\smallskip

\begin{lemma}  \label{continuity of the evaluation}
 Given  $ \, {V^\ell(L)}_h \in \text{\rm (LQUEAd)}_{A_h} \, $  and  $ \, K_h := {V^\ell(L)}_h^{\,*} \, $,  consider the two-sided ideal  $ \; I_{K_h} := \, \big\{\, \chi \in K_h \,\big|\, \big\langle 1_{{V^\ell(L)}_h} , \chi \big\rangle \in h \, A_h \,\big\} \, $  of  $ K_h \, $,  as well as its powers  $ \, I_{K_h}^n \, (n \in \N) \, $.  Then, for every  $ \, u \in V^\ell(L)_h \, $  and every  $ \, r \in \N \, $,  there exists  $ \, t_r \in \N \, $  such that  $ \, \big\langle u , I_{K_h}^{t_r} \big\rangle \in h^r A_h \; $.
                                                          \par
   The same property holds if one considers the left dual  $ \, K_h := {\big({V^\ell(L)}_h\big)}_* \; $  of  $ \, {V^\ell(L)}_h \; $.
\end{lemma}

\begin{proof}
 Thanks to the previous lemma, there exists  $ \, t_r \in \N \, $  such that
  $$  \Delta^{t_r}(u)  \; = \;  \delta_0 + h \, \delta_1 + h^2 \, \delta_2 + \cdots + h^{r-1} \, \delta_{r-1} + h^r \, \delta_r  $$
for some elements  $ \; \delta_0 \, , \, \dots \, , \, \delta_r \in {V^\ell(L)}_h^{\widehat{\otimes}\, t_r} \, $  such that  $ \, \delta_0 \, , \, \dots \, , \, \delta_{r-1} \, $  contain at least  $ r $  terms equal to  $ 1 \, $.  From this fact and the properties of the natural pairing  $ \, \big\langle \ ,\ \big\rangle \, $  between  $ {V^\ell(L)}_h $  and its right dual  $ \, K_h := {V^\ell(L)}_h^{\,*} \, $  it is easy to see that  $ \; \big\langle \phi \, , \, u \big\rangle \in h^r A_h \; $  for all  $ \, \phi \in I_{K_h}^{\,t_r} \, $,  whence the claim.
\end{proof}

\medskip

   We are now ready for our first important result about linear duality of ``quantum groupoids''.  In a nutshell, it claims that the left and the right dual of a left, resp.~right, quantum universal enveloping algebroid are both right, resp.~left, quantum formal series algebroids.

\medskip

\begin{theorem}  \label{dual_QUEAd's=QFSAd's} {\ }
 \vskip4pt
   {\it (a)}\,  If  $ \; {V^\ell(L)}_h \in \text{\rm (LQUEAd)}_{A_h} \, $,  then  $ \; {V^\ell(L)}_h^{\;*} \, , {{V^\ell(L)}_h}_{\,*} \in \text{\rm (RQFSAd)}_{A_h} \; $,  with semiclassical limits  (cf.~\S \ref{further_J(L)_comp})
  $$  {V^\ell(L)}_h^{\;*} \,\Big/ h \, {V^\ell(L)}_h^{\;*}  \; \cong \;  {V^\ell(L)}^* \, =: \, J^r(L)
\qquad  \text{and}  \qquad  {{V^\ell(L)}_h}_{\,*} \,\Big/ h \, {{V^\ell(L)}_h}_{\,*}  \; \cong \;
{V^\ell(L)}_* \, \cong J^r(L)  $$
Therefore  $ {V^\ell(L)}_h^{\;*} \, $  and   $ {{V^\ell(L)}_h}_{\,*} \, $  are  quantization of  $ \, J^r(L) \, $.
                                                                       \par
   Moreover, the structure of Lie-Rinehart algebra induced on  $ L^* $  by the quantization  $ {V^\ell(L)}_h^{\;*} $  of  $ J^r(L) $   --- according to  Theorem \ref{semiclassical_limit-J^r(L)}  ---   is the same as that induced by the quantization  $ {V^\ell(L)}_h $  of  $ \, V^\ell(L) $   --- according to  Theorem \ref{semiclassical_limit-V^ell(L)};  therefore, the structure of Lie-Rinehart bialgebra induced on  $ L $  is the same in either case.
                                                                       \par
   On the other hand, the structure of Lie-Rinehart algebra induced on  $ L^* $  by the quantization  $ {{V^\ell(L)}_h}_{\;*} $  of  $ \, {V^\ell(L)}_* \cong J^r(L) \, $  is opposite to that induced by the quantization  $ {V^\ell(L)}_h $  of  $ \, V^\ell(L) \, $.  Thus the structures of Lie-Rinehart bialgebra induced on  $ L $  in the two cases are coopposite to each other:  $ {V^\ell(L)}_h $  provides a quantization of the Lie-Rinehart bialgebra  $ L \, $,  while  $ {{V^\ell(L)}_h}_{\;*} $  provides a quantization of the coopposite Lie-Rinehart bialgebra  $ L_{\text{\it coop}} $   ---  cf.~Remarks \ref{props_Lie-bialg}{\it (e)}.
 \vskip4pt
   {\it (b)}\,  If  $ \, {V^r(L)}_h \in \text{\rm (RQUEAd)}_{A_h} \, $,  then  $ \, {}_*{V^r(L)}_h \, , \phantom{|}^*{V^r(L)}_h \in \text{\rm (LQFSAd)}_{A_h} \; $,  with semiclassical limits  (cf.~\S \ref{further_J(L)_comp})
  $$  {}_*{{V^r(L)}_h} \,\Big/ h \, {}_*{\!}{V^r(L)_h}  \; \cong \;  {}_*{V^r(L)} :=J^\ell (L)
   \qquad  \text{and}  \qquad  \phantom{|}^*{V^r(L)}_h \,\Big/ h \, \phantom{|}^*{V^r(L)}_h  \; \cong \;  \phantom{|}^*{V^r(L)} \cong  J^\ell(L)  $$
Therefore  $ {}_*{V^r(L)}_h \, $ and  $ \phantom{|}^*{V^r(L)}_h \, $ are quantizations of  $ \, J^\ell(L) := {}_*{V^r(L)} \; $.
                                                                       \par
   Moreover, the structures of Lie-Rinehart algebra induced on  $ L^* $  by the quantization  $ {}_*{V^r(L)}_h $  of  $ J^\ell(L) $   --- according to  Theorem \ref{semiclassical_limit-J^ell(L)}  ---   is the same as that induced by the quantization  $ {V^r(L)}_h $  of  $ \, V^r(L) $   --- according to  Theorem \ref{semiclassical_limit-V^r(L)}.
                                                                       \par
 On the other hand, the structure of a Lie-Rinehart algebra induced on  $ L^* $  by the quantization  $ \phantom{|}^*{V^r(L)}_h $  of  $ \, {}^*{V^r(L)} \cong J^\ell(L) \, $  is opposite to that induced by the quantization  $ {V^r(L)}_h $  of  $ \, V^r(L) \, $.  Thus the structures of Lie-Rinehart bialgebra induced on  $ L $  in the two cases are coopposite to each other:  $ {V^\ell(L)}_h $  provides a quantization of the Lie-Rinehart bialgebra  $ L \, $,  while  $ \phantom{|}^*{V^r(L)}_h $  provides a quantization of the coopposite Lie-Rinehart bialgebra  $ L_{\text{\it coop}} $   ---  cf.~Remarks \ref{props_Lie-bialg}{\it (e)}.
\end{theorem}

\begin{proof}
   {\it (a)}\,  We shall start by proving that if  $ \, {V^\ell(L)}_h \! \in \! \text{\rm (LQUEAd)}_{A_h} \, $,  then  $ \, {V^\ell(L)}_h^{\,*} \! \in \! \text{\rm (RQFSAd)}_{A_h} \, $.
                                                                   \par
   As we saw in  \S \ref{notation_V^ell(L)_h^*},  the right dual  $ \, K_h := {V^\ell(L)}_h^{\,*} \, $  of  $ {V^\ell(L)}_h $  has a canonical structure of  $ A^e $--ring.  Moreover, it is endowed with a map  $ \; \partial_h : {V^\ell(L)}_h^{\,*} \longrightarrow A_h \, \big( \chi \mapsto \big\langle 1_{{V^\ell(L)}_h} , \chi \big\rangle \big) \, $,  which has all the properties of a ``counit'' in a right bialgebroid and defines the two-sided ideal  $ \, I_{K_h} := \partial_h^{-1}\big( h \, A_h \big) \, $.  What we still have to prove is that
 \vskip4pt
   $ \bullet \quad $  $ K_h := {V^\ell(L)}_h^{\,*} \; $  is complete for the  $ I_{K_h} $--adic  topology;
 \vskip4pt
   $ \bullet \quad $  there exists a suitable coproduct  $ \; \Delta_h : K_h := {V^\ell(L)}_h^{\,*} \!\rightarrow K_h \,\widetilde{\otimes}_{A_h} K_h = {V^\ell(L)}_h^{\,*} \,\widetilde{\otimes}_{A_h} \! {V^\ell(L)}_h^{\,*} \; $,  which makes  $ \, K_h := {V^\ell(L)}_h^{\,*} \, $  into a topological right bialgebroid;
 \vskip4pt
   $ \bullet \quad $  $ K_h \Big/ h \, K_h  \, = \,  {V^\ell(L)}_h^{\;*} \Big/ h \, {V^\ell(L)}_h^{\;*} \; $  is isomorphic to  $ {V^\ell(L)}^* $  as topological right bialgebroid.
 \vskip7pt
   We begin by looking for an isomorphism  $ \; {V^\ell(L)}_h^{\;*} \Big/ h \, {V^\ell(L)}_h^{\;*} \, \cong \, {V^\ell(L)}^* \; $.  For this, we distinguish two cases, the  {\sl free\/}  one and the  {\sl general\/}  one.
 \vskip4pt
\noindent
 {\ } --- \  {\it  $ \underline{\text{Free case}} $:} \,  $ L $  {\sl is a free  $ A $--module},  of finite type.
 \vskip2pt
   In this case, let us fix an  $ A $--basis  $ \{\, \overline{e}_1 , \dots , \overline{e}_n \} \, $  of the  $ A $--module  $ L \, $;  then we lift each  $ \overline{e}_i $  to some  $ \, e_i \in {V^\ell(L)}_h \; $.  Then any element of  $ {V^\ell(L)}_h $  can be written as the  $ h $--adic  limit of elements of the form  $ \; \sum_{(a_1,\dots,a_n) \in \N^n} t_l(c_{a_1,\dots,a_n}) \, e_1^{a_1} \cdots e_n^{a_n} \; $  in which almost all  $ c_{a_1,\dots,a_n} $'s  are zero.
                                                                                       \par
   For a given  $ \, \lambda \in {V^\ell(L)}^* \, $,  set  $ \, \overline{\alpha}_{a_1, \dots ,a_n} :=  \lambda\big( \overline{e}_1^{\,a_1} \cdots \overline{e}_n^{\,a_n} \!\big) \in A \, $  for all  $ \, \underline{a} := (a_1, \dots, a_n) \in \N^n \, $.  We lift each  $ \overline{\alpha}_{a_1, \dots, a_n} $  to some  $ \, \alpha_{a_1,\dots,a_n} \in A_h \, $,  with the assumption that if  $ \, \overline{\alpha}_{a_1,\dots,a_n} = 0 \, $  then we choose  $ \, \alpha_{a_1,\dots,a_n} = 0 \; $.  Now we set
  $$  \Lambda \, \Big(\, {\textstyle \sum_{(a_1,\dots,a_n) \in \N^n}} \, t_l(c_{a_1,\dots,a_n}) \, e_1^{a_1} \cdots e_n^{a_n} \Big)  \; := \;  {\textstyle \sum_{(a_1,\dots,a_n) \in \N^n}} \, c_{a_1,\dots ,a_n} \, \alpha_{a_1,\dots ,a_n}  $$
This defines a map  $ \Lambda $  from the right  $ A_h $--submodule  of  $ {V^\ell(L)}_h $  spanned by all the monomials  $ e_1^{a_1} \cdots e_n^{a_n} $  to  $ A_h \, $:  as the  $ h $--adic completion of this submodule is nothing but  $ {V^\ell(L)}_h \, $,  this map uniquely extends (by continuity) to a map  $ \; \Lambda : {V^\ell(L)}_h \longrightarrow A_h \; $.  By construction, we have  $ \, \Lambda \in {V^\ell(L)}_h^{\;*} \; $,  and  $ \Lambda $  is a lifting of  $ \lambda \, $, that is  $ \; \Lambda \! \mod h \, {V^\ell(L)}_h^{\;*} = \lambda \; $.  This guarantees that the canonical map  $ \; {V^\ell(L)}_h^{\;*} \,\Big/ h \, {V^\ell(L)}_h^{\;*} \relbar\joinrel\longrightarrow {V^\ell(L)}^* \; $,  which is obviously injective, is also surjective.
%
%%%
%%%
  \eject
%%%
%%%
%  \vskip4pt
%
\noindent
 {\ } --- \  {\it  $ \underline{\text{General case}} $:} \,  $ L $  {\sl is a projective  $ A $--module},  of finite type.
 \vskip2pt
   As in  \S \ref{L-R_proj->free},  we introduce a projective  $ A $--module  $ Q $  such that  $ \, L \oplus Q = F \, $  is a finite free  $ A $--module.  Fix an  $ A $--basis  $ \{b_1, \dots ,b_n\} $  of  $ F \, $,  and set  $ \, Y = k \, b_1 \oplus k \, b_2 \cdots \oplus k \, b_n \, $,  so that  $ \, F = A \otimes_k Y \, $.  The basis  $ \{b_1, \dots ,b_n\} $  also defines a good basis  $ {\{ \overline{e}_i \}}_{i \in T := \N \times \{1,\dots,n\}} \, $  of  $ L_Q \; $.
                                                                                       \par
    Now let  $ \, \lambda \in {V(L)}^* \; $.  We extend  $ \lambda $  to some  $ \, \lambda' \in {V(L_Q)}^* \, $  by setting  $ \, \lambda'{\big|}_{V(L) \otimes {S(Z)}^+} := 0 \; $.  Now we can adapt the arguments of the free case to construct a lifting  $ \, \Lambda' \in {V(L)_{h,Y}}^{\;*} \, $  of  $ \lambda' \; $.  Then  $ \; \Lambda := \Lambda'{\big|}_{{V(L)}_h} \! \in {V(L)}_h^{\;*} \, $  is a lifting  of  $ \lambda $  as required.
                                                                                       \par
   Thus one sees again that the canonical map  $ \; {V^\ell(L)}_h^{\;*} \,\Big/ h \, {V^\ell(L)}_h^{\;*} \longrightarrow {V^\ell(L)}^* \; $  is a bijection.
 \vskip7pt
   On  $ {V^\ell(L)}_h^{\;*} $  we have already considered an algebraic structure of  ``$ A^e $--ring  with counit'': the same structure then is inherited by its quotient  $ \; {V^\ell(L)}_h^{\;*} \,\Big/ h \, {V^\ell(L)}_h^{\;*} \; $.  On the other hand,  $ {V^\ell(L)}^* $  is a right bialgebroid, hence in particular it is an  ``$ A^e $--ring  with counit'' as well.  The canonical bijection  $ \; {V^\ell(L)}_h^{\;*} \Big/ h \, {V^\ell(L)}_h^{\;*} \!\longrightarrow {V^\ell(L)}^* \; $  above is clearly compatible with this additional structure.  In particular, this implies that  $ \; \text{\sl Ker}\,(\partial_h) \! \mod h \, K_h  \, \cong \,  \text{\sl Ker}\,(\partial_{V^\ell(L)})  \, =: \,  \J_{{V^\ell(L)}^*} \; $.
                                                                \par
   Now consider  $ \, I_{K_h} := \partial_h^{-1}\big( h \, A_h \big) \, $,  which can be written as  $ \; I_{K_h} = \text{\sl Ker}\,(\partial_h) \, + \, h \, K_h \; $.  As we know that  $ {V^\ell(L)}^* $  is  $ \J_{{V^\ell(L)}^*} $--adically  complete  (cf.~\S \ref{def_J^r(L)}),  from  $ \; \text{\sl Ker}\,(\partial_h) \! \mod h \, K_h \, \cong \, \J_{{V^\ell(L)}^*} \; $  and  $ \; I_{K_h} = \text{\sl Ker}\,(\partial_h) \, + \, h \, K_h \; $  we can easily argue that  $ \, K_h := {V^\ell(L)}_h^{\;*} $  is  $ I_{K_h} \! $--adically  complete.
 \vskip9pt
   Now we look for a suitable coproduct.  To this end, we shall show that the natural ``coproduct'' given by the recipe in  \S \ref{bialgd-struct-dual}  does the job.  The problem is to prove the existence of an isomorphism from  $ \; {V^\ell(L)}_h^{\,*} \, {{}_\blacktriangleleft \widetilde{\otimes}_\blacktriangleright} \, {V^\ell(L)}_h^{\,*} \; $   --- the completion of  $ \, {V^\ell(L)}_h^{\,*} \otimes {V^\ell(L)}_h^{\,*} \, $  with respect to the topology defined by the filtration  $ \, {\Big\{ \sum_{p+q=n} I_h^{\,p} \otimes I_h^{\,q} \Big\}}_{n \in \N} \, $  ---   to  $ \; \text{\sl Hom}_{(-,A_h)} \big( {\big( {V^\ell(L)}_h \, {{}_\triangleleft \otimes_{\scriptscriptstyle \blacktriangleright}} {V^\ell}(L)_h \big)}_\triangleleft \, , \, A_h \,\big) \; $.  Indeed   --- more precisely --- there exists  (cf.~\S \ref{bialgd-struct-dual})  a natural map  $ \chi $  from  $ \, {V^\ell(L)}_h^{\,*} \otimes {V^\ell(L)}_h^{\,*} \, $  to
%%%
  \break
%%%
  $ \; \text{\sl Hom}_{(-,A_h)} \big(\! {\big( {V^\ell(L)}_h \, {{}_\triangleleft \otimes_{\scriptscriptstyle \blacktriangleright}} {V^\ell}(L)_h \big)}_\triangleleft \, , \, A_h \big) \, $;
%%%
%%%
we now show that this  $ \chi $  actually extends to a (continuous) map   --- which, by abuse of notation, we still denote by  $ \chi $  ---   from  $ \; {V^\ell(L)}_h^{\,*} \, {{}_\blacktriangleleft \widetilde{\otimes}_\blacktriangleright} \, {V^\ell(L)}_h^{\,*} \; $  to
%%%
  \break
%%%
  $ \; \text{\sl Hom}_{(-,A_h)} \big( {\big( {V^\ell(L)}_h \, {{}_\triangleleft \otimes_{\scriptscriptstyle \blacktriangleright}} {V^\ell}(L)_h \big)}_\triangleleft \, , \, A_h \,\big) \; $.
%%%
%%%
                                                                    \par
   To begin with, fix  $ \, u \in {V^\ell(L)}_h \, $.  For every  $ \, r \in \N \, $,  there exists  $ \, t_r \in \N \, $  such that  $ \, \Delta^{t_r}(u) \, $  expands as  $ \; \Delta^{t_r}(u) = \delta_0 + h \, \delta_1 + h^2 \, \delta_2 + \cdots + h^r \, \delta_r \; $  as in  Lemma \ref{property of the coproduct}:  in particular, every  $ \, \delta_i \in {V^\ell(L)}_h^{\otimes\, t_r} \, $  with  $ \, 0 \leq i \leq r-1 \, $  contains at least  $ \, r \, $  terms equal to  $ 1 \, $.  As the canonical evaluation pairing between  $ {V^\ell(L)}_h $  and  $ \, K_h := {V^\ell(L)}_h^{\,*} \, $  is a bialgebroid right pairing   --- in the sense of  Definition \ref{left and right 2}  ---   the formulas for such pairings imply at once (by induction) that  $ \; \big\langle\, u \, , \, I_{K_h}^{\,t} \,\big\rangle \subseteq h^r A_h \; $  for all  $ \, t \geq t_r \, $.  By the same arguments, given  $ \, v, w \in {V^\ell(L)}_h \, $  we see that, for every  $ \, r \in \N \, $,  one has
  $$  \big\langle\, v \otimes w \, , \, I_{K_h}^{\,t'} \! \otimes I_{K_h}^{\,t''} \,\big\rangle  \; \subseteq \;  h^r A_h   \qquad  \text{for all \ }  \, t' + t'' \gg 0   \eqno (5.1)  $$
   \indent   Now let  $ \, \Lambda \in {V^\ell(L)}_h^{\,*} \, {{}_\blacktriangleleft \widetilde{\otimes}_\blacktriangleright} \, {V^\ell(L)}_h^{\,*} \, $.  Then  $ \Lambda $  is the limit of a sequence  $ \, {(\Lambda_n)}_{n \in \N} \, $   --- with  $ \, \Lambda_n \in  {V^\ell(L)}_h^{\,*} {{}_\blacktriangleleft \otimes_\blacktriangleright} {V^\ell(L)}_h^{\,*} \, $  for all  $ n $  ---   for the topology defined by the filtration  $ \, {\Big\{ \sum_{p+q=n} I_h^{\,p} \otimes I_h^{\,q} \Big\}}_{n \in \N} \, $;  in particular, for each  $ \, t \in \N \, $  one has
  $$  \big( \Lambda_{n'} - \Lambda_{n''} \big)  \; \in \;  {\textstyle \sum\limits_{t'+t''=t}} I_{K_h}^{\,t'} \! \otimes I_{K_h}^{\,t''}  \qquad \qquad  \text{for all \ } \, n', n'' \gg 0   \eqno (5.2)  $$
By (5.1) and (5.2) together we get that for all  $ \, r \in \N \, $  one has
  $$  \big( \chi(\Lambda_{n'}) - \chi(\Lambda_{n''}) \big)(v \otimes w) = \chi\big(\Lambda_{n'} - \Lambda_{n''}\big)(v \otimes w):= \big\langle\, v \otimes w \, , \, \Lambda_{n'} - \Lambda_{n''} \,\big\rangle  \; \subseteq \;  h^r A_h  $$
for all  $ \, n', n'' \gg 0 \, $;  in other words,  $ \; {\big\{ \chi(\Lambda_n\big)(v \otimes w) \big\}}_{n \in \N} \,  $  is a Cauchy sequence for the  $ h $--adic  topology in  $ A_h \, $;  as the latter is  $ h $--adically  complete (and separated), there exists a unique, well-defined limit  $ \, \lim\limits_{n \to \infty} \chi(\Lambda_n\big)(v \otimes w) \in A_h \; $.  In the end, we can set
 $ \; \chi\big(\Lambda\big)(v \otimes w) \, := \, \lim\limits_{n \to \infty} \chi(\Lambda_n\big)(v \otimes w) \, $;
 \, this defines a (continuous) map extending the original one, namely
  $$  \chi \, : \, {V^\ell(L)}_h^{\,*} \, {{}_\blacktriangleleft \widetilde{\otimes}_\blacktriangleright} \, {V^\ell(L)}_h^{\,*} \, \relbar\joinrel\longrightarrow \, \text{\sl Hom}_{(-,A_h)} \big( {\big( {V^\ell(L)}_h \, {{}_\triangleleft \otimes_{\scriptscriptstyle \blacktriangleright}} {V^\ell}(L)_h \big)}_\triangleleft \, , \, A_h \,\big)   \eqno (5.3)  $$
   \indent   To complete our argument, we need a few more steps.  In order to ease the notation, we shall write  $ \; X{\big|}_{h=0} := X \Big/ h \, X \; $  for every  $ k[[h]] $--module  $ X \, $.
                                                                  \par
   First, with the same arguments used to prove that  $ \; \text{\sl Hom}_{(-,A_h)}\big( {V^\ell(L)}_h \, , A_h \big){\Big|}_{h=0} \! = \, {V^\ell(L)}_h^{\;*}{\Big|}_{h=0} \; $  has a canonical bijection with  $ \, \text{\sl Hom}_{(-,A)} \big( V^\ell(L) \, , A \,\big) =: {V^\ell(L)}^* \, $  we can also prove that
  $$  \text{\sl Hom}_{(-,A_h)} \big( {\big( {V^\ell(L)}_h \, {{}_\triangleleft \otimes_{\scriptscriptstyle \blacktriangleright}} {V^\ell}(L)_h \big)}_\triangleleft \, , \, A_h \,\big){\Big|}_{h=0}  \, \cong \;  \text{\sl Hom}_{(-,A)} \big( V^\ell(L) \otimes V^\ell(L) \, , \, A \,\big)   \eqno (5.4)  $$
Similarly, the same arguments once more can be adapted to prove that
  $$  {V^\ell(L)}_h^{\,*} \, {{}_\blacktriangleleft \widetilde{\otimes}_\blacktriangleright} \, {V^\ell(L)}_h^{\,*}{\Big|}_{h=0}  \, \cong \;  {V^\ell(L)}^* \,\widetilde{\otimes}\, {V^\ell(L)}^*  \quad \big( = \, J^r(L) \,\widetilde{\otimes}\, J^r(L) \; , \;\;  \text{cf.~\S \ref{def_J^r(L)}} \,\big)   \eqno (5.5)  $$
Finally, by construction the reduction modulo  $ h $  of the map  $ \chi $  in (5.3), call it  $ \overline{\chi} \, $,  is nothing but the map
  $$  \widetilde{\vartheta} \, : \, {J^r(L)}_\blacktriangleleft \widetilde{\otimes}_{\,{}_{\scriptstyle \blacktriangleright}}\! J^r(L) \, = \, {V^\ell(L)}^*_{\,\blacktriangleleft} \widetilde{\otimes}_{\,{}_{\scriptstyle \blacktriangleright}}\! {V^\ell(L)}^* \relbar\joinrel\longrightarrow {\big( {V^\ell(L)}_\triangleleft \! \otimes_{\,{}_\blacktriangleright}\!\! V^\ell(L) \big)}^{\!*}  $$
considered in  \S \ref{def_J^r(L)}.  Therefore   --- since  $ \, \text{\sl Hom}_{(-,A)} \big( V^\ell(L) \otimes V^\ell(L) \, , \, A \,\big) =: {\big( {V^\ell(L)}_\triangleleft \! \otimes_{\,{}_\blacktriangleright}\!\! V^\ell(L) \big)}^{\!*} \, $,  and taking into account the isomorphisms in (5.5--5) ---   as  $ \, \overline{\chi} = \widetilde{\vartheta} \, $  is a  $ k $--linear  isomorphism we can deduce that  $ \chi $  is an isomorphism as well.
                                                                  \par
   The outcome now is that  $ \, K_h := {V^\ell(L)}_h^{\;*} \, $  endowed with the previously constructed structure   --- including the coproduct map given by the recipe in  \S \ref{bialgd-struct-dual}  ---   is a topological right bialgebroid, complete with respect to the  $ I_{K_h} \! $--adic  topology.  In addition, the bijection  $ \; {V^\ell(L)}_h^{\;*} \,\Big/ h \, {V^\ell(L)}_h^{\;*} \longrightarrow {V^\ell(L)}^* \; $  found above by construction happens to be a right bialgebroid isomorphism.
 \vskip9pt
   Our next task is the following.  Denote by  $ \, \big( L^* \, , {[\,\ ,\ ]}' , \omega' \big) \, $  and  $ \, \big( L^* \, , {[\,\ ,\ ]}'' , \omega'' \big) \, $  the structures of Lie-Rinehart bialgebras induced on  $ L^* $  respectively by  Theorem \ref{semiclassical_limit-J^r(L)}   --- for  $ \, J^r(L) := {V^\ell(L)}_h^{\;*} \, $  ---  and by  Theorem \ref{semiclassical_limit-V^ell(L)}   --- applied to  $ {V^\ell(L)}_h \, $.  We must prove that  $ \, \omega' = \omega'' \, $  and  $ \, {[\,\ ,\ ]}' = {[\,\ ,\ ]}'' \, $.  To this end, recall that, by  Remarks \ref{props_Lie-bialg}{\it (b)},  $ \omega'' $  and  $ {[\,\ ,\ ]}'' $  are uniquely determined by the conditions
  $$  \omega''(\Phi)(a) \, = \big\langle \delta_L\!(a) \, , \Phi \big\rangle  \;\; ,  \quad  \big\langle\, \Theta \, , {[\Phi\,,\Psi]}'' \,\big\rangle \, = \, \omega''(\Phi)\big(\langle\, \Theta \, , \Psi \rangle\big) \, - \, \omega''(\Psi)\big(\langle \Theta \, , \Phi \rangle\big) \, - \, \big\langle \delta_L\!(\Theta) \, , \Phi \otimes \Psi \big\rangle  $$
(for all $ \, \Phi \, , \Psi \in L^* \, $,  $ \, \Theta \in L \, $,  $ \, a \in A \, $),  where  $ \delta_L\!(a) $  and  $ \delta_L\!(\Theta) $  are defined by the formula for  $ \delta $  in  Theorem \ref{semiclassical_limit-V^ell(L)}.  Therefore, it is enough for us to prove that (for all $ \, \Phi \, , \Psi \in L^* \, $,  $ \, \Theta \in L \, $,  $ \, a \in A \, $)
  $$  \omega'(\Phi)(a) \, = \big\langle \delta_L\!(a) \, , \Phi \big\rangle  \; ,  \;\;  \big\langle\, \Theta , {[\Phi\,,\Psi]}' \,\big\rangle \, = \, \omega''(\Phi)\big(\!\langle\, \Theta , \Psi \rangle\!\big) - \omega''(\Psi)\big(\!\langle \Theta , \Phi \rangle\!\big) - \big\langle \delta_L\!(\Theta) \, , \Phi \otimes \Psi \big\rangle  \!\quad   \eqno (5.6)  $$
   \indent   In order to prove (5.6), we choose liftings  $ \, \phi' , \psi' \in {J^r(L)}_h := {V^\ell(L)}_h^{\;*} \, $,  with the additional condition that  $ \, \phi' , \psi' \in \J_{\scriptscriptstyle {J^r(L)}_h} := \text{\sl Ker}\,\big( \partial_{{J^r(L)}_h} \big) \, $  (such a choice is always possible), a lifting  $ \, \theta \in {V^\ell(L)}_h \, $  of  $ \Theta \, $  and a lifting  $ \, a' \in A_h \, $  of  $ \, a \, $.  Now direct computation gives
  $$  \displaylines{
   \omega'(\Phi)(a)  \;\; = \;\;  \Big(\! \big( h^{-1} \big( \phi' \, t_r\big(a'\big) - t_r\big(a'\big) \, \phi' \big) \big) \!\!\mod h \, {J^r(L)}_h \,\Big) \!\!\!\mod \J_{\scriptscriptstyle J(L)}  \;\; =   \hfill  \cr
    = \,  \partial_{{J(L)}_{\!h}}\! \bigg(\! {{\; \phi' \, t_r\big(a'\big) - t_r\big(a'\big) \, \phi' \;} \over h} \bigg) \!\!\!\!\mod h \, A_h  \;
  = \;  \bigg\langle 1 \, , {{\; \phi' \, t_r\big(a'\big) - t_r\big(a'\big) \, \phi' \;} \over h} \bigg\rangle \!\!\mod h \, A_h  \; =   \hfill  \cr
   = \;  {{\; a' \, \big\langle\, 1 \, , \, \phi' \,\big\rangle \, - \, \big\langle\, 1 \, , \, t_r\big(a'\big) \, \phi'  \,\big\rangle \;} \over h} \!\mod h \, A_h  \;
   = \;  {{\; \big\langle\, 1 \, , \, \phi' \,\big\rangle \, a' \, - \, \big\langle\, 1 \, , \, t_r\big(a'\big) \, \phi'  \,\big\rangle \;} \over h} \!\mod h \, A_h  \; =   \hfill  \cr
   \hfill   = \;  {{\; \big\langle\, t^\ell\big(a'\big) \, , \phi' \,\big\rangle \, - \, \big\langle\, s^\ell\big(a'\big) \, , \phi'  \,\big\rangle \;} \over h} \!\!\mod h \, A_h  \; = \;  \bigg\langle\! {{\; t^\ell\big(a'\big) - s^\ell\big(a'\big) \;} \over h} \, , \phi' \bigg\rangle \!\!\mod h \, A_h  \; = \;  \big\langle\, \delta_L\!(a) \, , \Phi \,\big\rangle  }  $$
where  $ \, \big\langle\,\ ,\ \big\rangle \, $  denotes the natural evaluation pairing between  $ {V^\ell(L)}_h $  and its right dual  $ {V^\ell(L)}_h^{\;*} \, $,  we exploited the fact that this pairing is a  {\sl right\/}  bialgebroid pairing (cf.~Definitions \ref{left and right 1}  and  \ref{left and right 2})  and the fact that  $ \; \big\langle 1_{\scriptscriptstyle {V^\ell(L)}_h} \, , \, \phi' \,\big\rangle \, =: \partial_{{J^r(L)}_h}\big(\phi'\big) = \, 0 \; $  because  $ \, \phi' \in \J_{\scriptscriptstyle {J^r(L)}_h} := \text{\sl Ker}\,\big( \partial_{{J^r(L)}_h} \big) \, $  by assumption.
     Thus the first identity in (5.6) is verified.
                                                                  \par
   As to the second identity, we write  $ \, \Delta(\theta) = \theta_{(1)} \otimes \theta_{(2)} \, $  as  $ \, \Delta(\theta) = \theta \otimes 1 + 1 \otimes \theta + h \, \sum_{[\theta]} \theta_{[1]} \otimes \theta_{[2]} \, $,  so that  $ \; \big( \sum_{[\theta]} \theta_{[1]} \otimes \theta_{[2]} \big) \mod h \, {V^\ell(L)}_h \displaystyle{{{\,}_\triangleleft}{\mathop \times_A}{{\,}_\triangleright}} {V^\ell(L)}_h \, =: \, \Delta^{[1]}(\Theta) \; $   --- notation of  Definition \ref{semiclassical_limit-V^ell(L)}.  Then by direct computation we find
%%%
%%%
 \eject
%%%
%%%
  $$  \displaylines{
   \Big\langle \Theta \, , \, {[\Phi,\Psi]}' \Big\rangle  \,\; = \;  \bigg\langle \theta \, , \, {{\; \phi' \, \psi' \, - \, \psi' \, \phi' \;} \over h} \bigg\rangle \!\!\!\mod h \, A_h  \,\; = \;  h^{-1} \, \Big(\! \big\langle\, \theta \, , \, \phi' \, \psi' \, \big\rangle \, - \, \big\langle\, \theta \, , \, \psi' \phi' \, \big\rangle \Big) \!\!\!\mod h \, A_h  \; =   \hfill  \cr
   = \;  h^{-1} \, \Big(\! \big\langle\, 1 \, , \, t_r\big( \big\langle\, \theta \, , \phi' \,\big\rangle \big) \, \psi' \, \big\rangle \, - \, \big\langle\, 1 \, , \, t_r\big( \big\langle\, \theta \, , \psi' \,\big\rangle \big) \, \phi' \, \big\rangle \Big) \!\!\!\mod h \, A_h \; +   \hfill  \cr
   \hfill   + \; {\textstyle \sum_{[\theta]}} \Big( \big\langle\, \theta_{[2]} \, , \, t_r\big( \big\langle\, \theta_{[1]} \, , \phi' \,\big\rangle \big) \, \psi' \, \big\rangle \, - \, \big\langle\, \theta_{[2]} \, , \, t_r\big( \big\langle\, \theta_{[1]} \, , \psi' \,\big\rangle \big) \, \phi' \, \big\rangle \Big) \!\!\!\mod h \, A_h  \; =  \cr
   = \;  h^{-1} \, \Big(\! \big\langle\, s^\ell\big( \big\langle\, \theta \, , \phi' \,\big\rangle \big) \, , \, \psi' \, \big\rangle \, - \, \big\langle\, s^\ell\big( \big\langle\, \theta \, , \psi' \,\big\rangle \big) \, , \, \phi' \, \big\rangle \Big) \!\!\!\mod h \, A_h \; +   \hfill  \cr
   \hfill   + \; {\textstyle \sum_{[\theta]}} \Big( \big\langle\, s^\ell\big( \big\langle\, \theta_{[1]} \, , \phi' \,\big\rangle \big) \, \theta_{[2]} \, , \, \psi' \, \big\rangle \, - \, \big\langle\, s^\ell\big( \big\langle\, \theta_{[1]} \, , \psi' \,\big\rangle \big) \, \theta_{[2]} \, , \, \phi' \, \big\rangle \Big) \!\!\!\mod h \, A_h  \; =  \cr
   = \;  \Big( h^{-1} \, \Big\langle\, s^\ell\big( \big\langle\, \theta \, , \phi' \,\big\rangle \big) - t^\ell\big( \big\langle\, \theta \, , \phi' \,\big\rangle \big) \, , \, \psi' \, \Big\rangle \, - \, h^{-1} \, \Big\langle\, s^\ell\big( \big\langle\, \theta \, , \psi' \,\big\rangle \big) - t^\ell\big( \big\langle\, \theta \, , \psi' \,\big\rangle \big)\, , \, \phi' \, \Big\rangle \Big) \!\!\!\mod h \, A_h \; +   \hfill  \cr
   \hfill   + \; {\textstyle \sum_{[\theta]}} \Big( \big\langle\, t^\ell\big( \big\langle\, \theta_{[1]} \, , \phi' \,\big\rangle \big) \, \theta_{[2]} \, , \, \psi' \, \big\rangle \, - \, \big\langle\, t^\ell\big( \big\langle\, \theta_{[1]} \, , \psi' \,\big\rangle \big) \, \theta_{[2]} \, , \, \phi' \, \big\rangle \Big) \!\!\!\mod h \, A_h  \; =  \cr
   = \;  \bigg( \bigg\langle\, {{\, s^\ell\big( \big\langle\, \theta \, , \phi' \,\big\rangle \big) - t^\ell\big( \big\langle\, \theta \, , \phi' \,\big\rangle \big) \,} \over h} \, , \, \psi' \, \bigg\rangle \, - \, \bigg\langle\, {{\, s^\ell\big( \big\langle\, \theta \, , \psi' \,\big\rangle \big) - t^\ell\big( \big\langle\, \theta \, , \psi' \,\big\rangle \big) \,} \over h} \, , \, \phi' \, \bigg\rangle \bigg) \!\!\!\mod h \, A_h \; +   \hfill  \cr
   \hfill   + \; {\textstyle \sum_{[\theta]}} \Big( \big\langle\, \theta_{[2]} \, , \, \psi' \, \big\rangle \, \big\langle\, \theta_{[1]} \, , \phi' \,\big\rangle \, - \, \big\langle\, \theta_{[2]} \, , \, \phi' \, \big\rangle \, \big\langle\, \theta_{[1]} \, , \psi' \,\big\rangle \Big) \!\!\!\mod h \, A_h  \; =  \cr
   = \;\;  - \, \big\langle\, \delta_L\!\big( \big\langle\, \Theta \, , \Phi \,\big\rangle \big) \, , \, \Psi \, \big\rangle \; + \; \big\langle\, \delta_L\!\big( \big\langle\, \Theta \, , \Psi \,\big\rangle \big) \, , \, \Phi \, \big\rangle \; + \; \big\langle \Delta^{[1]}(\Theta) - {\Delta^{[1]}(\Theta)}_{2,1} \; , \, \Phi \otimes \Psi \big\rangle  \;\; =  \hfill  \cr
   \hfill   = \;\;  \omega''(\Phi)\big(\! \big\langle\, \Theta \, , \Psi \,\big\rangle \!\big) \; - \; \omega''(\Psi)\big(\! \big\langle\, \Theta \, , \Phi \,\big\rangle \!\big) \; - \; \big\langle \delta_L\!(\Theta) \, , \, \Phi \otimes \Psi \big\rangle  }  $$
%
%%%
%  \eject
%%
% \noindent
%%%
%
 where we exploited the properties of a right bialgebroid pairing   --- in particular, the identity  $ \; \big\langle\, t^\ell(\alpha) \, , \, \chi' \, \big\rangle \, = \, \big\langle\, 1 \, , \, \chi' \, \big\rangle \, \alpha \; $  ---   the fact that  $ \, \big\langle\, 1 \, , \phi' \,\big\rangle = \partial_{\scriptscriptstyle {J^r(L)}_h}\big(\phi'\big) = 0 \, $,  $ \, \big\langle\, 1 \, , \psi' \,\big\rangle = \partial_{\scriptscriptstyle {J^r(L)}_h}\big(\psi'\big) = 0 \, $,  the fact that  $ \, s^\ell(\kappa) \equiv t^\ell(\kappa) \!\mod h \, {V^\ell(L)}_h \; $  and the fact (already proved) that  $ \, \omega'' = \omega' \, $.
 This proves the second identity in (5.6).
 \vskip7pt
   Finally, we have to deal with  $ {{V^\ell(L)}_h}_{\;*} \, $.  Acting much like for  $ {V^\ell(L)}_h^{\;*} \, $,  one proves that  $ {{V^\ell(L)}_h}_{\;*} $  is indeed a topological right bialgebroid, whose specialization modulo  $ h $  is just  $ \, {V^\ell(L)}_* \cong J^\ell(L) \, $,  hence we can claim that  $ \, {{V^\ell(L)}_h}_{\;*} \in \text{\rm (RQFSAd)}_{A_h} \, $  is a quantization of  $ J^\ell(L) \, $.
                                                                     \par
   As to the last part of claim  {\it (a)},  concerning the two Lie-Rinehart algebra structures induced on  $ L^* \, $,  we can again proceed like for  $ {V^\ell(L)}_h^{\;*} \, $:  the difference in the outcome   --- a minus sign ---   now is due to the fact that the natural pairing (given by evaluation) among the left bialgebroid  $ {V^\ell(L)}_h $  and the right bialgebroid  $ {{V^\ell(L)}_h}_{\;*} $  is a  {\sl left\/}  bialgebroid pairing  (cf.~Definitions \ref{left and right 1}  and  \ref{left and right 2})   --- whereas in the case of  $ {V^\ell(L)}_h $  and  $ {V^\ell(L)}_h^{\;*} $  it is a  {\sl right\/}  bialgebroid pairing.  Full detail computations are left to the reader.
 \vskip5pt
   {\it (b)}\,  The proof given for claim  {\it (a)\/}  clearly adapts to claim  {\it (b)\/}  as well, by the same arguments.  Otherwise, one can deduce claim  {\it (b)\/}  directly from claim  {\it (a)\/}  using general isomorphisms such as  $ \; {}_*\big(U^{\text{\it op}}_{\text{\it coop}}\big) \, \cong \, (U^*)^{\text{\it op}}_{\text{\it coop}} \; $  and  $ \; {}^*\big( U^{\text{\it op}}_{\text{\it coop}} \big) \, \cong \, {(U_*)}^{\text{\it op}}_{\text{\it coop}} \; $  (see  Remark \ref{dual_vs._op-coop}).
\end{proof}

\smallskip

 \subsection{Linear duality for QFSAd's}  \label{lin-dual_QFSAd's}

\smallskip

   {\ } \quad   Much like for their classical counterparts, the duals for QFSAD's have to be meant in topological sense.  Indeed, we introduce now a suitable definition of ``continuous'' dual of a (R/L)QFSAd:

\medskip

\begin{definition}  \label{cont-dual_QFSAd's}
 Let  $ \, K_h \in \text{\rm (RQFSAd)}_{A_h} \, $.  Let  $ \; I_h \, := \, \big\{\, \lambda \in K_h \,\big|\, \partial_h(\lambda) \in hA_h \,\big\} \; $.
                                                                   \par
   We denote by  $ {}_\star{}K_h $  the  $ k[[h]] $--submodule  of  $ {}_*{}K_h $  of all (left  $ A_h $--linear)   maps from  $ K_h $  to  $ A_h $  which are  {\sl continuous}  for the  $ I_h $--adic  topology on  $ K_h $  and the  $ h $--adic  topology on  $ A_h \, $.
                                                                   \par
   We denote by  $ {}^\star{}K_h $  the  $ k[[h]] $--submodule  of  $ {}^*{}K_h $  of all (right  $ A_h $--linear)   maps from  $ K_h $  to  $ A_h $  which are  {\sl continuous}  for the  $ I_h $--adic  topology on  $ K_h $  and the  $ h $--adic  topology on  $ A_h \, $.
 \vskip3pt
   In a similar way, we define also ``continuous dual'' objects  $ \, {K_h}_{\,\star} \; \big(\! \subseteq {K_h}_{\,*} \big) \, $  and  $ \, K_h^{\,\star} \; \big(\! \subseteq K_h^{\,*} \big) \, $  for every  $ \, K_h \in \text{\rm (LQFSAd)}_{A_h} \, $.
\end{definition}

\smallskip

   It is time for our second result about linear duality of ``quantum groupoids''.  In short, it claims that the left and the right continuous dual of a left, resp.~right, quantum formal series algebroid are both right, resp.~left, quantum universal enveloping algebroids.

\smallskip

\begin{theorem}  \label{dual_QFSAd's=QUEAd's}  {\ }
 \vskip4pt
    {\it (a)}\,  If  $ \, {J^r(L)}_h \in \text{\rm (RQFSAd)}_{A_h} \, $,  then  $ \; {}_\star{}{J^r(L)}_h \, , \, {}^\star{\!}{J^r(L)}_h \in \text{\rm (LQUEAd)}_{A_h} \, $,  with semiclassical limits  (cf.~Remark \ref{diff_double-dual})
  $$  {}_\star{}{J^r(L)}_h \,\Big/ h \, {}_\star{}{J^r(L)}_h  \; \cong \;  {}_\star{J^r(L)}  \, = \,  V^\ell(L)
   \qquad  \text{and}  \qquad  {}^\star{\!}{J^r(L)}_h \,\Big/ h \, {}^\star{\!}{J^r(L)}_h  \; \cong \;  {}^\star{\!}{J^r(L)}  \, \cong \,  V^\ell\big(L\big)  $$
Therefore  $ {}_\star{}{J^r(L)}_h \, $  and  $ {}^\star{}{J^r(L)}_h \, $  are quantizations of  $ \; V^\ell(L) = {}_\star{J^r(L)} \; $.
                                                                       \par
   Moreover, the structure of Lie-Rinehart bialgebra induced on  $ L $  by the quantization  $ {}_\star{}{J^r(L)}_h $  of  $ \, V^\ell(L) $   --- according to  Theorem \ref{semiclassical_limit-V^ell(L)}  ---   is the same as that induced by the quantization  $ {J^r(L)}_h $  of  $ \, J^r(L) $   --- according to  Theorem \ref{semiclassical_limit-J^r(L)}.
                                                                       \par
   On the other hand, the structure of Lie-Rinehart algebra induced on  $ L^* $  by the quantization  $ {}^\star{\!}{J^r(L)}_h $  of  $ \, V^\ell\big(L\big) $  is opposite to that induced by the quantization  $ {J^r(L)}_h $  of  $ J^r(L) \, $.  Therefore, the structures of Lie-Rinehart bialgebra induced on  $ L $  in the two cases are coopposite to each other: in short,  $ {J^r(L)}_h $  provides a quantization of the Lie-Rinehart bialgebra  $ L \, $,  while  $ {}^\star{\!}{J^r(L)}_h $  provides a quantization of the coopposite Lie-Rinehart bialgebra  $ L_{\text{\it coop}} $   ---  cf.~Remarks \ref{props_Lie-bialg}{\it (e)}.
 \vskip3pt
    {\it (b)}\,  If  $ \, {J^\ell(L)}_h \in \text{\rm (LQFSAd)}_{A_h} \, $,  then  $ \; {J^\ell(L)}_h^{\;\star} \, , \, {{J^\ell(L)}_h}_{{\,}_{\scriptstyle \star}} \in \text{\rm (RQUEAd)}_{A_h} \, $,  with semiclassical limits (cf.~Remark \;\ref{diff_double-dual})
  $$  {J^\ell(L)}_h^{\;\star} \,\Big/ h \, {J^\ell(L)}_h^{\;\star}  \; \cong \;  {J^\ell(L)}^\star  \, = \,  V^r(L)  \qquad  \text{and}  \qquad  {{J^\ell(L)}_h}_{{\,}_{\scriptstyle \star}} \,\Big/ h \, {{J^\ell(L)}_h}_{{\,}_{\scriptstyle \star}}  \; \cong \;  {J^\ell(L)}_\star  \, \cong \,  V^r\big(L\big)  $$
Therefore  $ {J^\ell(L)}_h^{\;\star} \, $  and  $ {{J^\ell(L)}_h}_{{\,}_{\scriptstyle \star}} \, $  are quantizations of  $ \, V^r(L) = {J^\ell(L)}^\star \, $.
                                                                       \par
   Moreover, the structure of Lie-Rinehart bialgebra induced on  $ L $  by the quantization  $ {J^\ell(L)}_h^{\;\star} $  of  $ \, V^r(L) $   --- according to  Theorem \ref{semiclassical_limit-V^r(L)}  ---   is the same as that induced by the quantization  $ {J^\ell(L)}_h $  of  $ \, J^\ell(L) $   --- according to  Theorem \ref{semiclassical_limit-J^ell(L)}.
                                                                       \par
   On the other hand, the structure of Lie-Rinehart algebra induced on  $ L^* $  by the quantization  $ {{J^\ell(L)}_h}_{{\,}_{\scriptstyle \star}} $  of  $ \, V^r\big(L\big) $  is opposite to that induced by the quantization  $ {J^\ell(L)}_h $  of  $ J^\ell(L) \, $.  Therefore, the structures of Lie-Rinehart bialgebra induced on  $ L $  in the two cases are coopposite to each other: in short,  $ {J^\ell(L)}_h $  provides a quantization of the Lie-Rinehart bialgebra  $ L \, $,  while  $ {{J^\ell(L)}_h}_{{\,}_{\scriptstyle \star}} $  provides a quantization of the coopposite Lie-Rinehart bialgebra
 $ L_{\text{\it coop}} $   ---  cf.~Remarks \ref{props_Lie-bialg}{\it (e)}.
\end{theorem}

\begin{proof}
 {\it (a)}\,  To simplify notation set  $ \, K_h := {J^r(L)}_h \; $.  We begin dealing with  $ {}_\star{}K_h \, $,  in several steps.
 \vskip4pt
   The main point in the proof is the following.  By definition,  $ {}_\star{}K_h $  is contained in  $ {}_*{}K_h \, $:  by the recipe in  \S \ref{dual_bialgds},  the latter is ``almost'' a left bialgebroid over  $ A_h \, $,  as it has a natural structure of  $ A_h^{\,e} $--ring  with ``counit'', and also a ``candidate'' as coproduct.  Then the natural pairing among  $ {}_*{}K_h $  and  $ K_h $  (given by evaluation), hereafter denoted  $ \, \langle\,\ ,\ \rangle \, $,  is an  $ A_h^{\,e} $--right  pairing  (cf.~Definition \ref{left and right 1}),  and also a bialgebroid right pairing  (cf.~Definition \ref{left and right 2})  --- as far as this makes sense.  Basing on this, we shall presently show that this structure on  $ {}_*{}K_h $   --- which makes it an  $ A_h^{\,e} $--ring  and even ``almost a left  $ A_h^{\,e} $--bialgebroid'',  actually does restrict to  $ {}_\star{}K_h \, $,  making it into a left  $ A_h^{\,e} $--bialgebroid.  Also, the evaluation will then provide a natural bialgebroid right pairing between  $ {}_\star{}K_h $  and  $ K_h \, $.
                                                                        \par
   Along the way, we shall prove also that  $ {}_\star{}K_h $  has semiclassical limit  $ V^\ell(L) \, $,  and finally that the Lie-Rinehart bialgebra structure on  $ L $  induced by it is the same as that induced by  $ \, K_h := {J^r(L)}_h \; $.
 \vskip7pt
   {\it (1)}\,  First we prove that the source and target maps of  $ {}_*{}K_h $  (as given in  \S \ref{dual_bialgds})  actually map into  $ {}_\star{}K_h \, $,  that is  $ \, s^\ell_*(A_h) \subseteq {}_\star{}K_h \, $  and  $ \, t^\ell_*(A_h) \subseteq {}_\star{}K_h \; $.  We shall prove it by showing that, for any  $ \, a \in A_h \, $,  one has  $ \, \big\langle s^\ell_*(a) \, , \, I_h^{\,n} \big\rangle \subseteq h^n A_h \, $,  $ \, \big\langle t^\ell_*(a) \, , \, I_h^{\,n} \big\rangle \subseteq h^n A_h \, $,  for all  $ \, n \in \N \, $.
 \vskip3pt
   For  $ t^\ell_* \, $,  if  $ \, \phi \in I_h \, $, then  $ \; \big\langle\, t^\ell_*(a) \, , \, \phi \,\big\rangle = \big\langle\, 1 \, , \phi \,\big\rangle \, a \, \in \, h \, A_h = h^1 A_h \; $;  thus  $ \, \big\langle\, t^\ell_*(a) \, , \, I_h^{\,1} \,\big\rangle \subseteq h^1 A_h \; $.
                                                               \par
   Now assume by induction that  $ \, \big\langle\, t^\ell_*(a) \, , \, I_h^{\,m} \big\rangle \subseteq h^m A_h \, $.  Let  $ \, \psi \in I_h^{\,m} \, $  and  $ \, \chi \in I_h \, $;  then
  $$  \big\langle\, t^\ell_*(a) \, , \psi \, \chi \,\big\rangle  \; = \;  \big\langle\, 1 \, , \psi \, \chi \,\big\rangle \, a  \; = \;  \big\langle\, s^\ell_*\big( \langle\, 1 \, , \psi \,\rangle \big) \, 1 \, , \, \chi \,\big\rangle \, a  $$
thus by the induction hypothesis and the case  $ \, m = 1 \, $  we see that  $ \, \big\langle t^\ell_*(a) \, , \psi \, \chi \big\rangle \in h^{m+1} A_h \; $.
                                                                 \par
   The case of  $ s^\ell_* \, $   --- being totally similar ---   is left to the reader.
 \vskip7pt
   {\it (2)}\,  Let us show that if  $ \, \omega , \omega' \in {}_\star{}K_h \, $,  then  $ \; \omega \, \omega' \in {}_\star{}K_h \; $.  Given  $ \, n \in \N \, $,  let  $ \, p, q \in \N \, $  be such that  $ \, \big\langle \omega \, , I_h^p \,\big\rangle \in h^n A_h \, $  and  $ \, \big\langle \omega' , I_h^q \,\big\rangle \in h^n A_h \, $.  Now take  $ \, \eta \in I_h^{p+q} \, $.  Then the identity
  $$  \big\langle\, \omega \, \omega' , \, \eta \,\big\rangle  \; = \;  \big\langle \, \omega \, t^\ell_*\big( \big\langle \omega' , \, \eta_{(2)} \big\rangle \big) \, , \, \eta_{(1)} \big\rangle  \; = \;  \big\langle \omega \, , \, \eta_{(1)} \, s_r\big( \big\langle \omega' , \, \eta_{(2)} \big\rangle \big) \big\rangle  $$
taking into account that  $ \; \Delta\big(I_h^{p+q}\big) \, \subseteq \! \sum\limits_{r+s=p+q} \! I_h^{\,r} \otimes_{\!{}_{\scriptstyle A_h}} \! I_h^{\,s} \; $  because  $ \; \Delta(I_h) \, \subseteq \, K_h \otimes_{\!{}_{\scriptstyle A_h}} \! I_h + I_h \otimes_{\!{}_{\scriptstyle A_h}} \! K_h \; $,  proves that  $ \, \big\langle\, \omega \, \omega' , \, I^{p+q}_h \,\big\rangle \in h^n A_h \; $.  Thus  $ {}_\star{}K_h $  is a subring of  $ {}_*{}K_h $   --- even an  $ A_h^{\,e} $--subring,  by  {\it (1)}.
 \vskip7pt
   {\it (3)}\,  Let us show that  $ {}_\star{}K_h $  is topologically free.  First we prove that it is complete for the  $ h $--adic  topology.  Indeed, as  $ K_h $  is topologically free (for its own  $ h $--adic  topology), so is  $ \text{\sl Hom}_{k[[h]]}(K_h \, , A_h) \, $  as well.  Now let  $ \, {(\lambda_n)}_{n \in \N} \, $  be a Cauchy sequence of elements in  $ {}_\star{}K_h \, $;  then this sequence converges to a unique limit  $ \, \lambda \in \text{\sl Hom}_{k[[h]]}(K_h \, , A_h) \, $.  Then it is easy to see that  $ \, \lambda \in \text{\sl Hom}_{A_h}(K_h \, , A_h) \; $.
                                                                    \par
   Now we show that  $ \, \lambda \in {}_\star{}K_h \; $.  Given  $ \, n \in \N \, $,  there exists  $ n_1 \in \N \, $  such that  $ \, \lambda_{n_1} - \lambda \, $  takes values in  $ \, h^n A_h \; $.  As  $ \, \lambda_{n_1} \in {}_\star{}K_h \, $,  there exists  $ \, n_2 \in \N \, $  such that  $ \, \big\langle \lambda_{n_1} \, , I_h^{n_2} \big\rangle \in h^n A_h \; $.  But then we have  $ \, \big\langle \lambda \, , I_h^{n_2} \big\rangle \in h^n A_h \, $  and so we conclude that  $ \, \lambda \in {}_\star{}K_h \; $.
                                                                    \par
   Finally, as  $ {}_\star{}K_h $  is complete for the  $ h $--adic  topology and without torsion, it is topologically free.
 \vskip7pt
   {\it (4)}\,  Now we show that  $ \; {}_\star{}K_h \Big/ h \, {}_\star{}K_h \, = \, {}_{{}_{\scriptstyle \star}}\big( K_h \big/ \,h\,K_h \big) \, = \, {}_\star{J^r(L)} \, = \, V^\ell(L) \; $.
 \vskip3pt
   Let  $ \, \lambda \in {}_\star{}K_h \, $,  so that  $ \lambda $  as a map from  $ K_h $  (with the  $ I_h $--adic topology) to  $ A_h $  (with the  $ h $--adic  topology) is continuous.  Then  $ \lambda $ induces (modulo  $ h \, $)  a map  $ \; \overline{\lambda} : J^r(L) \longrightarrow A \; $  which is  $ 0 $  on  $ \, \J^n \, $  for  $ \, n \gg 0 \, $,  where  $ \, \J := \text{\sl Ker}\,\big(\partial_{\scriptscriptstyle J^r(L)}\big) \, $.  We claim that the kernel of the map  $ \, \chi : \lambda \mapsto \overline{\lambda} \, $  is  $ \, h\,{}_\star{}K_h \, $:  indeed, it is obvious that  $ \; h\,{}_\star{}K_h \subseteq \text{\sl Ker}\,(\chi) \; $,  and the inverse inclusion follows from the fact that  $ A_h $  is topologically free.  Therefore we have an injective map  $ \; \overline{\chi} : {}_\star{}K_h \big/ h\,{}_\star{}K_h \longrightarrow {}_{{}_{\scriptstyle \star}} \big( K_h \big/ \,h\,K_h \big) \, = \, {}_\star{J^r(L)} \, = \, V^\ell(L) \; $  induced by  $ \chi $  (modulo  $ h \, $),  and we are left to show that  $ \overline{\chi} $  is surjective too.
 \vskip3pt
   We distinguish two cases:
 \vskip4pt
   {\it $ \underline{\text{Finite free case}} $:}  Assume that  $ L $  as an  $ A $--module  is  {\sl free of finite type}.  Let  $ \, \{ \overline{e}_1 \, , \dots , \overline{e}_n \} \, $  be an  $ A $--basis  of  $ L \, $.  Then $ \; \big\{\, \underline{\overline{e}}^{\underline{\alpha}} := \overline{e}_1^{\alpha_1} \cdots \overline{e}_n^{\alpha_n} \,\big|\, \underline{\alpha} := (\alpha_1\,,\dots,\,\alpha_n) \in \N^n \,\big\} \, $  is a basis of  $ V^\ell(L) \, $,  by the Poincar\'e-Birkhoff-Witt theorem.  Define  $ \, \overline{\xi}_i \in K := J^r(L) \, $  by  $ \; \big\langle\, \overline{\xi}_i \, , \overline{e}_1^{\alpha_1} \cdots \overline{e}_n^{\alpha_n} \big\rangle \, = \, \prod_{j=1}^n \delta_{\alpha_j,\,\delta_{i,j}} \; $.
                                                              \par
   Let  $ \, \xi_i \in K_h \, $  be a lifting of  $ \overline{\xi}_i $  such that  $ \, \partial_h(\xi_i) = 0 \, $.  We denote (ordered) monomials in the  $ \overline{\xi}_i $'s  or in the  $ \xi_i $'s  by  $ \, \underline{\overline{\xi}}^{\,\underline{\alpha}} := \overline{\xi}_1^{\,\alpha_1} \cdots \overline{\xi}_n^{\,\alpha_n} \, $  and  $ \, \underline{\xi}^{\underline{\alpha}} := \xi_1^{\alpha_1} \cdots \xi_n^{\alpha_n} \, $  respectively.  Note that  $ \, \underline{\xi}^{\underline{\alpha}} \in I_h^{|\underline{\alpha}|} \, $,  where  $ \, |{\underline{\alpha}}| := \sum_{i=1}^n \alpha_i \; $.  Let  $ \; \lambda \in {}_{{}_{\scriptstyle \star}}\! \big( K_h \big/ \,h\,K_h \big) = {}_\star{J^r(L)} = V^\ell(L) \; $  be given: we write  $ \, \bar{a}_{\underline{\alpha}} := \big\langle\, \lambda \, , \, \underline{\overline{\xi}}^{\,\underline{\alpha}} \,\big\rangle \in A \, $,  and note that all but finitely many of the  $ \bar{a}_{\underline{\alpha}} $'s  are zero.  Let  $ \, a_{\underline{\alpha}} \in A_h \, $  be any lifting of  $ \bar{a}_{\underline{\alpha}} $  (for all  $ \, \underline{\alpha} \in \N^n \, $),  with the condition that whenever  $ \, \bar{a}_{\underline{\alpha}} = 0 \, $  we take also  $ \, a_{\underline{\alpha}} = 0 \, $.  Now we define  $ \, \Lambda \in {}_*{}K_h \, $  by setting  $ \, \big\langle\, \Lambda \, , \underline{\xi}^{\underline{\alpha}} \,\big\rangle := a_{\underline{\alpha}} \; $.  As  $ \; I_h^{\,m} = \hskip-5pt \sum\limits_{|\underline{\alpha}| + s \geq m} \hskip-5pt h^s \, \underline{\xi}^{\underline{\alpha}} \, t_r(A_h) \, $,  \, it is easy to check that if  $ \, n \, \in \N \, $  then  $ \, \big\langle\, \Lambda \, , I_h^{\,m} \,\big\rangle \subseteq h^n A_h \, $  for  $ \, m \gg 0 \, $.  Hence  $ \, \Lambda \in {}_{\star}K_h \, $,  and by construction  $ \, \overline{\chi}\;\big(\, \Lambda \!\! \mod h\,{}_\star{}K_h \big) = \lambda \, $,  so that the map  $ \overline{\chi} $  is onto, q.e.d.
 \vskip4pt
   {\it $ \underline{\text{General case}} $:}  By our overall assumption,  $ L $  as an  $ A $--module  is  {\sl projective of finite type}.  Then we resume the setup and notation of in  \S \ref{L-R_proj->free}:  there exists a finitely generated projective  $ A $--module  $ Q $  such that  $ \, L \oplus Q = F \, $  is a finite free  $ A $--module,  and we consider the free  $ A $--module  $ \, L_Q := L \oplus \big( A \otimes_k Z \big) \, $  with  $ \, Z := Y \oplus Y \oplus Y \oplus \cdots \; $.  From an  $ A $--basis  $ \, \{ b_1 \, , \dots , \, b_n \} \, $  of  $ Y $  we get a ``good basis'' of elements  $ \bar{e}_t $  indexed by  $ \, T := \N \times \{1,\dots,n\} \, $,  i.e.~$ \; L_Q = \mathop{\oplus}\limits_{t \in T} k \, \bar{e}_t \; $.  Fixing on  $ T $  any total order, the PBW theorem yields  $ \, \big\{\, \underline{\bar{e}}^{\,\underline{\alpha}} := \prod_{t\in T} \bar{e}_t^{\,\alpha_t} \,\big|\, \underline{\alpha} = {(\alpha_t)}_{t \in T} \in T^{(\N)} \big\} \, $  is an  $ A $--basis  of  $ V^\ell(L_Q) \, $.  Let  $ \, \overline{\xi}_j \, $  be the element of  $ \, J^r_f(L_Q) \, $  defined by
  $ \; \big\langle\, \overline{\xi}_j \, , \, \underline{\bar{e}}^{\,\underline{\alpha}} \,\big\rangle = 1 \; $  if  $ \, \underline{\alpha} = {(\alpha_t = \delta_{t,j})}_{t \in T} \, $,  $ \; \big\langle\, \bar{\xi}_j \, , \, \underline{\bar{e}}^{\,\underline{\alpha}} \,\big\rangle = 0 \; $  otherwise.
 If  $ \; A{\big[\big[ {\{ X_t \}}_{t \in T} \big]\big]}_f := \hskip-5pt \bigcup\limits_{i_1 < \cdots < i_n} \hskip-7pt A[[ X_{i_1} \, , \dots, X_{i_n} ]] \; $,  \, then one has  $ \, J^r_f(L_Q) = A{\big[\big[ {\big\{ \overline{\xi}_t \big\}}_{t \in T} \big]\big]}_f \; $.

 Now consider the quantization  $ K_{h,Y} $  of  $ J^r_f(L_Q) $   --- cf.~\S \ref{ext-QFSAd's}.  Recall  (cf.~\S \ref{ext-QFSAd's})  that

 \vskip-5pt
  $$  K_{h,Y}  \;\; := \;\;  h\text{--adic completion of} \;\;\; {\textstyle \sum_{n \in \N}} \, K_h\widetilde{ \otimes}_k {S(Y^*)}^{\widetilde{\otimes}\, n} \otimes 1 \otimes 1 \otimes 1 \cdots  $$

 \vskip-1pt
\noindent
 where  $ \, K_h \, \widetilde{\otimes}_k \, {S(Y^*)}^{\widetilde{\otimes}\, n} \otimes 1 \otimes 1 \cdots \; $
is the  $ \, \big( {\big( {S(Y^*)}^{\otimes n} \big)}^+ \! \otimes 1 \otimes 1 \cdots \big) $--adic  completion of  $ \, K_h \otimes_k \, {S(Y^*)}^{\otimes n} \otimes 1 \otimes 1 \cdots \; $.  By construction, every  $ \overline{\xi}_i $  belongs to some  $ \, K \widetilde{\otimes}_k S(Y^*)^{\widetilde{\otimes}\, n_i} \otimes_k 1 \otimes_k 1 \otimes_k \cdots \; $  ($ \, n_i \in \N \, $).  Let  $ \xi_i $  be any lifting of  $ \overline{\xi}_i $  in  $ \, K_h \widetilde{\otimes}_k S(Y^*)^{\widetilde{\otimes}\, n_i} \otimes_k 1 \otimes_k 1 \otimes_k \cdots \; $  such that  $ \, \big( \partial_h \otimes \epsilon_{S(Z^*)} \big)(\xi_i) = 0 \; $.  Given  $ \, a \in A_h \, $,  we denote again by  $ a $  the element  $ \, t^r(a) \in t^r(A_h) \subseteq K_h \, $.  Let also  $ \, \sigma : A \lhook\joinrel\longrightarrow A_h \, $  be a section of the natural projection map from  $ A_h $  to  $ A \, $,  let  $ \, \J_{h,Y} := \text{\sl Ker}\,(\partial_h) = \partial_h^{-1}\big(\{0\}\big) \, $  and  $ \, I_{h,Y} := \partial_h^{-1}\big(h\,A_h\big) \, $:  taking into account that  $ \, t^r(A) = A^{\text{\it op}} \, $  and  $ \, t^r(A_h) = A_h^{\text{\it op}} \, $,  one has
 \vskip-8pt
  $$  \displaylines{
   K_{h,Y}  \, = \,  \Big\{\, {\textstyle \sum_{n \in \N}} \, h^n P^\sigma_n\big(\underline{\xi}\big) \,\Big|\, P_n \in \! {\big[\big[ {\{ X_t \}}_{t \in T} \big]\big]}_{\!f}\,A^{\text{\it op}} \Big\}  \, = \,  \Big\{\, {\textstyle \sum_{n \in \N}} \, h^n P_n\big(\underline{\xi}\big) \,\Big|\, P_n \in \! {\big[\big[ {\{ X_t \}}_{t \in T} \big]\big]}_{\!f}\,A_h^{\text{\it op}} \Big\}  \cr
   I_{h,Y}  \, = \,  \big( h \, , {\{ \xi_t \}}_{t \in T} \big)  \qquad ,  \quad \qquad  \J_{h,Y}  \, = \,  {\textstyle \sum_{t \in T}} \, \xi_t \, K_{h,Y}  }  $$
 \vskip-3pt
\noindent
 where round braces stand for ``two-sided ideal generated by'',  and  $ \, {\big[\big[ {\{ X_t \}}_{t \in T} \big]\big]}_f\,A \, $,  respectively  $ \, {\big[\big[ {\{ X_t \}}_{t \in T} \big]\big]}_f\,A_h \, $,  denotes the ring of formal power series with coefficients  {\sl on the right\/}  chosen in  $ A \, $,  respectively in  $ A_h \, $,  each one involving only finitely many indeterminates  $ X_t \, $.

\vskip3pt

   Now,  $ L_Q $  as an  $ A $--module  is free but not finite; however,  $ J^r_f(L_Q) $  and its quantization  $ K_{h,Y} $  have enough ``finiteness'' behavior as to let the arguments for the finite free case apply again.  In other words, the analysis we carried on for the finite free case can be applied again in the present, general context working with  $ K_{h,Y} \, $.  Indeed, let us remark that

 \vskip-10pt
  $$  I_{h,Y}  \, := \,  h\text{--adic completion of} \;\, \bigg(\, {\textstyle \sum\limits_{n \in \N}} \, I_h \,\widetilde{\mathop\otimes\limits_{\scriptscriptstyle k}}\, S(Y^*)^{\widetilde{\otimes}_k n} \mathop\otimes\limits_{\scriptscriptstyle k} 1 \mathop\otimes\limits_{\scriptscriptstyle k} \cdots \, +
   {\textstyle \sum\limits_{n \in \N_+}} \! K_h \,\widetilde{\mathop\otimes\limits_{\scriptscriptstyle k}}\, \Big( {S(Y^*)}^{\widetilde{\otimes}_k n} \Big)^{\!+} \! \mathop\otimes\limits_{\scriptscriptstyle k} 1 \mathop\otimes\limits_{\scriptscriptstyle k} \cdots \bigg)  $$
 \vskip-4pt
\noindent
 while on the other hand  $ \; V^\ell(L_Q) = V^\ell(L) \oplus \big( V^\ell(L) \otimes_k {S(Z)}^+ \big) \; $.  Now let  $ \, K := J^r(L) \, $  and  $ \, \lambda \in {}_\star{}K \, $.  As
 $ \; J^r_f(L_Q) = K \oplus \, \sum_{n \in \N_+} \!
 K \widetilde{\otimes}_k {\big( {S(Y^*)}^{\widetilde{\otimes}\, n} \big)}^{\! +} \! \otimes_k 1 \otimes_k \cdots \; $,  we can extend  $ \lambda $  to an element  $ \, \mu \in {}_\star{}J^r_f(L_Q) \, $  by  $ \; \mu\big|_{\sum_{n \in \N_+} \! K \, \widetilde{\otimes}_k \, {({S(Y^*)}^{\otimes n})}^{\! +} \! \otimes_k 1 \otimes_k \cdots} := 0 \; $  and  $ \; \mu{\big|}_K := \lambda \; $.  By the arguments used in the finite free case,  $ \mu $  can be lifted to an element  $ \, M \in {}_{\star_f}K_{h,Y} \, $;  then  $ \, \Lambda = M{\big|}_{K_h} \in {}_\star{}K_h \, $  is a lift of  $ \lambda \; $.  So the (injective) map  $ \; \overline{\chi} : {}_\star{}K_h \big/ h\,{}_\star{}K_h \longrightarrow {}_{{}_{\scriptstyle \star}}\! \big( K_h \big/ \,h\,K_h \big) \, = \, {}_\star{J^r(L)} \, = \, V^\ell(L) \; $  is surjective.
 \vskip7pt
   {\it (5)}\,  Let us now show that  $ \; \Delta\big( {}_\star{}K_h \big) \subseteq {}_\star{}K_h \, \widehat{\otimes}_{\!{}_{\scriptstyle A_h}} \! {}_\star{}K_h \; $  for the ``coproduct map''  $ \Delta $  given by the transpose map of the multiplication in  $ \, K_h \, $.
                                                                   \par
   Let  $ \; \Lambda \in {}_\star{}K_h \; $.  We know that modulo  $ h $  one has  $ \; \overline{\Delta(\Lambda)} \in {}_\star{}K \otimes_{\!{}_{\scriptstyle A}} \! {}_\star{}K \; $.  Now write  $ \; \overline{\Delta(\Lambda)} = \sum \lambda^{(1)} \otimes \lambda^{(2)} \; $  (a finite sum) with  $ \, \lambda^{(1)}, \lambda^{(2)} \in {}_\star{}K \; $,  and let  $ \Lambda_h^{(1)} $  and  $ \Lambda_h^{(2)} $  in  $ {}_\star{}K_h $  be liftings of  $ \lambda^{(1)} $  and  $ \lambda^{(2)} $,  i.e.~$ \, \overline{\Lambda_h^{(1)}} = \lambda^{(1)} \, $  and  $ \, \overline{\Lambda_h^{(2)}} = \lambda^{(2)} \, $:  then  $ \; \Delta(\Lambda_h) - \sum \Lambda_h^{(1)} \otimes \Lambda_h^{(2)} \in h \, \big( {}_*{}K_h \, \widehat{\otimes}_{\!{}_{\scriptstyle A_h}} \! {}_*{}K_h \big) \, $,  so that  $ \; h^{-1} \big( \Delta(\Lambda_h) - \sum \Lambda_h^{(1)} \otimes \Lambda_h^{(2)} \big) \in {}_*{}K_h \, \widehat{\otimes}_{\!{}_{\scriptstyle A_h}} \! {}_*{}K_h \; $.  In addition, whenever  $ \, p + q \gg 0 \, $  one has also  $ \; \big\langle\, \Delta(\Lambda_h) - \sum \Lambda_h^{(1)} \otimes \Lambda_h^{(2)} \, , I_h^{\,p} \otimes I_h^{\,q} \,\big\rangle \in h^2 A_h \; $;  therefore we find that
  $$  \overline{h^{-1} \big( \Delta(\Lambda_h) - \sum \Lambda_h^{(1)} \otimes \Lambda_h^{(2)} \big)}
\in {}_\star{}K \otimes_{\!{}_{\scriptstyle A}} \! {}_\star{}K  $$
We can carry on this argument and eventually show that  $ \, \Delta(\Lambda_h) \in {}_\star{} K_h \, \widehat{\otimes}_{\!{}_{\scriptstyle A_h}} \! {}_\star{}K_h \; $,  q.e.d.

 \vskip7pt

   {\it (6)}\,  Altogether, the steps  {\it (1)--(5)\/}  above prove that  $ {}_\star{\!}K_h $  is a LQUEAd (over  $ A_h $),  whose semiclassical limit  $ \, {}_\star{}K_h \big/ h\,{}_\star{}K_h \, $  is exactly isomorphic (as a left bialgebroid over  $ A \, $)  to  $ V^\ell(L) \, $.  Now we show that the structure of Lie-Rinehart bialgebra induced on  $ L $  by the quantization  $ {}_\star{}K_h $  of  $ V^\ell(L) $  is the same as that induced by the quantization  $ K_h $  of  $ J^r(L) \, $.  To this end, let  $ {[\,\ ,\ ]'} $,  $ \omega' $,  be the Lie bracket and the anchor map on  $ L^* $  induced by  $ {}_\star{}K_h \, $,  and  $ [\,\ ,\ ]'' \, $,  $ \omega'' \, $,  those induced by  $ K_h \, $.
                                                                                  \par
   We proceed like in the proof of  Theorem \ref{dual_QUEAd's=QFSAd's}.  Our goal is to prove that  $ \, \omega' = \omega'' \, $  and  $ \, {[\,\ ,\ ]}' = {[\,\ ,\ ]}'' \, $;  thus recall that (cf.~Remarks \ref{props_Lie-bialg}{\it (b)\/})  $ \omega' $  and  $ {[\,\ ,\ ]}' $  are uniquely determined by
  $$  \omega'(\Phi)(a) \, = \big\langle \delta_L\!(a) \, , \Phi \big\rangle  \;\; ,  \quad  \big\langle\, \Theta \, , {[\Phi\,,\Psi]}' \,\big\rangle \, = \, \omega'(\Phi)\big(\langle\, \Theta \, , \Psi \rangle\big) \, - \, \omega'(\Psi)\big(\langle \Theta \, , \Phi \rangle\big) \, - \, \big\langle \delta_L\!(\Theta) \, , \Phi \otimes \Psi \big\rangle  $$
(for all $ \, \Phi \, , \Psi \in L^* \, $,  $ \, \Theta \in L \, $,  $ \, a \in A \, $),  where  $ \delta_L\!(a) $  and  $ \delta_L\!(\Theta) $  are defined by the formula for  $ \delta $  in  Theorem \ref{semiclassical_limit-V^ell(L)}.  Thus it is enough to prove that (for all  $ \, \Phi \, , \Psi \in L^* \, $,  $ \, \Theta \in L \, $,  $ \, a \in A \, $)
  $$  \omega''(\Phi)(a) \, = \big\langle \delta_L\!(a) \, , \Phi \big\rangle  \; ,  \;\;  \big\langle\, \Theta , {[\Phi\,,\Psi]}'' \,\big\rangle \, = \, \omega''(\Phi)\big(\!\langle\, \Theta , \Psi \rangle\!\big) - \omega''(\Psi)\big(\!\langle \Theta , \Phi \rangle\!\big) - \big\langle \delta_L\!(\Theta) \, , \Phi \otimes \Psi \big\rangle  \!\quad   \eqno (5.7)  $$
   \indent   To prove (5.7), choose liftings  $ \, \phi' , \psi' \in {J^r(L)}_h =: K_h \, $,  with the additional condition that  $ \, \phi' , \psi' \in \J_{\scriptscriptstyle {J^r(L)}_h} := \text{\sl Ker}\,\big( \partial_{{J^r(L)}_h} \big) \, $  (this is always possible), a lifting  $ \, \theta \in {V^\ell(L)}_h := {}_\star{J^r(L)}_h \, $  of  $ \Theta \, $  and a lifting  $ \, a' \in A_h \, $  of  $ \, a \, $.  Now direct computation gives
  $$  \displaylines{
   \omega'(\Phi)(a)  \;\; = \;\;  \big\langle\, \delta_L\!(a) \, , \Phi \big\rangle  \;\; = \;\;   \hfill  \cr
   = \;\;  h^{-1} \big\langle\, t^\ell_*\big(a'\big) - s^\ell_*\big(a'\big) \, , \, \phi' \,\big\rangle \!\!\mod h \, A_h  \;\; = \;\;  h^{-1} \big\langle\, 1 \, , \phi' \, s_r\big(a'\big) - s_r\big(a'\big) \, \phi' \,\big\rangle  \!\!\mod h \, A_h  \;\; =  \cr
    \hfill   = \;\;  \bigg( {{\; \phi' \, s_r\big(a'\big) - s_r\big(a'\big) \, \phi' \;} \over h} \!\!\!\mod h \, K_h \bigg) \!\! \mod \J_{\scriptscriptstyle J^r(L)}  \;\; = \;\;  \omega''(\Phi)(a)  }  $$
where we exploited the fact that the involved pairing a  {\sl right\/}  bialgebroid pairing  (cf.~Definitions \ref{left and right 1}  and  \ref{left and right 2}).
  Thus the first identity in (5.7) is verified.
                                                                  \par
   As to the rest, we write  $ \, \Delta(\theta) = \theta_{(1)} \otimes \theta_{(2)} \, $  as  $ \, \Delta(\theta) = \theta \otimes 1 + 1 \otimes \theta + h \, \theta_{[1]} \otimes \theta_{[2]} \, $,  so that  $ \; \big( \theta_{[1]} \otimes \theta_{[2]} \big) \mod h \, {V^\ell(L)}_h \displaystyle{{{\,}_\triangleleft}{\mathop \times_A}{{\,}_\triangleright}} {V^\ell(L)}_h \, =: \, \Delta^{[1]}(\Theta) \, $,  as in  Definition \ref{semiclassical_limit-V^ell(L)}.  Moreover, let us set  $ \, \phi := \phi' \!\!\mod h \, {J^r(L)}_h \, $,  $ \, \psi := \psi' \!\!\mod h \, {J^r(L)}_h \, $,  which are lifts of  $ \Phi $  and  $ \Psi $  in  $ J^r(L) \, $,  and actually belong to  $ \J_{\scriptscriptstyle J^r(L)} \, $.  Then direct computation gives
  $$  \Big\langle \Theta \, , \, {[\Phi,\Psi]}'' \Big\rangle  \,\; = \;\,  \big\langle\, \Theta \, , \, \{\phi,\psi\} \,\big\rangle  \,\; = \;  \bigg\langle \theta \, , \, {{\; \phi' \, \psi' \, - \, \psi' \, \phi' \;} \over h} \bigg\rangle \!\!\!\mod h \, A_h  $$
 Now, in the proof of  Theorem \ref{dual_QUEAd's=QFSAd's}   --- namely, to prove the second part of (5.6) ---   we saw that
   $$  \bigg\langle \theta \, , \, {{\; \phi' \, \psi' \, - \, \psi' \, \phi' \;} \over h} \bigg\rangle \!\!\!\mod h \, A_h  \;\; = \;\;  \omega'(\Phi)\big(\! \big\langle\, \Theta \, , \Psi \,\big\rangle \!\big) \; - \; \omega'(\Psi)\big(\! \big\langle\, \Theta \, , \Phi \,\big\rangle \!\big) \; - \; \big\langle \delta_L\!(\Theta) \, , \, \Phi \otimes \Psi \big\rangle  $$
so that the second identity in (5.7) is proved.
 \vskip7pt
   At last,
let now cope with the case of  $ {}^\star{}K_h \, $.  Clearly, we can proceed much like for  $ {}_\star{}K_h \, $:  one proves that  $ \, {}^\star{}K_h = {}^\star{\!}{J^r(L)}_h \, $  is a topological left bialgebroid, whose specialization modulo  $ h $  is  $ \, {}^\star{\!}J^r(L) \cong V^\ell(L) \, $,  hence we can claim that  $ \, {}^\star{\!}{J^r(L)}_h \in \text{\rm (LQUEAd)}_{A_h} \, $  is a quantization of  $ V^\ell(L) \, $.
                                                                     \par
   A difference arises about the last part of claim  {\it (a)},  concerning the two Lie-Rinehart algebra structures induced on  $ L^* \, $:  indeed, the difference in the outcome   --- a minus sign ---   is due to the fact that the natural pairing (given by evaluation) among the left bialgebroid  $ {}^\star{\!}{J^r(L)}_h $  and the right bialgebroid  $ {J^r(L)}_h $  is now a  {\sl left\/}  bialgebroid pairing  (cf.~Definitions \ref{left and right 1}  and  \ref{left and right 2})   --- while for  $ {}_\star{}{J^r(L)}_h $  and  $ {J^r(L)}_h $  it is a  {\sl right\/}  one.
 Explicit computations are (again) much like those in the proof of  Theorem \ref{dual_QUEAd's=QFSAd's}  (for the very last part of claim  {\it (a)\/}),  just as it occurs for  $ \, {}_\star{}K_h = {}_\star{}{J^r(L)}_h \; $.
 \vskip9pt
   {\it (b)}\,  The arguments used to prove claim  {\it (a)\/}  clearly adapt to claim  {\it (b)\/}  as well. Otherwise,
 one can deduce  {\it (b)\/}  directly from claim  {\it (a)\/}  using general isomorphisms such as  $ \; {}_\star\big(U^{\text{\it op}}_{\text{\it coop}}\big) \, \cong \, (U^\star)^{\text{\it op}}_{\text{\it coop}} \; $  and  $ \; {}^\star\big( U^{\text{\it op}}_{\text{\it coop}} \big) \, \cong \,
 {(U_\star)}^{\text{\it op}}_{\text{\it coop}} \; $   --- see Remark \ref{dual_vs._op-coop}.
\end{proof}

\smallskip

 \subsection{Functoriality of linear duality for quantum groupoids}  \label{funct-lin-dual_q-groupd's}

\smallskip

   {\ } \quad   The results in  Sections \ref{lin-dual_QUEAd's}  and  \ref{lin-dual_QFSAd's}  about the duality constructions for quantum bialgebroids can be improved.  Indeed, they can be cast in the following, functorial version (cf.~Definition  \ref{def-QUEAd}  and  \ref{def-QFSAd}  for notation), which is the main outcome of this section:

\medskip

\begin{theorem}  \label{linear-duality}
 Left and right duals yield pairs of well-defined  {\sl contravariant}  functors
  $$  \displaylines{
   \text{\rm (LQUEAd)}_{A_h} \!\!\longrightarrow \text{\rm (RQFSAd)}_{A_h} \; ,  \;\;  H_h \mapsto H_h^{\,\ast} \; ,  \;\qquad
      \text{\rm (RQFSAd)}_{A_h} \!\!\longrightarrow \text{\rm (LQUEAd)}_{A_h} \; ,  \;\;  K_h \mapsto {{}_{\scriptstyle \star}}K_h  \cr
   \text{\rm (LQUEAd)}_{A_h} \!\!\longrightarrow \text{\rm (RQFSAd)}_{A_h} \; ,  \;\;  H_h \mapsto {H_h}_{\,{}_{\scriptstyle \ast}} \; ,  \;\qquad
      \text{\rm (RQFSAd)}_{A_h} \!\!\longrightarrow \text{\rm (LQUEAd)}_{A_h} \; ,  \;\;  K_h \mapsto {}^\star{\!}K_h  \cr
   \text{\rm (RQUEAd)}_{A_h} \!\!\longrightarrow \text{\rm (LQFSAd)}_{A_h} \; ,  \;\;  H_h \mapsto {}^\ast{}\!H_h \; ,
\;\qquad
      \text{\rm (LQFSAd)}_{A_h} \!\!\longrightarrow \text{\rm (RQUEAd)}_{A_h} \; ,  \;\;  K_h \mapsto {K_h}_{\,{}_{\scriptstyle \star}}  \cr
   \text{\rm (RQUEAd)}_{A_h} \!\!\longrightarrow \text{\rm (LQFSAd)}_{A_h} \; ,  \;\;  H_h \mapsto {}_{\scriptstyle \ast}H_h \; ,
\;\qquad
      \text{\rm (LQFSAd)}_{A_h} \!\!\longrightarrow \text{\rm (RQUEAd)}_{A_h} \; ,  \;\;  K_h \mapsto K_h^{\,\star}  }  $$
which are (pairwise) inverse to each other, hence yield pairs of antiequivalences of categories.
\end{theorem}

\begin{proof}
 It is clearly enough to present the proof for just one pair of functors, say those in first line.
 \vskip4pt
   Let  $ \, H_h = {V^\ell(L)}_h \in \text{\rm (LQUEAd)}_{A_h} \, $.  For any  $ \, \lambda \in H_h^{\,\ast} \, $  and any  $ \, \eta \in H_h \, $,  let  $ \, \text{\it ev}_\eta(\lambda) := \lambda(\eta) \, $,  and consider the map  $ \; H_h \longrightarrow {{}_{\scriptstyle \star}}\big( H_h^{\,\ast} \big) \; $  given by  $ \, \eta \mapsto \text{\it ev}_\eta \, $;  note that  {\it a priori\/}  this map takes values in  $ {{}_{\scriptstyle *}}\big( H_h^{\,\ast} \big) \, $,  but Lemma \ref{continuity of the evaluation}  actually proves that every  $ \text{\it ev}_{\eta} $  belongs to  $ {}_\star(H^*) \, $.
                                                                      \par
  Now, this map is an isomorphism in  $ \text{\rm (LQUEAd)}_{A_h} \, $  because it is an isomorphism modulo  $ h \, $.  The other points can also be proved, by standard arguments, in a similar way.
\end{proof}

\bigskip

\section{Drinfeld's functors and quantum duality}  \label{Drinf-functs_q-duality}

\smallskip

   {\ } \quad   In this section we present the main new contribution in this paper, namely the definition of Drinfeld's functors and the equivalences   --- instead of antiequivalences! ---   of categories established via them among (left or right) QUEAd's and QFSAd's.

\smallskip

 \subsection{The Drinfeld's functor  $ \, {(\ )}^\vee \, $}  \label{Dr-functors_vee}

\smallskip

\begin{definition}
 Let  $ \, K_h \in \text{\rm (RQFSAd)}_{A_h} \, $.  We set  $ \, I_h := \partial_h^{-1}(h\,A_h) \, $  and  $ \, \J_h := \text{\sl Ker}(\partial_h) \, $,  where  $ \partial_h $  is the counit of  $ K_h \, $;  then one has  $ \; I_h = \J_h + h \, K_h \; $.  We define
 $ \; K_h^\times := \, s^r(A_h) + {\textstyle \sum\limits_{n \in \N_+}} h^{-n} \, I_h^n  \; = \;  s^r(A_h) \, + {\textstyle \sum\limits_{n \in \N_+}} h^{-n} s^r(A_h) \, \J_h^n \;  $,  which is a  $ k[[h]] $--submodule of  $ \, K_F := k((h)) {\mathop \otimes\limits_{k[[h]]}} K_h \, $,  and we denote by  $ \, K_h^\vee $  the  $ h $--adic completion of the  $ k[[h]] $--module  $ K_h^\times \, $.
                                                          \par
   Moreover, in an entirely similar way we define  $ K_h^\vee $  for any  $ \, K_h \in \text{\rm (LQFSAd)}_{A_h} \, $.
\end{definition}

\smallskip

\begin{remarks}  {\ }
 \vskip3pt
   {\it (a)}\,  Note that  $ \J_h $  is not an  $ (A_h\,, A_h) $--subbimodule  of  $ K_h \, $,  in general.  Indeed, if  $ \, a \in A_h \, $  and  $ \, \psi \in \J_h \, $,  it is clear (from the properties of the counit of a right bialgebroid) that  $ \; \psi \, s_r(a) \, , \, \psi \, t_r(a) \in \J_h \; $;  but we cannot prove in general that  $ \, s_r(a) \, \psi \, $  and  $ \, t_r(a) \, \psi \, $  belong to  $ \J_h \, $.  On the other hand, one has that  $ I_h $  instead  {\sl is\/}   definitely an  $ (A_h\,, A_h) $--subbimodule.  For this reason, it is better to (define and) describe  $ K_h^\times $  and  $ K_h^\vee $  using  $ I_h $  than using  $ \J_h \, $.
 \vskip3pt
   {\it (b)}\,  Let  $ K $  be a LQFSAd, respectively a RQFSAd.  Then  $ {( K_h )}^{\text{\it op}}_{\text{\it coop}} $  is a RQFSAd, respectively a LQFSAd.  It easily follows from definitions that  $ \, {\big(\! {( K_h )}^{\text{\it op}}_{\text{\it coop}} \big)}^{\!\vee} = {\big( K_h^\vee \big)}^{\text{\it op}}_{\text{\it coop}} \; $.

\end{remarks}

\smallskip

\begin{free text}  \label{descr-K_h^vee}
 {\bf Description of  $ K_h^\vee \, $.}\,  Directly from its very definition, we can find out a description of  $ K_h^\vee \, $.  This is very neat in the case when the Lie-Rinehart algebra  $ L $   --- such that  $ K_h $  is a quantization of  $ J^r(L) $  or  $ J^\ell(L) $  ---   is of finite free type (as an  $ A $--module),  and can be reduced somehow to that case when  $ L $  instead is just of finite projective type.  Thus we distinguish these two cases.
 \vskip5pt
   {\it  (a)  $ \, \underline{\text{\sl Finite free case}} $:} \,  Let us assume that  $ L $  (as an  $ A $--module,  of finite type) is free.  Then we can explicitly describe  $ K_h^\vee \, $,  as follows.  Fix an  $ A $--basis  $ \, \{ \overline{e}_1 \, , \dots , \, \overline{e}_n \} \, $  of  $ L \, $,  and let  $ \bar{\xi}_i $  be the element of  $ \, \text{\sl Hom}\big( V^\ell(L) , A \big) = {V^\ell(L)}^* = J^r(L) \, $  defined (using standard multiindex notation) by
  $$  \big\langle\, \overline{\xi}_i \, , \, \overline{e}^{\,\underline{\alpha}} \,\big\rangle  \; = \;  \big\langle\, \overline{\xi}_i \, , \, \overline{e}_1^{\,\alpha_1} \cdots \overline{e}_n^{\,\alpha_n} \big\rangle  \; := \; \delta_{\alpha_1, 0} \cdots \delta_{\alpha_i,1} \cdots \delta_{\alpha_n,0}   \eqno  \forall \;\; \underline{\alpha} = (\alpha_1\,, \dots , \, \alpha_n) \in \N^n  \quad  $$
Let  $ \xi_i $  be an element of  $ K_h $  lifting  $ \bar{\xi}_i $  and such that  $ \, \partial_h(\xi_i) = 0 \; $.  If  $ \, a \in A_h \, $,  we shall write again  $ a $  to denote the element  $ \, t^r(a) \in K_h \; $.  We have the following descriptions
  $$  \displaylines{
   K_h  \; = \;  \big\{\, {\textstyle \sum_{\underline{d} \in \N^n}} \, \xi_1^{d_1} \cdots \xi_n^{d_n} \, a_{\underline{d}} \;\big|\; a_{\underline{d}} \in A_h^{\text{\it op}} \, , \; \forall \; \underline{d} \in \N^n \big\}
          \,\; \cong \;  A[[\,X_1,\dots,X_n]][[h\,]]  \cr
   I_h  \, = \, \big( h \, , \xi_1 \, , \dots , \xi_n \big)  \qquad ,  \quad \qquad  \J_h  \, = \, {\textstyle \sum_{i=1}^n} \, \xi_i \, K_h  }  $$
where the first line item is a (right)  $ A_h^{\text{\it op}} $--module  of formal power series (convergent in the  $ I_h $--adic  topology) and the last isomorphism is one of topological  $ k $--modules,  while round braces in second line stand once again for ``two-sided ideal generated by''.  By this and the very definition it follows that, writing  $ \, \check{\xi}_i := h^{-1} \xi_i \, $,  one has (the last isomorphism being one of topological  $ k $--modules)
   $$  K_h^\vee  \; = \;  \Big\{\, {\textstyle \sum_{\underline{b} \in \N^{n+1}}} h^{b_0} \, \check{\xi}_1^{\,b_1} \cdots \check{\xi}_n^{\,b_n} \, a_{\underline{b}} \;\Big|\; a_{\underline{b}} \in A_h^{\text{\it op}} \, , \; \forall \; \underline{b} \;\Big\}
 \,\; \cong \;  A\big[\check{X}_1,\dots,\check{X}_n\big][[h\,]]  $$
where the sum denotes formal series which are convergent in the  $ h $--adic  topology, and then also
  $$  \J_h^\vee  \; := \;  h^{-1} \J_h  \; = \; {\textstyle \sum_{i=1}^n} \, \check{\xi}_i \, K_h^\vee  \; = \;  \text{\rm right ideal of  $ \displaystyle{K_{h}^\vee} $  generated by the  $ \check{\xi}_i $'s}  $$
 \vskip5pt
   {\it  (b)  $ \, \underline{\text{\sl Finite projective case}} $:} \,  Assume now that  $ L $  (as an  $ A $--module)  is just finite projective (as usual in this work).  Like in  Subsection \ref{quant_proj-free}, we fix a finite projective  $ A $--module  $ Q $  such that  $ \, L \oplus Q = F \, $  is a finite free  $ A $--module,  we write  $ \; F = A \otimes_k Y \; $  where  $ Y $  is the  $ k $--span  of an  $ A $--basis  of  $ F \, $,  and we construct the (infinite dimensional) Lie-Rinehart algebra  $ \; L_Q = L \oplus \big( A \otimes_k Z \big) \; $  with  $ \; Z = Y \oplus Y \oplus Y \oplus \cdots \; $.  Then, for  $ \, {J^r(L)}_h := K_h \, $,  we can introduce the right bialgebroid  $ \, K_{h,Y} := {J^r(L)}_{h,Y} \, $  as in  \S \ref{ext-QFSAd's}:  namely (with notation as in  \S \ref{ext-QFSAd's}),  we recall that
  $$  K_{h,Y}  \;\; := \;\;  h\text{--adic completion of} \;\;\; {\textstyle \sum_{n \in \N}} \,
K_h \widetilde{\otimes}_k {S(Y^*)}^{\widetilde{\otimes}\, n} \otimes 1 \otimes 1 \otimes 1 \cdots  $$
 $ \, {\big( {S(Y^*)}^{\otimes\, n} \big)}^+ \, $  being the kernel of the natural counit map of  $ \, {S(Y^*)}^{\otimes\, n} \, $  and  $ \; K_h \widetilde{\otimes}_k {S(Y^*)}^{\widetilde{\otimes}\, n} \otimes 1 \cdots \; $  the  $ \, \big( {\big( {S(Y^*)}^{\otimes n} \big)}^+ \! \otimes 1 \otimes 1 \cdots \big) $--\,adic  completion of  $ \; K_h \otimes_k {S(Y^*)}^{\otimes n} \otimes 1 \otimes 1 \cdots \; $.
                                                              \par
   Furthermore, let  $ \partial_h $  be the counit of  $ K_{h,Y} \, $,  and  $ \; I_{h,Y} := \partial_h^{-1}\big( h \, A_h \big) \; $.  Then we have also
  $$  I_{h,Y}  \; := \;  h\text{--adic completion of} \;\, \Big(\, {\textstyle \sum\limits_{r \in \N}} \, I_h \mathop{\widetilde{\otimes}}\limits_k {S(Y^*)}^{\widetilde{\otimes}\, r} \otimes 1 \otimes \cdots  \, + \,
 {\textstyle \sum\limits_{s \in \N}} \, K_h \mathop{\widetilde{\otimes}}\limits_k
 {\big( {S(Y^*)}^{\widetilde{\otimes}\, s} \big)}^+ \otimes 1 \otimes \cdots \Big)  $$
   Basing upon these remarks, we can define  $ K_{h,Y}^{\,\vee} $  and describe it as above: namely, one has
  $$  K_{h,Y}^\vee  \, = \,  h\text{--adic completion of} \;\;
   {\textstyle \sum\limits_{n,m}} \, {\textstyle \sum\limits_{r+s=n}} \!
   h^{-n} \, \J_h^{\,r} \widetilde{\otimes}_k
{\big( {\big( S(Y^*)^{\widetilde{\otimes} m} \big)}^+ \big)}^s \otimes 1 \otimes 1 \otimes \cdots  \, = \,  K_h^\vee \, \widehat{\otimes}_k \, S(Z^{*_f})  $$
where  $ \, Z^{*_{\!{}_f}} = Y^* \oplus Y^* \oplus Y^* \oplus \cdots \; $.

   \indent   Let now  $ \, {\{e_t\}}_{t \in T := \N \times \{1,\dots,n\}} \, $  be a good basis of the  $ A $--module  $ L_Q \, $.  From the proof of  Theorem \ref{dual_QFSAd's=QUEAd's}  (step  {\it (4)\/} for the general case) we can select elements  $ \; \xi_t \in K_{h,Y} \; (t \in T) \; $  such that
  $$  \displaylines{
   K_{h,Y}  \; = \;  \Big\{\, {\textstyle \sum}_{n \in \N} \, h^n P_n\big( \big\{ \xi_t \big\}_{t \in T} \big) \;\Big|\; P_n \in {\big[\big[ {\{ X_t \}}_{t \in T} \big]\big]}_f \, A_h^{\text{\it op}} \,\Big\}  \,\; \cong \;  A{\big[\big[ {\{ X_t \}}_{t \in T} \big]\big]}_f[[h\,]]  \cr
   I_{h,Y}  \, = \, \big( h \, , {\{ \xi_t \}}_{t \in T} \big)  \qquad ,  \quad \qquad  \J_{h,Y}  \, = \, {\textstyle \sum_{t \in T}} \, \xi_t \, K_{h,Y}  }  $$
where  $ \, {\big[\big [ {\{ X_t \}}_{t \in T} \big] \big ]_f}\,A_h \, $  is the ring of formal power series  with coefficients  {\sl on the right\/}  chosen in  $ A_h \; $  involving only finitely many indeterminates  $ X_t \; $.
One easily finds, letting  $ \, \check{\xi}_t := h^{-1} \xi_t \, $,
that
   $$  K_{h,Y}^{\,\vee}  \,\; = \;\,  \Big\{\, {\textstyle \sum_{n \in \N}} \, h^n \, P_n\big( \big\{ \check{\xi}_t \big\}_{t \in T} \big) \;\Big|\; P_n \in \big[ \big\{\, \underline{\check{\xi}} \,\big\}_{t \in T} \big]A_h^{\text{\it op}} \, , \; \forall \; n \in \N \,\Big\}  \,\; \cong \;  A\big[ {\big\{ \check{X}_t \big\}}_{t \in T} \big][[h\,]]   $$
where the sum denotes formal series convergent in the  $ h $--adic  topology, and  $ \, {\big[ {\{ X_t \}}_{t \in T} \big]}\,A_h^{\text{\it op}} \, $  denotes the ring of polynomials with coefficients  {\sl on the right\/}  chosen in  $ A_h^{\text{\it op}} \; $.  We find also
  $$  \J_{h,Y}^\vee  \; := \;  h^{-1} \J_{h,Y}  \; = \; {\textstyle \sum_{t \in T}} \, \check{\xi}_t \, K_{h,Y}^\vee  \; = \;  \text{\rm right ideal of  $ \displaystyle{K_{h,Y}^\vee} $  generated by the  $ \check{\xi}_t $'s}   \eqno \diamondsuit  $$
\end{free text}

\bigskip

   It is time for the main result of this subsection.  In short, it claims that the construction  $ \, K_h \mapsto K_h^{\,\vee} \, $,  starting from a quantization of  $ L $   --- of type  $ J^{r/\ell}(L) $  ---  provides a quantization of the  {\sl dual\/}  Lie-Rinehart bialgebra  $ L^* $   --- of type  $ V^{r/\ell}(L^*) \; $;  moreover, this construction is functorial.

\begin{theorem}  \label{compute vee x QFSAd}  {\ }
 \vskip3pt
   (a)\,  Let  $ \, {J^r(L)}_h \in \text{\rm (RQFSAd)}_{A_h} \, $,  where  $ L $  is a Lie-Rinehart algebra which, as an  $ A $--module, is projective of finite type.  Then:
 \vskip3pt
   --- (a.1)\,  $ \, {J^r(L)}_h^{\,\vee} \in \text{\rm (RQUEAd)}_{A_h} \, $,  with semiclassical limit  $ \; {J^r(L)}_h^{\,\vee} \Big/ h \, {J^r(L)}_h^{\,\vee} \, \cong \, V^r(L^*) \, $.  Moreover, the structure of Lie-Rinehart bialgebra induced on  $ L^* $  by the quantization  $ {J^r(L)}_h^{\,\vee} $  of  $ V^r(L^*) $   --- as in  Theorem \ref{semiclassical_limit-V^r(L)}  ---   is dual to that induced on  $ L $  by the quantization  $ {J^r(L)}_h $  of  $ J^r(L) $  --- as in  Theorem \ref{semiclassical_limit-J^r(L)};
 \vskip3pt
   --- (a.2)\,  the definition of  $ \, {J^r(L)}_h \mapsto {J^r(L)}_h^{\,\vee} \, $  extends to morphisms in  $ \text{\rm (RQFSAd)} \, $,  so that we have a well defined (covariant) functor  $ \; {(\ )}^\vee : \text{\rm (RQFSAd)} \longrightarrow \text{\rm (RQUEAd)} \; $.
 \vskip5pt
   (b)\,  Let  $ \, {J^\ell(L)}_h \in \text{\rm (LQFSAd)}_{A_h} \, $,  where  $ L $  is a Lie-Rinehart algebra which, as an  $ A $--module, is projective of finite type.  Then:
 \vskip3pt
   --- (b.1)\,  $ \, {J^\ell(L)}_h^{\,\vee} \in \text{\rm (LQUEAd)}_{A_h} \, $,  with semiclassical limit  $ \; {J^\ell(L)}_h^{\,\vee} \Big/ h \, {J^\ell(L)}_h^{\,\vee} \, \cong \, V^\ell(L^*) \, $.  Moreover, the structure of Lie-Rinehart bialgebra induced on  $ L^* $  by the quantization  $ {J^\ell(L)}_h^{\,\vee} $  of  $ \, V^\ell(L^*) $   --- as in  Theorem \ref{semiclassical_limit-V^ell(L)}  ---   is dual to that induced on  $ L $  by the quantization  $ {J^\ell(L)}_h $  of  $ \, J^\ell(L) $  --- as in  Theorem \ref{semiclassical_limit-J^ell(L)};
 \vskip3pt
   --- (b.2)\,  the definition of  $ \, {J^\ell(L)}_h \mapsto {J^\ell(L)}_h^{\,\vee} \, $  extends to morphisms in  $ \text{\rm (LQFSAd)} \, $,  so that we have a well defined (covariant) functor  $ \; {(\ )}^\vee : \text{\rm (LQFSAd)} \longrightarrow \text{\rm (LQUEAd)} \; $.
 \vskip3pt
\end{theorem}

\begin{proof}
 {\it (a)}\,  In order to ease notation, let us write  $ \, K_h := {J^r(L)}_h \, $.
 \vskip3pt
   By definition,  $ K_h^\times $  is the unital  $ k [[h]]$--subalgebra  of  $ \, {\big( K_h \big)}_F := k((h)) \otimes_{k[[h]]} K_h \, $  generated by  $ \, h^{-1} I_h \, $  and  $ s^r(A_h) \, $:  thus it is automatically a unital  $ k[[h]] $--algebra.  It follows that  $ K_h^\vee $  is a unital  $ k[[h]] $--algebra  too, complete in the  $ h $--adic  topology.  Moreover,  $ I_h $  is an  $ (A_h,A_h) $--subbimodule of  $ K_h \, $:  this implies at once that  $ K_h^\times $  and  $ K_h^\vee $  are  $ (A_h,A_h) $--bimodules  too.  As  $ {\big( K_h \big)}_F $  is torsionless, so are  $ K_h^\times $  and its completion  $ K_h^\vee \, $;  also,  $ K_h^\vee $  is separated and complete, so it is topologically free.
 \vskip3pt
   Let us now see that the coproduct in  $ K_h $  induces a coproduct   --- in a suitable,  $ h $--adical  sense ---   for  $ K_h^\vee $  as well.  Given any  $ \, \phi \in I_h \, $,  we write  $ \, \Delta(\phi ) = \phi_{(1)} \otimes \phi_{(2)} \, $   --- a formal series (in  $ \Sigma $--notation)  ---   convergent in the  $ I_h $--adic  topology of  $ K_h \, $.  Writing  $ \, \phi_{(1)} \, $  and  $ \, \phi_{(2)} \, $  as
  $$  \displaylines{
   \phi_{(1)} \, = \, \phi_{(1)}^{+_s} + s^r\big(\partial_h\big(\phi_{(1)}\big)\big) \; ,  \qquad    \phi_{(2)} \, = \, \phi_{(2)}^{+_t} + t^r\big(\partial_h\big(\phi_{(2)}\big)\big) \; } $$
we have seen that
 $ \; \Delta(\phi) \, = \, \phi_{(1)}^{+_s} \otimes \phi_{(2)} + s^r\big(\partial_h\big(\phi_{(1)}\big)\!\big) \otimes \phi_{(2)}^{+_t} + s^r\!\big(\partial_h(\phi)\big) \otimes 1 \; $  belongs to the space  $ \, \big( I_h \,\widetilde{\otimes}_{\!{}_{\scriptstyle A_h}} K_h + K_h \,\widetilde{\otimes}_{\!{}_{\scriptstyle A_h}} I_h + h \, s^r\!(A_h) \,\widetilde{\otimes}_{\!{}_{\scriptstyle A_h}} 1 \big) \, $.
All this implies $ \Delta\big( h^{-1} \phi \big)  \in K_h^\vee \,\widetilde{\otimes}_{A_h} K_h^\vee  .$

In addition, we must observe the following.  Every  $ \, \phi \in I_h \, $  expands as an  $ I_h $--adically  convergent series  $ \, \phi = \sum_{n \in \N_+} \phi_n \, $  with  $ \, \phi_n \in I_h^{\,n} \, $  for all  $ \, n \in \N_+ \, $;  but then  $ \, \phi_n \in I_h^{\,n} = h^n {(h^{-1} I_h)}^n \, $  for every  $ n $  and so  $ \, h^{-1} \phi \, $  expands as a series  $ \, h^{-1} \phi = \sum_{n \in \N_+} h^{n-1} {(h^{-1} I_h)}^n \, $  which is convergent in the  $ h $--adic  topology of  $ K_h^\vee \, $.  As a byproduct of this analysis, we can apply the same argument to  $ \, \Delta\big( h^{-1} \phi \big) \, $  and thus realize that it is actually a well defined element of  $ \, K_h^\vee \,\widehat{\otimes}_{A_h} K_h^\vee\, $,  the  $ h $--adic  completion of  $ \, K_h^\vee \otimes_{A_h} K_h^\vee \, $.  Finally, it is clear that in fact  $ \, \Delta\big( h^{-1} \phi \big) \, $  even belongs to the Takeuchi product inside  $ \, K_h^\vee \otimes_{A_h} K_h^\vee \, $,  as the parallel property is true for  $ \Delta(\phi) $  inside  $ \, K_h \,\widetilde{\otimes}_{A_h} K_h \, $.
                                                          \par
   As  $ K^\times $  is generated   --- as an algebra ---   by  $ \, h^{-1} I_h \, $  and  $ s^r(A_h) \, $,  and  $ K_h^\vee $  is its  $ h $--adic  completion, we finally conclude that the coproduct of  $ K_h $  does provide a well defined coproduct for  $ K_h^\vee \, $,  making it into a (topological) right bialgebroid over  $ A_h \, $.
 \vskip3pt
   Moreover, by construction  $ K_h^\vee $  is isomorphic (as a  $ k[[h]] $--module)  to  $ \; \big( K_h^\vee \big/ h\,K_h^\vee \big)[[h]] \; $.
 \vskip3pt
   What we are left to prove   --- for claim  {\it (a.1)}  ---   is that  $ \, \overline{K_h^\vee} := K_h^\vee \Big/ h\,K_h^\vee \, $  be isomorphic to  $ V^r(L') $  for some Lie-Rinehart bialgebra, and that such  $ L' $   --- with its structure of (Lie-Rinehart)  {\sl bialgebra\/}  induced by this very quantization ---   is isomorphic to  $ L^* $  with its structure of Lie-Rinehart bialgebra dual to that induced on  $ L $  by the quantization  $ \, K_h := {J^r(L)}_h \, $  we started from.
 \vskip5pt
   We follow the strategy in  \cite{Gavarini}  and  \cite{Etingof and Schiffmann}.  So far we saw that  $ K_h^\vee $  is a deformation of the right bialgebroid  $ \; K_h^\vee \Big/ h \, K_h^\vee \; $:  then we shall apply  Proposition \ref{right-bialgd_V^r(L)}  (and the remarks after it) to show that the latter is indeed of the form  $ V^r(L') \, $,  with  $ \, L' \cong L^* \, $.  For computations hereafter we fix some notation:  $ \, \J_h := \text{\sl Ker}(\partial_h) \, $,  $ \, K := K_h \Big/ h\, K_h \, $  and  $ \, \J := \text{\sl Ker}(\partial) \, $  for  $ \, \partial := \partial_K \, $.  Also, from  Theorem \ref{dual_QFSAd's=QUEAd's}{\it (a)\/}  we consider  $ \, {V^\ell(L)}_h := {}_\star{}K_h = {}_\star{J^r(L)}_h \in \text{\rm (LQUEAd)}_{A_h} \, $  so that  $ \, {J^r(L)}_h = {V^\ell(L)}_h^* \; $.
 \vskip7pt
   We proceed in several steps.
 \vskip5pt
 \noindent
\;\ ---    $ \bullet $ \  For all  $ \, a \in A_h \, $,  we have  $ \; s_r(a) \equiv t_r(a) \, \mod h \, K_h^\vee \; $.
 \vskip2pt
   Indeed, one has  $ \; \big( s^r(a) - t^r(a) \big) \in \J_h \, \subseteq \, I_h \, = \, h \, h^{-1} I_h \, \subseteq \, h \, K_h^\vee \; $,  whence the claim.
 \vskip5pt

 \noindent
\;\ ---    $ \bullet $ \  The set  $ \; P^r\big(\, \overline{K_h^\vee} \,\big) \; $  of (right) primitive elements of  $ \, \overline{K_h^\vee} := K_h^\vee \Big/ h \, K_h^\vee \, $  --- cf.~Proposition \ref{right-bialgd_V^r(L)}  ---  has a natural structure of right Lie-Rinehart algebra, induced by specialization from  $ K_h^\vee \, $.
 \vskip2pt
   Indeed, this is entirely standard.  Both the Lie bracket  $ \, [\,\ ,\ ] \, $  and the anchor map  $ \, \omega \, $  are recovered as semiclassical limits of commutators from the multiplicative structure and source/target structure of the ``quantum right bialgebroid''  $ K_h^\vee \, $.  Namely, for any  $ \, x, y \in P^r\big(\, \overline{K_h^\vee} \,\big) \, $  and  $ \, a \in A \, $,  choose any lifts  $ \, x', y' \in K_h^\vee \, $   and  $ \, a' \in A_h \, $  of them: then defining
  $$  \displaylines{
   a.x \, := \, x' \, s^r(a') \mod h \, K_h^\vee  \quad ,  \qquad  [x,y] \, := \, x' \, y' - y' \, x'  \mod h \, K_h^\vee  \cr
   \omega(x)(a)  \, := \,  \partial_h\big( x' \, s^r(a') - s^r(a') \, x' \big)  \mod h \, A_h  }  $$
it is a routine matter to check that  $ \, P^r\big(\, \overline{K_h^\vee} \,\big) \, $  is made into a Lie-Rinehart algebra over  $ A \, $.
 \vskip5pt
 \noindent
\;\ ---    $ \bullet $ \  Set  $ \; \J_h^\vee := h^{-1} \J_h \, \big(\! \subseteq K_h^\vee \big) \; $  and  $ \; \overline{\J_h^\vee} := \J_h^\vee \mod h \, K_h^\vee \; $;  then  $ \, \overline{\J_h^\vee} \, $  is a Lie-Rinehart subalgebra of  $ \, P^r\big(\, \overline{K_h^\vee} \,\big) \; $.
 \vskip2pt
   Indeed, let  $ \, \phi \in \J_h \, $,  and set  $ \, \phi^\vee := h^{-1} \phi \in \J_h^\vee \, $.  Then acting as in the first part of the proof (with notation introduced therein) we get
  $$  \displaylines{
   \Delta(\phi)  \, = \,  \phi \otimes 1 + 1 \otimes \phi + \phi_{(1)}^{+_s} \otimes \phi_{(2)}^{+_t}  \; \in \;  \phi \otimes 1 \, + \, 1 \otimes \phi \, + \, \J_h \,\widetilde{\otimes}_{A_h} \J_h  }  $$
thanks to the assumption  $ \, \phi \in \J_h \, $  (and to several identities holding true in any right bialgebroid).  As  $ \; \J_h = h \, h^{-1} \J_h = h \, \J_h^\vee \, \subseteq \, h \, K_h^\vee \; $,  we end up with
 $ \; \Delta\big(\phi^\vee\big)  \, = \,  \phi^\vee \otimes 1 \, + \, 1 \otimes \phi^\vee \, + \, h \, \big( K_h^\vee \,\widehat{\otimes}_{A_h} K_h^\vee \big) \; $,
 \, so that  $ \; \overline{\phi^\vee} := \phi^\vee \mod h \, K_h^\vee \; $  is primitive in  $ \, \overline{K_h^\vee} \, $.  This proves that  $ \; \overline{\J_h^\vee} \subseteq P^r\big(\, \overline{K_h^\vee} \,\big) \; $.
                                                                \par
   Finally,  $ \, \overline{\J_h^\vee} \, $  is a Lie-Rinehart subalgebra of  $ \, P^r\big(\, \overline{K_h^\vee} \,\big) \, $  if and only if it is a (right)  $ A $--submodule,  closed for the Lie bracket.  Now, by definition  $ \J_h $  is a right ideal in  $ K_h \, $,  and this implies   --- by construction ---   that  $ \, \overline{\J_h^\vee} \, $  is a (right)  $ A $--submodule.  As to the Lie bracket, if  $ \, x, y \in \overline{\J_h^\vee} \, $  we have by definition  $ \; [x,y] \, := \, x' \, y' - y' \, x' \mod h \, K_h^\vee \; $  for any choice of liftings  $ \, x', y' \in K_h^\vee \, $  of  $ x $  and  $ y \, $.  On the other hand, we can clearly choose  $ \, x', y' \in \J_h^\vee \, $,  so that  $ \, x' = h^{-1} \chi \, $,  $ \, y' = h^{-1} \eta \, $,  for some  $ \, \chi, \eta \in \J_h \, $;  then we have
  $$  x' \, y' - y' \, x'  \; = \;  h^{-2} (\chi \, \eta - \eta \, \chi)  \; \in \;  h^{-2} \big( \J_h \,{\textstyle \bigcap}\, h \, K_h \big)  \; = \;  h^{-2} h \, \J_h  \; = \;  h^{-1} \J_h  \; =: \;  \J_h^\vee  $$
since  $ \J_h $  is a right ideal and  $ \, K_h \big/ h \, K_h \cong J^r(L) \, $  is commutative.  It follows that  $ \, [x,y] \in \overline{\J_h^\vee} \, $,  q.e.d.
 \vskip5pt
 \noindent
\;\ ---    $ \bullet $ \  We will now show that  $ \,\; \J_h^\vee \,\;{\textstyle \bigcap}\; h\,K_h^\vee  \,\; = \;\,  \J_h + \J_h^\vee \, \J_h  \,\; = \;\,  h \, \J_h^\vee + h \, {\big( \J_h^\vee \big)}^2 \;\, $.
 \vskip2pt
   Indeed, the second identity in the claim is a trivial consequence of  $ \, \J_h^\vee := h^{-1} \J_h \, $.  As to the first one, as  $ \, K_h = {J^r(L)}_h \, $,  we distinguish two cases: either  $ L $  is free (as an  $ A $--module),  or not.
                                                                \par
   If  $ L $  is free, then the identity  $ \; \J_h^\vee \;{\textstyle \bigcap}\; h\,K_h^\vee \, = \, \J_h + \J_h^\vee \, \J_h \; $  is an easy, direct consequence of the description of  $ \J_h^\vee $  given in  \S \ref{descr-K_h^vee}  here above in the free case   --- i.e.~part  {\it (a)}.
                                                                \par
   If instead  $ L $  is not free, then we proceed as follows.  First consider  $ K_{h,Y} $  and  $ \J_{h,Y} \, $,  and construct from them  $ K_{h,Y}^\vee $  and  $ \J_{h,Y}^\vee \, $.  In this case, the description of  $ \J_{h,Y}^\vee  $  given in  \S \ref{descr-K_h^vee},  part  {\it (b)},  implies again easily the identity  $ \; \J_{h,Y}^\vee \;{\textstyle \bigcap}\; h\,K_{h,Y}^\vee \, = \, \J_{h,Y} + \J_{h,Y}^\vee \, \J_{h,Y} \; $.  Now consider the map  $ \; \pi^{\scriptscriptstyle Y} \! : K_{h,Y} \longrightarrow K_h \; $, \, introduced in  \S \ref{ext-QFSAd's}{\it (b)},  for  $ \, {J^r(L)}_h := K_h \, $  and  $ \, {J^r(L)}_{h,Y} := K_{h,Y} \; $:  this is a an epimorphism of right bialgebroids, thus in particular  $ \; \pi^{\scriptscriptstyle Y}\!\big(\J_{h,Y}\big) = \J_h \; $.  Then it follows at once that  $ \pi^{\scriptscriptstyle Y} $  canonically induces another epimorphism of right bialgebroids  $ \; \check{\pi}^{\scriptscriptstyle Y} \! : K_{h,Y}^\vee \longrightarrow K_h^\vee \; $  such that  $ \; \check{\pi}^{\scriptscriptstyle Y}\!\big(\J_{h,Y}^\vee\big) = \J_h^\vee \; $.  But then, using  $ \, \pi^{\scriptscriptstyle Y} $  and  $ \check{\pi}^{\scriptscriptstyle Y} $  and the identity  $ \; \J_{h,Y}^\vee \;{\textstyle \bigcap}\; h\,K_{h,Y}^\vee \, = \, \J_{h,Y} + \J_{h,Y}^\vee \, \J_{h,Y} \; $  we easily deduce the identity  $ \; \J_h^\vee \;{\textstyle \bigcap}\; h\,K_h^\vee \, = \, \J_h + \J_h^\vee \, \J_h \; $  we were looking for.
 \vskip5pt
 \noindent
\;\ ---    $ \bullet $ \  There exists an  $ A $--linear  isomorphism  $ \,\; \psi : \J_h^\vee \Big/ \Big( h \, \J_h^\vee + h \, \big( \J_h^\vee \big)^2 \Big) \, \cong \; \overline{\J_h^\vee} \;\, $   --- hence hereafter we shall identify  $ \, \overline{\J_h^\vee} \, $  and  $ \, \J_h^\vee \Big/ \Big( h \, \J_h^\vee + h \, \big( \J_h^\vee \big)^2 \Big) \, $  via  $ \psi $  and  $ \psi^{-1} \, $.
 \vskip2pt
   Indeed, the natural projection map  $ \; K_h^\vee \!\relbar\joinrel\twoheadrightarrow \overline{K_h^\vee} := K_h^\vee \big/ h \, K_h^\vee \; $,  \, whose kernel is  $ \, h \, K_h^\vee \, $,  yields by restriction a similar map  $ \; \J_h^\vee \relbar\joinrel\twoheadrightarrow \overline{\J_h^\vee} := \J_h^\vee \Big/ \big( \J_h^\vee \bigcap h \, K_h^\vee \big ) \, $  whose kernel is  $ \; \big( \J_h^\vee \bigcap h \, K_h^\vee \big ) \; $.  By the previous step, we have  $ \; \J_h^\vee \;{\textstyle \bigcap}\; h\,K_h^\vee \, = \, h \, \J_h^\vee + h \, {\big( \J_h^\vee \big)}^2 \; $,  \, whence we get an  $ A $--linear  isomorphism.
 \vskip5pt
 \noindent
\;\ ---    $ \bullet $ \  There exists an  $ A $--linear  isomorphism  $ \,\; \sigma : \, \overline{\J_h^\vee} \, \cong \, \J_h^\vee \Big/ \Big( h \, \J_h^\vee + h \, \big( \J_h^\vee \big)^2 \Big) \, \cong \; \J \Big/ \J^{\,2} =: L^* \;\, $,  where  $ \, \J \equiv \J_{J^r(L)} := \text{\sl Ker}\,\big(\partial_{\scriptscriptstyle J^r(L)}\big) \, $,  given by  $ \; \overline{h^{-1} y} \mapsto \sigma\big(\, \overline{h^{-1} y} \,\big) := \overline{y} \! \mod \J^{\,2} \; $.
 \vskip2pt
   Indeed, there exists a natural projection map  $ \; \sigma'' : \J_h \relbar\joinrel\twoheadrightarrow \J_h \big/ h\,\J_h = \J \relbar\joinrel\twoheadrightarrow \J \big/ \J^{\,2} =: L^* \; $,  whose kernel is  $ \, \big( h \, \J_h + \J_h^{\,2} \big) \; $.  Then  $ \; \sigma' :  \J_h^\vee := h^{-1} \J_h \relbar\joinrel\twoheadrightarrow \J \big/ \J^{\,2} =: L^* \; \big( h^{-1} y \mapsto \sigma'\big(h^{-1} y\big) := \sigma''(y) \big) \; $  is a well defined  $ k $--linear  map, whose kernel is
   $ \; \big( \J_h + h^{-1} \J_h^{\,2} \big) = \Big( h \, \J_h^\vee + h \, \big( \J_h^\vee \big)^2 \Big) \; $.  Therefore  $ \sigma' $  canonically induces a  $ k $--linear  isomorphism  $ \; \sigma : \, \overline{\J_h^\vee} \, \cong \, \J_h^\vee \Big/ \Big( h \, \J_h^\vee + h \, \big( \J_h^\vee \big)^2 \Big) \;{\buildrel \cong \over {\lhook\joinrel\relbar\joinrel\relbar\joinrel\twoheadrightarrow}}\; \J \big/ \J^{\,2} =: L^* \; $  given by  $ \; \overline{h^{-1} y} \mapsto \sigma\big(\, \overline{h^{-1} y} \,\big) := \overline{\sigma''(y)} \; $;  also, it is straightforward to check that this is  $ A $--linear  too.
 \vskip5pt
 \noindent
\;\ ---    $ \bullet $ \  We have  $ \; \overline{K_h^\vee} \in \text{\rm (RQUEAd)}_{A_h} \, $,  namely  $ \; \overline{K_h^\vee} \,\cong\, V^r\big(L'\big) \; $  for the Lie-Rinehart  $ A $--algebra  $ \, L' := \overline{\J_h^\vee} \, $  (with the Lie-Rinehart structure mentioned above).
 \vskip2pt
   Indeed, what we proved so far shows that  $ \, L' := \overline{\J_h^\vee} \, $  is a Lie-Rinehart subalgebra of  $ P^r\big(\, \overline{K_h^\vee} \,\big) \, $,  which together with  $ A $  generates  $ \overline{K_h^\vee} $  (as an algebra) and is finite projective as an  $ A $--module  (since it is isomorphic, as an  $ A $--module,  to  $ L^* \, $,  see above).  Therefore, all conditions in  Remark \ref{rPrim-Vr(L)}  are fulfilled, so it applies and gives  $ \; \overline{K_h^\vee} \,\cong\, V^r\big(L'\big) \; $  for  $ \, L' := \overline{\J_h^\vee} = P^r\big(\, \overline{K_h^\vee} \,\big) \, $.
 \vskip5pt
 \noindent
\;\ ---    $ \bullet $ \  There exists on the Lie-Rinehart algebra  $ L' $  a unique structure of Lie-Rinehart bialgebra, canonically induced from the quantization  $ K_h^\vee $  of  $ V^r\big(L'\big) \, $.
 \vskip2pt
   In fact, this is just a direct consequence of  Theorem \ref{semiclassical_limit-V^r(L)}.
 \vskip5pt
 \noindent
\;\ ---    $ \bullet $ \  The  $ A $--linear  isomorphism  $ \, \sigma : \, \overline{\J_h^\vee} \, \cong \, \J_h^\vee \Big/ \Big( h \, \J_h^\vee + h \, \big( \J_h^\vee \big)^2 \Big) \cong \, \J \Big/ \J^{\,2} =: L^* \, $  is actually an isomor\-phism of Lie-Rinehart bialgebras over  $ A \, $.
 \vskip2pt
   In order to prove this, we must show that  $ \sigma $  preserves the Lie bracket, the anchor map and the differential  $ \delta $  (cf.~Definition \ref{def_L-R_bialgebras})  on either side.
                                                     \par
   For the Lie bracket, let  $ \, x, y \in \overline{\J_h^\vee} \, $:  given  $ \, \chi , \eta \in \J_h \, $  such that  $ \, x = h^{-1} \chi \, $,  $ \, y = h^{-1} \eta \, $,  we have
  $$  [x,y]  \,\; = \,\;  h^{-2} (\chi \, \eta - \eta \, \chi)  \! \mod h \, K_h^\vee  \,\; = \;\,  h^{-2} h \, \zeta  \! \mod h \, K_h^\vee  \,\; = \;\,  h^{-1} \zeta  \! \mod h \, K_h^\vee  $$
for some  $ \, \zeta \in \J_h \, $.  But then also  $ \,\; \overline{\zeta} \, := \, \zeta  \! \mod h \, K_h \, =: \, \big\{ \overline{\chi} , \overline{\eta} \big\} \;\, $   --- where  $ \; \overline{\alpha} := \alpha \! \mod h \, K_h \; $  for all  $ \, \alpha \in K_h \, $  ---   by  Theorem \ref{semiclassical_limit-J^r(L)}.  Now the Poisson bracket of  $ \, K_h \big/ h\,K_h \, $  restricted to  $ \J_h $  pushes down to the Lie bracket of  $ \, \J_h \big/ \J_h^{\,2} =: L^* \, $;  thus setting  $ \, X := \overline{\chi} \! \mod \J^{\,2} \, $,  $ \, Y := \overline{\eta} \! \mod \J^{\,2} \; \big(\! \in \J \big/ \J^{\,2} =: L^* \big) $,  we have  $ \; [X,Y] = \big\{ \overline{\chi} , \overline{\eta} \big\} \! \mod \J^{\,2} \, = Z \; $.  Now, by construction we have  $ \; X \, = \, \sigma(x) \; $,  $ \; Y \, = \, \sigma(y) \; $,  and the previous analysis eventually gives also  $ \; \sigma\big([x,y]\big) \, = \, Z \, = [X,Y] = \, \big[\sigma(x),\sigma(y)\big] \; $.
                                                     \par
   For the anchor map, let  $ \, x \in \overline{\J_h^\vee} \, $,  $ \, \chi \in \J_h \, $,  $ \, X \in \J \big/ \J^{\,2} = L^* \, $  as above, and take  $ \, a \in A \, $  and  $ \, a' \in A_h \, $  such that  $ \; a' \!\! \mod h \, A_h = a \; $.  Then direct computations give
 \vskip-14pt
  $$  \displaylines{
   \qquad   \omega(x)(a)  \; = \;  \partial_h\big( h^{-1} \chi \, s^r(a') - s^r(a') \, h^{-1} \chi \big)  \!\! \mod h \, A_h  \; =   \hfill  \cr
   \hfill   = \;  \partial \Big( h^{-1} \big( \chi \, s^r(a') - s^r(a') \, \chi \big)  \!\! \mod h \, K_h \Big)  \; = \;  \omega(X)(a)   \qquad  }  $$
 \vskip-4pt
\noindent
 which means  $ \; \omega(x) = \omega(X) = \omega\big(\sigma(x)\big) \; $,  that is  $ \, \sigma \, $  preserves the anchor, q.e.d.
 \vskip5pt
   Finally, in order to compare the two differentials on  $ \overline{\J_h^\vee} $  and  $ L^* \, $,  respectively denoted  $ \delta' $  and  $ \delta'' \, $,  recall that in any Lie-Rinehart bialgebra  $ \big( \mathcal{L} \, , \mathcal{A} \big) $   --- in the present case  $ (L^*,A) $  ---   the differential  $ \delta_{\mathcal{L}} $  is related with the Lie bracket and the anchor map by the identities
  $$  \big\langle\, \text{\rm f} , \delta_{\mathcal{L}}(a) \,\big\rangle  \, = \,  \omega_{\mathcal{L}^*}\!(\lambda)(a)  \; ,  \!\!\quad  \langle\, \text{\rm f} \otimes \mu , \delta_{\mathcal{L}}(x) \,\rangle  \, = \,  \omega_{\mathcal{L}^*}\!(\text{\rm f})\big(\! \langle\, \text{\rm m} , x \rangle \!\big) - \omega_{\mathcal{L}^*}\!(\text{\rm m})\big(\! \langle\, \text{\rm f} , x \rangle \!\big) - \big\langle {[\,\text{\rm f}\,,\text{\rm m}]}_{\mathcal{L}^*} , x \big\rangle  $$
for all  $ \, x \in \mathcal{L} \, $,  $ \, \text{\rm f} \, , \text{\rm m} \in \mathcal{L}^* \, $,  $ \, a \in A \, $   --- see  Remarks \ref{props_Lie-bialg}{\it (b)}.  We apply this to  $ \, \big( \mathcal{L} \, , \mathcal{A} \big) = (L^*,A) \, $.
 \vskip4pt
   For the differential on  $ A \, $,  we must prove that  $ \, \sigma\big(\delta'(a)\big) = \delta''(a) \, $  for all  $ \, a \in A \, $,  which amounts to showing that  $ \, \big\langle\, \text{\rm f} \, , \, \sigma\big(\delta'(a)\big) \big\rangle = \big\langle\, \text{\rm f} \, , \, \delta''(a) \big\rangle \, $  for all  $ \, a \in A \, $  and all  $ \, \text{\rm f} \in L \, $.  For this comparison, recall that  $ \, {V^\ell(L)}_h := {}_\star{J^r(L)}_h \in \text{\rm (LQUEAd)}_{A_h} \, $  is a quantization of  $ V^\ell(L) \, $,  by  Theorem \ref{dual_QFSAd's=QUEAd's}{\it (a)\/};  moreover, the natural pairing between  $ {V^\ell(L)}_h $  and  $ {J^r(L)}_h $  (given by evaluation) is a right bialgebroid pairing.  Now choose a lifting  $ \, a' \in A_h \, $  of  $ \, a \in A \, $  and a lifting  $ \, f' \in {V^\ell(L)}_h \, $  of  $ \, \text{\rm f} \in L \, $:  more precisely, we choose  $ \, f' \in \text{\sl Ker}\,\big(\epsilon_{\scriptscriptstyle {V^\ell(L)}_h}\big) \, $.  Then direct computation gives
 \vskip-14pt
  $$  \displaylines{
   \big\langle\, \text{\rm f} \, , \, \sigma\big(\delta'(a)\big) \big\rangle  \;\; =
    \;\;  h \cdot \big\langle\, f' \, , \, \delta'(a) \big\rangle  \!\! \mod \;h \; A_h=
    \;\;  \big\langle\, f' \, , \, s^r\big(a'\big) - t^r\big(a'\big) \big\rangle  \!\! \mod h \, A_h  \hfill \cr
    = \;\;  \big\langle\, f' \, s^\ell\big(a'\big) - s^\ell\big(a'\big) \, f' \, , \, 1 \,\big\rangle  \!\! \mod h \, A_h  \;\;
  = \;\;  \big\langle\, f' \, s^\ell\big(a'\big) - t^\ell\big(a'\big) \, f' \, , \, 1 \,\big\rangle  \!\! \mod h \, A_h  \cr
   = \;\;  \big\langle\, f' \, s^\ell\big(a'\big) \, , \, 1 \,\big\rangle  \!\! \mod h \, A_h  \;\;
      = \;\;  \epsilon_{\scriptscriptstyle V^\ell(L)}(\,\text{\rm f}\,a)  \;\; = \;\;  \omega_L(\,\text{\rm f}\,)(a)  \;\; = \;\;  \big\langle\, \text{\rm f} \, , \, \delta''(a) \,\big\rangle  }  $$
(cf.~\S \ref{V^ell(L)_left-bialgd}  for the last but one identity).  This proves that  $ \, \sigma\big(\delta'(a)\big) = \delta''(a) \, $  for all  $ \, a \in A \, $.
 \vskip4pt
   For the differential on  $ L^* \, $,  consider  $ \, x := \overline{\chi^\vee} = \overline{h^{-1} \chi} \in \overline{\J_h^\vee} \, $,  with  $ \, \chi \in \J_h \, $;
   then we have  $ \, \sigma(x) := \overline{\chi} \!\! \mod \J^{\,2} =: X \in \J \big/ \J^{\,2} = L^* \, $.  Our goal is to prove that  $ \, (\sigma \otimes \sigma)\big( \delta'(x) \big) = \delta''\big(\sigma(x)\big) \, $.
                                                                     \par
   Write  $ \, \Delta(\chi) = \chi_{(1)} \otimes \chi_{(2)} \, $  as  $ \, \Delta(\chi) = \chi \otimes 1 + 1 \otimes \chi + \sum_{[\theta]} \chi_{[1]} \otimes \chi_{[2]} \, $;  then we have  $ \, \Delta\big(\chi^\vee\big) = \chi^\vee \otimes 1 + 1 \otimes \chi^\vee + h \, \sum_{[\theta]} \chi_{[1]}^\vee \otimes \chi_{[2]}^\vee \, $   --- where  $ \, \chi_{[i]}^\vee := h^{-1} \chi_{[i]} \, $,  for  $ \, i \in \{1,2\} \, $  ---   so that  $ \; \delta'(x) := - \sum_{[\theta]} x_{[1]} \otimes x_{[2]} + \sum_{[\theta]} x_{[2]} \otimes x_{[1]} \; $  with  $ \, x_{[i]} := \overline{\chi_{[i]}^\vee} \, $  for  $ \, i \in \{1,2\} \, $.  In all this,  $ \, \chi^\vee := h^{-1} \chi \, $  is a lifting of  $ \, x \in L' \, $  in  $ \, {V^r(L')}_h := {J^r(L)}_h^{\,\vee} \, $,  and  $ \chi $  is a lifting of  $ \, X := \sigma(x) \, $  in  $ {J^r(L)}_h \, $;  in addition, we can assume that  $ \, \partial_h(\chi) = 0 \, $.  We
 adopt similar remarks, and notation,
 for  $ x_{[i]} \, $,  $ \chi_{[i]} $  and  $ \, X_{[i]} := \sigma(x_{[i]}) \, $  with  $ \, i \in \{1,2\} \, $.  Now for  $ \, \text{\rm f} , \text{\rm m} \in L \, $  and liftings  $ \, f' , m' \in {V^\ell(L)}_h \, $  of them, direct calculation yields
  $$  \displaylines{
   \big\langle\, \text{\rm f} \,\otimes\, \text{\rm m} \, , \, \delta''\big(\sigma(x)\big) \big\rangle  \; = \;  \big\langle\, \text{\rm f} \,\otimes\, \text{\rm m} \, , \, \delta''(X) \big\rangle  \; = \;  \omega''_L(\,\text{\rm f}\,)\big( \langle\, \text{\rm m} \, , X \,\rangle \big) \, - \, \omega''_L(\,\text{\rm m}\,)\big( \langle\, \text{\rm f} \, , X \,\rangle \big) \, - \, \big\langle\, {[\,\text{\rm f},\text{\rm m}]}''_L \, , X \,\big\rangle  \; =   \hfill  \cr
   \hfill
   = \;\;  \epsilon_{\scriptscriptstyle V^\ell(L)}\big(\, \text{\rm f} \, \langle\, \text{\rm m} \, , X \,\rangle \big) \, - \, \epsilon_{\scriptscriptstyle V^\ell(L)}\big(\, \text{\rm m} \, \langle\, \text{\rm f} \, , X \,\rangle \big) \, - \, \big\langle\, \text{\rm f} \, \text{\rm m} - \text{\rm m} \, \text{\rm f} \, , X \,\big\rangle  \;\; =  \cr
   = \;\;  \Big(\; \big\langle\, f' \cdot t^\ell\big( \langle\, m' \, , \chi \,\rangle \big) \, , \, 1 \,\big\rangle \, -
    \, \big\langle\, m' \cdot t^\ell\big( \langle\, f' , \chi \,\rangle \big) \, , \, 1 \,\big\rangle \, - \, \big\langle\, f' \, m' - m' \, f' , \chi \,\big\rangle \;\Big)  \mod \, h\;A_h  \;\; =   \hfill  \cr
   = \;\;  \Big(\; \big\langle\, f' \cdot t^\ell\big( \langle\, m' \, , \chi \,\rangle \big) \, , \, 1 \,\big\rangle \, - \, \big\langle\, m' \cdot t^\ell\big( \langle\, f' , \chi \,\rangle \big) \, , \, 1 \,\big\rangle \;\, -   \qquad \qquad \qquad  \cr
   \hfill   - \; \big\langle\, f' \cdot t^\ell\big( \langle\, f' , \chi_{(2)} \rangle \big) \, , \chi_{(1)} \big\rangle \, +
   \, \big\langle\, m' \cdot t^\ell\big( \langle\, f' , \chi_{(2)} \rangle \big) \, ,
   \chi_{(1)} \big\rangle \;\Big)  \mod \; h \, A_h  \;\; =   \cr
%%
%   = \;\;  \Big(\; \big\langle\, f' \cdot t^\ell\big( \langle\, m' \, , \chi \,\rangle \big) \, , \, 1 \,\big\rangle \, - \, \big\langle\, m' \cdot t^\ell\big( \langle\, f' , \chi \,\rangle \big) \, , \, 1 \,\big\rangle  \, -
%   \, \big\langle\, f' \cdot t^\ell\big( \langle\, m' \, , 1 \rangle \big) \, , \chi \big\rangle \,\; +   \hfill  \cr
   %+ \,\; \big\langle\, m' \cdot t^\ell\big( \langle\, f' , 1 \rangle \big) \, , \chi \big\rangle \, -
%   \, \big\langle\, f' \cdot t^\ell\big( \langle\, m' \, , \chi \rangle \big) \, , \, 1 \big\rangle \, + \, \big\langle\, m' \cdot t^\ell\big( \langle\, f' , \chi \rangle \big) \, , \, 1 \big\rangle  \;\, -  \cr
 %  \hfill   - \;\, \big\langle\, f' \cdot t^\ell\big( \langle\, m' \, , \chi_{[2]} \rangle \big) \, , \chi_{[1]} \big\rangle \, +
  % \, \big\langle\, m' \cdot t^\ell\big( \langle\, f' , \chi_{[2]} \rangle \big) \, , \chi_{[1]} \big\rangle \;\Big)  \mod \; h \, A_h  \;\; =  \cr
%%
   = \;\;  \Big(\; \big\langle\, m' \cdot t^\ell\big( \langle\, f' , \chi_{[2]} \rangle \big) \, , \chi_{[1]} \big\rangle \, - \, \big\langle\, f' \cdot t^\ell\big( \langle\, m' \, , \chi_{[2]} \rangle \big) \, , \chi_{[1]} \big\rangle \;\Big)  \mod \; h \, A_h  \;\; =   \hfill  \cr
   = \;\;  \Big(\; \big\langle\, m'  , \, \chi_{[1]}s^r\big( \langle\, f' ,
        \chi_{[2]} \rangle \big) \,\big\rangle \,
     - \, \big\langle\, f' , \,\chi_{[1]}
     s^r\big( \langle\, m' , \chi_{[2]}\,\big\rangle \big) \big\rangle \;\Big)  \mod \; h \, A_h \;\; =  \cr
   \hfill   = \;\;  \Big(\; \big\langle\, m'  , \, \chi_{[1]}t^r\big( \langle\, f' , \chi_{[2]} \rangle \big) \,\big\rangle \,
     - \, \big\langle\, f' , \,\chi_{[1]}
     t^r\big( \langle\, m' , \chi_{[2]}\,\big\rangle \big) \big\rangle \;\Big)  \mod \; h \, A_h \;\; =  \cr
   = \;\;  \Big(\; \big\langle\, m' , \, \chi_{[1]} \big \rangle \, \langle\, f' , \chi_{[2]} \rangle \,     - \, \big\langle\, f' , \,\chi_{[1]} \big\rangle \, \langle\, m' , \chi_{[2]}\,\big\rangle \;\Big)  \mod \; h \, A_h  \;\; =   \hfill  \cr
  = \;\;  \big\langle\, \text{\rm m} \, , \, \sigma\big(x_{[1]}\big) \big \rangle
     \, \langle\, \text{\rm f} \, , \sigma\big(x_{[2]}\big) \rangle \,
      - \, \big\langle\, \text{\rm f} \, , \,\sigma\big(x_{[1]}\big) \big \rangle
       \, \langle\, \text{\rm m} \, , \sigma\big(x_{[2]}\big)\,\big\rangle  \;\; =   \qquad \qquad  \cr
   \hfill   = \;\;  \big\langle\, \text{\rm f} \,\otimes\, \text{\rm m} \, , \, (\sigma \otimes \sigma)\big( {\Delta^{[1]}(x)}_{2,1} - \Delta^{[1]}(x) \big) \,\big\rangle  \;\; = \;\; \big\langle\, \text{\rm f} \,\otimes\, \text{\rm m} \, , \, (\sigma \otimes \sigma)\big(\delta'(x)\big) \,\big\rangle  }  $$
Here above we used the fact that  $ \, s^r\big( \big\langle f' , \chi_{[2]} \big\rangle \big) - t^r\big( \big\langle m' , \chi_{[2]} \big\rangle \big ) \, $  belongs to  $ \J_h \, $,  so that we have  $ \; \chi_{[1]} \big( s^r\big(\! \big\langle f' , \chi_{[2]} \big\rangle \!\big) - t^r\big(\! \big\langle f' , \chi_{[2]} \big\rangle \big) \!\big) \in \J_h^2 \; $  and  $ \; \big\langle m', \chi_{[1]}\big( s^r\big(\! \big\langle f' , \chi_{[2]} \big\rangle \!\big) - t^r\big(\! \big\langle f' , \chi_{[2]} \big\rangle \big) \!\big) \big\rangle = 0 \! \mod h \, A_h \, $.  Thus  $ \; \big\langle\, \text{\rm f} \,\otimes\, \text{\rm m} \, , \, \delta''\big(\sigma(x)\big) \big\rangle \, = \, \big\langle\, \text{\rm f} \,\otimes\, \text{\rm m} \, , \, (\sigma \otimes \sigma)\big(\delta'(x)\big) \big\rangle \; $  for  $ \, \text{\rm f} \, , \text{\rm m} \in L \, $,  so  $ \, \delta''\big(\sigma(x)\big) = (\sigma \otimes \sigma)\big(\delta'(x)\big) \, $.
 \vskip6pt
   In the end, all the above eventually completes the proof of claim  {\it (a.1)}.
 \vskip8pt
   As to claim  {\it (a.2)},  let  $ \big( K_h \, , A_h \, , s^r_{K_h} , t^r_{K_h} , \Delta \, , \partial_{K_h} \big) \, $  and  $ \, \big( \varGamma_h \, , B_h \, , s^r_{\varGamma_h} , t^r_{\varGamma_h}, \Delta \, , \partial_{\varGamma_h} \big) \, $  be two RQFSAd's, and let  $ \; (f,\phi) : K_h \longrightarrow \varGamma_h \; $  be a morphism between them in  $ \text{\rm (RQFSAd)} \, $.  The very definition of morphism in  $ \text{\rm (RQFSAd)} $  imply at once that  $ \; \phi\big(s^r_{K_h}\!(A_h)\big) \subseteq s^r_{\varGamma_h}\!(B_h) \; $   --- because  $ \, \phi \circ s^r_{K_h} = s^r_{\varGamma_h}\circ f \, $  ---   and  $ \, \phi\big(I_{K_h}\big) \subseteq I_{\varGamma_h} \, $   --- because  $ \, \partial_{\varGamma_h} \! \circ \phi = \partial_{K_h} \, $  ---   hence also  $ \, \phi\big( h^{-1} \, I_{K_h} \big) \subseteq h^{-1} \, I_{\varGamma_h} \, $
   for the natural,  $ k((h)) $--linear  extension of  $ \; \phi : K_h \longrightarrow \varGamma_h \; $  to  $ \; \phi^\times : {(K_h)}_F \longrightarrow {(\varGamma_h)}_F \; $.  By construction, this implies that  $ \phi^\times $  defines by restriction a morphism  $ \; \phi^\times : K_h^\times \longrightarrow \varGamma_h^\times \; $,  and this in turn extends by  $ h $--adic  continuity to a well defined morphism  $ \; \phi^\vee : K_h^\vee \longrightarrow \varGamma_h^\vee \; $  in the category  $ \text{\rm (RQUEAd)} \, $.
 \vskip9pt
   {\it (b)}\,  A direct proof of  {\it (b)\/}  can be given mimicking that of  {\it (a)}.  Otherwise, it can be deduced from  {\it (a)\/}  (and, clearly, the r{\^o}les of the two results in this deduction can be reversed) as follows.
                                                                          \par
   If  $ \, \varGamma_h := {J^\ell(L)}_h \in  \text{\rm (LQFSAd)}_{A_h} \, $,  then  $ \, {(\varGamma_h)}_{\text{\it coop}}^{\text{\it op}} \in \text{\rm (RQFSAd)}_{A_h} \, $;  thus by claim  {\it (a)\/}  we have that   $ \, {\big( {(\varGamma_h)}_{\text{\it coop}}^{\text{\it op}} \big)}^{\!\vee} \in \text{\rm (RQUEAd)}_{A_h} \, $.  Now, by construction  $ \, {\big( {(\varGamma_h)}_{\text{\it coop}}^{\text{\it op}} \big)}^{\!\vee} = {\big( \varGamma_h^{\,\vee} \big)}^{\text{\it op}}_{\text{\it coop}} \; $,  hence we deduce that  $ \, \varGamma_h^{\,\vee} \in \text{\rm (LQUEAd)}_{A_h} \, $.  All other aspects of the claim also follow from this argument.
\end{proof}

\smallskip

 \subsection{The Drinfeld's functor(s)  $ \, (\ )' = \, {}'\!{(\ )} \, $} \label{Dr-functors_rprime-lprime}

\smallskip

   {\ } \quad   We introduce now a second type of Drinfeld's functor, denoted  $ \, H \mapsto H' \, $.  Just like for the functor  $ \, H \mapsto H^\vee \, $,  this also is inspired by the similar notion introduced for ``quantum'' Hopf algebras (see \cite{Gavarini});  nevertheless, in this case we must be more careful, as we shall presently explain.

\smallskip

   Let  $ H_h $  be a left (or a right) bialgebroid.  If  $ \, s^\ell_h = t^\ell_h =: \iota^\ell_h \, $,  then we can define  $ H' $  as in the ``classical'' framework of quantum Hopf algebra deformations.  Let us shortly recall it.  Set  $ \; \delta_n = {(\text{\sl id}_H - s^\ell \circ \epsilon)}^{\otimes n} \circ \Delta^n \, $,  \, where  $ \; {(\text{\sl id}_H - s^\ell \circ \epsilon)}^{\otimes n} \; $  is the projection of  $ H^{\otimes n} $  onto $ \J^{\otimes n} $  defined by the decomposition  $ \, H = \J \oplus s_\ell(A) \, $,  with  $ \, \J := \text{\sl Ker}(\epsilon) \, $:  then we define
  $$  H'  \; := \;  \big\{\, a \in H \,\big|\; \delta_n(a) \in h^n H^n \;\; \forall \; n \in \N \,\big\} \,\;\; \subseteq \;\;\, H  \quad .  $$

\smallskip

   If instead $ s^\ell $  and  $ t^\ell $  {\sl do not\/}  coincide, then the projection of  $ H^{\otimes n} $  onto  $ \J^{\otimes n} $  is not defined, because the  $ \big( A_h \otimes A_h^{\text{\it op}} \big) $--module  $ \J_h $  does not have a complement in  $ H_h \, $.  Therefore, as $ s^\ell $  and  $ t^\ell $  do not necessarily coincide, we adopt the following definition:

\medskip

\begin{definition}  \label{def_H'-'H}
 As above, we use notation  $ \; {(H_h)}_F := k((h)) \otimes_{k[[h]]} H_h \; $.
 \vskip3pt
   (a) \,  If  $ \, H_h \in \text{\rm (LQUEAd)}_{A_h} \, $,  we define
  $$  H_h^{\,\prime}  \; := \;  \big\{\, \eta \in {(H_h)}_F \,\big|\, \big\langle\, \eta \, , {(H_h^{\;*})}^\times \,\big\rangle \in A_h \,\big\}  \quad ,  \qquad
       {}^{\prime\!}H_h  \; := \;  \big\{\, \eta \in {(H_h)}_F \,\big|\, \big\langle\, \eta \, , {((H_h)_*)}^\times \,\big\rangle \in A_h \,\big\}  $$
 \vskip3pt
   (b) \,  If  $ \, H_h \in \text{\rm (RQUEAd)}_{A_h} \, $,  we define
  $$  H_h^{\,\prime}  \; := \;  \big\{\, \eta \in {(H_h)}_F \,\big|\, \big\langle\, \eta \, , {(\,{}^*\!H_h)}^\times \,\big\rangle \in A_h \,\big\}  \quad ,  \qquad
      {}^{\prime\!}H_h  \; := \;  \big\{\, \eta \in {(H_h)}_F \,\big|\, \big\langle\, \eta \, , {(\,{}_*\!(H_h))}^\times \,\big\rangle \in A_h \,\big\}  $$
\end{definition}

\smallskip

\begin{proposition}  \label{props_H'_&_'H}
 Let  $ \, H_h \in \text{\rm (LQUEAd)}_{A_h} \, $.  Then
 \vskip5pt
\noindent
 \quad  {\it (a)} \qquad  $ H_h^{\,\prime} \,\subseteq\, H_h \; ,   \qquad  {}^{\prime\!}H_h \,\subseteq\, H_h $
 \vskip2pt
   \quad   $ H_h^{\,\prime} \,=\, {}_*\big( {(H_h^{\;*})}^\times \big) \,=\, {}_*\big( {(H_h^{\;*})}^\vee \big) \; ,  \qquad  {}^{\prime\!}H_h \,=\, {}^*\big( {{(H_h)}_*}^{\!\!\times} \big) \,=\, {}^*\big( {{(H_h)}_*}^{\!\!\vee} \big) $
 \vskip2pt
   $ \; H_h^{\,\prime} \,=\, \big\{\, \lambda : H_h^* \!\rightarrow A_h \,\big|\, \lambda(u+u') = \lambda(u) + \lambda(u') \, , \; \lambda(u \, t^r(a)) = a \, \lambda (u) \, , \;\; \lambda\big(I_{H_{h}^*}^n\big) \subseteq h^n A_h \;\,  \forall \; n \,\big\} $
 \vskip2pt
   $ \; {}^{\prime\!}H_h \,=\, \big\{\, \lambda : H_{h*} \!\rightarrow A_h \,\big|\, \lambda(u+u') = \lambda(u) + \lambda(u') \, , \; \lambda(u \, s^r(a)) = \lambda(u) \, a \, , \;\; \lambda\big(I_{H_{h*}}^n\big) \subseteq h^n A_h \;\,  \forall \; n \,\big\} $
 \vskip5pt
\noindent
 \quad  {\it (b)} \,  The analogous results hold if  $ \, H_h \in \text{\rm (RQUEAd)}_{A_h} \, $.
\end{proposition}

\begin{proof}
 The proof is the same as in  \cite{Gavarini},  hence we do not need to reproduce it.
\end{proof}

\smallskip

\begin{remark}
 If  $ \, H_h \in \text{\rm (LQUEAd)}_{A_h} \, $,  then  $ \, \big( {(H_h)}^{\text{\it op}}_{\text{\it coop}} \big)' = \big( {}^{\prime\!}H_h \big)^{\text{\it op}}_{\text{\it coop}} \, $.  This follows from the following three remarks:
 \vskip3pt
   --- if  $ U $  is any left bialgebroid, then  $ \, (U_*)^{\text{\it op}}_{\text{\it coop}} \cong {}^*(U^{\text{\it op}}_{\text{\it coop}}) \, $  as left bialgebroids;
 \vskip3pt
   --- if  $ W $  is any right bialgebroid, then  $ \, {({}^*W)}^{\text{\it op}}_{\text{\it coop}} \cong {\big( W^{\text{\it op}}_{\text{\it coop}} \big)}_* \, $  as right bialgebroids;
 \vskip3pt
   --- the functor  $ {(\ )}^\vee $  commutes with the functor  $ (\ )^{\text{\it op}}_{\text{\it coop}} \, $.
 \vskip4pt
   Similarly, one has  $ \, {}^{\prime\!}\big( {(H_h)}^{\text{\it op}}_{\text{\it coop}} \big) = {(H_h')}^{\text{\it op}}_{\text{\it coop}} \, $.  Finally, in the same way one finds also the parallel identities  $ \, \big( {(H_h)}^{\text{\it op}}_{\text{\it coop}} \big)' = \big( {}^{\prime\!}H_h \big)^{\text{\it op}}_{\text{\it coop}} \, $  and  $ \, {}^{\prime\!}\big( {(H_h)}^{\text{\it op}}_{\text{\it coop}} \big) = {(H_h')}^{\text{\it op}}_{\text{\it coop}} \, $  for every  $ \, H_h \in \text{\rm (RQUEAd)}_{A_h} \, $.
\end{remark}

\smallskip

\begin{free text}
 {\bf Explicit description of  $ \, {}^{\prime\!}H_h \, $.}  For a given  $ \, H_h \in \text{\rm (LQUEAd)}_{A_h} \, $,  we can describe  $ {}^{\prime\!}H_h \, $  quite explicitly.  Write  $ \, H_h = \J_h \oplus s_\ell(A_h) \, $,  and let  $ \pi_s $  be the projection of  $ H_h $  onto  $ \J_h \, $:  this is  {\sl not\/}  a morphism of  $ \big( A_h \otimes A_h^{\text{\it op}} \big) $--modules.  We need another lemma, whose proof is left to the reader:
\end{free text}

\smallskip

\begin{lemma}\label{study of pi_s}
 For any  $ \, u \in H_h \, $  and  $ \, a \in A_h \, $,  one has  $ \; \pi_s \big( s^\ell(a) \, u \big) = s^\ell(a) \, \pi_s(u) \; $.
                                                                    \par
   If in addition  $ \; t^\ell(a) - s^\ell(a) \, = \, h \, j \; $  for some  $ \, j \in \J_h \, $,  then  $ \; \pi_s \big( t^\ell(a) \, u \big) \, = \, s^\ell(a) \, \pi_s(u) + h \, \pi_s(j\,u) \; $.
\end{lemma}

\medskip

   The operator  $ \pi_s^{\otimes n} $  is not defined on  $ \, H_h \otimes_{A_h} H_h \otimes_{A_h} \cdots \otimes_{A_h} H_h \, $.  If  $ \, u_1 \otimes \cdots \otimes u_n \in H_h \otimes_{A_h} H_h \otimes_{A_h} \cdots \otimes_{A_h} H_h \, $, then  $ \, \pi_s(u_1) \otimes \cdots \otimes \pi_s(u_n) \, $  depends on the way of writing of  $ \, u_1 \otimes \cdots \otimes u_n \, $.  We will say that the component of  $ \, \sum u_1 \otimes \cdots \otimes u_n \, $  in  $ \J_h^{\otimes n} $  is defined up to  $ \, h^n \J_h^{\otimes n} \, $  if  $ \, \sum u_1 \otimes \cdots \otimes u_n = \sum v_1 \otimes \cdots \otimes v_n \, $  implies  $ \, \sum \pi_s(u_1) \otimes \cdots \otimes \pi_s(u_n) - \sum \pi_s(v_1) \otimes \cdots \otimes \pi_s(v_n) \in h^n \J_h^{\otimes n} \, $.

\medskip

\begin{lemma}  \label{comp-s_Delta(u)}
 Let  $ \, u \in H_h \, $  and  $ \, n \in \N_+ \, $.  If the component of  $ \, \Delta^n(u) $  in  $ \J_h^{\otimes n} $  is defined up to  $ \, h^n \J_h^{\otimes n} \, $  and belongs to  $ \, h^n \J_h^{\otimes n} \, $,  then the component of  $ \Delta^{n+1}(u) $  is defined up to  $ \, h^{n+1} \J_h^{\otimes (n+1)} \, $   --- hence it makes sense to say that it belongs to  $ \, h^{n+1} \J_h^{\otimes (n+1)} \, $.
\end{lemma}

\begin{proof}
 If the component of  $ \Delta^{n}(u) $  in  $ \J_h^{\otimes n} $  belongs to  $ \, h^n \J_h^{\otimes n} \, $,  then  $ \Delta^n(u) $  can be written as
  $$  \Delta^n(u)  \,\; = \;\,  {\textstyle \sum} \; h^n \, \phi_1 \otimes \cdots \otimes \phi_n  \; + \;
\text{\sl other terms}  $$
where all the  $ \phi_i $'s  are in  $ \J_h $  and ``other terms'' stands for a sum of homogeneous tensors containing (as tensor factors) elements of  $ s_\ell(A_h) $  which do not occur in the computation of the component of  $ \Delta^{n+1}(u) $  in  $ \J_h^{n+1} \, $.  Assume that  $ \Delta^{n+1}(u) $  can be written, for some  $ \, a \in A \, $,  as
  $$  \Delta^{n+1}(u)  \,\; = \;\,  {\textstyle \sum} \; h^n \, \chi_1 \otimes \cdots \otimes t^\ell(a)  \chi_i \otimes \chi_{i+1} \otimes \cdots \otimes \chi_{n+1} \; + \; \text{\sl other terms}  $$
 \vskip-11pt
\noindent
 or
  $$  \Delta^{n+1}(u)  \,\; = \;\,  {\textstyle \sum} \; h^n \, \chi_1 \otimes \cdots \otimes \chi_i \otimes s^\ell(a) \, \chi_{i+1} \otimes \cdots \otimes \chi_{n+1} \; + \; \text{\sl other terms}  $$
and let us compute  $ \pi_s^{\otimes (n+1)}\big(\Delta^{n+1}(u)\big) $  in both cases.

In the second case,  $ \pi_s^{\otimes (n+1)}\big(\Delta^{n+1}(u)\big) $  can be written as
  $$  \pi_s^{\otimes (n+1)}\big(\Delta^{n+1}(u)\big)  \,\; = \;\,  {\textstyle \sum} \; h^n \, \pi_s(\chi_1) \otimes \cdots \otimes \pi_s(\chi_i) \otimes s_\ell(a) \, \pi_s(\chi_{i+1}) \otimes \cdots \otimes \pi_s(\chi_{n+1})  $$
In the first case, if we write  $ \, t^\ell(a) - s^\ell(a) = h \, j \, $  (with  $ \, j \in \J_h \, $)  and use the previous lemma, we get
  $$  \displaylines{
   \pi_s^{\otimes (n+1)}\big(\Delta^{n+1}(u)\big)  \,\; = \;\,
   {\textstyle \sum} \; h^n \, \pi_s(\chi_1) \otimes \cdots \otimes t^\ell(a) \, \pi_s(\chi_i) \otimes \pi_s(\chi_{i+1}) \otimes \cdots \otimes \pi_s(\chi_{n+1})  \; +   \hfill  \cr
   \hfill   + \;  {\textstyle \sum} \; h^n \, \pi_s(\chi_1) \otimes \cdots \otimes \pi_s(\chi_{i-1}) \otimes
h \, \big( - j \, \pi_s(\chi_i) + \pi_s(j \, \chi_i) \big) \otimes \pi_s(\chi_{i+1}) \cdots \otimes \pi_s(\chi_{n+1})  }  $$
Taking the difference between the two computations we find
  $$  h^n \, \pi_s(\chi_1) \otimes \cdots \otimes \pi_s(\chi_{i-1}) \otimes h \, \big( - j \, \pi_s(\chi_i) + \pi_s(j \, \chi_i) \big) \otimes \pi_s(\chi_{i+1}) \cdots \otimes \pi_s(\chi_{n+1})  $$
which does belong to  $ \, h^{n+1} \J^{\otimes (n+1)} \, $,  \, q.e.d.
\end{proof}

\smallskip

   {\bf Notation:}  If the component of  $ \Delta^n(u) $  in  $ \J_h^{\otimes n} $  is defined up to  $ \, h^n \J^{\otimes n} \, $,  we shall write it as  $ \, \delta_s^{n}(u) \, $.  Then the condition  $ \, \delta_s^n(u) \in h^n \J^{\otimes n} \, $  perfectly makes sense.  Hereafter we shall write  $ \, \delta_s^n(u) \in h^n \J^{\otimes n} \, $  to mean that  $ \delta_s^n(u) $  is well defined   --- i.e., the component of  $ \Delta^n(u) $  in  $ \J^{\otimes n} $  is well defined ---   up to  $ \, h^n \J^{\otimes n} \, $  {\sl and\/}  it belongs to  $ \, h^n \J^{\otimes n} \, $.  For the rest of the discussion, we introduce the notation
  $$  \delta_s(H_h)  \; := \;  \big\{\, u \in H_h \,\big|\; \delta^n_s(u) \in h^n \J_h^{\otimes n} \;\, \forall \; n \in \N_+ \big\}  $$

\smallskip

   We need again a couple of technical results:

\smallskip

\begin{proposition}  \label{prop-delta_s}
 Let  $ \, u \in \delta_s(H_h) \, $.  Then  $ \Delta(u) $  can be written as
  $$  \Delta(u)  \,\; = \;\,  u \otimes 1 \, + \, {\textstyle \sum} \, u'_{(1)} \otimes u'_{(2)}   \eqno  \text{with}  \quad  u'_{(1)} \in \delta_s(H_h)  \quad  \text{and}  \quad  u'_{(2)} \in h \, \J_h  \qquad  $$
\end{proposition}

\begin{proof}
 \vskip3pt
   {\it First case:}  $ L $  is a finite free as an  $ A $--module.
 \vskip3pt
   Let  $ \, \big\{ \overline{e}_1 \, , \overline{e}_2 \, , \dots \, , \overline{e}_n \big\} \, $  be a basis of the $ A $--module  $ L \, $:  we lift each  $ \overline{e}_i $  to an element  $ \, e_i \in H_h \, $  such that  $ \, \epsilon(e_i) = 0 \, $.  Let  $ \, u \in \delta^s(H_h) \, $.  We write  $ \Delta(u) $  as
  $$  \Delta(u)  \,\; = \;\,  u' \otimes 1 \, + \, {\textstyle \sum_{\underline{\alpha} \in \N^n \setminus \{\underline{0}\}}} \, u_{\underline{\alpha}} \otimes e^{\underline{\alpha}}  \qquad
  {\rm with}  \quad  \lim\limits_{\mid \underline{\alpha} \mid \to +\infty}
  \big|\big| u_{\underline{\alpha}} \big|\big| = 0  $$
for suitable  $ \, u' , u_{\underline{\alpha}} \in H_h \, $.  The relation  $ \; m_{\scriptscriptstyle H_h}\big( (s^\ell \circ \epsilon ) \otimes \text{\sl id\/} \big)\big(\Delta(u)\big) = u \; $
gives  $ \, u' = u \; $.  Thus we have
  $$  \Delta(u)  \,\; = \;\,  u \otimes 1 \, + \, {\textstyle \sum_{\underline{\alpha} \in \N^n \setminus \{\underline{0}\}}} \, u_{\underline{\alpha}} \otimes e^{\underline{\alpha}}  $$
The relation  $ \; m_{\scriptscriptstyle H_h}\big((s^\ell \circ \epsilon ) \otimes  \text{\sl id\/}) \big(\Delta(u)\big) = u \; $  yields the identity
  $$  {\textstyle \sum_{\underline{\alpha} \in \N^n \setminus \{\underline{0}\}}} \,  s^\ell \big( \epsilon(u_{\underline{\alpha}\,}) \big) \, e^{\underline{\alpha}}  \,\; = \;\,  u - s^\ell\big( \epsilon(u) \big)  $$
As  $ \, u \in \delta_s(H_h) \, $ , one has  $ \, u - s^\ell\big(\epsilon(u)\big) \in h \, \J_h \, $,  which  implies that  $ \, s^\ell\big( \epsilon(u_{\underline{\alpha}\,}) \big) \in h \, H_h \; $;  hence  $ \, \overline{s^\ell\big( \epsilon(u_{\underline{\alpha}\,}) \big)} = \overline{s_\ell} \big(\, \overline{\epsilon(u_{\underline{\alpha}\,})} \,\big) = 0 \in H_h \Big/ h \, H_h \; $.  As  $ \overline{s_\ell} $  is injective, we get  $ \overline{\epsilon(u_{\underline{\alpha}\,})} = 0 \, $,  i.e.~$ \, \epsilon(u_{\underline{\alpha}\,}) \in h \, A_h \; $.
                                                                \par
   If  $ \, n > 1 \, $,  one has
  $$  \delta^n_s(u)  \; = \;  {\textstyle \sum_{\underline{\alpha} \in \N^n}} \, \delta^{n-1}_s(u_{\underline{\alpha}\,}) \otimes e^{\underline{\alpha}}  \;\; \in \;\,  h^n \, \J^{\otimes n}  $$
which implies  $ \, \delta^{n-1}_s(u_{\underline{\alpha}\,}) \in h^n \, \J^{\otimes (n-1)} \, $  and  $ \, u_{\underline{\alpha}} \in \delta_s(H_h) \, $.  Let  $ \, \widetilde{u}_{\underline{\alpha}} = \pi_s(u_{\underline{\alpha}\,}) = u_{\underline{\alpha}} - s^\ell\big( \epsilon(u_{\underline{\alpha}\,}) \big) \; $. For all  $ \, n \geq 1 \, $,  one has  $ \, \delta_s^n(\widetilde{u}_{\underline{\alpha}\,}) = \delta_s^n(u_{\underline{\alpha}\,}) \in h^{n+1} \, \J^n \; $.  In particular for  $ \, n = 1 \, $  we get  $ \; \widetilde{u}_{\underline{\alpha}} = h \, w_{\alpha} \; $  for some  $ \, w_{\alpha} \in \delta_s(H_h) \; $.  The element  $ \, u_{\underline{\alpha}} \, $  can be written as  $ \, u_{\underline{\alpha}} = h \, \big( w_{\underline{\alpha}} + s^\ell\big( h^{-1} \epsilon(u_{\underline{\alpha}\,}) \big) \big) \, \in \, h \, \delta_s(H_h) \; $.
 \vskip5pt
   {\it Second case:}  $ L $  is finite projective as an  $ A $--module.
 \vskip3pt
   Like in  Subsection \ref{quant_proj-free}, we fix a finite projective  $ A $--module  $ Q $  such that  $ \, L \oplus Q = F \, $  is a finite free $ A $--module.  We fix an  $ A $--basis  $ \, B := \{ e_1 \, , \dots \, , e_n\} \, $  of  $ F \, $:  then we call  $ Y $  the  $ k $--span  of  $ B \, $,  so that we can write  $ \; F = A \otimes_k Y \; $.  Moreover, we construct the (infinite dimensional) Lie-Rinehart algebra  $ \; L_Q = L \oplus \big( A \otimes_k Z \big) \, $,  with  $ \; Z = Y \oplus Y \oplus Y \oplus \cdots \; $,  which has a good basis  $ \, {\{ e_i \}}_{i \in T := \N \times \{1,\dots,n\}} \, $  defined by  $ B \, $.  Like in  \S \ref{ext-QUEAd's},  we can define  $ H_{h,Y} $  and  $ \delta_s(H_{h,Y\!}) \, $.  Now given  $ \, u \in \delta_s(H_{h,Y}\!) \, $,  we can write  $ \Delta(u) $  as follows:
  $$  \Delta(u)  \,\; = \;\,  u \otimes 1 \, + \, {\textstyle \sum_{\underline{\alpha} \in T^{(\N)} \setminus \{\underline{0}\}}} \, u_{\underline{\alpha}} \otimes e^{\underline{\alpha}}
   \qquad  \text{\rm with}  \quad
   \lim\limits_{\mid \underline{\alpha} \mid +\varpi(\underline{\alpha}) \rightarrow +\infty} \big\|u_{\underline{\alpha}\,}\big\| = 0  $$
Then the same reasoning as above shows that the proposition is true for  $ H_{h,Y} $  in the r{\^o}le of  $ H_h \, $.
                                                                                           \par
   Recall  (cf.~\S \ref{ext-QUEAd's})  that  $ \, H_{h,Y} = H_h \oplus \big( H_h \,\widehat{\otimes}_k\, {S(Z)}^+ \big) \, $  where  $ \, H_h \,\widehat{\otimes}_k\, {S(Z)}^+ \, $  is the  $ h $--adic completion of  $ \, H_h \otimes_k {S(Z)}^+ \, $,  with  $ \, Z = Y \oplus Y \oplus Y \oplus \cdots \; $;  the natural projection  $ \; \pi_{\scriptscriptstyle Y} : H_{h,Y} \relbar\joinrel\twoheadrightarrow H_h \; $  is then a morpism of left bialgebroids.  Moreover, if  $ \J_{h,Y} $  is the kernel of the counit of  $ H_{h,Y} \, $,  we have  $ \, \J_{h,Y} = \J_h \oplus \big( H_h \,\widehat{\otimes}_k\, {S(Z)}^+ \big) \, $.  Now it is easy to see that, if  $ \, v \in \delta_s(H_{h,Y}) \, $,  then  $ \, \pi_{\scriptscriptstyle Y}(v) \in \delta_s(H_h) \, $.
                                                                                           \par
   Now let  $ \, u \in \delta_s(H_h) \, $.  By the result for  $ H_{h,Y} \, $,  we know that  $ \Delta(u) $  can be written as
  $$  \Delta(u)  \,\; = \;\,  u \otimes 1 \, + \, {\textstyle \sum} \, u'_{(1)} \otimes u'_{(2)}   \eqno  \text{with}  \quad  u'_{(1)} \in \delta_s(H_{h,Y})  \quad  \text{and}  \quad  u'_{(2)} \in h \, \J_{h,Y}  \qquad  $$
As  $ \, \pi_{\scriptscriptstyle Y}(u) = u \, $,  applying  $ \, \pi_{\scriptscriptstyle Y} \otimes \pi_{\scriptscriptstyle Y} \, $  to the previous identity we get
  $$  \Delta(u)  \,\; = \;\,  u \otimes 1 \, + \, {\textstyle \sum} \, \pi_{\scriptscriptstyle Y}(u'_{(1)}) \otimes \pi_{\scriptscriptstyle Y}(u'_{(2)})  $$
with  $ \; \pi_{\scriptscriptstyle Y}(u'_{(1)}) \in \pi_{\scriptscriptstyle Y}\big(\delta_s(H_{h,Y})\big) = \delta_s(H_{h}) \; $  and  $ \; \pi_{\scriptscriptstyle Y}(u_{(2)}) \in h \, \pi_{\scriptscriptstyle Y}\big(\J_{h,Y})\big) = h \, \J_h \; $,  \, q.e.d.
\end{proof}

\medskip

\begin{lemma}  \label{delta_s  and action}
 $ s^\ell(A_h) \cdot \delta_s(H_h) \, \subseteq \, \delta_s(H_h) \; $  and
$ \, t^\ell(A_h) \cdot \delta_s(H_h) \, \subseteq \, \delta_s(H_h) \; $.
\end{lemma}

\begin{proof}
 Let  $ \, u \in \delta_s(H_h) \, $  and  $ \, a \in A_h \, $.  The properties  $ \; s^\ell(a) \, u \in \delta_s(H_h) \; $  follows from the following properties:  $ \; \pi_s\big( s^\ell(a) \, u \big) = s^\ell(a) \, \pi_s(u) \; $  and  $ \; \Delta^n\big(s^\ell(a)\big) = s^\ell(a) \otimes 1 \otimes \cdots \otimes 1 \; $.  Let us now show that  $ \; \delta^n_s \big(t^\ell(a)\,u\big) \in h^n \, \J^{\otimes n} \; $  for all  $ \, n \in \N \, $.  Write  $ \; t^\ell(a) - s^\ell(a) = h \, j \; $  with  $ \, j \in \J_h \, $.
                                                                                         \par
   For  $  \, n = 1 \, $,  by  Lemma \ref{study of pi_s}  we have  $ \; \pi_s\big( t^\ell(a) \, u \big) = s^\ell(a) \, \pi_s(u) + h \, \pi_s\big( j \, u \big) \, \in \, h \, \J \; $.
                                                                                         \par
   For  $ \, n > 1 \, $,  let us show that  $ \; \delta^{n}_s\big(t^\ell(a)\big) \in h^n \, \J^{\otimes n} \; $.  Set $ \; \Delta (u) \, = \, u \otimes 1 \, + \, u'_{(1)} \otimes u'_{(2)} \; $  with  $ \, u'_{(1)} \in \delta_s(H_h) \, $,  $ \, u'_{(2)} \in h \, \J_h \, $  (cf.~Proposition \ref{prop-delta_s}).  Then  $ \; \Delta\big(t^\ell(a)\,u\big) \, = \, u \otimes t^\ell(a) \, + \, u'_{(1)} \otimes t^\ell(a) \, u'_{(2)} \, $,  \, hence  $ \; \delta^{n}_s\big(t^\ell(a)\,u\big) \, = \, \delta^{n-1}_s(u) \otimes \pi_s\big(t^\ell(a)\big) \, + \, \delta^{n-1}_s(u'_{(1)}) \otimes \pi_s\big( t^\ell(a) \, u'_{(2)} \big) \; $,  thus  $ \, \delta^{n}_s\big( t^\ell(a) \, u \big) \, \in \, h^n \, \J^{\otimes n} \; $.
\end{proof}

\smallskip

   We are now ready for the first key result of this subsection:

\smallskip

\begin{theorem}  \label{'H=delta_s(H)}
 With assumptions and notation as above, we have
  $$  {}^{\prime\!}H_h  \; = \;  \big\{\, u \in H_h \,\big|\, \delta^n_s(u) \in h^n \, \J_h^{\otimes n} \;\; \forall \; n \in \N_+ \,\big\}  \, =: \, \delta_s(H_h)  $$
\end{theorem}

\begin{proof}
 To begin with, we show that  $ \, \delta_s(H_h) \subseteq {}^{\prime\!}H_h \, $.  To this end, we prove that for any  $ \, u \in \delta_s(H) \, $  we have  $ \; \big\langle\, u \, , I_{{(H_h)}_*}^{\,n} \big\rangle \subseteq h^n A_h \; $  for all  $ \, n \in \N_+ \, $,  using induction on  $ n \, $.
                                                                                   \par
   Take  $ \, n = 1 \, $.  As  $ \, u \in \delta_s(H) \, $,  note that  $ \, \delta^1(u) \in h \, \J_{h} \, $  implies  $ \; u = h \, j + s^\ell\big(\epsilon(u)\big) \; $  with  $ \, j \in \J_h \, $.  Then one has
  $$  \big\langle\, u \, , I_{{H_h}_*} \big\rangle  \,\; =\;\,  h \, \big\langle \, j \, , I_{{H_h}_*} \big\rangle \, + \, \epsilon(u) \, \big\langle\, 1 \, , I_{{H_h}_*} \big\rangle  \; \in \; h \, A_h  $$
  Now assume  $ \, n > 1 \, $.  For our  $ \, u \in \delta_s(H) \, $,  set  $ \; \Delta(u) = u \otimes 1 + u'_{(1)} \otimes u'_{(2)} \, $  with  $ \, u'_{(1)} \in \delta_s(H_h) \, $  and  $ \, u'_{(2)} \in h \, \J_h \, $  as in  Proposition \ref{prop-delta_s}.  Let  $ \, \alpha \in I_{{H_h}_*}^n \, $  be of the form  $ \, \alpha = \alpha_1 \, \alpha_2 \, $  with  $ \, \alpha_1 \in I_{{H_h}_*} \, $  and  $ \, \alpha_2 \in I_{{H_h}_*}^{n-1} \, $:  then, as the pairing  $ \langle\,\ ,\ \rangle $  between  $ H_h $  and  $ {H_h}_* $  is a left bialgebroid pairing, we have
  $$  \big\langle\, u \, , \, \alpha_1 \, \alpha_2 \,\big\rangle  \,\; = \;\,  \big\langle\, t^\ell\big( \langle\, u'_{(2)} \, , \, \alpha_1 \,\rangle \big) \, u'_{(1)} \, , \, \alpha_2 \,\big\rangle \, + \, \big\langle\, t^\ell\big( \langle\, 1 \, , \, \alpha_1 \,\big\rangle \big) \, u \, , \, \alpha_2 \,\big\rangle  \; \in \;  h^n \, A_h  $$
by the induction hypothesis and the case  $ \, n = 1 \, $  (also using the two previous lemmas).
 \vskip4pt
   Conversely, let us now show that  $ \, {}^{\prime\!}H_h \subseteq \delta_s (H_h) \, $.  To this end, we prove (by induction on  $ n $)  that for any  $ \, u \in {}^{\prime\!}H_h \, $  one has  $ \, \delta_s^n(u) \in h^n \J_h^{\otimes n} \, $  for all  $ \, n \in \N \, $.
                                                                           \par
   For  $ \, n = 1 \, $.  As  $ \, u \in{}^{\prime\!}H_h \, $  we have  $ \; \big\langle\, u \, , I_{{H_h}_*} \big\rangle \, \subseteq \, h \, A_h \; $;  on the other hand,  $ \; \delta^1_s(u) = u - s^\ell\big(\epsilon(u)\big) \; $  by definition.  Then we have  $ \; \big\langle\, \delta^1_s(u) \, , \lambda \,\big\rangle \, \in \, h \, A_h \, $  if  $ \, \lambda \in I_{H_{h*}} \, $,  because
  $$  \big\langle u - s^\ell\big(\epsilon(u)\big) \, , \lambda \big\rangle \, = \, \big\langle u \, , \lambda \big\rangle - \big\langle s^\ell\big(\epsilon(u)\big) \, , \lambda \big\rangle \, = \, \big\langle u \, , \lambda \big\rangle - \epsilon(u) \, \big\langle 1 \, , \lambda \big\rangle \, = \, \big\langle u \, , \lambda \big\rangle - \epsilon(u) \, \partial(\lambda) \, \in \, h \, A_h  $$
On the other hand, clearly  $ \; \delta^1_s(u) = u - s^\ell\big(\epsilon(u)\big) \, \in \, \J_h \, $,  \, hence  $ \; \delta^1_s(u) \, \in \, \J_h \cap h \, H_h \, = \, h \, \J_h \; $.
                                                                           \par
   Let now  $ \, n > 1 \, $,  and assume by induction that  $ \; \delta^{n-1}_s(u') \in h^{n-1} \J^{\otimes (n-1)} \; $ for all  $ \, u' \in {}^\prime{}H_h \, $.  For our  $ \, u \in {}^\prime{}H_h \, $,  write  $ \; \Delta(u) =  u_{(1)} \otimes u_{(2)} \, $  with  $ \, u_{(1)} $,  $ \, u_{(2)} \in {}^\prime{}H \, $.  As  $ \, \Delta^n(u)  = \Delta^{n-1}\big(u_{(1)}\big) \otimes u_{(2)} \, $,   we get  $ \; \delta_s^n(u) \, = \, \delta_s^{n-1}\big(u_{(1)}\big) \otimes \delta_s^1\big(u_{(2)}\big) \;  \in \, h^n \J^{\otimes n} \; $  by the induction hypothesis and the case  $ \, n = 1 \, $.
\end{proof}

\medskip

\begin{free text}
  {\bf Explicit description of  $ \, H_h^{\,\prime} \, $.}  We shall now give an explicit description of $ H_h^{\,\prime} \, $:  this will be entirely similar to that for  $ {}^{\prime\!}H_h \, $,  thus we shall only outline the main steps, without dwelling into details   --- which can be easily filled in by the reader.
                                                                           \par
   Write  $ \, H_h = \J_h \oplus t_\ell(A_h) \, $,  and let  $ \pi_t $  be the projection of  $ H_h $  onto  $ \J_h \, $:  once again, this is  {\sl not\/}  a morphism of  $ \big( A_h \otimes A_h^{\text{\it op}} \big) $--modules.  The operator  $ \pi_t^{\otimes n} $  is not defined on  $ \, H_h \otimes_{A_h} H_h \otimes_{A_h} \cdots \otimes_{A_h} H_h \, $:  indeed, if  $ \; u_1 \otimes \cdots \otimes u_n \, \in \, H_h \otimes_{A_h} H_h \otimes_{A_h} \cdots \otimes_{A_h} H_h \, $,  \, then  $ \, \pi_t(u_1) \otimes \cdots \otimes \pi_t(u_n) \, $  depends on the way of writing  $ \, u_1 \otimes \cdots \otimes u_n \, $.  We say that the component of  $ \, \sum u_1 \otimes \dots \otimes u_n \, $  in  $ \J_h^{\otimes n} $  is defined up to  $ \, h^n \J_h^{\otimes n} \, $  if  $ \, \sum u_1 \otimes \cdots \otimes u_n = \sum v_1 \otimes \cdots \otimes v_n \, $  yields  $ \; \sum \pi_t(u_1) \otimes \cdots \otimes \pi_t(u_n) - \sum \pi_t(v_1) \otimes \cdots \otimes \pi_t(v_n) \, \in \, h^n \J_h^{\otimes n} \; $.
                                                                           \par
   The following lemma is the parallel of  Lemma \ref{comp-s_Delta(u)},  with similar proof.  {\sl Note that the statement is formally the same},  but actually the ``componentes'' to which one refers in the two claims are defined with respect to different projectors   --- namely  $ \pi_s^{\otimes n} $  or  $ \pi_t^{\otimes n} $  ---   in the two cases.
\end{free text}

\smallskip

\begin{lemma}  \label{comp-t_Delta(u)}
 Let  $ \, u \in H_h \, $.  If the component of  $ \Delta^n(u) $  in  $ \J_h^{\otimes n} $  is defined up to  $ \, h^n \J_h^{\otimes n} \, $  and  belongs to  $ \, h^n \J_h^{\otimes n} \, $,  then the component of  $ \Delta^{n+1}(u) $  is defined up to  $ \, h^{n+1} \J_h^{\otimes (n+1)} \, $   --- hence it makes sense to say that it belongs to  $ \, h^{n+1} \J_h^{\otimes (n+1)} \, $.
\end{lemma}

\medskip

   {\bf Notation:}  If the component of  $ \Delta^n(u) $  in  $ \J_h^{\otimes n} $  is defined up to  $ \, h^n \J^{\otimes n} \, $  (in the above sense), we shall write it as  $ \, \delta_t^n(u) \, $.  Then the condition  $ \, \delta_t^n(u) \in h^n \J^{\otimes n} \, $  perfectly makes sense.  Thus we shall write  $ \, \delta_t^n(u) \in h^n \J^{\otimes n} \, $  to mean that  $ \delta_t^n(u) $  is well defined (i.e., the component of  $ \Delta^n(u) $  in  $ \J^{\otimes n} $,  in the above sense, is well defined) up to  $ \, h^n \J^{\otimes n} \, $ and it belongs to  $ \, h^n \J^{\otimes n} \, $.  Also, we set
  $$  \delta_t(H_h)  \; := \;  \big\{\, u \in H_h \,\big|\;\, \delta^n_t(u) \in h^n \J_h^{\otimes n} \;\, \forall \; n \in \N_+ \big\}  $$
   \indent   Arguing like for  $ {}^{\prime\!}H_h \, $,  we can then prove the following, analogous characterization of  $ \, H_h^{\,\prime} \; $:

\smallskip

\begin{theorem}  \label{H'=delta_t(H)}
 With assumptions and notation as above, we have
  $$  H_h^{\,\prime}  \; = \;  \big\{\, u \in H_h \,\big|\, \delta^n_t(u) \in h^n \J_h^{\otimes n} \;\; \forall \; n \in \N_+ \big\}  \, =: \, \delta_t(H_h)  $$
\end{theorem}

\smallskip

\begin{remark}
The study of  $ {}^{\prime\!}H_h $  and  $ H_h^{\,\prime} $  we have done for LQUEAd holds for RQUEAd as well.  One can check it directly (via the same arguments) or, besides, deducing the results for RQUEAd's from those for LQUEAd's in force of the general identities  $ \, {\big( H_h^{\,\prime} \big)}^{\text{\it op}}_{\text{\it coop}} = {}^\prime \big( {(H_h)}_{\text{\it coop}}^{\text{\it op}} \big) \, $.
\end{remark}

\smallskip

   Thanks to the characterizations in  Theorem \ref{'H=delta_s(H)}  and  Theorem \ref{H'=delta_t(H)}  we can eventually prove the following\, remarkable result:

\smallskip

\begin{theorem}  \label{'H=H'}
 Let  $ H_h $  be a LQUEAd or a RQUEAd.  Then  $ \; H_h^{\,\prime} \, = \, {}^{\prime\!}H_h \; $.
\end{theorem}

\begin{proof}
 We begin with  $ H_h $  being an LQUEAd.  We show, that for any  $ \, u \in \delta_s(H_h) \, $  we have  $ \; \delta_t^n(u) \in h^n \J^{\otimes n} \; $  for all  $ \, n \in \N \, $,  by induction on  $ n $.
                                                               \par
   For  $ \, n = 1 \, $,  one has
  $$  \delta_{t}^1(u)  \,\; = \;\,  u - t_l\big(\epsilon(u)\big)  \,\; = \;\,  u - s_\ell\big(\epsilon(u)\big) \, + \, s_\ell\big(\epsilon(u)\big) \, - \, t_\ell\big(\epsilon(u)\big)  $$
As  $ \, s^\ell - t^\ell = 0 \mod h \, $,  one has  $ \, s^\ell\big(\epsilon(u)\big) - t^\ell\big(\epsilon(u)\big) \in h \, A_h \, $.  Moreover, we have also  $ \, \epsilon\big( s^\ell\big(\epsilon(u)\big) - t^\ell\big(\epsilon(u)\big) \big) = 0 \, $,  so that  $ \, s^\ell\big(\epsilon(u)\big) - t^\ell\big(\epsilon(u)\big) \, $  belongs to  $ \, \J_h \cap h \, A_h = h \, \J_h \, $.  Thus  $ \, \delta_t^1(u) \in h \, \J_h \, $,  q.e.d.
                                                               \par
   For  $ \, n > 1 \, $,  let us write  $ \; \Delta(u) = u \otimes 1 + u'_{(1)} \otimes u'_{(2)} \, $  with  $ \, u'_{(1)} \in \delta_s(H_h) \, $  and  $ \, u'_{(2)} \in h \, \J_h \, $  as in  Proposition \ref{prop-delta_s}.  Then one has  $ \; \delta^n_t(u) = \delta^{n-1}_t\big(u'_{(1)}\big) \otimes \delta^1_t\big(u'_{(2)}\big) \; $,  which is an element of  $ \, h^n \J^{\otimes n} \, $  thanks to the induction hypothesis.
                                                               \par
   By the above we have proved the inclusion  $ \, \delta_s(H_h) \subseteq \delta_t(H_h) \, $;  the reverse inclusion can be shown in the same way, so to give  $ \, \delta_s(H_h) = \delta_t(H_h) \, $.  By  Theorem \ref{'H=delta_s(H)}   --- giving  $ \, {}^{\prime\!}H_h \, = \, \delta_t(H_h) \, $  ---   and  Theorem \ref{H'=delta_t(H)}   --- giving  $ \, H_h^{\,\prime} = \delta_s(H_h) \, $  ---   this eventually implies  $ \; H_h^{\,\prime} \, = \, {}^{\prime\!}H_h \; $.
 \vskip4pt
   For  $ H_h $  a RQUEAd, we can provide a direct proof by the same arguments used for a LQUEAd; otherwise, we can deduce the result for RQUEAd's from that for LQUEAd's, as follows.
                                                               \par
   If  $ H_h $  is a RQUEAd, then  $ \, {(H_h)}_{\text{\it coop}}^{\text{\it op}} $  is a LQUEAd; then we have the chain of identities  $ \; {(H_h^{\,\prime})}^{\text{\it op}}_{\text{\it coop}} = {}^{\prime\!}\big({(H_h)}^{\text{\it op}}_{\text{\it coop}}\big) = {\big( {(H_h)}^{\text{\it op}}_{\text{\it coop}} \big)}' = {\big( {}^{\prime\!}H_h \big)}^{\text{\it op}}_{\text{\it coop}} \; $,  \, whence  $ \; {}^{\prime\!}H_h = H_h^{\,\prime} \; $  follows too.
\end{proof}

\smallskip

   We are now ready for the main result of this subsection.  In short, it claims that the construction  $ \, H_h \mapsto {}^{\prime\!}H_h = H_h^{\,\prime} \, $,  starting from a quantization of  $ L $   --- of type  $ V^{\ell/r}(L) $  ---  provides a quantization of the  {\sl dual\/}  Lie-Rinehart bialgebra  $ L^* $   --- of type  $ J^{\ell/r}(L^*) \; $;  moreover, this construction is functorial.

\medskip

\begin{theorem}  \label{compute_'H=H'}
 \vskip3pt
   (a)\,  Let  $ \, {V^\ell(L)}_h \in \text{\rm (LQUEAd)}_{A_h} \, $,  where  $ L $  is a Lie-Rinehart algebra which, as an  $ A $--module, is projective of finite type.  Then:
 \vskip3pt
   --- (a.1)\,  $ {}^{\prime}{V^\ell(L)}_h = {V^\ell(L)}_h^{\,\prime} \in \text{\rm (LQFSAd)}_{A_h} \, $,  with semiclassical limit  $ \; {V^\ell(L)}_h^{\,\prime} \Big/ h \, {V^\ell(L)}_h^{\,\prime} \cong J^\ell(L^*) \; $.  Moreover, the structure of Lie-Rinehart bialgebra induced on  $ L^* $  by the quantization  $ {V^\ell(L)}_h^{\,\prime} $  of  $ \, J^\ell(L^*) $  is dual to that on  $ L $  by the quantization  $ {V^\ell(L)}_h $  of  $ \, V^\ell(L) \, $;
 \vskip3pt
   --- (a.2)\,  the definition of  $ \, {V^\ell(L)}_h \mapsto {}^{\prime}{V^\ell(L)}_h = {V^\ell(L)}_h^{\,\prime} \, $  extends to morphisms in  $ \text{\rm (LQUEAd)} \, $,  so that we have a well defined (covariant) functor  $ \; {}'{(\ )} = {(\ )}' : \text{\rm (LQUEAd)} \longrightarrow \text{\rm (LQFSAd)} \; $.
 \vskip5pt
   (b)\,  Let  $ \, {V^r(L)}_h \in \text{\rm (RQUEAd)}_{A_h} \, $,  where  $ L $  is a Lie-Rinehart algebra which, as an  $ A $--module, is projective of finite type.  Then:
 \vskip3pt
   --- (b.1)\,  $ {}^{\prime}{V^r(L)}_h = {V^r(L)}_h^{\,\prime} \in \text{\rm (RQFSAd)}_{A_h} \, $,  with semiclassical limit  $ \; {V^r(L)}_h^{\,\prime} \Big/ h \, {V^r(L)}_h^{\,\prime} \cong J^r(L^*) \; $.  Moreover, the structure of Lie-Rinehart bialgebra induced on  $ L^* $  by the quantization  $ {V^r(L)}_h^{\,\prime} $  of  $ \, J^r(L^*) $  is dual to that on  $ L $  by the quantization  $ {V^r(L)}_h $  of  $ \, V^r(L) \, $;
 \vskip3pt
   --- (b.2)\,  the definition of  $ \, {V^r(L)}_h \mapsto {}^{\prime}{V^r(L)}_h = {V^r(L)}_h^{\,\prime} \, $  extends to morphisms in  $ \text{\rm (RQUEAd)} \, $,  so that we have a well defined (covariant) functor  $ \; {}'{(\ )} = {(\ )}' : \text{\rm (RQUEAd)} \longrightarrow \text{\rm (RQFSAd)} \; $.
\end{theorem}

\begin{proof}
 {\it (a)}\,  Given  $ \, {V^\ell(L)}_h \in \text{\rm (LQUEAd)}_{A_h} \, $,  we know that  $ \, {J^r(L)}_h := {V^\ell(L)}_h^{\,*} \in \text{\rm (RQFSAd)}_{A_h} \, $,  by  Theorem \ref{dual_QUEAd's=QFSAd's}{\it (a)\/};  then  $ \, {V^r(L^*)}_h := {J^r(L)}_h^{\,\vee} \in \text{\rm (RQUEAd)}_{A_h} \, $  is a quantization of  $ V^r(L^*) \, $,  by  Theorem \ref{compute vee x QFSAd}.
 By  Proposition \ref{props_H'_&_'H},  $ \, {\big(V^\ell(L)_h\big)}' = {}_*{\big( {J^r(L)}_h^{\,\vee\,} \big)} \, $  is a quantization of  $ J^\ell(L^*) \, $,  by  Theorem \ref{dual_QUEAd's=QFSAd's}.  In all this,  $ L^* $  stands for the  $ A $--module  dual to  $ L $  endowed with the Lie-Rinehart bialgebra structure dual to that defined on  $ L $  by the quantization  $ {V^\ell(L)}_h $   --- according to  Theorem \ref{semiclassical_limit-V^ell(L)}.  This completes the proof of  {\it (a.1)}.
%
%%%
% Now, there is a unique way to endow  $ {\big(V^\ell(L)_h\big)}' $  with a left bialgebroid
% structure such that the standard (non degenerate) pairing between  $ {\big({V^\ell(L)}_h\big)}' $  % and  $ \, {\big( {({V^\ell(L)}_h)}^* \big)}^{\!\vee} = {J^r(L)}_h^{\,\vee} \, $  is a
% bialgebroid left pairing: indeed, by  Proposition \ref{props_H'_&_'H}  this pairing
% identifies  $ {\big(V^\ell(L)_h\big)}' $  with the left dual of  $ \, {\big(
% {({V^\ell(L)}_h)}^* \big)}^{\!\vee} = {J^r(L)}_h^{\,\vee} \, $,  that is  $ \;
% \big(V^\ell(L)_h\big)' = {}_*{\big( {J^r(L)}_h^{\,\vee} \big)} \; $  as right bialgebroids.
% Then
%  $$  {\big(V^\ell(L)_h\big)}'  \,\; = \;\,  {}_*{\big( {J^r(L)}_h^{\,\vee\,} \big)}
% \,\; = \;\,  {}_*{\big( {V^r(L^*)}_h \big)}  \,\; = \;\,  {J^\ell(L^*)}_h  $$
% %
% where  $ \; {J^\ell(L^*)}_h = {}_*{\big( {V^r(L^*)}_h \big)} \in \text{\rm (LQFSAd)}_{A_h} \; $
% is a quantization of  $ J^\ell(L^*) \, $,  by  Theorem \ref{dual_QUEAd's=QFSAd's}.  In all this,
% $ L^* $  stands for the  $ A $--module  dual to $ L $  endowed with the Lie-Rinehart bialgebra
% structure dual to that defined on  $ L $  by the quantization  $ {V^\ell(L)}_h $   --- according
% to  Theorem \ref{semiclassical_limit-V^ell(L)}.
%                                                                             \par
%   This completes the proof of  {\it (a.1)}.
%%%
%
 \vskip3pt
   As to  {\it (a.2)},  let  $ \, H_h = {V^\ell\big(L_{\scriptscriptstyle A}\big)}_h \, $  be a LQUEAd over  $ A_h $  and  $ \, \varGamma_h = {V^\ell\big(L_{\scriptscriptstyle B}\big)}_h \, $  a LQUEAd over  $ B_h \, $,  and let  $ \, \phi := (f,F) \, $  be a morphism of left bialgebroids among them.  Set  $ \, \J_{\scriptscriptstyle H_h} := \text{\sl Ker}\,(\epsilon_{\scriptscriptstyle H_h}) \, $  and  $ \, \J_{\scriptscriptstyle \varGamma_h} := \text{\sl Ker}\,(\epsilon_{\scriptscriptstyle \varGamma_h}) \, $.  Then  $ \, F\big( \J_{\scriptscriptstyle H_h} \big) \subseteq \J_{\scriptscriptstyle \varGamma_h} \, $  by the property  $ \, \epsilon_{\scriptscriptstyle \varGamma_h} \circ F = f \circ \epsilon_{\scriptscriptstyle H_h} \, $  of a morphism of bialgebroids.  Similarly, one has  $ \; F^{\otimes n} \circ \Delta_{\scriptscriptstyle H_h}^n = \Delta_{\scriptscriptstyle \varGamma_h}^{\otimes n} \circ F \; $  and  $ \; F \circ s_{\scriptscriptstyle H_h}^\ell = s_{\scriptscriptstyle \varGamma_h}^\ell \, $;  from this, one easily sees that $ \; \delta^n_s\big(F(u)\big) = F^{\otimes n} \big( \delta^n_s(u)\big) \; $.  From all this we get  $ \; F(H_h^{\,\prime}) \subseteq \varGamma_h^{\,\prime} \; $,  \, so the restriction of the morphism  $ (f,F) $  between  $ H_h $  and  $ \varGamma_h $  provides a morphism in  $ \text{\rm (LQFSAd)} $  between  $ H_h^{\,\prime} $  and  $ \varGamma_h^{\,\prime} \; $.
 \vskip7pt
   {\it (b)}\,  A direct proof for claim  {\it (b)\/}  can be given by the same arguments used for  {\it (a)}.  Otherwise, we can deduce  {\it (b)\/}  from  {\it (a)\/}  as follows.
                                                                                      \par
   If  $ \, H_h \in \text{\rm (RQUEAd)} \, $,  then  $ \, {(H_h)}^{\text{\it op}}_{\text{\it coop}} \in \text{\rm (LQUEAd)} \, $  and  $ \, {\big( {(H_h)}^{\text{\it op}}_{\text{\it coop}} \big)}' = {\big( {}^{\prime\!}H_h \big)}^{\text{\it op}}_{\text{\it coop}} \, $,  so that  $ \, H_h^{\,\prime} = {}^{\prime\!}H_h = {\big( {\big( {(H_h)}^{\text{\it op}}_{\text{\it coop}} \big)}' \big)}^{\text{\it op}}_{\text{\it coop}} \, $.  From this we can easily deduce claim  {\it (b)\/}  from claim  {\it (a)}.
\end{proof}

\smallskip

\begin{free text}
 {\bf Description of  $ {V^\ell(L)}_h^{\,\prime} $  when  $ L $  is a (finite type) free  $ A $--module.}  Let  $ L $  be a Lie-Rinehart algebra which, as an  $ A $--module,  is free of finite type.  Let  $ {V^\ell(L)}_h \in \text{\rm (LQUEAd)}_{A_h} \, $  be a quantization of  $ V^\ell(L) \, $;  by the freeness of  $ L \, $,  we can provide an explicit description of  $ \, {V^\ell(L)}_h^{\,\prime} \, $,  much like that given in  \cite{Gavarini}  for the similar case of quantum universal enveloping algebras.
 \vskip4pt
   First of all, consider  $ \, K_h := {V^\ell(L)}_h^{\,*} \equiv {J^r(L)}_h \in \text{\rm (RQFSAd)}_{A_h} \, $,  which (cf.~Theorem \ref{dual_QUEAd's=QFSAd's})  is a quantization of  $ J^r(L) \, $.  From  Proposition \ref{props_H'_&_'H}  we have  $ \, {V^\ell(L)}_h^{\;\prime} \cong {}_*\big(K_h^\vee\big) \; $.
                                                                                       \par
   Let  $ \, {\{\overline{e}_i\}}_{i \in \{1,\dots,n\}} \, $  be a basis of the free  $ A $--module  $ L \, $.  Then (by the Poincar\'e-Birkhoff-Witt theorem) the set of ordered monomials  $ \, {\big\{ \overline{e}^{\,\underline{\alpha}} \big\}}_{\underline{\alpha} \in \N^n} \, $  is an  $ A $--basis  of  $ V^\ell(L) $,  where  $ \; \overline{e}^{\,\underline{\alpha}} := \overline{e}_1^{\,\alpha_1} \cdots \overline{e}_n^{\,\alpha_n} \; $.
                                                                                       \par
   Let  $ \; \overline{\xi}_i \in \text{\sl Hom}\,\big( V^\ell(L) , A \big) \, \cong \, \widehat{S_A(L^*)} \; $  be defined by  $ \; \big\langle\, \overline{\xi}_i \, , \overline{e}_1^{\,\underline{\alpha}} \,\big\rangle \, = \, \delta_{\alpha_1 , \, 0} \cdots \delta_{\alpha_i , 1} \cdots \delta_{\alpha_n , \, 0} \; $.
 Then the ordered monomials  $ \; {1 \over {\,\underline{\alpha}!\,}} \, \overline{\xi}^{\,\underline{\alpha}} \; $  (with  $ \, \underline{\alpha}! := \alpha_1! \cdots \alpha_n! \, $)  is a pseudobasis   --- i.e., a basis in topological sense ---   of the  $ A $--module  $ J^r(L) $  dual to the PBW basis  $ \, {\big\{ \overline{e}^{\,\underline{\alpha}} \big\}}_{\underline{\alpha} \in \N^n} \, $.
                                                                                       \par
   Lift  $ \, {\{\, \overline{\xi}_i\}}_{i \in \{1,\dots,n\}} \, $  to a subset  $ \, {\{\xi_i\}}_{i \in \{1,\dots,n\}} \, $  in  $ \, {J^r(L)}_h = K_h \, $  such that  $ \, \partial_h(\xi_i) = 0 \, $;  then  $ \, {\big\{ {1 \over {\,\alpha!\,}} \, \xi^{\underline{\alpha}} \big\}}_{\underline{\alpha} \in \N^n} \, $  is a topological pseudo-basis of  $ \, {J^r(L)}_h = K_h \, $.  Let  $ \, {\big\{ \theta^{\underline{\alpha}} \big\}}_{\underline{\alpha} \in \N^n} \, $  be the topological basis of  $ {V^\ell(L)}_h $  dual to  $ \, {\big\{ {1 \over {\,\underline{\alpha}!\,}} \, \xi^{\,\underline{\alpha}} \big\}}_{\underline{\alpha} \in \N^n} \, $;  then  $ \, {\big\{ \theta^{\underline{\alpha}} \big\}}_{\underline{\alpha} \in \N^n} \, $  is a lift of the PBW basis  $ \, {\big\{ \overline{e}^{\,\underline{\alpha}} \big\}}_{\underline{\alpha} \in \N^n} \, $  of  $ V^\ell(L) \; $.  Using these tools, a straightforward analysis shows that
  $$  {V^\ell(L)}_h^{\,\prime}  \,\; = \;\,  \big\{\, {\textstyle \sum_{\underline{\alpha}}} \, t^\ell(a_{\underline{\alpha}\,}) \, h^{|\underline{\alpha}|} \, \theta_{\underline{\alpha}} \;\big|\; a_{\underline{\alpha}} \in A_h \,\big\}  $$
where the summation symbol denotes  $ h $--adically  convergent series.
\end{free text}

\smallskip

 \subsection{Quantum duality for quantum groupoids}  \label{q-duality_q-groupds}

\smallskip

   {\ } \quad   We consider now the composition of two Drinfeld's functors.  We shall prove that the functors  $ {(\ )}^\vee $  and  $ \, (\ )' = {}^{\hskip1pt\prime}\!{(\ )} \, $  are actually inverse to each other, so that they establish equivalences of categories  $ \; \text{\rm (RQFSAd)} \, \cong \text{\rm (RQUEAd)} \; $  and  $ \; \text{\rm (LQFSAd)} \, \cong \text{\rm (LQUEAd)} \; $.  Our result reads as follows:

\smallskip

\begin{theorem}  \label{quantum_duality}
 {\ }
 \vskip4pt
   (a) \,  If  $ \, K_h \in \text{\rm (RQFSAd)} \, $,  then  $ \; {\big(K_h^{\,\vee}\big)}' = K_h = \!\phantom{\big|}'\big(K_h^{\,\vee}\big) \; $.
 \vskip4pt
   (b) \,  If  $ \, K_h \in \text{\rm (LQFSAd)} \, $,  then  $ \; {\big(K_h^{\,\vee}\big)}' = K_h = \!\phantom{\big|}'\big(K_h^{\,\vee}\big) \; $.
 \vskip4pt
   (c) \,  If  $ \, H_h \in \text{\rm (LQUEAd)} \, $,  then  $ \; {\big( H_h^{\,\prime} \big)}^{\!\vee} = H_h = {\big( {}'{\!}H_h \big)}^{\!\vee} \; $.
 \vskip4pt
   (d) \,  If  $ \, H_h \in \text{\rm (RQUEAd)} \, $,  then  $ \; {\big( H_h^{\,\prime} \big)}^{\!\vee} = H_h = {\big( {}'{\!}H_h \big)}^{\!\vee} \; $.
 \vskip4pt
   (e) \,  The functors  $ \; {(\ )}^\vee : \text{\rm (RQFSAd)} \rightarrow \text{\rm (RQUEAd)} \, $  and  $ \; (\ )' = {}'{(\ )} : \text{\rm (RQUEAd)} \rightarrow \text{\rm (RQFSAd)} \, $  are inverse to each other, hence they are equivalences of categories.  Similarly for the functors  $ \; {(\ )}^\vee : \text{\rm (LQFSAd)} \longrightarrow \text{\rm (LQUEAd)} \; $  and  $ \; (\ )' = {}'{(\ )} : \text{\rm (LQUEAd)} \longrightarrow \text{\rm (LQFSAd)} \; $.
\end{theorem}

\begin{proof}
 Clearly, claim  {\it (e)\/}  is just a consequence of the previous items in the statement.  We begin by focusing on claim  {\it (a)\/}:  we assume that  $ \, K_h \in \text{\rm (RQFSAd)}_{A_h} \, $  and we shall prove that  $ \; {\big(K_h^{\,\vee}\big)}' = K_h \; $.
 \vskip5pt

   Let us show that  $ \, K_h \subseteq {\big(K_h^{\,\vee}\big)}' \, $.
 \vskip3pt
   Given  $ \lambda $  is in  $ K_h \, $,  consider its  $ n $--th  iterated coproduct  $ \, \Delta^n(\lambda) = \lambda_{(1)} \otimes \cdots \otimes \lambda_{(n)} \, $;  if we write every  $ \lambda_{(i)} $  as  $ \, \lambda_{(i)} = \lambda'_{(i)} + \lambda''_{(i)} \, $  with  $ \, \lambda'_{(i)} := \lambda_{(i)} - s^r_h\big(\partial(\lambda_{(i)})\big) \in \J_h := \text{\sl Ker}\,\big( \partial_{K_h} \big) \, $  and  $ \, \lambda''_{(i)} := s^r_h\big( \partial(\lambda_{(i)}) \big) \in s^r_h(A_h) \, $,  then expanding again  $ \Delta^n(\lambda) $  we can write it as a sum  $ \, \Delta^n(\lambda) = \sum \lambda^\circ_{(1)} \otimes \cdots \otimes \lambda^\circ_{(n)} \, $  in which  $ \, \lambda^\circ_i \in \J_h \, $  or  $ \, \lambda^\circ_i \in s^r_h(A_h) \, $  for every  $ i = 1, \dots, n \, $.
                                                             \par
   Now let  $ \, \alpha_1 \, , \dots , \alpha_n \in I_{\!{}_{{}_*\!}(K_h^{\vee})} := \epsilon_{{}_{{}_*\!}(K_h^{\vee})}^{\,-1}\big(h\,A_h\big) \, $.  As every  $ \alpha_j $  belongs to $ {}_{{}_{\scriptstyle *}}\!\big( K_h^\vee \big) \, $,  it defines a map from  $ \J_h $  to  $ \, h \, A_h \, $.  Hence  $ \; \big\langle \alpha_i \, , \lambda_j \big\rangle \in h \, A_h \; $  and one has  $ \; \big\langle \alpha_1 \cdots \alpha_n \, , \, \lambda \big\rangle \in h^n A_h \; $.  Thus, for any  $ \, n \in \N \, $,  we have that  $ \lambda $  defines a map  $ \; \Lambda_n : h^{-n} \, I_{\!{}_{{}_*\!}(K_h^{\vee})}^{\,n} \longrightarrow A_h \; $.  Clearly all these  $ \Lambda_n $'s  match together to define an element  $ \; \Lambda \in {\Big(\! \big( {}_{{}_{\scriptstyle *}}\!\big(K_h^{\,\vee}\big) \big)^{\!\vee} \Big)}^* \! = {\big( K_h^\vee \big)}' \; $;  thus we end up with a natural map  $ \; K_h \longrightarrow {\big( K_h^\vee \big)}' \; \big( \lambda \mapsto \Lambda \big) \, $,  which is clearly injective.  This yields the inclusion  $ \; K_h \subseteq {\big( K_h^\vee \big)}' \; $.
 \vskip4pt
   To prove the converse inclusion  $ \; K_h \supseteq {\big( K_h^\vee \big)}' \; $,  one  proceeds exactly like in  \cite{Gavarini}   --- we leave the details to the reader.  Similarly, we leave to the reader the proof of  {\it (b)},  analogous to that of  {\it (a)}.
 \vskip9pt
  To prove claim  {\it (c)},  consider  $ \, H_h \in \text{\rm (LQUEAd)} \, $.  We have  $ \, K_h := H_h^{\,\ast} \in \text{\rm (RQFSAd)} \, $,  and  $ \; H_h^{\,\prime} = \! {}_{{}_{\scriptstyle \ast}\!}{\Big(\! {\big( H_h^{\,\ast} \big)}^{\!\vee} \Big)} = \! {}_{{}_{\scriptstyle \ast}\!}{\big( K_h^{\,\vee} \big)} \; $  by  Proposition \ref{props_H'_&_'H}{\it (a)}.  Now  $ \, \varGamma_h := K_h^{\,\vee} \in \text{\rm (RQUEAd)} \, $  by  Theorem \ref{compute vee x QFSAd}, and then  $ \, {}'\!\varGamma_h = {\big( {({}_*{}\varGamma_h)}^\vee \,\big)}^* \, $  by
 Proposition \ref{props_H'_&_'H}{\it (b)},  which implies  $ \, {}_\star{\big({}'\!\varGamma_h\big)} = {}_{{}_{\scriptstyle \star}}\!\Big(\! {\big( {({}_*{}\varGamma_h)}^\vee \,\big)}^* \Big) = {({}_*{}\varGamma_h)}^\vee \, $.
Altogether   --- also exploiting claim  {\it (a)}  ---   this gives
  $$  {\big( H_h^{\,\prime} \big)}^{\!\vee}  \, = \,  {\big( {}_{{}_{\scriptstyle \ast}\!}{\big( K_h^{\,\vee} \big)} \big)}^{\!\vee} = {({}_*{}\varGamma_h)}^\vee  \, = \,  {}_\star{\big({}'\!\varGamma_h\big)}  \, = \,  {}_\star{\big(\varGamma_h^{\,\prime}\big)}  \, = \,  {}_\star{\Big(\! {\big(K_h^{\,\vee}\big)}' \Big)}  \, = \,  {}_\star{}K_h  \, = \,  {}_\star{\big( H_h^{\,\ast} \big)}  \, = \,  H_h  $$
   \indent   This proves  {\it (c)},  and the proof of  {\it (d)\/}  is entirely similar again.
\end{proof}

\bigskip

\section{An example}  \label{example_twistor}

\smallskip

   {\ } \quad   In this last section we apply the main construction of the paper   --- duality functors and Drinfeld's functors ---   to a toy model, namely a simple (yet non trivial!) quantum groupoid.
                                                                          \par
   We consider the two dimensional Lie  $ k $--algebra  $ \; \mathfrak{g} = k \, e_1 \oplus k \, e_2 \; $  with Lie bracket  $ \, [e_1,e_2] = e_1 \, $.  It is known that  $ \mathfrak{g}^* $  is a Poisson manifold: we consider  $ e_1 $  and  $ e_2 $  as coordinates on  $ \mathfrak{g}^* \, $,  denoting them by  $ x_1 $  and  $ x_2 $  respectively.  The Poisson structure on  $ \mathfrak{g}^* $  is determined by  $ \, \{x_1,x_2\} = [e_1,e_2] = e_1 \, $.
                                                                          \par
   Let us introduce the Lie  $ k[[h]] $--algebra  $ \; \mathfrak{g}_h := \, k[[h]] \, e_1 \oplus k[[h]] \, e_2 \; $  with non-zero Lie bracket  $ \, {[e_1,e_2]}_h := h \, e_1 \, $.  The  $ h $--adic  completion of the enveloping algebra  of  $ {\mathfrak g}_h \, $,  namely  $ \, A_h := \widehat{U(\mathfrak{g}_h}) \, $,   is a quantization of the Poisson algebra of polynomial functions on  $ \mathfrak{g}^* \, $,  namely  $ \, A = S(\mathfrak{g}) \, $.
                                                                          \par
   We write  $ {\mathcal D} $  for the ring of polynomial differential operators on  $ \mathfrak{g}^* \, $,  with  $ \, \partial_i := \dfrac {\partial } {\,\partial x_i\,} \, $,  $ \, i=1, 2 \, $.  It is the enveloping algebra of the Lie-Rinehart algebra  $ \, \big( S({\mathfrak g}) , \text{\sl Der}\big(S({\mathfrak g})\big) , \text{\sl id\/} \big) \, $.  We endow it with the standard left algebroid stucture and denote by  $ {\mathcal D}[[h]] $  the trivial deformation of this structure.

\smallskip

\begin{proposition}  \label{Prop-twistor}
 Fix notation  $ \, \theta_1 \! := \! x_1 \partial_1 \, $.  Then  $ \, {\mathcal F} \! := \!\! \sum\limits_{n=0}^\infty {\displaystyle {{\,h^n\,} \over {n!\,2^n}} \, {\Big( \theta_1 \! \otimes \partial_2 - \partial_2 \! \otimes \theta_1 \!\Big)}^{\!n}} $  is a twistor   --- cf.\ Definition \ref{def-twistor}  ---   for  $ \, {\mathcal D}[[h]] \, $.
\end{proposition}

\begin{proof}
 It is a straightforward computation.
\end{proof}

\smallskip

   We will now denote by  $ {\mathcal D}_h $  the twist of  $ {\mathcal D}[[h]] $  by  $ {\mathcal F} \, $.  As an algebra,  $ {\mathcal D}_h $  is isomorphic to  $ \, \big(S({\mathfrak g}) \otimes S({\mathfrak g}^*)\big)[[h]] \, $.  The deformation of  $ \, A = S({\mathfrak g}) \, $  defined by  $ {\mathcal F} $  is  $ \, A_h = \widehat{U(\mathfrak{g}_h)} \, $,  the  $ h $--adic  completion of the universal enveloping algebra  $ U(\mathfrak{g}_h) $  of  $ \mathfrak{g}_h \, $.  The source map  $ s^\ell_{\mathcal F} $  (an algebra morphism) is determined by
  $$  s^\ell_{\mathcal F}(x_1)  \; = \;  {\textstyle \sum\limits_{n=0}^{\infty}} \, \dfrac{1}{\,n!\,} \, \dfrac{h^n}{\,2^n\,} \, x_1 \, \partial_2^{\,n}  \quad ,  \qquad
      s^\ell_{\mathcal F}(x_2)  \; = \;  x_2 \, - \, h \, x_1 \, \partial_1  $$
The target  $ t^\ell_{\mathcal F} $  (an algebra antimorphism) the coproduct  $ \Delta_{\mathcal F} $  and the counit  $ \epsilon $  are determined by
  $$  \begin{array}l
    \qquad   t^\ell_{\mathcal F}(x_1)  \; = \;  \sum\limits_{n=0}^{\infty} \dfrac{(-1)^n}{\,n!\,} \, \dfrac{h^n}{\,2^n\,} \, x_1 \, \partial_2^{\,n}  \quad ,  \qquad
   t^\ell_{\mathcal F}(x_2)  \; = \;  x_2 \, + \, h \, x_1 \, \partial_1  \\
   \Delta_{\mathcal F}(X)  \; = \;  {\mathcal F}^{\# -1} \big( \Delta(X) \cdot {\mathcal F} \big) \quad ,  \quad \qquad
 \epsilon \Big( x_1^{\alpha_1} \, x_2^{\alpha_2} \, \partial_1^{\,\beta_1} \partial_2^{\,\beta_2} \Big)  \; = \;  x_1^{\alpha_1} \, x_2^{\alpha_2} \, \partial_1^{\,\beta_1} \partial_2^{\,\beta_2}(1)
      \end{array}  $$
(cf.\ Theorem \ref{twisted_left-bialgd}).  Explicitly,  $ \mathcal{F} $  can be lifted to an element  $ \, \widetilde{\mathcal F} \in \big( {\mathcal D} \otimes_k {\mathcal D} \big)[[h]] \, $  defined by
  $$  \widetilde{\mathcal F}  \; = \;  \exp \bigg( \dfrac{h}{\,2\,} \, \bigg( \theta_1 \otimes \dfrac{\partial}{\,\partial x_2\,} \, - \, \dfrac{\partial}{\,\partial x_2\,} \otimes \theta_1 \bigg) \bigg)  \, = \,  {\textstyle \sum\limits_{n=0}^{\infty}} \, \dfrac{1}{\,n!\,} \, \dfrac{h^n}{\,2^n\,} \, {\bigg( \theta_1 \otimes \dfrac{\partial}{\,\partial x_2\,} \, - \,
\dfrac{\partial}{\,\partial x_2\,} \otimes \theta_1 \bigg)}^{\!n}  \; \in \;  {\mathcal D} \otimes_k {\mathcal D}[[h]]  $$
this element  $ \widetilde{\mathcal F} $  is invertible in  $ \, \big( {\mathcal D} \otimes_k {\mathcal D} \big)[[h]] \, $  and one has
  $$  \displaylines{
   \widetilde{\mathcal F}^{-1}  \; = \;\,  \exp \bigg( - \dfrac{h}{\,2\,} \, \bigg( \theta_1 \otimes \dfrac{\partial}{\,\partial x_2\,} \, - \, \dfrac{\partial}{\,\partial x_2\,} \otimes \theta_1 \bigg) \bigg) \in \;\;  \big( {\mathcal D} \otimes_k {\mathcal D} \big)[[h]]  }  $$
In turn, the element  $ \widetilde{\mathcal F}^{-1} $  defines an element  $ \; {\mathcal G} \in \, {\mathcal D}[[h]] \,{\widehat{\otimes}_{A_{\mathcal F}}\, {\mathcal D}[[h]]} \, $,  \, namely
  $$  {\mathcal G}  \;\; = \;\;  {\textstyle \sum\limits_{n=0}^{\infty}} \, (-1)^n \, \dfrac{h^n}{\,2^n\,n!\,} \, {\bigg( \theta_1 \otimes \dfrac{\partial}{\,\partial x_2\,} \, - \, \dfrac{\partial}{\,\partial x_2\,} \otimes \theta_1 \bigg)}^{\!n} \, \in \, {\mathcal D}[[h]] \,{\widehat{\otimes}_{A_{\mathcal F}}\, {\mathcal D}[[h]]}  $$
Now the map
  $$  {\mathcal F}^{\#} \, : \, {\mathcal D}[[h]] \,\widehat{\otimes}_{A_{\mathcal F}}\, {\mathcal D}[[h]] \longrightarrow {\mathcal D}[[h]] \,\widehat{\otimes}_A{}\, {\mathcal D}[[h]] \;\; ,  \qquad
   h_1 \otimes h_2 \, \mapsto \, {\mathcal F} \cdot (h_1 \otimes h_2)  $$
is indeed invertible, its inverse being
  $$  {\mathcal F}^{\#-1} \, : \, {\mathcal D}[[h]] \,\widehat{\otimes}_A{}\, {\mathcal D}[[h]] \longrightarrow {\mathcal D}[[h]] \,\widehat{\otimes}_{A_{\mathcal F}}\, {\mathcal D}[[h]] \;\; ,  \qquad
   h_1 \otimes h_2 \, \mapsto \, {\mathcal G} \cdot (h_1 \otimes h_2)  $$

   We will compute now the dual bialgebroids  $ {\big( {\mathcal D}_h \big)}_* $  and  $ {\big( {\mathcal D}_h \big)}^* \, $.

\bigskip

   {\bf Computation of  $ {\big( {\mathcal D}_h \big)}_* \, $:} \,  We shall use the isomorphism
  $$  ({\mathcal D}_h )_* \longrightarrow \text{\sl Hom}({\mathcal D},A)[[h]] \;\; ,  \qquad
   \lambda \, \mapsto \, \Big(\, \partial_1^{\,a} \, \partial_2^{\,b} \, \mapsto \, \big\langle \lambda \, , \, \partial_1^{\,a} \, \partial_2^{\,b} \,\big\rangle \,\Big)  $$

   \indent   Let  $ \, de_1 , de_2 \in {\big( {\mathcal D}_h \big)}_* \, $  be such that
$ \; \big\langle\, de_1 \, , \, \partial_1^{\,a} \, \partial_2^{\,b} \,\big\rangle \, = \, \delta_{1,a} \, \delta_{0,b} \; $,  $ \; \big\langle\, de_2 \, , \, \partial_1^{\,a} \, \partial_2^{\,b} \,\big\rangle \, = \, \delta_{0,a} \, \delta_{1,b} \; $.  Similarly, let  $ \, e_1 , e_2 \in {\big( {\mathcal D}_h \big)}_* \, $  be such that  $ \; \big\langle\, e_1 \, , \, \partial_1^{\,a} \, \partial_2^{\,b} \,\big\rangle \, = \, x_1 \, \delta_{0,a} \, \delta_{0,b} \; $,  $ \; \big\langle\, e_2 \, , \, \partial_1^{\,a} \, \partial_2^{\,b} \,\big\rangle \, = \, x_2 \, \delta_{0,a} \, \delta_{0,b} \; $.
                                                   \par
   A direct computation shows that
  $$  \big\langle\, de_1 \cdot_h de_2 \, , \, \partial_1^{\,a} \, \partial_2^{\,b} \,\big\rangle  \;\; = \;\;
   \left\{
      \begin{array}l
  0  \quad \text{\rm if} \quad  a \geq 2  \quad \text{\rm or} \quad  b \geq 2  \\
  1  \quad \text{\rm if} \quad  a = 1  \quad \text{\rm and} \quad  b = 1  \\
  - \dfrac{h}{\,2\,}  \quad \text{\rm if} \quad  a = 1  \quad \text{\rm and} \quad  b = 0  \\
  0  \quad \text{\rm if} \quad  a = 0  \quad \text{\rm and} \quad  b = 1
      \end{array}
   \right.  $$
Similarly
  $$  \big\langle\, de_{2}\cdot_h de_1 \, , \, \partial_1^{\,a} \, \partial_2^{\,b} \,\big\rangle  \;\; = \;\;
   \left\{
      \begin{array}l
  0  \quad \text{\rm if} \quad  a \geq 2  \quad \text{\rm or} \quad  b \geq 2  \\
  1  \quad \text{\rm if} \quad  a = 1  \quad \text{\rm and} \quad  b = 1  \\
  \dfrac{h}{\,2\,}  \quad \text{\rm if} \quad  a = 1  \quad \text{\rm and} \quad  b = 0  \\
  0  \quad \text{\rm if} \quad  a = 0  \quad \text{\rm and} \quad  b = 1
      \end{array}
   \right.  $$
Hence  $ \; de_1 \cdot_h de_2 \, - \, de_2 \cdot_h de_1 \, = \, -h \, de_1 \; $.  Set  $ \; \check{de_i} := h^{-1} \, de_i \; $.  This equality can be written as
  $$  \check{de}_1 \cdot_h \check{de}_2 \, - \, \check{de}_2 \cdot_h \check{de}_1  \; = \;  -\check{de}_1  $$
                                                        \par
   Similarly, the following equalities can be established:
  $$ \displaylines{
  \check{de_1} \cdot_h e_2 - e_2 \cdot \check{de_1} \, = \, -e_1 \quad,
   \quad e_1 \cdot e_2 - e_2 \cdot_h e_1  \; = \;  h \, e_1  \quad ,
 \qquad  \check{de_1} \cdot_h e_1  \; = \;  e_1 \cdot_h \check{de_1}  \cr
   \check{de_2} \cdot_h e_2  \; = \;  e_2 \cdot_h \check{de_2}  \!\quad ,
 \!\!\!\qquad  \check{de_2} \cdot_h e_1 - e_1 \cdot_h \check{de_2}  \; = \;  e_1  \!\quad ,
 \!\!\!\qquad  s^r_*(e_i)  \, = \,  e_i  \!\quad ,
 \!\!\!\qquad  t^r_*(e_i)  \, = \,  e_i + h\,\check{de_i}  }  $$
From the properties of the coproduct, one gets  $ \; \Delta(e_i) \, = \, 1 \otimes e_i \; $,  $ \; \Delta(e_i + h \, \check{de_i}) \, = \, (e_i + h \, \check{de_i}) \otimes 1 \; $,  from which we deduce  $ \; \Delta(\check{de_i}) \, = \, \check{de_i} \otimes 1 + 1 \otimes \check{de_i} \; $.  The coproduct on  $ \, {\big( {({\mathcal D_h})}_* \big)}^\vee \, $  is now determined.

   Let us also point out the counit of  $ \, {\big( {({\mathcal D_h})}_* \big)}^\vee \, $:  it is given by  $ \; \partial({\check{de_i}}) \, = \, 0 \; $  and  $ \; \partial(e_i) \, = \, e_i \;\, $.

\smallskip

\begin{remark}
 \, Let us introduce the Lie algebra  $ \, {\mathfrak g}_1 \, $  such that  $ \; {\mathfrak g}_1 \cong {\mathfrak g } \, = \, k \, \check{de_1} \oplus k \, \check{de_2} \; $  (as a  $ k $--vector  space) and  $ \; {[\,\ ,\ ]}_1 := -{[\,\ ,\ ]}_{\mathfrak g} \; $.  Then  $ {\mathfrak g}_1 $  acts on  $ \, {\mathfrak g}_h = k[[h]] e_1 \oplus k[[h]]e_2 \, $  by derivations, via
  $$  {\mathfrak g} \longrightarrow \text{\sl Der}\,({\mathfrak g}_h) \; ,  \qquad
   \check{de}_1 \mapsto \left\{
      \begin{array}{l}
         e_1  \mapsto  0  \\
         e_2  \mapsto  -e_1
      \end{array}  \right.
 \; ,  \qquad
   \check{de}_2 \mapsto \left\{
      \begin{array}{l}
         e_2  \mapsto  0\\
         e_1  \mapsto  e_1
      \end{array}  \right.  $$
 We may perform the semi direct product  $ \, {\mathfrak g}_1 \ltimes {\mathfrak g}_h \, $  and  $ \, {\big( {({\mathcal D}_h)}_* \big)}^\vee \, $  is isomorphic to  $ \, U\big( {\mathfrak g}_1 \ltimes {\mathfrak g}_h \big) \, $  as an algebra but not as a bialgebroid.
\end{remark}
   \indent   Let us now compute  $ \, {}'{\mathcal D}_h \, $.  We proceed in several steps.

\smallskip

\noindent
 $ \bullet $  \  {\sl Let us show that  $ \, h \, \partial_2 \in {}'{\mathcal D}_h \, $}.
                                                                                 \par
   We shall show that  $ \; \big\langle\, h \, \partial_2 \, , \, de_1^{a_1} \, de_2^{a_2} \big\rangle = 0 \; $  if  $ \, (a_1, a_2) \neq (0,1) \; $.  We have three cases:
                                                                                 \par
   {\it First case:}  $ \; a_2 = 0 \; $.  \quad  In this case it is obvious that  $ \; \big\langle\, \partial_2 \, , \, de_1^{a_1} \big\rangle = 0 \; $.
                                                                                 \par
   {\it Second case:}  $ \; a_2 = 1 \; $.  \quad  In this case we have
 $ \; \big\langle\, \partial_2 \, , \, de_1^{a_1} \, de_2 \big\rangle \; = \;   \left\{
  \begin{array}{l}
     0  \quad \text{if} \quad  a_1 \neq 0  \\
     1  \quad \text{if} \quad  a_1 = 0
  \end{array}
        \right. \; $
                                                                                 \par
   {\it Third case:}  $ \; a_2 > 1 \; $.  \quad  In this case the summands in  $ \, \Delta_{\mathcal F}\big(\partial_2\big) \, $  that might bring a non zero contribution to  $ \; \big\langle\, \partial_2 \, , \, de_1^{a_1} \, de_2^{a_2} \big\rangle \; $  are those of the form
 $ \; \partial_2^{\,a_2} \, \otimes \, \theta_1^{a_2 ' + a''_2} \,
\dfrac{(-1)^{a''_2}}{\,a'_2! \, a''_2!\,} \, \dfrac{h^{a_2 -1}}{\,2^{a_2 - 1}\,} \; $  with  $ \; a'_2 + a''_2 = a_2 - 1 \; $;
 \, but
 $ \sum\limits_{a'_2 + a''_2 = a_2 - 1} \hskip-13pt \partial_2^{a_2} \otimes \, \theta_1^{a_2 ' + a''_2} \, \dfrac{(-1)^{a''_2}}{\,a'_2! \, a''_2!\,} \, \dfrac{h^{a_2 -1}}{\,2^{a_2 - 1}\,}  \, = \, 0 \; $,
 \, so we find again  $ \; \big\langle\, h \, \partial_2 \, , \, de_1^{a_1} \, de_2^{a_2} \,\big\rangle = 0 \; $.

\smallskip

\noindent
 $ \bullet $  \  {\sl Let us show that  $ \, h \, \partial_1 \in {}'{\mathcal D}_h \, $}.  We will show that  $ \; \big\langle\, h \, \partial_1 \, , \, de_1^{a_1} \, de_2^{a_2} \,\big\rangle \, \in \, h^{a_1 +a_2 } \, A_h \;\, $.
                                                                                 \par
   We start by computing  $ \; \big\langle\, \theta_1 \, , \, de_1^{a_1} \, de_2^{a_2} \,\big\rangle \; $.

\vskip9pt

   {\it First case:}  $ \; a_2 = 0 \; $.  \quad  It is easy to check that
 $ \; \big\langle\, \theta_1 \, , \, (de_1)^{a_1} \big\rangle \, = \,
  \left\{
   \begin{array}{l}
      x_1  \quad \text{if} \quad  a_1 = 1  \\
      0    \quad \text{otherwise} \quad
   \end{array}
  \right. \;$
                                                                                 \par
   {\it Second case case:}  $ \; a_2 \geq 1 \; $.  \quad  The summands in  $ \, \Delta_{\mathcal F}\big( \partial_2 \big) \, $  that might bring a non zero contri\-bution
% to
$ \; \big\langle\, \theta_1 \, , \, de_1^{a_1} \, de_2^{a_2} \,\big\rangle \; $
  are those of the form
  $ \; \partial_2^{a_2} \otimes\, \theta_1^{a_2 '+ a_2'' + 1} \, \dfrac{(-1)^{a''_2}}{\,a'_2! \, a''_2!\,} \, \dfrac{h^{a_2 }}{\,2^{a_2}\,} \; $  with  $ \, a'_2 + a''_2 = a_2 \; $;
 \, but
 $ \; \sum\limits_{a'_2 +a''_2=a_2} \hskip-7pt
\partial_2^{a_2} \otimes\, \theta_1^{a_2 '+ a_2'' + 1} \,
\dfrac{(-1)^{a''_2}}{\,a'_2! \, a''_2!\,} \, \dfrac{h^{a_2 }}{\,2^{a_2}\,} \, = \, 0 \; $,  \,
hence in the end  $ \; \big\langle\, \theta_1 \, , (de_1)^{a_1} \, (de_2)^{a_2} \, \big\rangle \, = \, 0 \; $.

\smallskip

   In conclusion, we find
 $ \;\; \big\langle\, \theta_1 \, , \, (de_1)^{a_1} \, (de_2)^{a_2} \,\big\rangle \, = \,
     \left\{
   \begin{array}{l}
        x_1  \quad \text{if} \quad  (a_1,a_2) = (1,0)  \\
        0  \quad \text{otherwise}
   \end{array}
     \right. \; $

\smallskip

   Let us now compute  $ \; \big\langle\, \partial_1 \, , \, de_1^{a_1} \, de_2^{a_2} \,\big\rangle \; $.  Again we have several cases to consider.

\vskip9pt

   {\it First case:}  $ \; a_2 = 0 \; $.  \quad  It is easy to check that
 $ \; \big\langle\, \partial_1 \, , \, (de_1)^{a_1} \,\big\rangle \, = \,
     \left\{
   \begin{array}{l}
      1  \quad \text{if} \quad  a_1 = 1  \\
      0  \quad \text{otherwise}
   \end{array}
     \right. \; $

\smallskip

   {\it Second case:}  $ \; a_2 \geq 1 \; $.  \quad   In this case one has
  $$  \displaylines{
   0  \,\; = \;\,  \big\langle\, \theta_1 \, , \, (de_1)^{a_1} \, (de_2)^{a_2} \,\big\rangle  \,\; = \;\,  \bigg\langle s^\ell_{\mathcal F}(x_1) \, \partial_1 \, - \, {\textstyle \sum\limits_{n=1}^{+\infty}}\, \theta_1 \, \dfrac{h^n}{\,2^n\,n!} \, \partial_2^{\,n} \; , \, (de_1)^{a_1} \, (de_2)^{a_2} \bigg\rangle  \,\;    \hfill  \cr
   \hfill   = \;\,  x_1 \, \Big\langle\, \partial_1 \; , \, (de_1)^{a_1} \, (de_2)^{a_2} \Big\rangle \, - \, {\textstyle \sum\limits_{n=1}^{a_2-1}} \bigg\langle\, \theta_1 \, \dfrac{h^n}{\,2^n\, n!} \, \partial_2^{\,n} \; , \, (de_1)^{a_1} \, (de_2)^{a_2} \bigg\rangle
    \hfill  \cr
    + \; \bigg\langle \theta_1 \, \dfrac{h^{a_2}}{\,2^{a_2}\, a_2!} \; \partial_2^{a_2} \; , \, (de_1)^{a_1} \, (de_2)^{a_2} \bigg\rangle     }  $$
For  $ \, 1 \leq n \leq a_2 - 1 \, $,  the unique summands in  $ \, \Delta_{\mathcal F}\big( \theta_1 \, \partial_2^{\,n} \big) \, $  that may bring a non zero contribution to  $ \; \big\langle\, \theta_1 \, \partial_2^{\,n} \; , \, (de_1)^{a_1} \, (de_2)^{a_2} \,\big\rangle \; $  are those of the form
 $ \; \partial_2^{a_2} \otimes \theta_1^{a_2 - n + 1} \, \dfrac{h^{a_2 - n}}{\,2^{a_2 - n}\,} \, \dfrac{(-1)^{c_2}}{c_1! \, c_2!} \; $  with  $ \, c_1 + c_2 = a_2 - n \; $;
but
 $ \; \sum\limits_{c_1 + c_2 = a_2 - n} \hskip-7pt \partial_2^{a_2} \otimes\, \theta_1^{a_2 - n + 1} \, \dfrac{h^{a_2 - n}}{\,2^{a_2 - n}\,} \, \dfrac{(-1)^{c_2}}{c_1! \, c_2!}  \, = \,  0 \; $,
hence
 $ \; \sum\limits_{n=1}^{a_2-1} \! \bigg\langle \theta_1 \dfrac{h^n}{\,2^n\, n!} \, \partial_2^{\,n} \; , \, (de_1)^{a_1} \, (de_2)^{a_2} \bigg\rangle \, = \, 0 \; $.
                                                                   \par
   Finally, we remark that  $ \; \big\langle\, \theta_1 \, \partial_2^{a_2} \, , \, (de_1)^{a_1} \, (de_2)^{a_2} \,\big\rangle \; $  is zero if  $ \, a_1 \neq 1 \, $.  Hence, in any case, we have  $ \,\; \bigg\langle \theta_1 \, \partial_2^{a_2} \, \dfrac{h^{a_2}}{\,2^{a_2}\, a_2!} \; , \, (de_1)^{a_1} \, (de_2)^{a_2} \bigg\rangle \, \in \, h^{a_1+a_2-1} A_h \;\, $.

\smallskip

   In conclusion, we find  that in all cases one has  $ \,\; \big\langle\, \partial_1 \, , \, (de_1)^{a_1} \, (de_2)^{a_2} \,\big\rangle \, \in \, h^{a_1 + a_2 - 1} \, A_h \;\, $.

\medskip

   Now denote by  $ \, {\big\{ \eta_{a,b} \big\}}_{(a,b) \in \N^2} \, $  the topological basis of  $ {\mathcal D}_h $  dual to the basis  $ \, {\left\{\! \dfrac{\,de_1^a\,}{a!} \, \dfrac{\,de_2^b\,}{b!} \right\}}_{\!(a,b) \in \N^2} \, $.  We know that
 $ \; {}'{\mathcal D}_h \, = \, \Big\{\, {\textstyle \sum_{(a,b)\in \N^2}} \, s^\ell_{\mathfrak F}(\alpha_{a,b}) \, h^{a+b} \eta_{a,b} \;\Big|\; \alpha_{a,b} \in A_h \,\Big\} \; $.
As  $ \,\; h^{a+b} \, \eta_{a,b} = \, h^{a+b} \, \partial_1^{\,a} \, \partial_2^{\,b} \; $  modulo  $ \; h' \, {\mathcal D}_h \;\, $,  \, we have
 $ \; {}'{\mathcal D}_h \, = \, \Big\{\, {\textstyle \sum_{(a,b)\in \N^2}} \, s^\ell_{\mathfrak F}(\alpha_{a,b}) \, h^{a+b} \, \partial_1^{\,a} \, \partial_2^{\,b} \;\Big|\; \alpha_{a,b} \in A_h \,\Big\} \; $.

\bigskip

   {\bf Computation of  $ \, {({\mathcal D}_h)}^* \, $:} \,  We shall compute  $ {({\mathcal D}_h)}^* \, $,  using the isomorphism
  $$  {({\mathcal D}_h)}^* \longrightarrow \text{\sl Hom}({\mathcal D},A)[[h]] \;\; ,  \qquad
   \lambda \, \mapsto \, \Big(\, \partial_1^{\,a} \, \partial_2^{\,b} \, \mapsto \big\langle \lambda \, , \, \partial_1^{\,a} \, \partial_2^{\,b} \,\big\rangle \,\Big)  $$
   \indent   Let  $ \, de_1 , de_2 \in {({\mathcal D}_h)}^* \, $  be such that
 $ \; \big\langle\, de_1 \, , \, \partial_1^{\,a} \, \partial_2^{\,b} \,\big\rangle \, = \, \delta_{1,a} \, \delta_{0,b} \, $,
 $ \; \big\langle\, de_2 \, , \, \partial_1^{\,a} \, \partial_2^{\,b} \,\big\rangle  \; = \;  \delta_{0,a} \, \delta_{1,b} \; $.
Similarly, let  $ \, e_1 , e_2 \in {({\mathcal D}_h)}^* \, $  be such that
 $ \; \big\langle\, e_1 \, , \, \partial_1^{\,a} \, \partial_2^{\,b} \,\big\rangle \, = \, x_1 \, \delta_{0,a} \, \delta_{0,b}  \, $,
 $ \; \big\langle\, e_2 \, , \, \partial_1^{\,a} \, \partial_2^{\,b} \,\big\rangle \, = \, x_2 \, \delta_{0,a}\delta_{0,b} \; $.
Now set  $ \, \check{de_i} := h^{-1} de_i \, $  for  $ \, i = 1, 2 \, $.  Then the following equalities can be established:
  $$  \displaylines{
   e_1 \cdot e_2 - e_2 \cdot_h e_1  \; = \;  -h \, e_1  \quad ,
 \qquad  \check{de}_1 \cdot_h \check{de}_2 - \check{de}_{2} \cdot_h \check{de}_1  \; = \;  \check{de}_1  \cr
   \check{de_1} \cdot_h e_2 - e_2 \cdot \check{de_1}  \; = \;  e_1  \quad ,
 \qquad  \check{de_1} \cdot_h e_1  \; = \;  e_1 \cdot_h \check{de_1}  \cr
   \check{de_2} \cdot_h e_2  \; = \;  e_2 \cdot_h \check{de_2}  \quad ,
 \qquad  \check{de_2} \cdot_h e_1 - e_1 \cdot_h \check{de_2}  \; = \;  -e_1  }  $$
Moreover, source and target are
 $ \; s_r^*(x_i) \, = \, e_i + h \, \check{de_i} \, $,  $ \; t_r^*(x_i) \, = \, e_i \; $.
From the properties of the coproduct, one has also
 $ \; \Delta(e_i) \, = \, e_i \otimes 1 \, $,  $ \; \Delta(\check{de_i}) \, = \, \check{de_i} \otimes 1 + 1 \otimes \check{de_i} \; $.
Finally, the counit of  $ {\left( ({\mathcal D}_h) ^*\right)}^\vee \, $  is given by the formulas
 $ \; \partial\big({\check{de_i}}\big) \, = \, 0 \, $,  $ \; \partial\big(e_i\big) \, = \, e_i \; $.

\bigskip

   {\bf A right bialgebroid isomorphism  $ \; {\big( {({\mathcal D}_h)}_* \big)}^\vee \cong {\big( {({\mathcal D}_h)}^* \big)}^\vee \; $.}  From the above analysis, one sees that there exists a unique isomorphism of right bialgebroids  $ \; \phi : {\big( {({\mathcal D}_h)}_* \big)}^{\!\vee} \! \relbar\joinrel\longrightarrow {\big( {({\mathcal D}_h)}^* \big)}^\vee \; $  determined by  $ \,\; \phi (e_i) \, = \, e_i \, + \, h \, \check{de_i} \;\, $  and  $ \,\; \phi(\check{de_i})  \, = \, -\check{de_i} \;\, $.

\end{document}